\documentclass{article}

    \usepackage[final,nonatbib]{neurips_2020}

\usepackage[utf8]{inputenc} 
\usepackage[T1]{fontenc}    
\usepackage{hyperref}       
\usepackage{url}            
\usepackage{booktabs}       
\usepackage{amsfonts}       
\usepackage{nicefrac}       
\usepackage{microtype}      

\usepackage{amssymb, amsmath, amsthm, latexsym}
\usepackage{fullpage, color}
\usepackage{url}
\usepackage{algorithm}
\usepackage{algpseudocode}
\usepackage{tabularx}
\usepackage{paralist}
\usepackage{mathtools}
\usepackage{tcolorbox}
\usepackage{xcolor}

\usepackage{bbm} 

\usepackage{makecell}
\usepackage{multirow}
\usepackage{booktabs}

\usepackage{nicefrac}       

\usepackage[numbers]{natbib}
\usepackage{sidecap}

\usepackage{pifont}
\usepackage{subfigure}

\newcommand{\R}{\mathbb{R}}
\newcommand{\eqdef}{\stackrel{\text{def}}{=}}

\def\<#1,#2>{\left\langle #1,#2\right\rangle}

\usepackage{mdframed} 
\usepackage{thmtools}

\definecolor{shadecolor}{gray}{0.9}
\declaretheoremstyle[
headfont=\normalfont\bfseries,
notefont=\mdseries, notebraces={(}{)},
bodyfont=\normalfont,
postheadspace=0.5em,
spaceabove=1pt,
mdframed={
  skipabove=8pt,
  skipbelow=8pt,
  hidealllines=true,
  backgroundcolor={shadecolor},
  innerleftmargin=4pt,
  innerrightmargin=4pt}
]{shaded}

\declaretheorem[style=shaded,within=section]{definition}
\declaretheorem[style=shaded,sibling=definition]{theorem}

\declaretheorem[style=shaded,sibling=definition]{corollary}

\declaretheorem[style=shaded,sibling=definition]{lemma}

\newcommand{\circledOne}{\text{\ding{172}}}
\newcommand{\circledTwo}{\text{\ding{173}}}
\newcommand{\circledThree}{\text{\ding{174}}}
\newcommand{\circledFour}{\text{\ding{175}}}
\newcommand{\circledFive}{\text{\ding{176}}}

\usepackage[colorinlistoftodos,bordercolor=orange,backgroundcolor=orange!20,linecolor=orange,textsize=scriptsize]{todonotes}

\newcommand{\cF}{{\cal F}}

\newcommand{\cX}{{\cal X}}

\newcommand{\EE}{\mathbf{E}}

\def\R{\mathbb{R}}

\def\R{\mathbb R}

\def\EE{\mathbb E}
\def\PP{\mathbb P}

\def\la{\langle}
\def\ra{\rangle}

\def\tnabla{\widetilde{\nabla}}

\def\Bxi{\boldsymbol{\xi}}

\def\Bxi{\boldsymbol{\xi}}

\def\clip{\text{clip}}

\def\obf{\mathbbm{1}}

\usepackage{hyperref}

\usepackage{accents}
\newlength{\dhatheight}

\title{Stochastic Optimization with Heavy-Tailed Noise via \\ Accelerated Gradient Clipping}

\author{%
  Eduard Gorbunov\thanks{\texttt{eduard.gorbunov@phystech.edu}, \url{eduardgorbunov.github.io}} \\
  MIPT and HSE, Russia\\
  \And 
  Marina Danilova\thanks{\texttt{danilovamarina15@gmail.com}, \url{marinadanya.github.io}}\\
  ICS RAS and MIPT, Russia\\
  \And 
  Alexander Gasnikov\thanks{\texttt{gasnikov@yandex.ru}}\\
  MIPT and HSE, Russia\\
}

\begin{document}

\maketitle

\begin{abstract}
  In this paper, we propose a new accelerated stochastic first-order method called {\tt clipped-SSTM} for smooth convex stochastic optimization with heavy-tailed distributed noise in stochastic gradients and derive the first high-probability complexity bounds for this method closing the gap in the theory of stochastic optimization with heavy-tailed noise. Our method is based on a special variant of accelerated Stochastic Gradient Descent ({\tt SGD}) and clipping of stochastic gradients. We extend our method to the strongly convex case and prove new complexity bounds that outperform state-of-the-art results in this case. Finally, we extend our proof technique and derive the first non-trivial high-probability complexity bounds for {\tt SGD} with clipping without light-tails assumption on the noise. 
\end{abstract}

\section{Introduction}\label{sec:intro}
In this paper we focus on the following problem
\begin{equation}
    \min\limits_{x\in \R^n}f(x),\quad f(x) = \EE_\xi\left[f(x,\xi)\right],\label{eq:main_problem}
\end{equation}
where $f(x)$ is a smooth convex function and the mathematical expectation in \eqref{eq:main_problem} is taken with respect to the random variable $\xi$ defined on the probability space $(\cX, \cF, \PP)$ with some $\sigma$-algebra $\cF$ and probability measure $\PP$. Such problems appear in various applications of machine learning \cite{Goodfellow-et-al-2016,shalev2014understanding,shapiro2014lectures} and mathematical statistics \cite{spokoiny2012parametric}. Perhaps, the most popular method to solve problems like \eqref{eq:main_problem} is Stochastic Gradient Descent ({\tt SGD}) \cite{hardt2015train, nemirovski1978cesari, nemirovsky1983problem, robbins1951stochastic, shalev2011pegasos}. There is a lot of literature on the convergence in expectation of {\tt SGD} for (strongly) convex \cite{ghadimi2013stochastic, gorbunov2019unified, gower2019sgd, moulines2011non, needell2016stochastic, nemirovski2009robust, nguyen2018sgd} and non-convex \cite{davis2019stochastic, ghadimi2013stochastic, khaled2020better} problems under different assumptions on stochastic gradient. When the problem is good enough, i.e.\ when the distributions of stochastic gradients are \textit{light-tailed}, this theory correlates well with the real behavior of trajectories of {\tt SGD} in practice. Moreover, the existing \textit{high-probability} bounds for {\tt SGD} \cite{devolder2011stochastic,dvurechensky2016stochastic, nemirovski2009robust} coincide with its counterpart from the theory of convergence in expectation up to logarithmical factors depending on the confidence level.

However, there are a lot of important applications where the noise distribution in the stochastic gradient is significantly \textit{heavy-tailed} \cite{simsekli2019tail,zhang2019adam}. For such problems {\tt SGD} is often less robust and shows poor performance in practice. Furthermore, existing results for the convergence with high-probability for {\tt SGD} are also much worse in the presence of heavy-tailed noise than its ``light-tailed counterparts''. In this case, rates of the convergence in expectation can be insufficient to describe the behavior of the method. 

To illustrate this phenomenon we consider a simple example of stochastic optimization problem and apply {\tt SGD} with constant stepsize to solve it. After that, we present a natural and simple way to resolve the issue of {\tt SGD} based on the \textit{clipping} of stochastic gradients. However, we need to introduce some important notations and definitions before we start to discuss this example.

\subsection{Preliminaries}\label{sec:notation}
In this section we introduce the main part of notations, assumption and definitions. The rest is classical for optimization literature and stated in the appendix (see Section~\ref{sec:appendix_notation}). Throughout the paper we assume that at each point $x\in \R^n$ function $f$ is accessible only via stochastic gradients $\nabla f(x,\xi)$ such that 
\begin{eqnarray}
    \EE_\xi[\nabla f(x,\xi)] &=& \nabla f(x),\quad
    \EE_\xi\left[\left\|\nabla f(x,\xi) - \nabla f(x)\right\|_2^2\right] \le \sigma^2,\label{eq:bounded_variance_clipped_SSTM}
\end{eqnarray}
i.e.\ we have an access to the unbiased estimator of $\nabla f(x)$ with uniformly bounded by $\sigma^2$ variance where $\sigma$ is some non-negative number. These assumptions on the stochastic gradient are standard in the stochastic optimization literature \cite{ghadimi2012optimal, ghadimi2013stochastic, juditsky2011first, lan2012optimal, nemirovski2009robust}. Below we introduce one of the most important definitions in this paper.
\begin{definition}[light-tailed random vector]\label{def:light_tailed_distrib}
    We say that random vector $\eta$ has a light-tailed distribution, i.e.\ satisfies ``light-tails'' assumption, if there exist $\EE[\eta]$ and $\PP\left\{\left\|\eta - \EE[\eta]\right\|_2 > b\right\} \le 2\exp\left(-\frac{b^2}{2\sigma^2}\right)$ for all $b > 0$
\end{definition}
Such distributions are often called sub-Gaussian ones (see \cite{jin2019short} and references therein). One can show (see Lemma~2 from \cite{jin2019short}) that this definition is equivalent to
\begin{equation}
    \EE\left[\exp\left(\nicefrac{\left\|\eta - \EE[\eta]\right\|_2^2}{\sigma^2}\right)\right] \le \exp(1) \label{eq:light_tails_def_equivalent}
\end{equation}
up to absolute constant difference in $\sigma$. Due to Jensen's inequality and convexity of $\exp(\cdot)$ one can easily show that inequality \eqref{eq:light_tails_def_equivalent} implies $\EE[\|\eta - \EE[\eta]\|_2^2] \le \sigma^2$. However, the reverse implication does not hold in general. Therefore, in the rest of the paper by stochastic gradient with heavy-tailed distribution, we mean such a stochastic gradient that satisfies \eqref{eq:bounded_variance_clipped_SSTM} but not necessarily \eqref{eq:light_tails_def_equivalent}.

\subsection{Simple Motivational Example: Convergence in Expectation and Clipping}\label{sec:motivation}
In this section we consider {\tt SGD} $x^{k+1} = x^k - \gamma \nabla f(x^k,\xi^k)$ applied to solve the problem \eqref{eq:main_problem} with $f(x,\xi) = \nicefrac{\|x\|_2^2}{2} + \la \xi, x\ra$, where $\xi$ is a random vector with zero mean and the variance by $\sigma^2$ (see the details in Section~\ref{sec:toy_details}). The state-of-the-art theory (e.g.\ \cite{gorbunov2019unified, gower2019sgd}) says that convergence properties in expectation of {\tt SGD} in this case depend only on the stepsize $\gamma$, condition number of $f$, initial suboptimality $f(x^0) - f(x^*)$ and the variance $\sigma$, but does not depend on distribution of $\xi$. However, the trajectory of {\tt SGD} significantly depends on the distribution of $\xi$. To illustrate this we consider $3$ different distributions of $\xi$ with the same $\sigma$, i.e., Gaussian distribution, Weibull distribution \cite{weibull1951statistical} and Burr Type XII distribution \cite{burr1942cumulative, mclaughlin2001compendium} with proper shifts and scales to get needed mean and variance for $\xi$ (see the details in Section~\ref{sec:toy_details}). For each distribution, we run {\tt SGD} several times from the same starting point, the same stepsize $\gamma$, and the same batchsize, see typical runs in Figure~\ref{fig:sgd_toy_example}.
\begin{figure}[h]
    \centering
    \includegraphics[width=0.32\textwidth]{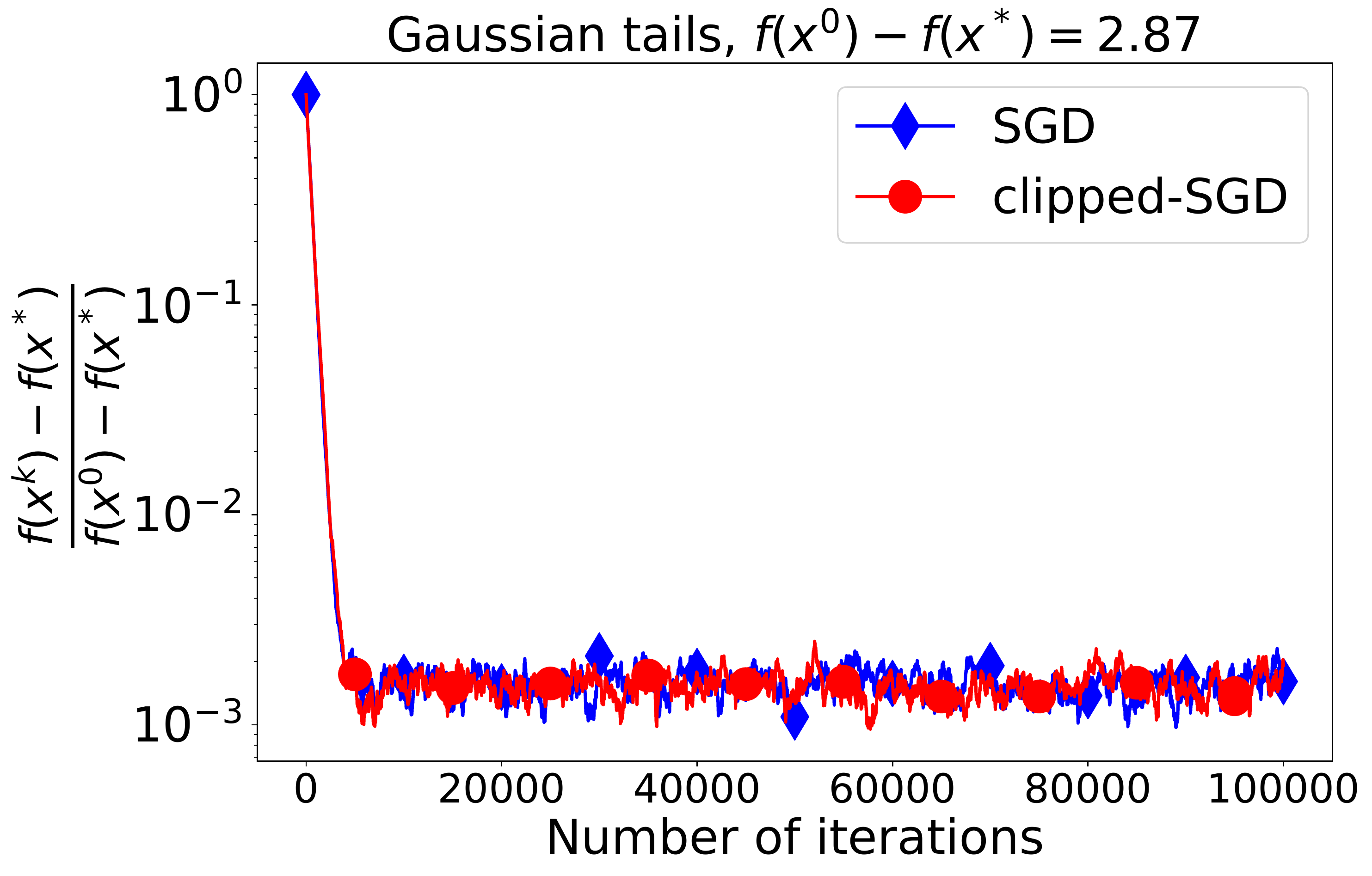}
    \includegraphics[width=0.32\textwidth]{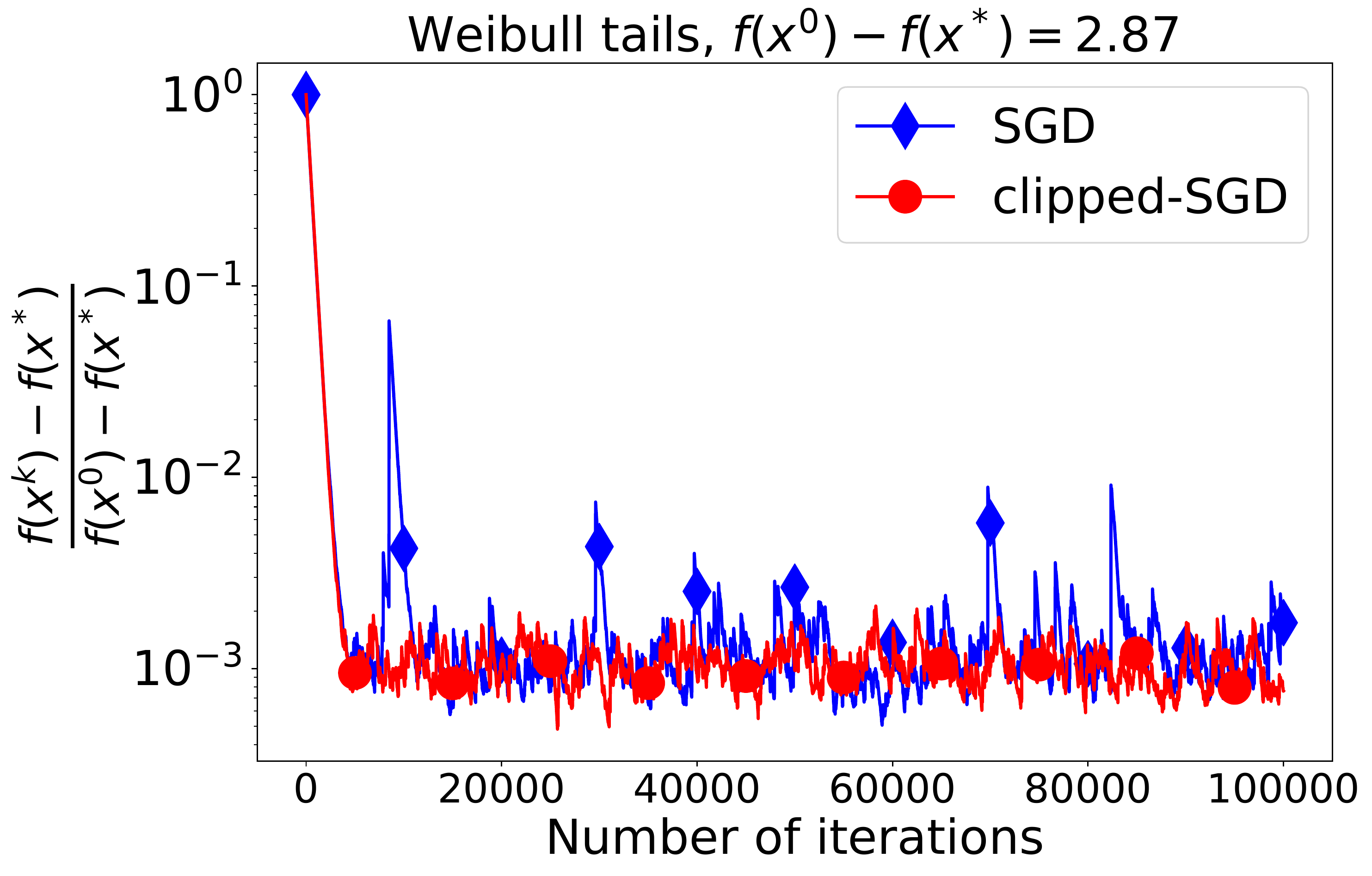}
    \includegraphics[width=0.32\textwidth]{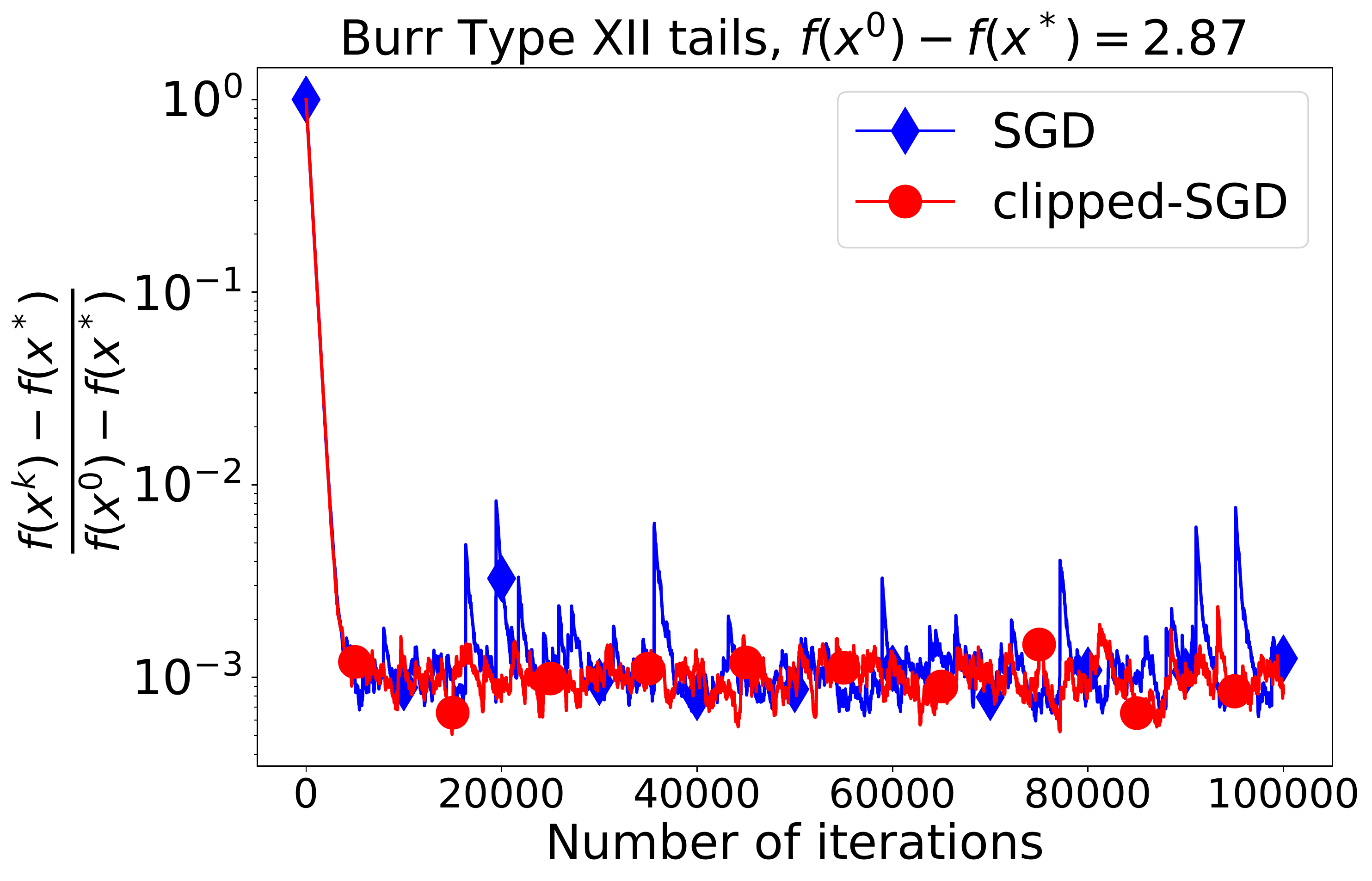}
    \caption{\scriptsize Typical trajectories of {\tt SGD} and {\tt clipped-SGD} applied to solve \eqref{eq:toy_problem} with $\xi$ having Gaussian, Weibull, and Burr Type XII tails.}
    \label{fig:sgd_toy_example}
\end{figure}
This simple example shows that {\tt SGD} in all $3$ cases rapidly reaches a neighborhood of the solution and then starts to oscillate there. However, these oscillations are significantly larger for the second and the third cases where stochastic gradients are heavy-tailed. Unfortunately, guarantees for the convergence in expectation cannot express this phenomenon, since in expectation the convergence guarantees for all $3$ cases are identical.

Moreover, in practice, e.g., in training big machine learning models, it is often used only a couple runs of {\tt SGD} or another stochastic method. The training process can take hours or even days, so, it is extremely important to obtain good accuracy of the solution \textit{with high probability}. However, as our simple example shows, {\tt SGD} fails to converge robustly if the noise in stochastic gradients is heavy-tailed which was also noticed for several real-world problems like training AlexNet \cite{krizhevsky2012imagenet} on CIFAR10 \cite{krizhevsky2009cifar} (see \cite{simsekli2019tail}) and training an attention model \cite{vaswani2017attention} via BERT \cite{devlin2018bert} (see \cite{zhang2019adam}). 


Clearly, since the distributions of stochastic gradients in the second and the third cases are heavy tailed the probability of sampling too large $\xi$ (in terms of the norm) and, as a consequence, too large $\nabla f(x,\xi)$ is high even if we are close to the solution. Once the current point $x^k$ is not too far from the solution and {\tt SGD} gets a stochastic gradient with too large norm the method jumps far from the solution. Therefore, we see large oscillations. Since the reason of such oscillations is large norm of stochastic gradient it is natural to \textit{clip} it, i.e., update $x^{k+1}$ according to $x^{k+1} = x^k - \gamma \min\{1,\nicefrac{\lambda}{\|\nabla f(x^k,\xi^k)\|_2}\}\nabla f(x^k,\xi^k).$ The obtained method is known in literature as {\tt clipped-SGD} (see \cite{gehring2017convolutional, Goodfellow-et-al-2016, menon2020can, merity2017regularizing, peters2018deep, zhang2019gradient, zhang2019adam} and references therein). Among the good properties of {\tt clipped-SGD} we emphasize its robustness to the heavy-tailed noise in stochastic gradients (see also \cite{zhang2019adam}). In our tests, trajectories of {\tt clipped-SGD} oscillate not significantly even for heavy-tailed distributions, and clipping does not spoil the rate of convergence. These two factors make {\tt clipped-SGD} preferable than {\tt SGD} when we deal with heavy-tailed distributed stochastic gradients (see further discussion in Section~\ref{sec:grad_clipping_extra}).



\subsection{Related Work}\label{sec:related_work}
\subsubsection{Smooth Stochastic Optimization: Light-Tailed Noise}\label{sec:light_tails}
In the light-tailed case high-probability complexity bounds and complexity bounds in expectation for {\tt SGD} and {\tt AC-SA} differ only in logarithmical factors of $\nicefrac{1}{\beta}$, see the details in Table~\ref{tab:cvx_case_comparison}. Such bounds were obtained in \cite{devolder2011stochastic} for {\tt SGD} in the convex case and then were extended to the $\mu$-strongly convex case in \cite{dvurechensky2016stochastic} for modification of {\tt SGD} called Stochastic Intermediate Gradient Method ({\tt SIGM}). Finally, optimal complexities were derived in \cite{ghadimi2012optimal,ghadimi2013optimal,lan2012optimal} for the method called {\tt AC-SA} in the convex case and for Multi-Staged {\tt AC-SA} ({\tt MS-AC-SA}) in the strongly convex case.

\subsubsection{Smooth Stochastic Optimization: Heavy-Tailed Noise}\label{sec:rel_work_heavy_tails}
Without light tails assumption the most straightforward results lead to $O(\nicefrac{1}{\beta^2})$ and $O(\nicefrac{1}{\beta})$ dependency on $\beta$ in the complexity bounds. Such bounds can be obtained from the complexity bounds for the convergence in expectation via Markov's inequality. However, for small $\beta$ these bounds become unacceptably poor. Classical results \cite{elisseeff2002stability,nesterov2008confidence,shalev2009stochastic} reduce these dependence to $O(\ln(\beta^{-1}))$ but they have worse dependence on $\varepsilon$ than corresponding results relying on light tails assumption.

For a long time the following question was open: \textit{is it possible to design stochastic methods having the same or comparable complexity bounds as in the light-tailed case but without light tails assumption on stochastic gradients?} In \cite{nazin2019algorithms} and \cite{davis2019low} the authors give a positive answer to this question \textit{but only partially}. Let us discuss the results from these papers in detail.

In \cite{nazin2019algorithms} Nazin et al.\ develop a new algorithm called Robust Stochastic Mirror Descent ({\tt RSMD}) which is based on a special truncation of stochastic gradients and derive complexity guarantees similar to {\tt SGD} in the convex case but without light assumption, see Table~\ref{tab:cvx_case_comparison}. This technique is very similar to gradient clipping. Moreover, in \cite{nazin2019algorithms} authors consider also composite problems with non-smooth composite term. However, in \cite{nazin2019algorithms} the optimization problem is defined on some \textit{compact} convex set $X$ with diameter $\Theta = \max\{\|x-y\|_2\mid x,y\in X\} < \infty$ and the analysis depends substantially on the boundedness of $X$. Using special restarts technique together with iterative squeezing of the set $X$ Nazin et al.\ extend their method to the $\mu$-strongly convex case, see Table~\ref{tab:str_cvx_case_comparison}.
Finally, in the discussion section of \cite{nazin2019algorithms} authors formulate the following question: \textit{is it possible to develop such \textbf{accelerated} stochastic methods that have the same or comparable complexity bounds as in the light-tailed case but do not require stochastic gradients to be light-tailed?}

In the strongly convex case the positive answer to this question was given by Davis et al.\ \cite{davis2019low} where authors propose a new method called {\tt proxBoost} that is based on robust distance estimation \cite{hsu2016loss,nemirovsky1983problem} and proximal point method \cite{martinet1970regularisation, martinet1972determination, rockafellar1976monotone}, see Table~\ref{tab:str_cvx_case_comparison}. However, this approach requires solving an auxiliary optimization problem at each iteration that can lead to poor performance in practice.

In our paper we close the gap in theory, i.e., we provide a positive answer to the following question:
\textit{Is it possible to develop such an accelerated stochastic method that have the same or comparable complexity bound as for {\tt AC-SA} in the convex case but do not require stochastic gradients to be light-tailed?}


\subsection{Our Contributions}\label{sec:contrib}
\begin{itemize}
    \item One of the main contributions of our paper is a new method called Clipped Stochastic Similar Triangles Method ({\tt clipped-SSTM}). For the case when the objective function $f$ is convex and $L$-smooth we derive the following complexity bound \textit{without light tails assumption on the stochastic gradients:} $O(\max\{\sqrt{\nicefrac{LR_0^2}{\varepsilon}}, \nicefrac{\sigma^2R_0^2}{\varepsilon^2}\}\ln(\nicefrac{LR_0^2}{\varepsilon\beta)}).$ This bound outperforms all known bounds for this setting (see Table~\ref{tab:cvx_case_comparison}) and up to the difference in logarithmical factors recovers the complexity bound of {\tt AC-SA} derived under light tails assumption. That is, in this paper we close the gap in theory theory of smooth convex stochastic optimization with heavy-tailed noise. Moreover, unlike in \cite{nazin2019algorithms}, we do not assume boundedness of the set where the optimization problem is defined, which makes our analysis more complicated. We also study different batchsize policies for {\tt clipped-SSTM}.
    \item Using restarts technique we extend {\tt clipped-SSTM} to the $\mu$-strongly convex objectives and obtain a new method called Restarted {\tt clipped-SSTM} ({\tt R-clipped-SSTM}). For this method we prove the following complexity bound (again, \textit{without light tails assumption on the stochastic gradients}): $O(\max\{\sqrt{\nicefrac{L}{\mu}}\ln(\nicefrac{\mu R^2}{\varepsilon}), \nicefrac{\sigma^2}{\mu\varepsilon}\}\ln(\nicefrac{L}{\mu\beta}\ln(\nicefrac{\mu R^2}{\varepsilon}))).$ Our bound outperforms the state-of-the-art result from \cite{davis2019low} in terms of the dependence on $\ln\frac{L}{\mu}$, see Table~\ref{tab:str_cvx_case_comparison} for the details.
    \item We prove the first high-probability complexity guarantees for {\tt clipped-SGD} in convex and strongly convex cases \textit{without light tails assumption on the stochastic gradients}, see Tables~\ref{tab:cvx_case_comparison}~and~\ref{tab:str_cvx_case_comparison}. The complexity we prove for {\tt clipped-SGD} in the convex case is comparable with corresponding bound for {\tt SGD} derived under light tails assumption. In the $\mu$-strongly convex case we derive a new complexity bound for the restarted version of {\tt clipped-SGD} ({\tt R-clipped-SGD}) which is comparable with its ``light-tailed counterpart''.
    \item We conduct several numerical experiments with the proposed methods in order to justify the theory we develop. In particular, we show that {\tt clipped-SSTM} can outperform {\tt SGD} and {\tt clipped-SGD} in practice even without using large batchsizes. Moreover, in our experiments we illustrate how clipping makes the convergence of {\tt SGD} and {\tt SSTM} more robust and reduces their oscillations.
\end{itemize}

\begin{table}[h]
    \centering
    \caption{Comparison of existing high-probability convergence results for stochastic optimization under assumptions \eqref{eq:bounded_variance_clipped_SSTM} for convex and $L$-smooth objectives. The second column contains an overall number of stochastic first-order oracle calls needed to achieve $\varepsilon$-solution with probability at least $1-\beta$. In the third column ``light'' means that $\nabla f(x,\xi)$ satisfies \eqref{eq:light_tails_def_equivalent} and ``heavy'' means that the result holds even in the case when \eqref{eq:light_tails_def_equivalent} does not hold. Column ``Domain'' describes the set where the optimization problem is defined. For {\tt RSMD} $\color{red}\Theta$ is a diameter of the set where the optimization problem is defined. We use red color to emphasize the restrictions we eliminate. 
    }
    \label{tab:cvx_case_comparison}
    \begin{tabular}{|c|c|c|c|}
        \hline
        Method & Complexity & Tails & Domain \\
        \hline
        \makecell{{\tt SGD} \cite{devolder2011stochastic}} & $O\left(\max\left\{\frac{L{R_0}^2}{\varepsilon}, \frac{\sigma^2{R_0}^2}{\varepsilon^2}\ln^2(\beta^{-1})\right\}\right)$ & {\color{red}light} & {\color{red} bounded} \\
        \hline
        \makecell{{\tt AC-SA} \cite{ghadimi2012optimal,lan2012optimal}} & $O\left(\max\left\{\sqrt{\frac{LR_0^2}{\varepsilon}}, \frac{\sigma^2R_0^2}{\varepsilon^2}\ln(\beta^{-1})\right\}\right)$ & {\color{red}light} & arbitrary\\
        \hline
        \makecell{{\tt RSMD} \cite{nazin2019algorithms}} & $O\left(\max\left\{\frac{L{\color{red}\Theta}^2}{\varepsilon}, \frac{\sigma^2{\color{red}\Theta}^2}{\varepsilon^2}\right\}\ln(\beta^{-1})\right)$ & heavy & {\color{red} bounded}\\
        \hline
        \makecell{{\tt clipped-SGD} {\scriptsize\color{blue}[This work]}} & $O\left(\max\left\{\frac{L{R_0}^2}{\varepsilon}, \frac{\sigma^2{R_0}^2}{\varepsilon^2}\right\}\ln(\beta^{-1})\right)$ & heavy & $\R^n$ \\
        \hline
        \makecell{{\tt clipped-SSTM} {\scriptsize\color{blue}[This work]}} & $O\left(\max\left\{\sqrt{\frac{LR_0^2}{\varepsilon}}, \frac{\sigma^2R_0^2}{\varepsilon^2}\right\}\ln\frac{LR_0^2+\sigma R_0}{\varepsilon\beta}\right)$ & heavy & $\R^n$ \\
        \hline
    \end{tabular}
\end{table}
\begin{table}[h]
    \centering
     \caption{Comparison of existing high-probability convergence results for stochastic optimization under assumptions \eqref{eq:bounded_variance_clipped_SSTM} for $\mu$-strongly convex and $L$-smooth objectives. The second column contains an overall number of stochastic first-order oracle calls needed to achieve $\varepsilon$-solution with probability at least $1-\beta$. In the third column ``light'' means that $\nabla f(x,\xi)$ satisfies \eqref{eq:light_tails_def_equivalent} and ``heavy'' means that the result holds even in the case when \eqref{eq:light_tails_def_equivalent} does not hold. Column ``Domain'' describes the set where the optimization problem is defined. For {\tt RSMD} $\color{red}\Theta$ is a diameter of the set where the optimization problem is defined and $R = \sqrt{\nicefrac{2(f(x^0) - f(x^*))}{\mu}}$, $r_0 = f(x^0) - f(x^*)$. We use red color to emphasize the restrictions we eliminate.
    }
    \label{tab:str_cvx_case_comparison}
    \begin{tabular}{|c|c|c|c|}
        \hline
        Method & Complexity & Tails & Domain \\
        \hline
        \makecell{{\tt SIGM} \cite{dvurechensky2016stochastic}} & $O\left(\max\left\{\frac{L}{\mu}\ln\frac{\mu R_0^2}{\varepsilon}, \frac{\sigma^2}{\mu\varepsilon}\ln\left(\beta^{-1}\ln\frac{\mu R_0^2}{\varepsilon}\right)\right\}\right)$ & {\color{red}light} & {arbitrary} \\
        \hline
        \makecell{{\tt MS-AC-SA} \cite{ghadimi2013optimal}} & $O\left(\max\left\{\sqrt{\frac{L}{\mu}}\ln\frac{LR_0^2}{\varepsilon}, \frac{\sigma^2}{\mu\varepsilon}\ln\left(\beta^{-1}\ln\frac{LR_0^2}{\varepsilon}\right)\right\}\right)$ & {\color{red}light} & arbitrary \\
        \hline
        \makecell{{\tt restarted-RSMD}\\ \cite{nazin2019algorithms}} & $ O\left(\max\left\{\frac{L}{\mu}\ln\left(\frac{\mu{\color{red}\Theta}^2}{\varepsilon}\right),\frac{\sigma^2}{\mu\varepsilon}\right\}\ln\left(\beta^{-1}\ln\frac{\mu{\color{red}\Theta}^2}{\varepsilon}\right)\right)$ & heavy & {\color{red} bounded} \\
        \hline
        \makecell{{\tt proxBoost} \cite{davis2019low}} & \makecell{$ O\left(\max\left\{\sqrt{\frac{L}{\mu}}\ln\left(\frac{LR_0^2{\color{red}\ln\frac{L}{\mu}}}{\varepsilon}\right), \frac{\sigma^2{\color{red}\ln\frac{L}{\mu}}}{\mu\varepsilon}\right\}\cdot C\right)$,\\where $C = {\color{red}\ln\left(\frac{L}{\mu}\right)}\ln\left(\frac{\ln\frac{L}{\mu}}{\beta}\right)$} & heavy & arbitrary\\
        \hline
        \makecell{{\tt clipped-SGD}\\ {\scriptsize\color{blue}[This work]}} & $ O\left(\max\left\{\frac{L}{\mu}, \frac{\sigma^2}{\mu\varepsilon}\cdot \frac{L}{\mu}\right\}\ln\left(\frac{r_0}{\varepsilon}\right)\ln\left(\frac{L}{\mu\beta}\ln\frac{r_0}{\varepsilon}\right)\right)$ & heavy & $\R^n$ \\
        \hline
        \makecell{{\tt R-clipped-SGD}\\ {\scriptsize\color{blue}[This work]}} & $ O\left(\max\left\{\frac{L}{\mu}\ln\frac{\mu R^2}{\varepsilon}, \frac{\sigma^2}{\mu\varepsilon}\right\}\ln\left(\frac{L}{\mu\beta}\ln\frac{\mu R^2}{\varepsilon}\right)\right)$ & heavy & $\R^n$\\
        \hline
        \makecell{{\tt R-clipped-SSTM}\\ {\scriptsize\color{blue}[This work]}} & $O\left(\max\left\{\sqrt{\frac{L}{\mu}}\ln\frac{\mu R^2}{\varepsilon}, \frac{\sigma^2}{\mu\varepsilon}\right\}\ln\left(\frac{L}{\mu\beta}\ln\frac{\mu R^2}{\varepsilon}\right)\right)$ & heavy & $\R^n$ \\
        \hline
    \end{tabular}
\end{table}

\subsubsection{Relation to \cite{zhang2019adam}}\label{sec:relation_to_zhangs_paper}
While Zhang et al. \cite{zhang2019adam} consider different setup, \cite{zhang2019adam} is highly relevant to our paper, and, in some sense, it complements our findings. In particular, it contains the analysis of several versions of {\tt clipped-SGD} establishing the rates of convergence \textit{in expectation} while we focus on the \textit{high-probability} complexity guarantees. Secondly, we consider convex and strongly convex cases while \cite{zhang2019adam} provides an analysis for non-convex and strongly convex problems. Finally, \cite{zhang2019adam} relies on the following assumption: there exist such $G> 0$ and $\alpha \in (1,2]$ that the stochastic gradient $g(x)$ satisfies $\mathbb{E}\|g(x)\|_2^\alpha \le G^\alpha$. This assumption implies the boundedness of the gradient of the objective function $f(x)$ which is quite restrictive and does not hold on the whole space for strongly convex functions. In our paper, we assume only boundedness of the variance. Moreover, we consider \textit{smooth} problems that allows us to accelerate {\tt clipped-SGD} and obtain {\tt clipped-SSTM}, while Zhang et al. \cite{zhang2019adam} provide non-accelerated rates.

\subsection{Paper Organization}\label{sec:paper_org}
The remaining part of the paper is organized as follows. In Section~\ref{sec:acc_method} we present {\tt clipped-SSTM} together with the main complexity result in the convex case that we prove for this method. Then, we present the first high-probability complexity bounds for {\tt clipped-SGD} for for the convex problems. In Section~\ref{sec:numerical_exp} we provide our numerical experiments justifying our theoretical results. Finally, in Section~\ref{sec:discussion} we provide some concluding remarks and discuss the limitations and possible extensions of the results developed in the paper. Due to the space limitations, we put the exact formulations of all theorems, results for the strongly convex problems and the full proofs in the Appendix (see Sections~\ref{sec:proofs_acc_method}~and~\ref{sec:missing_proofs_SGD}), together with auxiliary and technical results and additional experiments (see Section~\ref{sec:extra_exp}). Moreover, in Section~\ref{sec:sketch_clipped_SSTM} we present a sketch of the proof of the main convergence result for {\tt clipped-SSTM} and explain the intuition behind it.

\section{Accelerated {\tt SGD} with Clipping}\label{sec:acc_method}
In this section we consider the situation when $f(x)$ is convex and $L$-smooth on $\R^n$. For this problem we present a new method called Clipped Stochastic Similar Triangles Method ({\tt clipped-SSTM}, see Algorithm~\ref{alg:clipped-SSTM}). 
\begin{algorithm}[h]
\caption{Clipped Stochastic Similar Triangles Method ({\tt clipped-SSTM})}
\label{alg:clipped-SSTM}   
\begin{algorithmic}[1]
\Require starting point $x^0$, number of iterations $N$, batchsizes $\{m_k\}_{k=1}^N $, stepsize parameter $a$, clipping parameter $B$
\State Set $A_0 = \alpha_0 = 0$, $y^0 = z^0 = x^0$
\For{$k=0,\ldots, N-1$}
\State Set $\alpha_{k+1} = \frac{k+2}{2aL}$, $A_{k+1} = A_k + \alpha_{k+1}$, $\lambda_{k+1} = \frac{B}{\alpha_{k+1}}$
\State $x^{k+1} = \nicefrac{(A_k y^k + \alpha_{k+1} z^k)}{A_{k+1}}$
\State Draw fresh i.i.d.\ samples $\xi_1^k,\ldots,\xi_{m_{k}}^k$ and compute $\nabla f(x^{k+1},\Bxi^k) = \frac{1}{m_k}\sum_{i=1}^{m_k}\nabla f(x^{k+1},\xi_i^k)$
\State Compute $\tnabla f(x^{k+1},\Bxi^k) = \clip(\nabla f(x^{k+1},\Bxi^k),\lambda_{k+1})$ using \eqref{eq:clip_operator} 
\State $z^{k+1} = z^k - \alpha_{k+1}\tnabla f(x^{k+1},\Bxi^k)$
\State $y^{k+1} = \nicefrac{(A_k y^k + \alpha_{k+1} z^{k+1})}{A_{k+1}}$
\EndFor
\Ensure $y^N$ 
\end{algorithmic}
\end{algorithm}
In our method we use a clipped stochastic gradient that is defined in the following way:
\begin{equation}
    \clip(\nabla f(x,\Bxi), \lambda) = \min\left\{1,\nicefrac{\lambda}{\|\nabla f(x,\Bxi)\|_2}\right\}\nabla f(x,\Bxi)
      \label{eq:clip_operator}
\end{equation}
where $\nabla f(x,\Bxi)=\frac{1}{m}\sum_{i=1}^m\nabla f(x,\xi_i)$ is a mini-batched version of $\nabla f(x)$. That is, in order to compute $\clip(\nabla f(x,\Bxi), \lambda)$ one needs to get $m$ i.i.d.\ samples $\nabla f(x,\xi_1),\ldots,\nabla f(x,\xi_m)$, compute its average and then project the result $\nabla f(x,\Bxi)$ on the Euclidean ball with radius $\lambda$ and center at the origin.
Next theorem summarizes the main convergence result for {\tt clipped-SSTM}.
\begin{theorem}\label{thm:main_result_clipped_SSTM_main}
    Assume that function $f$ is convex and $L$-smooth. Then for all $\beta \in (0,1)$ and $N\ge 1$ such that $\ln(\nicefrac{4N}{\beta}) \ge 2$ we have that after $N$ iterations of {\tt clipped-SSTM} with $m_k = \Theta\left(\max\left\{1,\nicefrac{\sigma^2 \alpha_{k+1}^2N\ln(\nicefrac{N}{\beta})}{R_0^2}\right\}\right)$, $B = \Theta(\nicefrac{R_0}{\ln(\nicefrac{N}{\beta})})$ and $a = \Theta(\ln^2(\nicefrac{N}{\beta}))$ that $f(y^N) - f(x^*) = O(\nicefrac{aLR_0^2}{N^2})$ holds with probability at least $1-\beta$ where $R_0 = \|x^0 - x^*\|_2$. In other words, if we choose $a$ to be equal to the maximum from \eqref{eq:B_a_parameters_clipped_SSTM}, then the method achieves $f(y^N) - f(x^*) \le \varepsilon$ with probability at least $1-\beta$ after $O(\sqrt{\nicefrac{LR_0^2}{\varepsilon}}\ln(\nicefrac{LR_0^2}{\varepsilon\beta}))$ iterations and requires $O(\max\{\sqrt{\nicefrac{LR_0^2}{\varepsilon}}, \nicefrac{\sigma^2R_0^2}{\varepsilon^2}\}\ln(\nicefrac{LR_0^2}{\varepsilon\beta}))$ oracle calls.
\end{theorem}
The theorem says that for any $\beta\in(0,1)$ {\tt clipped-SSTM} converges to $\varepsilon$-solution with probability at least $1-\beta$ and requires exactly the same number of stochastic first-order oracle calls (up to the difference in constant and logarithmical factors) as optimal stochastic methods like {\tt AC-SA} \cite{ghadimi2012optimal,lan2012optimal} or Stochastic Similar Triangles Method \cite{gasnikov2016universal, gorbunov2019optimal}. However, our method \textit{achieves this rate under less restrictive assumption}. Indeed, Theorem~\ref{thm:main_result_clipped_SSTM_main} holds even in the case when the stochastic gradient $\nabla f(x,\xi)$ satisfies only \eqref{eq:bounded_variance_clipped_SSTM} and can have \textit{heavy-tailed} distribution. In contrast, all existing results that establish \eqref{eq:clipped_SSTM_oracle_complexity} and that are known in the literature hold only in the light-tails case, see Section~\ref{sec:light_tails}.

Finally, when $\sigma^2$ is big then Theorem~\ref{thm:main_result_clipped_SSTM_main} says that at iteration $k$ {\tt clipped-SGD} requires large batchsizes $m_k \sim k^2 N$ (see \eqref{eq:bathces_clipped_SSTM}) which is proportional to $\varepsilon^{-\nicefrac{3}{2}}$ for last iterates. It can make the cost of one iteration extremely high, therefore, we also consider different stepsize policies that remove this drawback in Section~\ref{sec:convergence_cvx_case_clipped_SSTM}. In particular, the following result shows that {\tt clipped-SSTM} achieves the same oracle complexity even with constant batchsizes $m_k$ when stepsize parameter $a$ is chosen properly.
\begin{corollary}\label{cor:clipped_sstm_small_stepsize_const_batch_main}
    Let the assumptions of Theorem~\ref{thm:main_result_clipped_SSTM} hold and $a = \Theta\left(\max\{1,\ln^2(\nicefrac{N}{\beta}), \nicefrac{\sqrt{\ln\nicefrac{N}{\beta}}\sigma N^{\nicefrac{3}{2}}}{LR_0}\}\right)$. Then $m_k = O(1)$ and {\tt clipped-SSTM} achieves $f(y^N) - f(x^*) \le \varepsilon$ with probability at least $1-\beta$ after $O(\max\{\sqrt{\nicefrac{LR_0^2}{\varepsilon}}, \nicefrac{\sigma^2R_0^2}{\varepsilon^2}\}\ln(\nicefrac{(LR_0^2+\sigma R_0)}{\varepsilon\beta}))$ iterations/oracle calls.
\end{corollary}

\section{{\tt SGD} with Clipping}\label{sec:clipped_SGD}
In this section we present our complexity results for {\tt clipped-SGD} (see Algorithm~\ref{alg:clipped-SGD}) in the convex case.
\begin{algorithm}[h]
\caption{Clipped Stochastic Gradient Descent ({\tt clipped-SGD})}
\label{alg:clipped-SGD}   
\begin{algorithmic}[1]
\Require starting point $x^0$, number of iterations $N$, batchsizes $\{m_k\}_{k=0}^{N-1}$, stepsize $\gamma > 0$, clipping level $\lambda > 0$
\For{$k=0,\ldots, N-1$}
\State Draw fresh i.i.d.\ samples $\xi_1^k,\ldots,\xi_{m_{k}}^k$ and compute $\nabla f(x^{k},\Bxi^k) = \frac{1}{m_k}\sum_{i=1}^{m_k}\nabla f(x^k,\xi_i^k)$
\State Compute $\tnabla f(x^{k},\Bxi^k) = \clip(\nabla f(x^{k},\Bxi^k),\lambda)$ using \eqref{eq:clip_operator} 
\State $x^{k+1} = x^k - \gamma \tnabla f(x^{k},\Bxi^k)$
\EndFor
\Ensure $\Bar{x}^N = \frac{1}{N}\sum_{k=0}^{N-1}x^k$ 
\end{algorithmic}
\end{algorithm}
Next theorem summarizes the main convergence result for {\tt clipped-SGD} in this case.
\begin{theorem}\label{thm:main_result_clipped_SGD_main}
    Assume that function $f$ is convex and $L$-smooth. Then for all $\beta \in (0,1)$ and $N\ge 1$ such that $\ln(\nicefrac{4N}{\beta}) \ge 2 $ we have that after $N$ iterations of {\tt clipped-SGD} with $\lambda = \Theta(LR_0)$ and $m_k = m = \Theta(\max\{1,\nicefrac{N\sigma^2}{R_0^2L^2\ln(\nicefrac{N}{\beta})}\})$ where $R_0 = \|x^0 - x^*\|_2$ and stepsize $\gamma = \nicefrac{1}{80L\ln(\nicefrac{4N}{\beta})}$ that $f(\Bar{x}^N) - f(x^*) = O(\nicefrac{LR_0^2\ln(\nicefrac{4N}{\beta})}{N})$ with probability at least $1-\beta$ where $\Bar{x}^N = \frac{1}{N}\sum_{k=0}^{N-1}x^k$.
    In other words, the method achieves $f(\Bar{x}^N) - f(x^*) \le \varepsilon$ with probability at least $1-\beta$ after $O\left(\nicefrac{LR_0^2}{\varepsilon}\ln(\nicefrac{LR_0^2}{\varepsilon\beta})\right)$ iterations and requires $O(\max\{\nicefrac{LR_0^2}{\varepsilon}, \nicefrac{\sigma^2R_0^2}{\varepsilon^2}\}\ln(\nicefrac{LR_0^2}{\varepsilon\beta}))$ oracle calls.
\end{theorem}
To the best of our knowledge, it is the first result for {\tt clipped-SGD} establishing non-trivial complexity guarantees for the convergence with high probability. Up to the difference in logarithmical factors our bound recovers the complexity bound for {\tt SGD} which was obtained under light tails assumption and the complexity bound for {\tt RSMD}. However, unlike in \cite{nazin2019algorithms}, we do not assume that the optimization problem is defined on the bounded set. The proof technique is similar to one we use to prove Theorem~\ref{thm:main_result_clipped_SSTM}. One can find the full proof in Section~\ref{sec:proof_cvx_SGD}.

\section{Numerical Experiments}\label{sec:numerical_exp}
We have tested\footnote{One can find the code here: \url{https://github.com/eduardgorbunov/accelerated_clipping}.} {\tt clipped-SSTM} and {\tt clipped-SGD} on the logistic regression problem, the datasets were taken from LIBSVM library \cite{chang2011libsvm}. To implement methods we use Python 3.7 and standard libraries. One can find additional experiments and details in Section~\ref{sec:additional_logreg}.

First of all, using standard solvers from {\tt scipy} library we find good enough approximation of the solution of the problem for each dataset. For simplicity, we denote this approximation by $x^*$. Then, we numerically study the distribution of $\|\nabla f_i(x^*)\|_2$ and plot corresponding histograms for each dataset, see Figure~\ref{fig:noise_distrib}.
\begin{figure}[h]
    \centering
    \includegraphics[width=0.32\textwidth]{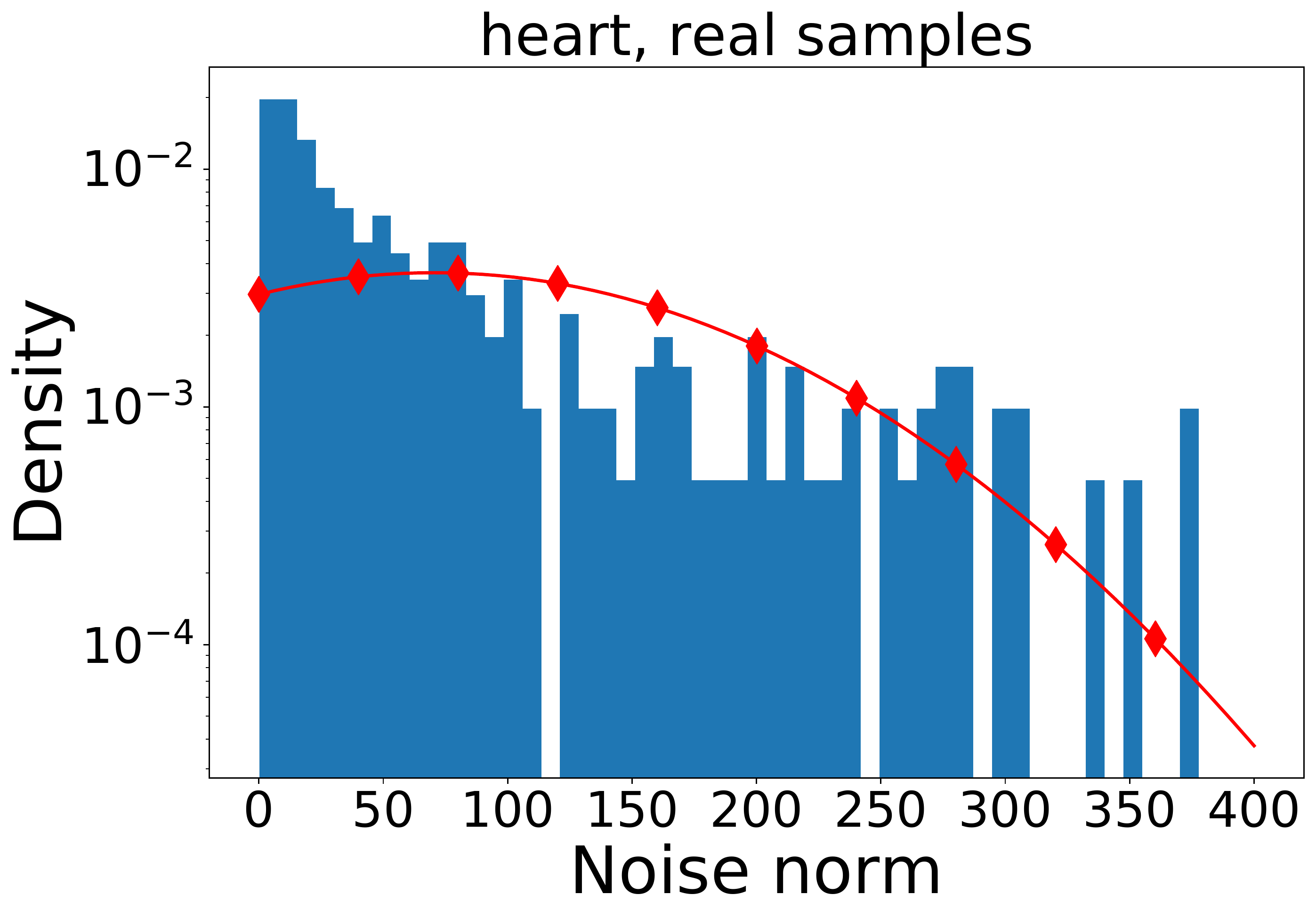}
    \includegraphics[width=0.32\textwidth]{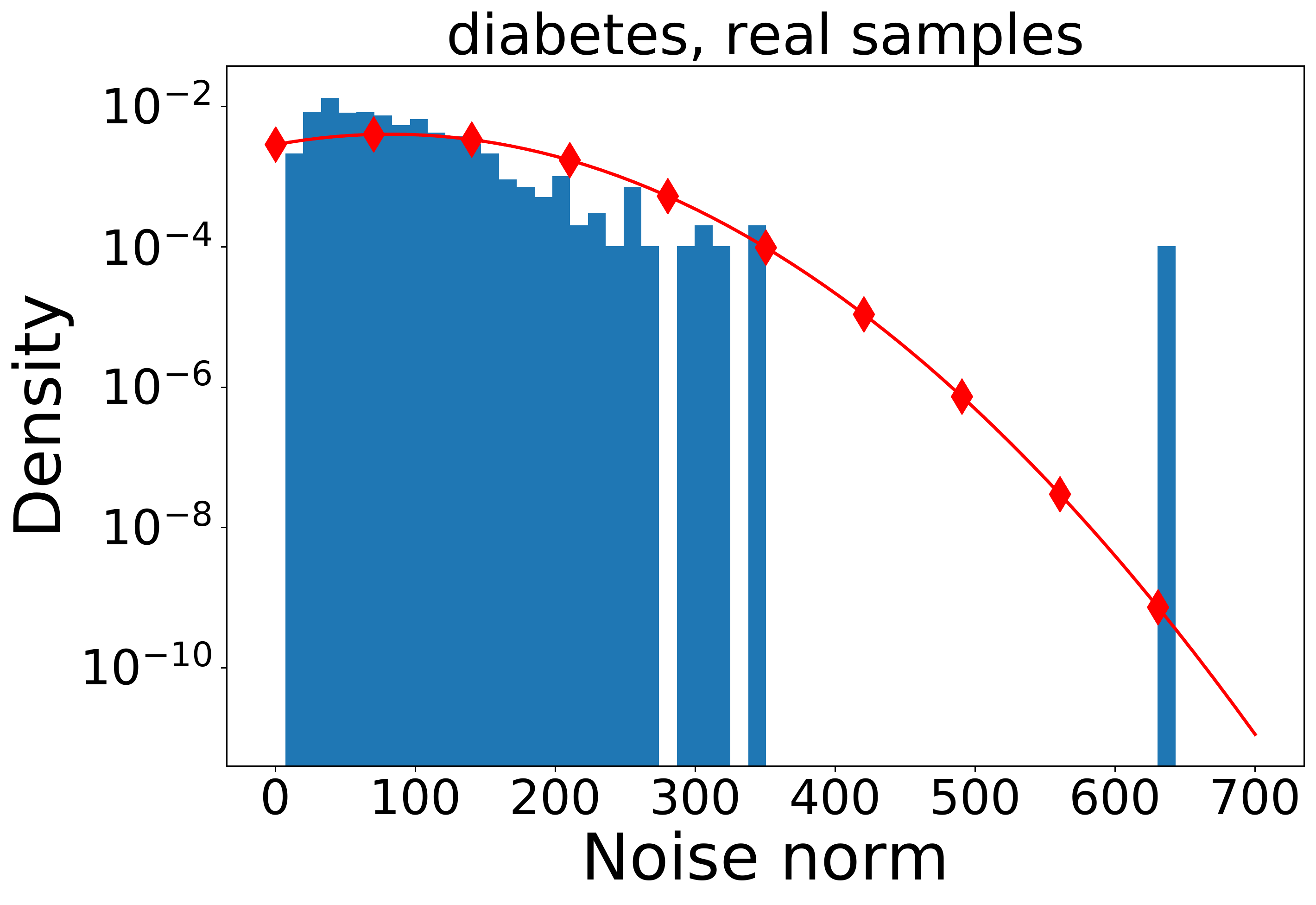}
    \includegraphics[width=0.32\textwidth]{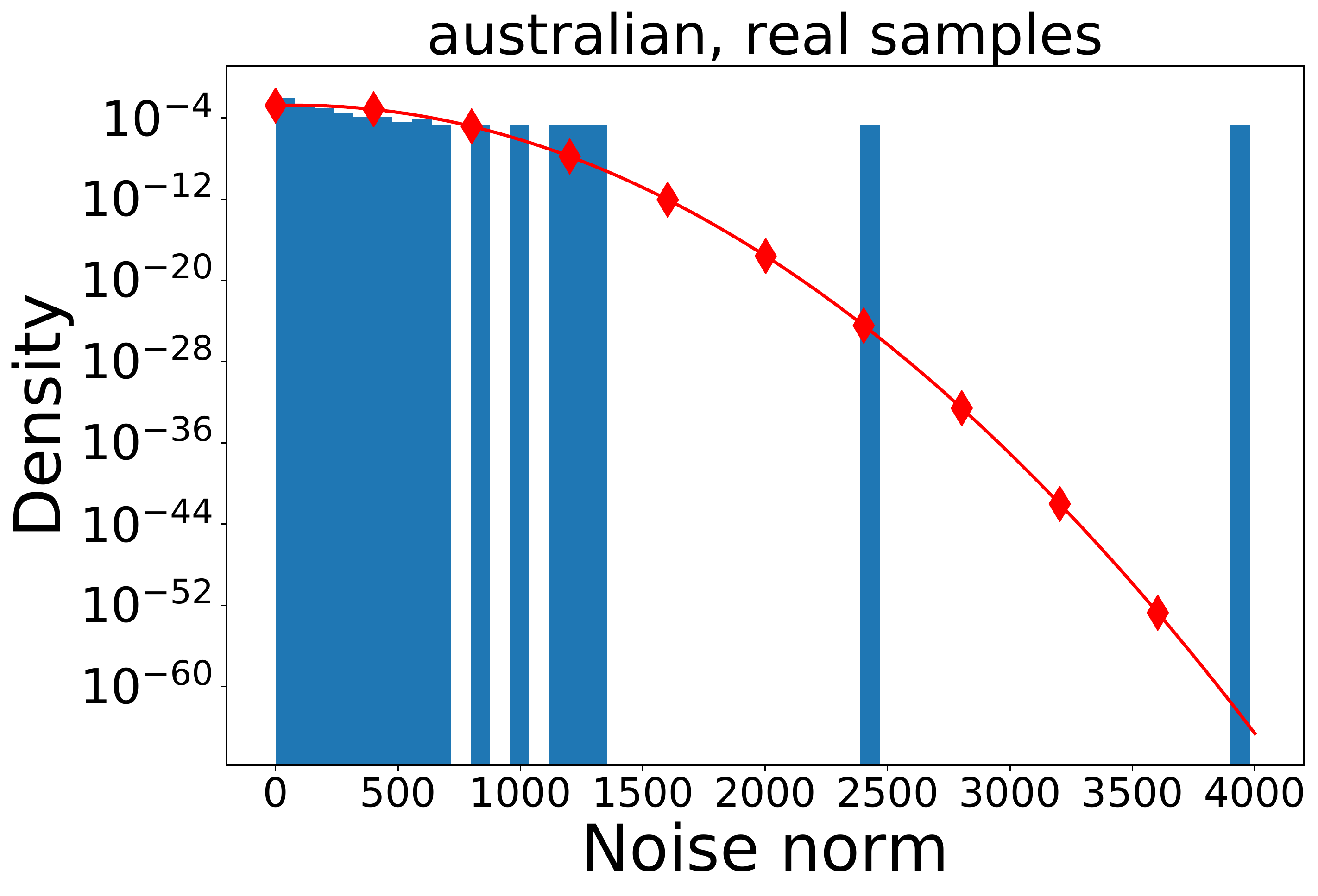}
    \caption{Histograms of $\|\nabla f_i(x^*)\|_2$ for different datasets. Red lines correspond to probability density functions of normal distributions with empirically estimated means and variances.}
    \label{fig:noise_distrib}
\end{figure}
These histograms hint that near the solution for {\tt heart} dataset tails of stochastic gradients are not heavy and the norm of the noise can be well-approximated by Gaussian distribution, whereas for {\tt diabetes} and {\tt australian} we see the presense of outliers that makes the distribution heavy-tailed.

Next, let us consider numerical results for {\tt SGD} and {\tt SSTM} with and without clipping applied to solve logistic regression problem on these $3$ datasets, see Figures~\ref{fig:heart_logreg}- \ref{fig:australian_logreg}. 
\begin{figure}[h]
    \centering
    \includegraphics[width=0.32\textwidth]{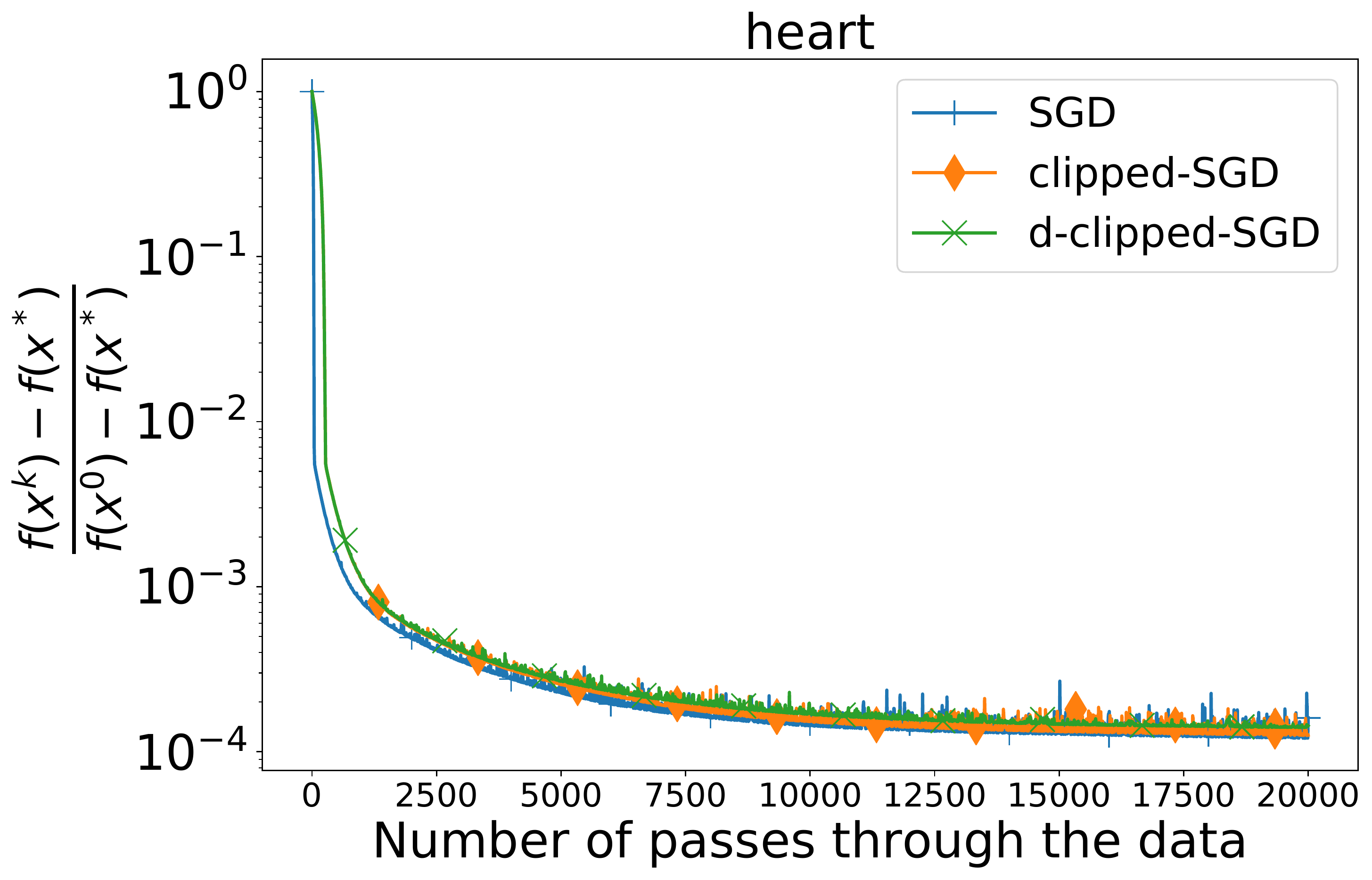}
    \includegraphics[width=0.32\textwidth]{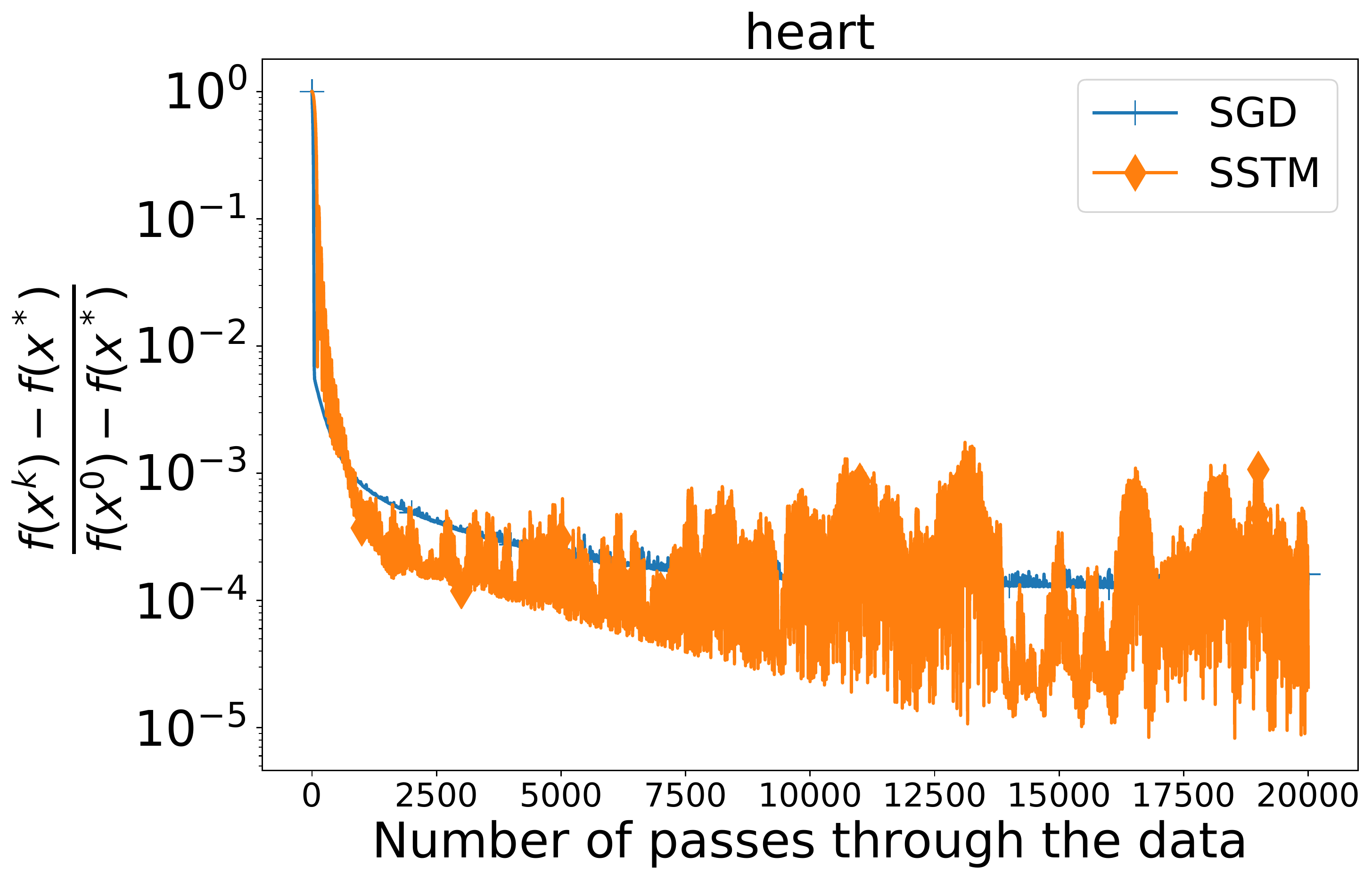}
    \includegraphics[width=0.32\textwidth]{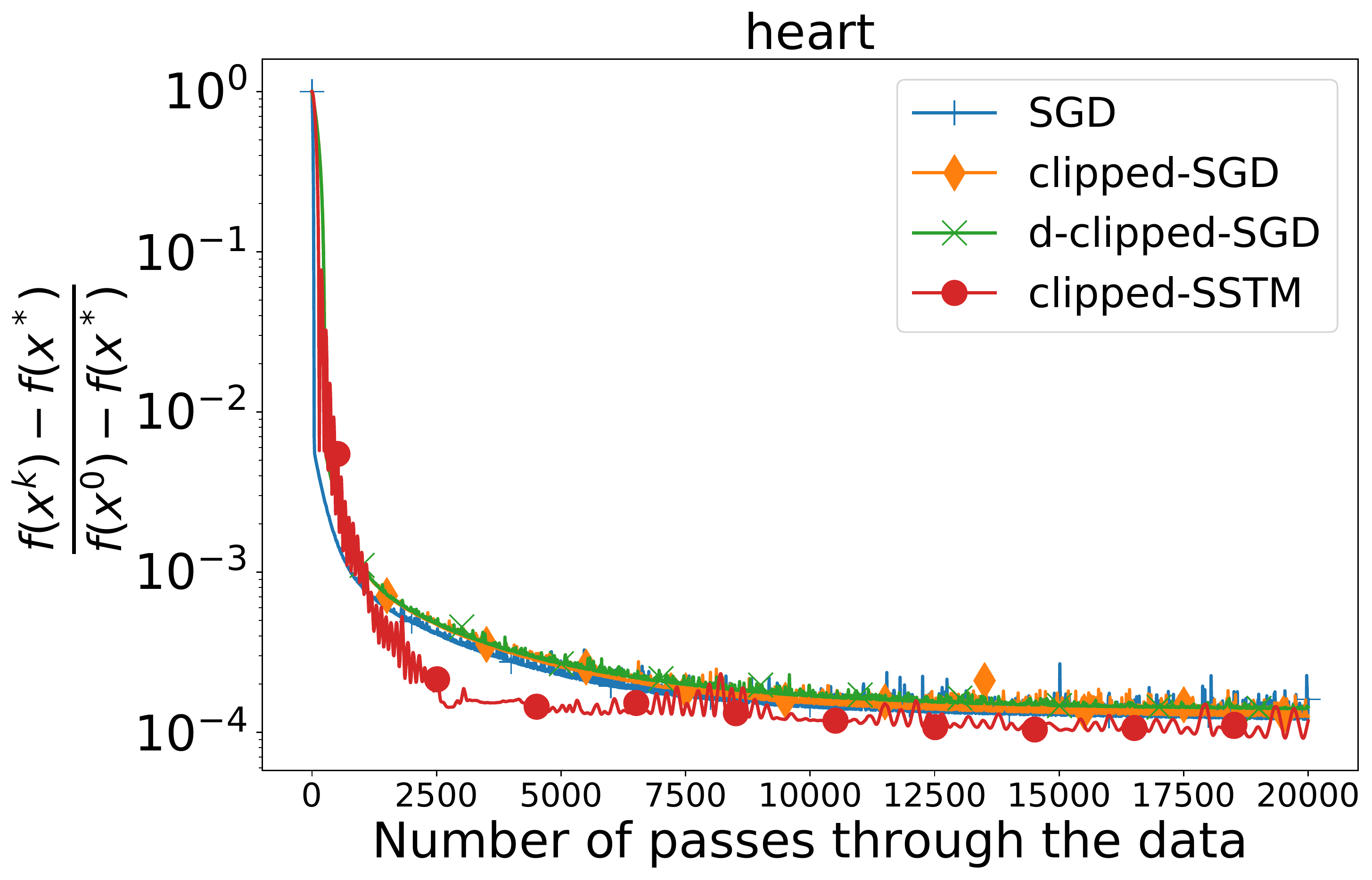}
    \caption{Trajectories of {\tt SGD}, {\tt clipped-SGD}, {\tt SSTM} and {\tt clipped-SSTM} applied to solve logistic regression problem on {\tt heart} dataset.}
    \label{fig:heart_logreg}
\end{figure}
For all methods we used constant batchsizes $m$, stepsizes and clipping levels were tuned, see Section~\ref{sec:additional_logreg} for the details. In our experiments we also consider {\tt clipped-SGD} with periodically decreasing clipping level $\lambda$ ({\tt d-clipped-SGD} in Figures), i.e.\ the method starts with some initial clipping level $\lambda_0$ and after every $l$ epochs or, equivalently, after every $\lceil\nicefrac{rl}{m}\rceil$ iterations the clipping level is multiplied by some constant $\alpha\in(0,1)$. 
\begin{figure}[h]
    \centering
    \includegraphics[width=0.32\textwidth]{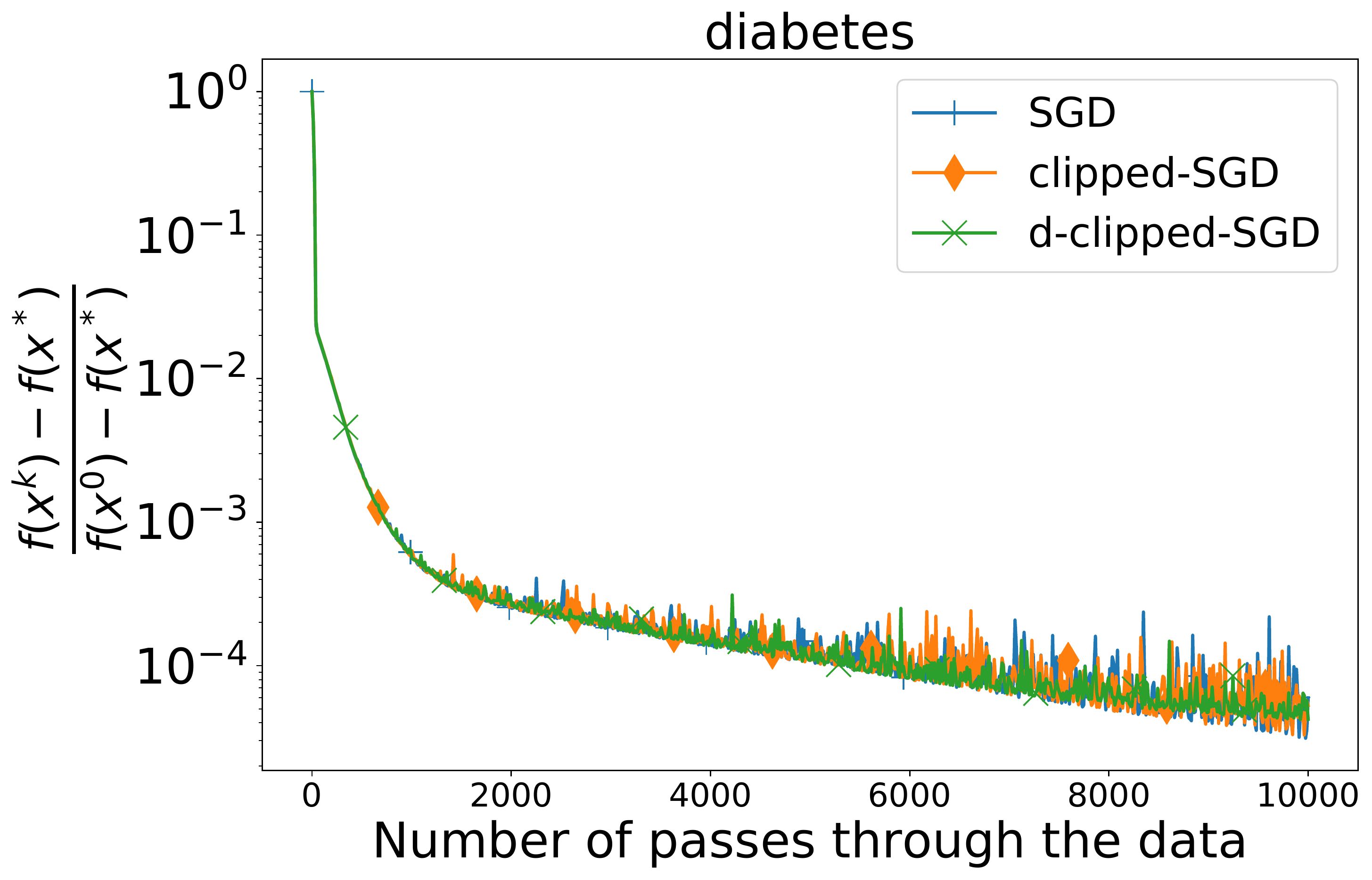}
    \includegraphics[width=0.32\textwidth]{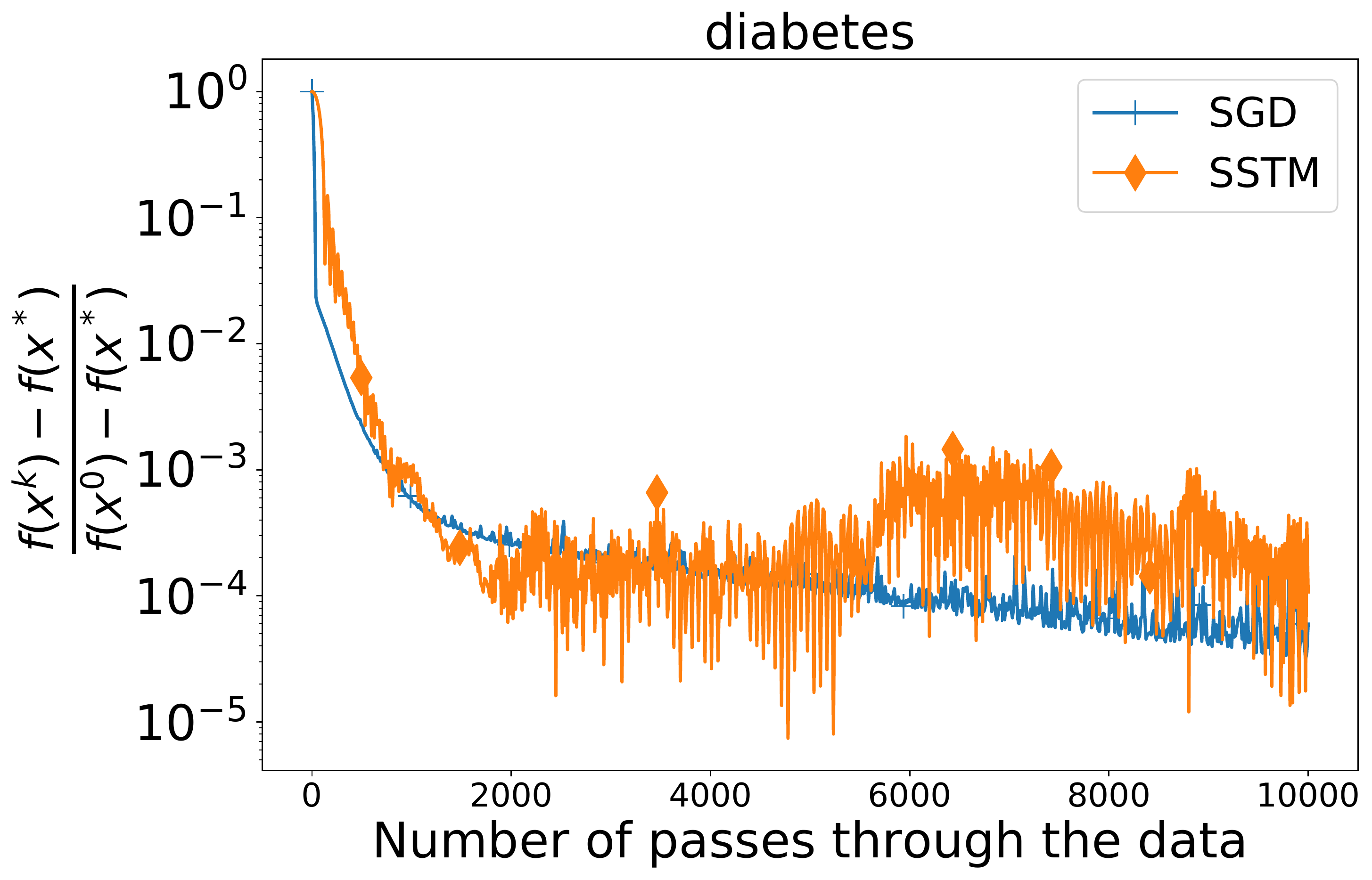}
    \includegraphics[width=0.32\textwidth]{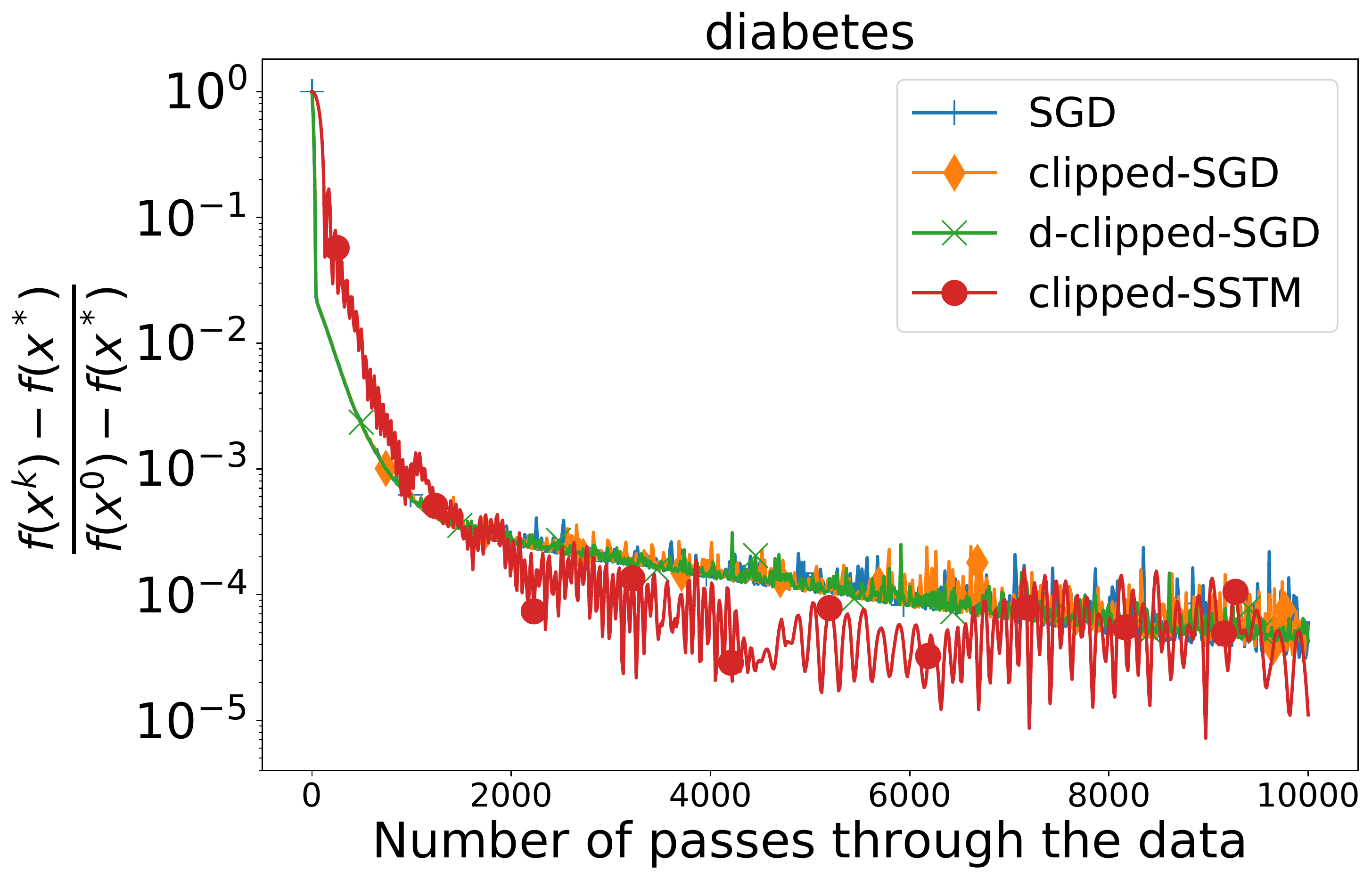}
    \caption{Trajectories of {\tt SGD}, {\tt clipped-SGD}, {\tt SSTM} and {\tt clipped-SSTM} applied to solve logistic regression problem on {\tt diabetes} dataset.}
    \label{fig:diabetes_logreg}
\end{figure}
\begin{figure}[h]
    \centering
    \includegraphics[width=0.32\textwidth]{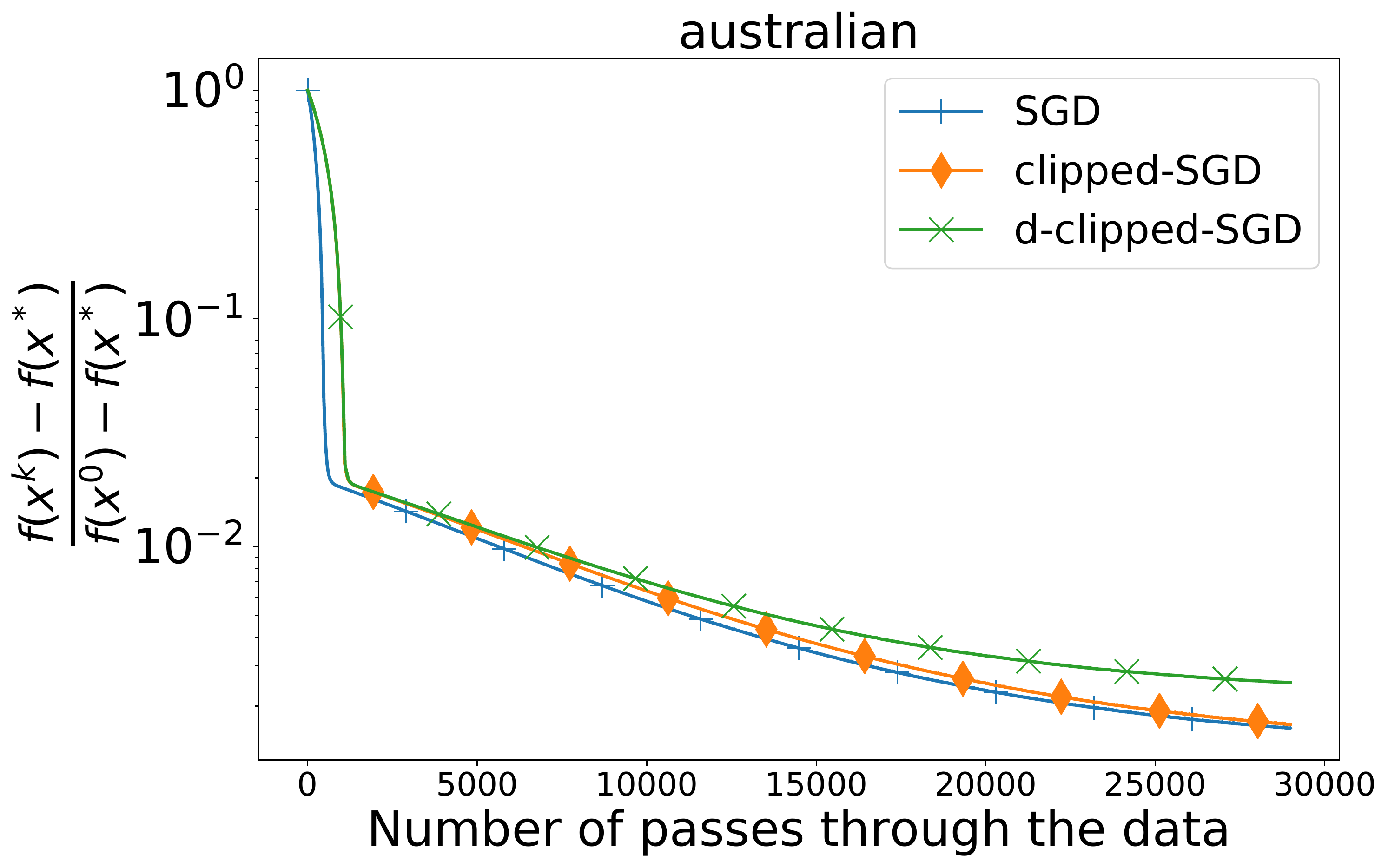}
    \includegraphics[width=0.32\textwidth]{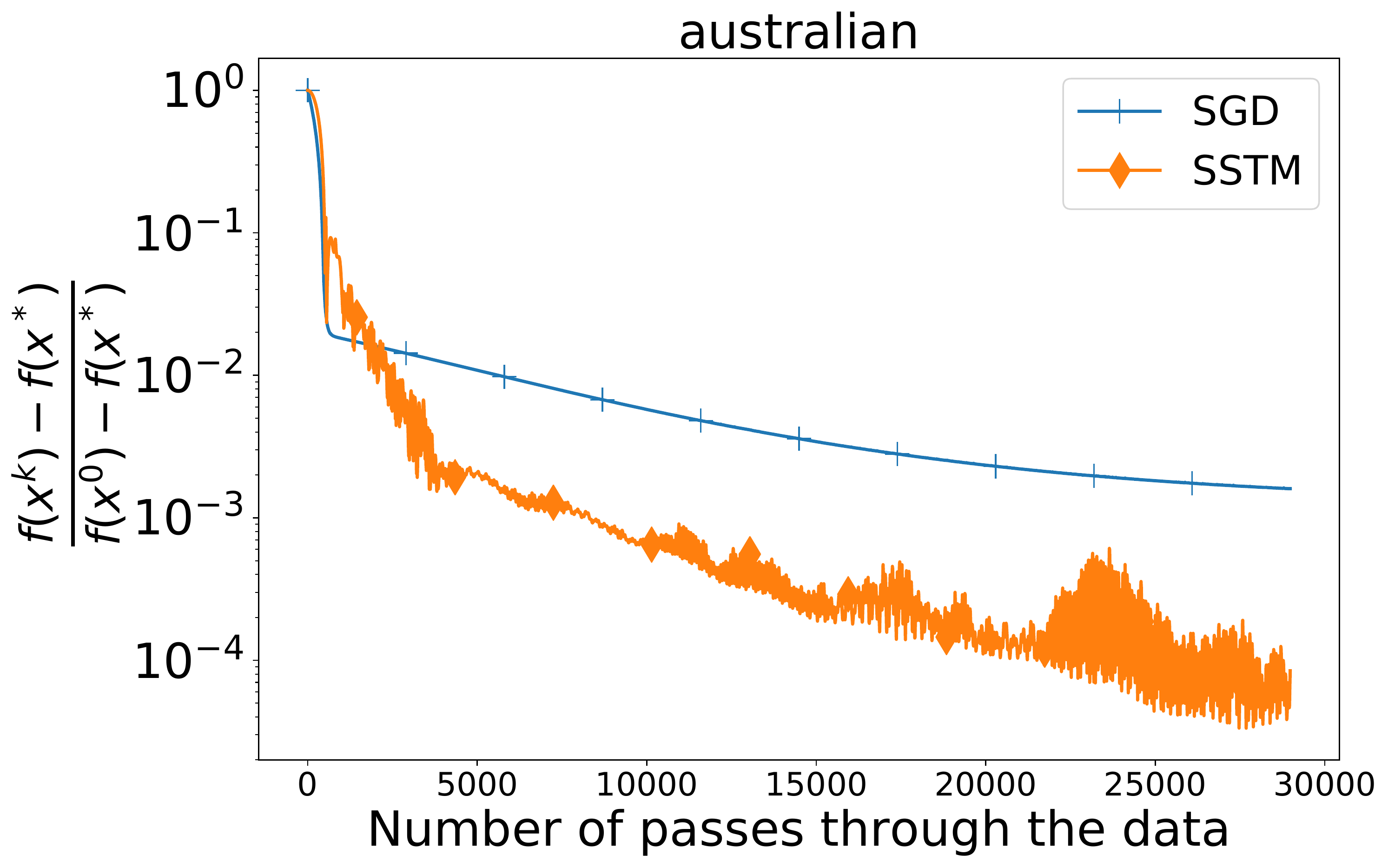}
    \includegraphics[width=0.32\textwidth]{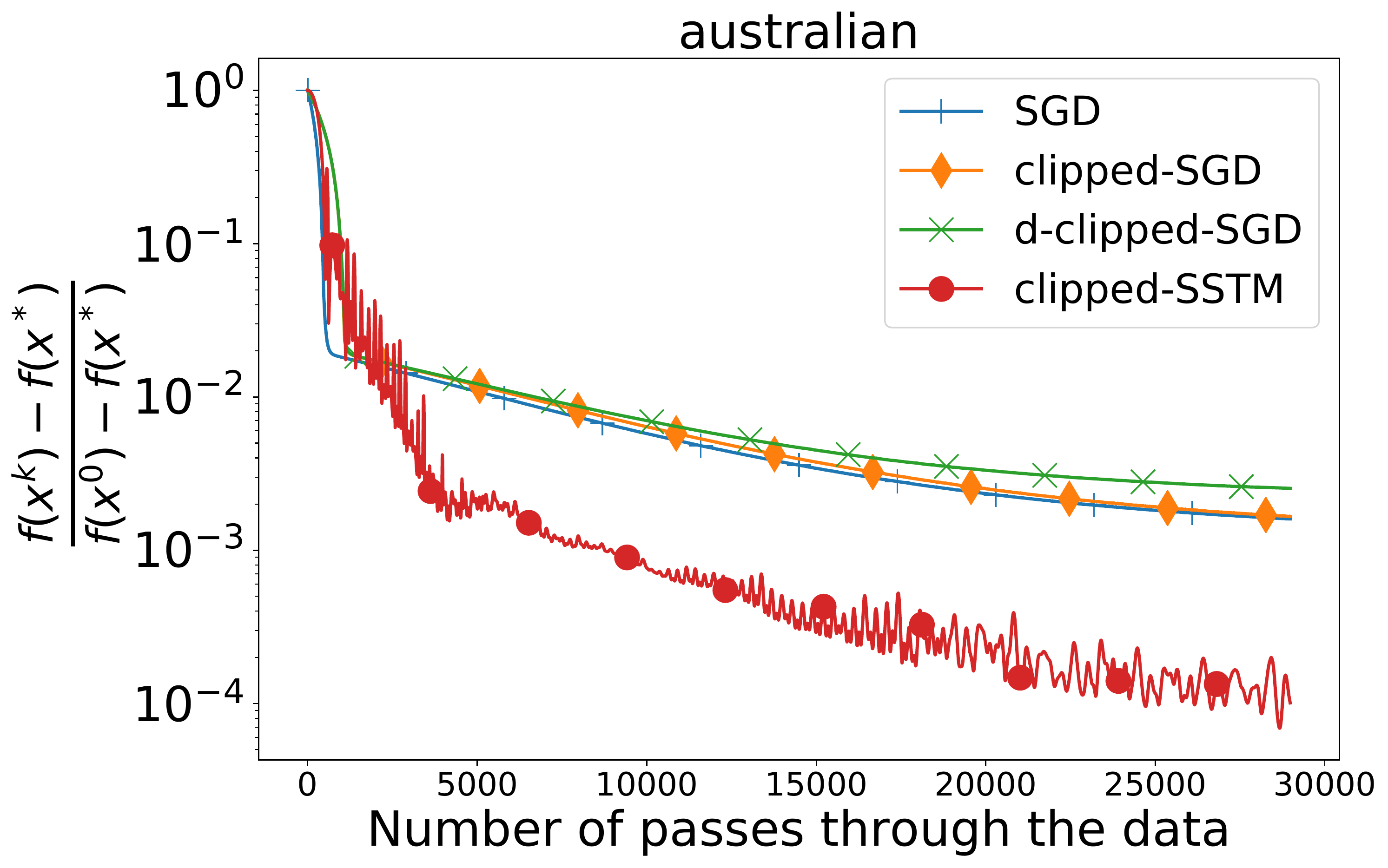}
    \caption{Trajectories of {\tt SGD}, {\tt clipped-SGD}, {\tt SSTM} and {\tt clipped-SSTM} applied to solve logistic regression problem on {\tt australian} dataset.}
    \label{fig:australian_logreg}
\end{figure}

Let us discuss the obtained numerical results. First of all, {\tt d-clipped-SGD} stabilizes the oscillations of {\tt SGD} even if the initial clipping level was high. In contrast, {\tt clipped-SGD} with too large clipping level $\lambda$ behaves similarly to {\tt SGD}. Secondly, we emphasize that due to the fact that we used small bathcsizes {\tt SSTM} has very large oscillations in comparison to {\tt SGD}. Actually, fast error/noise accumulation is a typical drawback of accelerated {\tt SGD} with small batchsizes \cite{kidambi2018insufficiency}. Moreover, deterministic accelerated and momentum-based methods often have non-monotone behavior (see \cite{danilova2018non} and references therein). However, to some extent {\tt clipped-SSTM} suffers from the first drawback less than {\tt SSTM} and has comparable convergence rate with {\tt SSTM}. Finally, in our experiments on {\tt heart} and {\tt australian} datasets {\tt clipped-SSTM} converges faster than {\tt SGD} and {\tt clipped-SGD} and oscillates little, while on {\tt diabetes} dataset it also converges faster than {\tt SGD}, but oscillates more if parameter $B$ is not fine-tuned.

We also want to mention that the behavior of {\tt SGD} on {\tt heart} and {\tt diabetes} datasets correlates with the insights from Section~\ref{sec:motivation} and our numerical study of the distribution of $\|\nabla f_i(x^*)\|_2$. Indeed, for {\tt heart} dataset {\tt SGD} has little oscillations since the distribution of $\|\nabla f_i(x^k) - \nabla f(x^k)\|_2$, where $x^k$ is the last iterate, is well concentrated near its mean and can be approximated by Gaussian distribution (see the details in Section~\ref{sec:additional_logreg}). In contrast, Figure~\ref{fig:diabetes_logreg} shows that {\tt SGD} oscillates more than in the previous example. One can explain such behavior using Figure~\ref{fig:noise_distrib} showing that the distribution of $\|\nabla f(x^*)\|_2$ has heavier tails than for {\tt heart} dataset.

However, we do not see any oscillations of {\tt SGD} for {\tt australian} dataset despite the fact that according to Figure~\ref{fig:noise_distrib} the distribution of $\|\nabla f_i(x^*)\|_2$ in this case has heavier tails than in previous examples. Actually, there is no contradiction and in this case it simply means that {\tt SGD} does not get close to the solution in terms of functional value, despite the fact that we used $\gamma = \nicefrac{1}{L}$. In Section~\ref{sec:additional_logreg} we present the results of different tests where we tried to use bigger stepsize $\gamma$ in order to reach oscillation region faster and show that in fact in that region {\tt SGD} oscillates significantly more, but clipping fixes this issue without spoiling the convergence rate.

\section{Discussion}\label{sec:discussion}

In this paper we close the gap in the theory of high-probability complexity bounds for stochastic optimization with heavy-tailed noise. In particular, we propose a new accelerated stochastic method~--- {\tt clipped-SSTM}~--- and prove the first accelerated high-probability complexity bounds for smooth convex stochastic optimization without light-tails assumption. Moreover, we extend our results to the strongly convex case and prove new complexity bounds outperforming the state-of-the-art results. Finally, we derive first high-probability complexity bounds for the popular method called {\tt clipped-SGD} in convex and strongly convex cases and conduct a numerical study of the considered methods.

However, our approach has several limitations. In particular, it significantly relies on the assumption that the optimization problem is defined on $\R^n$. Moreover, we do not consider regularized or composite problems like in \cite{nazin2019algorithms} and \cite{davis2019low}. However, in \cite{nazin2019algorithms} it is significant in the analysis that the set where the problem is defined is bounded and in \cite{davis2019low} the analysis works only for the strongly convex problems. It would also be interesting to generalize our approach to generally non-smooth problems using the trick from \cite{nesterov2015universal}.

\section*{Broader Impact}
Our contribution is primarily theoretical. Therefore, a broader impact discussion is not applicable.

\begin{ack}
The research of E.~Gorbunov and A.~Gasnikov was partially supported by the Ministry of Science and Higher Education of the Russian Federation (Goszadaniye) 075-00337-20-03. The research of Marina Danilova was funded by RFBR, project number 20-31-90073.
\end{ack}
{
\bibliography{references}
}
\clearpage

\part*{Appendix\\ \Large Stochastic Optimization with Heavy-Tailed Noise via \\ Accelerated Gradient Clipping}
\appendix

{\small
\tableofcontents
}

\section{Notations and Definitions}\label{sec:appendix_notation}
We use $\< x, y >$ to define standard inner product between two vectors $x,y\in\R^n$, i.e.\ $\< x, y > \eqdef \sum_{i=1}^n x_iy_i$, where $x_i$ is $i$-th coordinate of vector $x$, $i=1,\ldots,n$. Standard Euclidean norm of vector $x\in\R^n$ is defined as $\|x\|_2 \eqdef \sqrt{\< x, x >}$. 

We use $\PP\{\cdot\}$ to define probability measure which is always known from the context, $\EE[\cdot]$ denotes mathematical expectation, $\EE_\xi[\cdot]$ is used to define conditional mathematical expectation with respect to the randomness coming from $\xi$ only and $\EE\left[\xi\mid \eta\right]$ denotes mathematical expectation of $\xi$ conditional on $\eta$. In our proofs, we also use $\EE_k[\cdot]$ to denote conditional mathematical expectation with respect to all randomness coming from $k$-th iteration. For $\PP$-measurable set $X$ we use $\obf_X$ to denote indicator of event $X$, i.e.\ 
\begin{equation}
    \obf_X = \begin{cases}1,& \text{if event } X \text{ holds},\\ 0,& \text{otherwise}. \end{cases} \label{eq:indicator_def}
\end{equation}

Next, we introduce some standard definitions.
\begin{definition}[$L$-smoothness]
    Function $f$ is called $L$-smooth on $\R^n$ with $L > 0$ when it is differentiable and its gradient is $L$-Lipschitz continuous on $\R^n$, i.e.\ 
    \begin{equation}
        \|\nabla f(x) - \nabla f(y)\|_2 \le L\|x - y\|_2,\quad \forall x,y\in \R^n.\label{eq:L_smoothness}
    \end{equation}
\end{definition}
It is well-known that $L$-smoothness implies (see \cite{nesterov2018lectures})
\begin{eqnarray}
    f(y) &\le& f(x) + \la\nabla f(x),y-x \ra + \frac{L}{2}\|y-x\|_2^2\quad \forall x,y\in \R^n,\label{eq:L_smoothness_cor}
\end{eqnarray}    
and if $f$ is additionally convex, then
\begin{eqnarray}
    \|\nabla f(x) - \nabla f(y)\|_2^2 &\le& 2L\left(f(x) - f(y) - \la\nabla f(y),x-y \ra\right)\quad \forall x,y\in \R^n.\label{eq:L_smoothness_cor_2}
\end{eqnarray}
Since in this paper we focus only on smooth optimization problems we introduce strong convexity in the following way.
\begin{definition}[$\mu$-strong convexity]
    Differentiable function $f$ is called $\mu$-strongly convex on $\R^n$ with $\mu \ge 0$ if for all $x,y\in\R^n$ 
    \begin{equation}
    f(x) \ge f(y) + \langle\nabla f(y), x-y \rangle + \frac{\mu}{2}\|x-y\|_2^2. \label{eq:str_cvx_def}
\end{equation}
\end{definition}
In particular, $\mu$-strong convexity implies that for all $x\in\R^n$
\begin{equation}
    f(x) - f(x^*) \ge \frac{\mu}{2}\|x-x^*\|_2^2. \label{eq:str_cvx_cor}
\end{equation}

Throughout the paper, we use $x^*$ to denote any solution of problem \eqref{eq:main_problem} assuming its existence. By the complexity of stochastic first-order method we always mean the total number of stochastic first-order oracle calls that the method needs in order to produce such a point $\hat x$ that $f(\hat x) - f(x^*) \le \varepsilon$ with probability at least $1-\beta$ for some $\varepsilon > 0$ and $\beta\in(0,1)$. Finally, in the complexity bounds we often use $R_0$ to denote $\|x^0 - x^*\|_2$ where $x^0$ is the starting point of the method.
\clearpage
\section{Related Work: Additional Details}
\subsection{Related Work on Non-Smooth Stochastic Optimization}\label{sec:non_smooth_appendix}
Here we present an overview of existing results in the convex non-smooth case, i.e.\ when $f$ is still convex but not necessarily $L$-smooth and the stochastic gradients have a bounded second moment: $\EE_\xi[\|\nabla f(x,\xi)\|_2^2] \le M^2$ for all $x\in\R^n$. Under additional assumption that the stochastic gradients have light-tailed distribution it was shown that {\tt SGD} \cite{nemirovski2009robust} has $O\left(\nicefrac{M^2 R_0^2\ln(\beta^{-1})}{\varepsilon^2}\right)$ complexity and if additionally $f$ is $\mu$-strongly convex it was shown in \cite{juditsky2011first,juditsky2014deterministic} that the restarted version of {\tt SGD} has $O\left(\nicefrac{M^2\ln\left(\beta^{-1}\ln(M^2\mu^{-1}\varepsilon^{-1})\right)}{\mu\varepsilon}\right)$ complexity (see also \cite{hazan2014beyond,kakade2009generalization,rakhlin2011making}). Moreover, removing logarithmical factors from these bounds we get the complexity bounds of these methods for the convergence in expectation, i.e.\ needed number of oracle calls to find such $\hat x$ that $\EE[f(\hat x)] - f(x^*) \le \varepsilon$. That is, under light tails assumption high-probability complexity bounds and complexity bounds in expectation for {\tt SGD} and {\tt restarted-SGD} differ only in logarithmical factors of $\nicefrac{1}{\beta}$.

Unfortunately, for these methods the situation changes dramatically when the stochastic gradients are heavy-tailed. To the best of our knowledge, the best know bounds in the literature with the same dependency on $\varepsilon$ are $O\left(\nicefrac{M^2R_0^2}{\beta^2\varepsilon^2}\right)$ and $O\left(\nicefrac{M^2}{\mu\beta\varepsilon}\right)$. One can obtain these bounds using complexity results for the convergence in expectation and Markov's inequality. However, it leads to significantly worse dependence on $\beta$: instead of $O(\ln(\beta^{-1}))$ we get $O(\beta^{-2})$ and $O(\beta^{-1})$ dependence on the confidence level $\beta$. Furthermore, based on the well-known results on the distribution of sum of i.i.d.\ random variables (see Section~\ref{sec:sum_iid_rand_vars}) in \cite{gasnikov2014stochastic} authors consider the case when the tails of the distribution of stochastic gradient satisfy $\PP\{\|\nabla f(x,\xi) - \nabla f(x)\|_2 > s\} = O(s^{-\alpha})$ for $\alpha > 2$ and give the following complexity bounds without formal proofs that {\tt SGD} for convex problems and {\tt restarted-SGD} for $\mu$-strongly convex problems have following complexities:
\begin{equation*} \label{HTconv}
    O\left(M^2R^2\max\left\{ \frac{\ln\left(\beta^{-1}\right)}{\varepsilon^2},\left(\frac{1}{\beta\varepsilon^{\alpha}}\right)^{\frac{2}{3\alpha - 2}}\right\}\right),
\end{equation*}
\begin{equation*} \label{HTstrconv}
     O\left(\max\left\{ \frac{M^2\ln\left(\beta^{-1}\ln \frac{M^2}{\mu\varepsilon}\right)}{\mu\varepsilon},\left(\frac{M^2}{\mu\varepsilon}\right)^{\frac{\alpha}{3\alpha - 2}}\left(\beta^{-1}\ln \frac{M^2}{\mu\varepsilon}\right)^{\frac{2}{3\alpha - 2}}\right\}\right).
\end{equation*}
The first terms in maximums above correspond to the Central Limit Theorem regime, while the second terms correspond to the heavy-tailed regime, see Section~\ref{sec:sum_iid_rand_vars}. These bounds show that heavy tailed distributions of the stochastic gradients significantly spoil complexity bounds of {\tt SGD} and {\tt restarted-SGD} when the confidence level $\beta$ is small enough.

\subsection{Related Work on Gradient Clipping}\label{sec:grad_clipping_extra}
As we mentioned Section~\ref{sec:motivation} {\tt clipped-SGD} \cite{Goodfellow-et-al-2016, mikolov2012statistical,pascanu2013difficulty,usmanova2017master} is known to be robust to the noise in stochastic gradients and performs better than {\tt SGD} in the vicinity of extremely steep cliffs. Zhang et al.\ \cite{zhang2019adam} analyse the convergence of {\tt clipped-SGD} \textit{in expectation} for strongly convex and non-convex objectives under assumption that $\EE[\|\nabla f(x,\xi)\|_2^\alpha]$ is bounded for some $\alpha\in (1,2]$. For $\alpha < 2$ this assumption covers some heavy-tailed distributions of stochastic gradients appearing in practice. Moreover, in \cite{zhang2019adam} authors conduct several numerical tests showing that in some real-world problems where the noise in stochastic gradients is heavy-tailed {\tt clipped-SGD} converges faster than {\tt SGD}. In \cite{zhang2019gradient} Zhang et al.\ found that {\tt clipped-GD} is able to converge in non-convex case to the stationary point under the relaxed smoothness assumption with $O(\varepsilon^{-2})$ rate while Gradient Descent ({\tt GD}) can fail to converge with the same rate in this setting. A very similar approach based on the normalization of {\tt GD} is studied in \cite{hazan2015beyond,levy2016power}.
\clearpage

\section{Basic Facts}\label{sec:basic_facts}
In this section we enumerate for convenience basic facts that we use many times in our proofs.

\textbf{Fenchel-Young inequality.} For all $a,b\in\R^n$ and $\lambda > 0$
\begin{equation}
    |\la a, b\ra| \le \frac{\|a\|_2^2}{2\lambda} + \frac{\lambda\|b\|_2^2}{2}.\label{eq:fenchel_young_inequality}
\end{equation}

\textbf{Squared norm of the sum.} For all $a,b\in\R^n$
\begin{equation}
    \|a+b\|_2^2 \le 2\|a\|_2^2 + 2\|b\|_2^2.\label{eq:squared_norm_sum}
\end{equation}

\textbf{Inner product representation.} For all $a,b\in\R^n$
\begin{equation}
    \la a, b\ra = \frac{1}{2}\left(\|a+b\|_2^2 - \|a\|_2^2 - \|b\|_2^2\right) \label{eq:inner_product_representation}
\end{equation}

\textbf{Variance decomposition.} If $\xi$ is a random vector in $\R^n$ with bounded second moment, then
\begin{equation}
    \EE\left[\|\xi+a\|_2^2\right] = \EE\left[\left\|\xi-\EE[\xi]\right\|_2^2\right] + \left\|\EE[\xi]+a\right\|_2^2 \label{eq:variance_decomposition}
\end{equation}
for any deterministic vector $a\in\R^n$. In particular, this implies
\begin{equation}
     \EE\left[\left\|\xi-\EE[\xi]\right\|_2^2\right] \le \EE\left[\|\xi+a\|_2^2\right] \label{eq:variance_decomposition_2}
\end{equation}
for any deterministic vector $a\in\R^n$.

\section{Auxiliary Results}\label{sec:aux_results}
\subsection{Bernstein Inequality}\label{sec:bernstein}
\begin{lemma}[Bernstein inequality for martingale differences \cite{bennett1962probability,dzhaparidze2001bernstein,freedman1975tail}]\label{lem:Bernstein_ineq}
    Let the sequence of random variables $\{X_i\}_{i\ge 1}$ form a martingale difference sequence, i.e.\ $\EE\left[X_i\mid X_{i-1},\ldots, X_1\right] = 0$ for all $i \ge 1$. Assume that conditional variances $\sigma_i^2\eqdef\EE\left[X_i^2\mid X_{i-1},\ldots, X_1\right]$ exist and are bounded and assume also that there exists deterministic constant $c>0$ such that $\|X_i\|_2 \le c$ almost surely for all $i\ge 1$. Then for all $b > 0$, $F > 0$ and $n\ge 1$
    \begin{equation}
        \PP\left\{\Big|\sum\limits_{i=1}^nX_i\Big| > b \text{ and } \sum\limits_{i=1}^n\sigma_i^2 \le F\right\} \le 2\exp\left(-\frac{b^2}{2F + \nicefrac{2cb}{3}}\right).
    \end{equation}
\end{lemma}
\subsection{About the Sum of i.i.d.\ Random Variables with Heavy Tails}\label{sec:sum_iid_rand_vars}
In this section we present some classical results about the distribution of sum of i.i.d.\ random variables $\sum_{k=1}^N \xi_k$ with heavy tails \cite{borovkov2002probabilities}. As one can see from our proofs of main results for {\tt clipped-SSTM} and {\tt clipped-SGD} such sums play a central role in the analysis of convergence with high probability. Assume that $\{\xi_k\}$ is i.i.d. with $\EE [\xi_k] = 0$ and $\text{Var} [\xi_k] \eqdef \EE[(\xi_k - \EE[\xi_k])^2] = \sigma^2$. Assume also that $V(s) = \PP\left\{\xi_k \ge s\right\} = \Theta\left(s^{-\alpha}\right)$, where $\alpha > 2$. In this case 
\begin{equation*}
\PP\left\{\sum_{k=1}^N \xi_k \ge s\right\} \simeq 1 - \Phi\left(\frac{s}{\sqrt{\sigma^2N}}\right) + N\cdot V(s),
\end{equation*}
where $N \gg 1$ and $\Phi(x) = \frac{1}{2\pi}\int_{-\infty}^x \exp\left(-y^2/2\right)dy$. Since 
$$0.2\exp\left(-\frac{2x^2}{\pi}\right)\le 1-\Phi(x)\le\exp\left(-\frac{x^2}{2}\right),$$
we have\footnote{CLT = Central Limit Theorem.}
\begin{equation}\label{CLT}
\PP\left\{\sum_{k=1}^N \xi_k \ge s\right\} \simeq 1 - \Phi\left(\frac{s}{\sqrt{\sigma^2N}}\right), \quad s\le\sqrt{(\alpha - 2)\sigma^2N\ln N} \quad \text{(CLT regime)}
\end{equation}
and
\begin{equation}\label{HT}
\PP\left\{\sum_{k=1}^N \xi_k \ge s\right\} \simeq N\cdot V(s), \quad s > \sqrt{(\alpha - 2)\sigma^2N\ln N} \quad \text{(heavy-tailed regime).}
\end{equation}

This simple observation can play a significant role in deriving complexity results for non-smooth convex optimization under the assumption that stochastic gradients are heavy-tailed, see \cite{gasnikov2014stochastic} for the details.

\section{Technical Results}\label{sec:tech_results}

\begin{lemma}\label{lem:alpha_k}
    Consider two sequences of non-negative numbers $\{\alpha_k\}_{k\ge0}$ and $\{A_k\}_{k\ge 0}$ such that
    \begin{equation}
        \alpha_0 = A_0 = 0,\quad A_{k+1} = A_k + \alpha_{k+1},\quad \alpha_{k+1} = \frac{k+2}{2aL}\quad \forall k\ge0, \label{eq:alpha_k_definition}
    \end{equation}
    where $a,L>0$. Then for all $k \ge 0$
    \begin{eqnarray}
        A_{k+1} &=& \frac{(k+1)(k+4)}{4aL},\label{eq:A_k+1_formula}\\
        A_{k+1} &\ge& aL\alpha_{k+1}^2. \label{eq:A_k+1_alpha_k+1_rel}
    \end{eqnarray}
\end{lemma}
\begin{proof}
    By definition of $A_{k+1}$ we have that
    \begin{eqnarray*}
        A_{k+1} &=& \sum\limits_{l=1}^{k+1}\alpha_l = \frac{1}{2aL}\sum\limits_{l=1}^{k+1} (l+1) = \frac{(k+1)(k+4)}{4aL}.
    \end{eqnarray*}
    Using $(k+1)(k+4)\ge (k+2)^2$ together with the inequality above we derive \eqref{eq:A_k+1_alpha_k+1_rel}.
\end{proof}

\clearpage

\section{Accelerated {\tt SGD} with Clipping: Exact Formulations and Missing Proofs}\label{sec:proofs_acc_method}
In this section we provide exact formulations of all the results that we have for {\tt clipped-SSTM} and {\tt R-clipped-SSTM} together with the full proofs.
\subsection{Convex Case}\label{sec:cvx_case_clipped_SSTM_app}
Recall that in order to compute $\clip(\nabla f(x,\Bxi), \lambda)$ one needs to get $m$ i.i.d.\ samples $\nabla f(x,\xi_1),\ldots,\nabla f(x,\xi_m)$, compute its average
\begin{equation}
    \nabla f(x,\Bxi) = \frac{1}{m}\sum\limits_{i=1}^m\nabla f(x,\xi_i), \label{eq:mini_batched_stoch_grad_clipped_SSTM}
\end{equation}
and then project the result $\nabla f(x,\Bxi)$ on the Euclidean ball with radius $\lambda$ and center at the origin. We also notice that
\begin{eqnarray}
    \EE_\xi[\nabla f(x,\Bxi)] &=& \nabla f(x),\label{eq:mini_batched_unbiasedness_clipped_SSTM}\\
    \EE_\xi\left[\left\|\nabla f(x,\Bxi) - \nabla f(x)\right\|_2^2\right] &\le& \frac{\sigma^2}{m}.\label{eq:mini_batched_bounded_variance_clipped_SSTM}
\end{eqnarray}

\subsubsection{Convergence Guarantees for {\tt clipped-SSTM}}\label{sec:convergence_cvx_case_clipped_SSTM}
Next theorem summarizes the main convergence result for {\tt clipped-SSTM}.
\begin{theorem}\label{thm:main_result_clipped_SSTM}
    Assume that function $f$ is convex and $L$-smooth. Then for all $\beta \in (0,1)$ and $N\ge 1$ such that
    \begin{equation}
        \ln\frac{4N}{\beta} \ge 2 \label{eq:beta_N_condition_clipped_SSTM}
    \end{equation}
    we have that after $N$ iterations of {\tt clipped-SSTM} with
    \begin{equation}
        m_k = \max\left\{1,\frac{6000\sigma^2 \alpha_{k+1}^2N\ln\frac{4N}{\beta}}{C^2R_0^2}, \frac{10368\sigma^2 \alpha_{k+1}^2N}{C^2R_0^2}\right\},\label{eq:bathces_clipped_SSTM}
    \end{equation}
    \begin{equation}
        B = \frac{CR_0}{8\ln\frac{4N}{\beta}},\quad a \ge \max\left\{1,\frac{16\ln\frac{4N}{\beta}}{C},36\left(2\ln\frac{4N}{\beta} + \sqrt{4\ln^2\frac{4N}{\beta} + 2\ln\frac{4N}{\beta}}\right)^2\right\}, \label{eq:B_a_parameters_clipped_SSTM}
    \end{equation}
    that with probability at least $1-\beta$
    \begin{equation}
        f(y^N) - f(x^*) \le \frac{2aLC^2R_0^2}{N(N+3)}, \label{eq:main_result_clipped_SSTM}
    \end{equation}
    where $R_0 = \|x^0 - x^*\|_2$ and
    \begin{equation}
        C = \sqrt{5}. \label{eq:C_definition_clipped_SSTM}
    \end{equation}
    In other words, if we choose $a$ to be equal to the maximum from \eqref{eq:B_a_parameters_clipped_SSTM}, then the method achieves $f(y^N) - f(x^*) \le \varepsilon$ with probability at least $1-\beta$ after $O\left(\sqrt{\frac{LR_0^2}{\varepsilon}}\ln\frac{LR_0^2}{\varepsilon\beta}\right)$ iterations and requires
    \begin{equation}
        O\left(\max\left\{\sqrt{\frac{LR_0^2}{\varepsilon}}, \frac{\sigma^2R_0^2}{\varepsilon^2}\right\}\ln\frac{LR_0^2}{\varepsilon\beta}\right)\text{ oracle calls.} \label{eq:clipped_SSTM_oracle_complexity}
    \end{equation}
\end{theorem}
One can easily notice that multiplicative constant factors in formulas for $m_k$ and $a$ are too big and seem to be impractical, but in practice one can tune these constants to get good enough performance. That is, big constants in \eqref{eq:bathces_clipped_SSTM} and \eqref{eq:B_a_parameters_clipped_SSTM} are needed only in our analysis in order to get bound \eqref{eq:clipped_SSTM_oracle_complexity}.

Finally, when $\sigma^2$ is big then Theorem~\ref{thm:main_result_clipped_SSTM} says that at iteration $k$ {\tt clipped-SGD} requires large batchsizes $m_k \sim k^2 N$ (see \eqref{eq:bathces_clipped_SSTM}) which is proportional to $\varepsilon^{-\nicefrac{3}{2}}$ for last iterates. It can make the cost of one iteration extremely high, therefore, we consider different stepsize policies that remove this drawback.

\begin{corollary}\label{cor:different_stepsizes_clipped_SSTM}
    Let the assumptions of Theorem~\ref{thm:main_result_clipped_SSTM} hold.
    \begin{enumerate}
        \item \textbf{(Medium batchsize).} If $N$ and $\beta$ are such that $N\ln\frac{4N}{\beta}$ is bigger than the maximum from \eqref{eq:B_a_parameters_clipped_SSTM}, then for $a = N\ln\frac{4N}{\beta}$ we have 
        \begin{equation}
        m_k = \max\left\{1,\frac{6000\sigma^2 (k+2)^2}{4L^2NC^2R_0^2\ln\frac{4N}{\beta}}, \frac{10368\sigma^2 (k+2)^2}{4L^2C^2R_0^2N\ln^2\frac{4N}{\beta}}\right\}\label{eq:medium_bathces_clipped_SSTM}
    \end{equation}
    and the method achieves $f(y^N) - f(x^*) \le \varepsilon$ with probability at least $1-\beta$ after $O\left(\frac{LR_0^2}{\varepsilon}\ln\frac{LR_0^2}{\varepsilon\beta}\right)$ iterations and requires
    \begin{equation}
        O\left(\max\left\{\frac{LR_0^2}{\varepsilon}, \frac{\sigma^2R_0^2}{\varepsilon^2}\right\}\ln\frac{LR_0^2}{\varepsilon\beta}\right)\text{ oracle calls.} \label{eq:clipped_SSTM_oracle_complexity_medium_batches}
    \end{equation}
    \item \textbf{(Constant batchsize).} If $N$ and $\beta$ are such that $a_0 N^{\nicefrac{3}{2}}\sqrt{\ln\frac{4N}{\beta}}$ is bigger than the maximum from \eqref{eq:B_a_parameters_clipped_SSTM} for some positive constant $a_0$, then for $a = a_0 N^{\nicefrac{3}{2}}\sqrt{\ln\frac{4N}{\beta}}$ we have 
        \begin{equation}
        m_k = \max\left\{1,\frac{6000\sigma^2 (k+2)^2}{4a_0^2L^2N^2C^2R_0^2}, \frac{10368\sigma^2 (k+2)^2}{4a_0^2L^2C^2R_0^2N^2\ln\frac{4N}{\beta}}\right\}\label{eq:constant_bathces_clipped_SSTM}
    \end{equation}
    and the method achieves $f(y^N) - f(x^*) \le \varepsilon$ with probability at least $1-\beta$ after\newline $O\left(\frac{a_0^2L^2R_0^4}{\varepsilon^2}\ln\frac{a_0LR_0^2}{\varepsilon\beta}\right)$ iterations and requires
    \begin{equation}
        O\left(\max\left\{\frac{a_0^2L^2R_0^4}{\varepsilon^2}, \frac{\sigma^2R_0^2}{\varepsilon^2}\right\}\ln\frac{a_0LR_0^2}{\varepsilon\beta}\right)\text{ oracle calls.} \label{eq:clipped_SSTM_oracle_complexity_constant_batches}
    \end{equation}
    Finally, if $a_0 = \frac{\sigma}{LR_0}$, then $m_k = O(1)$ for $k=0,1,\ldots,N$ and {\tt clipped-SSTM} finds $\varepsilon$-solution with probability at least $1-\beta$ after $O\left(\frac{\sigma^2R_0^2}{\varepsilon^2}\ln\frac{\sigma R_0}{\varepsilon\beta}\right)$ iterations and requires $O(1)$ oracle calls per iteration.
    \end{enumerate}
\end{corollary}
In the first case batchsizes increase from $O(1)$ for $k=1$ to $O(\varepsilon^{-1})$ for $k=N$ and the overall complexity recovers the complexity of Robust Stochastic Mirror Descent ({\tt RSMD}) from \cite{nazin2019algorithms}. However, analysis from \cite{nazin2019algorithms} works only for the optimization problems on \textit{compact} convex sets, whereas our analysis handles an unconstrained optimization on $\R^n$. Despite the similarities of our approach and \cite{nazin2019algorithms}, it seems that the technique from \cite{nazin2019algorithms} cannot be generalized to obtain the complexity like in \eqref{eq:clipped_SSTM_oracle_complexity} due to the fast bias accumulation that appears because of the special truncation of stochastic gradients that is used in {\tt RSMD}.

In the second case the corollary establishes $\varepsilon^{-2}\ln(\varepsilon^{-1}\beta^{-1})$ rate for {\tt clipped-SSTM} with constant batchsizes, i.e.\ $m_k = O(1)$ for all $k$. The ability of {\tt clipped-SSTM} to converge with constant batchsizes makes it more practical and applicable for wider class of problems where it can be very expensive to compute large batchsizes, e.g.\ training deep neural networks. Moreover, when $\sigma$ is not too small, i.e.\ $\sigma^2 \ge L\varepsilon$, this rate is optimal (up to logarithmical factors) and also recovers the rate of {\tt RSMD}.

Finally, setting
\begin{eqnarray}
    a' &=& \max\left\{1,\frac{16\ln\frac{4N}{\beta}}{C},36\left(2\ln\frac{4N}{\beta} + \sqrt{4\ln^2\frac{4N}{\beta} + 2\ln\frac{4N}{\beta}}\right)^2\right\},\notag\\
    a &=& \max\left\{a', \frac{\sigma N^{\nicefrac{3}{2}}}{LR_0}\sqrt{\ln\frac{4N}{\beta}}\right\}\label{eq:clipped_sstm_formula_steps_coeff}
\end{eqnarray}
and $m_k$ as in \eqref{eq:bathces_clipped_SSTM}, we get $m_k = O(1)$ for $k = 0,1,\ldots, N$ and derive the following result.
\begin{corollary}\label{cor:clipped_sstm_small_stepsize_const_batch}
    Let the assumptions of Theorem~\ref{thm:main_result_clipped_SSTM} hold, $a$ is chosen as in \eqref{eq:clipped_sstm_formula_steps_coeff} and $m_k$ is computed via \eqref{eq:bathces_clipped_SSTM}. Then {\tt clipped-SSTM} achieves $f(y^N) - f(x^*) \le \varepsilon$ with probability at least $1-\beta$ after
    \begin{equation}
        O\left(\max\left\{\sqrt{\frac{LR_0^2}{\varepsilon}}, \frac{\sigma^2R_0^2}{\varepsilon^2}\right\}\ln\frac{LR_0^2+\sigma R_0}{\varepsilon\beta}\right)\text{ iterations/oracle calls.} \notag
    \end{equation}
\end{corollary}

\subsubsection{Sketch of the Proof of Theorem~\ref{thm:main_result_clipped_SSTM}}\label{sec:sketch_clipped_SSTM}
We start with the following lemma that is pretty standard in the analysis of Stochastic Similar Triangles Method, e.g.\ see the proof of Theorem~1 from \cite{dvurechenskii2018decentralize}.
\begin{lemma}\label{lem:main_opt_lemma_clipped_SSTM}
    Let $f$ be a convex $L$-smooth function and let stepsize parameter $a$ satisfy $a\ge 1$. Then after $N \ge 0$ iterations of {\tt clipped-SSTM} for all $z\in \R^n$ we have
    \begin{eqnarray}
        A_N\left(f(y^N) - f(z)\right) &\le& \frac{1}{2}\|z^0 - z\|_2^2 - \frac{1}{2}\|z^{N} - z\|_2^2 + \sum\limits_{k=0}^{N-1}\alpha_{k+1}\left\la \theta_{k+1}, z - z^{k}\right\ra\notag\\
        &&\quad + \sum\limits_{k=0}^{N-1}\alpha_{k+1}^2\left\|\theta_{k+1}\right\|_2^2 + \sum\limits_{k=0}^{N-1}\alpha_{k+1}^2\left\la\theta_{k+1},\nabla f(x^{k+1})\right\ra,\label{eq:main_opt_lemma_clipped_SSTM}\\
        \theta_{k+1} &\eqdef& \tnabla f(x^{k+1},\Bxi^k) - \nabla f(x^{k+1}).\label{eq:theta_k+1_def_clipped_SSTM}
    \end{eqnarray}
\end{lemma}
That is, if $z = x^*$, then the result above gives a preliminary upper bound for $A_N(f(y^N) - f(x^*))$. The first and the second terms in the r.h.s.\ of \eqref{eq:main_opt_lemma_clipped_SSTM} come from the analysis of Similar Triangles Method \cite{gasnikov2016universal} and three last terms have a stochastic nature. In particular, they explicitly depend on differences $\theta_{k+1} = \tnabla f(x^{k+1},\Bxi^k) - \nabla f(x^{k+1})$ between clipped mini-batched stochastic gradients and full gradients at $x^{k+1}$, so, if $\tnabla f(x^{k+1},\Bxi^k) = \nabla f(x^{k+1})$ with probability $1$, then we easily get needed convergence rate. However, we are interested in the more general case and, as a consequence, to continue the proof, we need to find a good enough upper bound for the last three terms from \eqref{eq:main_opt_lemma_clipped_SSTM}. In other words, we need to show that choosing parameters $a$, $m_k$ and $\lambda_{k+1}$ properly we can upper bound these terms by something that coincides with $\|z^0 - x^*\|_2^2$ up to numerical multiplicative constant. The proof of convergence result for {\tt RSMD} from \cite{nazin2019algorithms} where authors provide upper bound for similar sums hints that Bernstein's inequality (see Lemma~\ref{lem:Bernstein_ineq}) applied to estimate these terms can help us to reach our goal. In order to apply Bernstein's inequality one should derive tight bounds for such characteristics of $\tnabla f(x^{k+1},\Bxi^k)$ as upper bounds for the magnitude, bias, variance and distortion and the next lemma provides us with this.
\begin{lemma}\label{lem:main_stoch_lemma_clipped_SSTM}
    For all $k\ge 0$ the following inequality holds:
    \begin{equation}
        \left\|\tnabla f(x^{k+1},\Bxi^{k}) - \EE_{\Bxi^k}\left[\tnabla f(x^{k+1},\Bxi^{k})\right]\right\|_2 \le 2\lambda_{k+1}.\label{eq:magnitude_bound_clipped_SSTM}
    \end{equation}
    Moreover, if $\|\nabla f(x^{k+1})\|_2 \le \frac{\lambda_{k+1}}{2}$ for some $k\ge 0$, then for this $k$ we have:
    \begin{eqnarray}
        \left\|\EE_{\Bxi^k}\left[\tnabla f(x^{k+1},\Bxi^{k})\right] - \nabla f(x^{k+1})\right\|_2 &\le& \frac{4\sigma^2}{m_k\lambda_{k+1}},\label{eq:bias_bound_clipped_SSTM}\\
        \EE_{\Bxi^k}\left[\left\|\tnabla f(x^{k+1},\Bxi^{k}) - \nabla f(x^{k+1})\right\|_2^2\right] &\le& \frac{18\sigma^2}{m_k},\label{eq:distortion_bound_clipped_SSTM}\\
        \EE_{\Bxi^k}\left[\left\|\tnabla f(x^{k+1},\Bxi^{k}) - \EE_{\Bxi^k}\left[\tnabla f(x^{k+1},\Bxi^{k})\right]\right\|_2^2\right] &\le& \frac{18\sigma^2}{m_k}. \label{eq:variance_bound_clipped_SSTM}
    \end{eqnarray}
\end{lemma}
Clearly, clipping introduces a bias in $\tnabla f(x^{k+1},\Bxi^k)$ which influences the convergence of the method. Hence, the clipping level $\lambda_{k+1}$ should be chosen in a very accurate way. Below we informally describe what does it mean and present the sketch of the remaining part of the proof.

Imagine the ideal situation: $\nabla f(x^{k+1},\Bxi^k) = \nabla f(x^{k+1})$ with probability $1$ for all $k$, i.e.\ we have an access to the full gradients at points $x^{k+1}$. Then it is natural to choose $\lambda_{k+1}$ in such a way that $\clip(\nabla f(x^{k+1}), \lambda^{k+1}) = \nabla f(x^{k+1})$ in order to recover Similar Triangles Method ({\tt STM}) that converges with optimal rate in the deterministic case. In other words, one can pick $\lambda_{k+1}$ such that $\|\nabla f(x^{k+1})\|_2 \le \lambda_{k+1}$ and get an optimal method. Since we know that in this case the method should converge with $O(\nicefrac{1}{k^2})$ rate in terms of $f(x^k) - f(x^*)$ one can expect that the gradient's norm decays with $O(\nicefrac{1}{k})$ rate, so, one can choose $\lambda_{k+1}$ to be proportional to $\nicefrac{1}{k}$. It is exactly what we do when we define $\lambda_{k+1}$ as $\nicefrac{B}{\alpha_{k+1}}$.

The ideal case described above gives a good insight on how to choose $\lambda_{k+1}$ in the general case and can be described as follows: if we want to prevent our gradient estimator $\tnabla f(x^{k+1},\Bxi^k)$ from large deviations from $\nabla f(x^{k+1})$ with high probability, then it is needed to choose $\lambda_{k+1}$ such that $\|\nabla f(x^k)\|_2 \le c\lambda_{k+1}$ with high probability where $c < 1$ is some positive number. This choice guarantees that with high probability clipped mini-batched gradient $\tnabla f(x^{k+1},\Bxi^k)$ cannot deviates from $\nabla f(x^{k+1})$ significantly and, as a consequence, the convergence rate of {\tt clipped-SSTM} in terms of the number of iterations needed to achieve the desired accuracy of the solution with high probability becomes similar to the convergence rate of {\tt STM} up to some logarithmical factors depending on the confidence level. 

In particular, we choose $\lambda_{k+1}$ such that $\|\nabla f(x^{k+1})\|_2 \le \nicefrac{\lambda_{k+1}}{2}$ with high probability. Moreover, we derive this relation by induction via refined estimation of the three last terms from the r.h.s.\ of \eqref{eq:main_opt_lemma_clipped_SSTM} that is based on the new variant of advanced recurrences technique from \cite{gorbunov2019optimal, gorbunov2018accelerated}. The main trick there is in showing by induction that sequence $\|z^k - x^*\|_2$ is bounded by some constant multiplied by $\|x^0 - x^*\|_2$ and in deriving $\|\nabla f(x^{k+1})\|_2 \le \nicefrac{\lambda_{k+1}}{2}$ simultaneously for all $k=0,1,\ldots,N$. With such bounds and Lemma~\ref{lem:main_stoch_lemma_clipped_SSTM} in hand, it is possible to apply Bernstein's inequality to three sums from the r.h.s.\ of \eqref{eq:main_opt_lemma_clipped_SSTM} since all summands are bounded \textit{with high probability}. After applying Bernstein's inequality we adjust parameters $\alpha_{k+1}$ and $m_k$ in such a way that after rearranging the terms in the obtained upper bounds we get that r.h.s.\ in \eqref{eq:main_opt_lemma_clipped_SSTM} (with $z=x^*$) is smaller than $\|x^0 - x^*\|_2^2$ up to some multiplicative numerical constant. This finishes the proof.

To conclude, the key tools in our analysis are Bernstein's inequality (see Lemma~\ref{lem:Bernstein_ineq}) and advanced recurrences technique \cite{gorbunov2019optimal,gorbunov2018accelerated} that helps us to show boundedness of $\|z^N - x^*\|_2$ and $\|\nabla f(x^{k+1})\|_2 \le \nicefrac{\lambda_{k+1}}{2}$ with high probability. We provide detailed proofs of presented result in the Appendix (see Section~\ref{sec:proofs_accelerated}).

\subsection{Strongly Convex Case}\label{sec:str_cvx_clipped_SSTM}
In this section we assume additionally that $f(x)$ is $\mu$-strongly convex. For this case we modify Algorithm~\ref{alg:clipped-SSTM} and propose a new method called Restarted Clipped Similar Triangles Method ({\tt R-clipped-SSTM}), see Algorithm~\ref{alg:R-clipped-SSTM}.
\begin{algorithm}[h]
\caption{Restarted Clipped Stochastic Similar Triangles Method ({\tt R-clipped-SSTM})}
\label{alg:R-clipped-SSTM}   
\begin{algorithmic}[1]
\Require starting point $x^0$, number of iterations $N_0$ of {\tt clipped-SSTM}, number of {\tt clipped-SSTM} runs, batchsizes $\{m_k^0\}_{k=0}^{N_0-1}$, $\{m_k^1\}_{k=0}^{N_0-1}$, \ldots, $\{m_k^\tau\}_{k=0}^{N_0-1}$, stepsize parameter $a$, clipping parameters $\{B_t\}_{t=0}$
\State Set $\hat x^0 = x^0$
\For{$t = 0,1,\ldots, \tau-1$}
\State Run {\tt clipped-SSTM} (Algorithm~\ref{alg:clipped-SSTM}) for $N_0$ iterations with batchsizes $\{m_k^t\}_{k=1}^{N_0}$, stepsize parameter $a$, clipping parameter $B_t$ and starting point $\hat x^t$. Define the output of {\tt clipped-SSTM} by $\hat x^{t+1}$.
\EndFor
\Ensure $\hat x^{\tau}$ 
\end{algorithmic}
\end{algorithm}
At each iteration {\tt R-clipped-SSTM} runs {\tt clipped-SSTM} for $N_0$ iterations from the current point $\hat x^k$ and use its output as next iterate $\hat x^{k+1}$. In literature this approach is known as the restarts technique \cite{dvurechensky2016stochastic, juditsky2011first, juditsky2014deterministic, nemirovsky1983problem}. Choosing $N_0$ and parameters $m_k$, $a$ and $B$ in a proper way one can get an accelerated method for strongly convex objectives. Theorem below states the main convergence result for {\tt R-clipped-SSTM}.
\begin{theorem}\label{thm:main_result_R_clipped_SSTM}
    Assume that $f$ is $\mu$-strongly convex and $L$-smooth. If we choose $\beta \in (0,1)$, $\tau$ and $N_0\ge 1$ such that
    \begin{equation}
        \ln\frac{4N_0\tau}{\beta} \ge 2,\quad N_0 \ge C\sqrt{\frac{8aL}{\mu}}, \label{eq:beta_N_tau_condition_R_clipped_SSTM}
    \end{equation}
    and
    \begin{equation}
        m_k^t = \max\left\{1,\frac{6000\cdot 2^{t}\sigma^2 \alpha_{k+1}^2N_0\ln\frac{4N_0\tau}{\beta}}{C^2R^2}, \frac{10368\cdot 2^{t}\sigma^2 \alpha_{k+1}^2N_0}{C^2R^2}\right\},\label{eq:bathces_R_clipped_SSTM}
    \end{equation}
    \begin{equation}
        B_t = \frac{CR}{8\cdot 2^t\ln\frac{4N_0\tau}{\beta}}, \label{eq:B_parameter_R_clipped_SSTM}
    \end{equation}
    \begin{equation}
        a \ge \max\left\{1,\frac{16\ln\frac{4N_0\tau}{\beta}}{C},36\left(2\ln\frac{4N_0\tau}{\beta} + \sqrt{4\ln^2\frac{4N_0\tau}{\beta} + 2\ln\frac{4N_0\tau}{\beta}}\right)^2\right\}, \label{eq:a_parameter_R_clipped_SSTM}
    \end{equation} 
    where $R = \sqrt{\frac{2(f(x^0) - f(x^*))}{\mu}}$ and $C = \sqrt{5}$, then we have that after $\tau$ runs of {\tt clipped-SSTM} in {\tt R-clipped-SSTM} the inequality
    \begin{equation}
        f(\hat x^\tau) - f(x^*) \le 2^{-\tau} \left(f(x^0)-f(x^*)\right) \label{eq:main_result_R_clipped_SSTM}
    \end{equation}
    holds with probability at least $1-\beta$. That is, if we choose $a$ to be equal to the maximum from \eqref{eq:a_parameter_R_clipped_SSTM} and $N_0 \le C_1\sqrt{\frac{8aL}{\mu}}$ with some numerical constant $C_1 \ge C$, then the method achieves $f(\hat x^{\tau}) - f(x^*) \le \varepsilon$ with probability at least $1-\beta$ after
    \begin{equation}
        O\left(\sqrt{\frac{L}{\mu}}\ln\left(\frac{\mu R^2}{\varepsilon}\right)\ln\left(\frac{L}{\mu\beta}\ln\frac{\mu R^2}{\varepsilon}\right)\right)\text{ iterations (in total)}
    \end{equation}
    of {\tt clipped-SSTM} and requires
    \begin{equation}
        O\left(\max\left\{\sqrt{\frac{L}{\mu}}\ln\frac{\mu R^2}{\varepsilon}, \frac{\sigma^2}{\mu\varepsilon}\right\}\ln\left(\frac{L}{\mu\beta}\ln\frac{\mu R^2}{\varepsilon}\right)\right)\text{ oracle calls.} \label{eq:R_clipped_SSTM_oracle_complexity}
    \end{equation}
\end{theorem}
In other words, {\tt R-clipped-SSTM} has the same convergence rate as optimal stochastic methods for strongly convex problems like Multi-Staged {\tt AC-SA} ({\tt MS-AC-SA}) \cite{ghadimi2013optimal} or Stochastic Similar Triangles Method for strongly convex problems ({\tt SSTM{\_}sc}) \cite{gasnikov2016universal,gorbunov2019optimal}. Moreover, in Theorem~\ref{thm:main_result_R_clipped_SSTM} we \textit{do not assume} that stochastic gradients are sampled from sub-Gaussian distribution while corresponding results for {\tt MS-AC-SA} and {\tt SSTM{\_}sc} are substantially based on the light tails assumption. Our bound outperforms the state-of-the-art result from \cite{davis2019low} in terms of the dependence on $\ln\frac{L}{\mu}$. It is worth to mention here that using special restarts technique Nazin et al. \cite{nazin2019algorithms} generalize their method ({\tt RSMD}) for the strongly convex case, but since {\tt RSMD} is not accelerated their approach gives only non-accelerated convergence rate.

We also emphasize that big numerical factors in formulas for $m_k^t$ and $a$ are needed only in our analysis and in practice they can be tuned. However, when $\sigma^2$ is big bathsizes $m_k^t$ become of the order $k^2\varepsilon^{-1}$. It can make the cost of one iteration extremely high, therefore, as for {\tt clipped-SSTM} we consider a different stepsize policy removing this drawback.
\begin{corollary}\label{cor:stepsizes_r_clipped_SSTM}
    Let the assumptions of Theorem~\ref{thm:main_result_R_clipped_SSTM} hold. Assume that conditions \eqref{eq:beta_N_tau_condition_R_clipped_SSTM}, \eqref{eq:bathces_R_clipped_SSTM}, \eqref{eq:B_parameter_R_clipped_SSTM} and \eqref{eq:a_parameter_R_clipped_SSTM} are satisfied for
    \begin{equation}
        a = \Theta\left(\frac{\sigma^4\ln^2\frac{N_0\tau}{\beta}}{L\mu\varepsilon^2}\right),\quad N_0 = \Theta\left(\sqrt{\frac{aL}{\mu}}\right). \label{eq:N_0_a_choice}
    \end{equation}
    Then after $\tau = \lceil\ln(\nicefrac{\mu R^2}{2\varepsilon})\rceil$ runs of {\tt clipped-SSTM} in {\tt R-clipped-SSTM} the method achieves $f(\hat x^\tau) - f(x^*) \le \varepsilon$ with probability at least $1 - \beta$. Moreover, the total number of iterations of {\tt clipped-SSTM} equals
    \begin{equation}
        O\left(\frac{\sigma^2}{\mu\varepsilon}\ln\left(\frac{\mu R^2}{\varepsilon}\right)\ln\left(\frac{\sigma^2}{\mu\varepsilon\beta}\ln\frac{\mu R^2}{\varepsilon}\right)\right) \label{eq:complexity_r_clipped_SSTM_small_steps}
    \end{equation}
    with $m_k^t = O(1)$ for all $k = 0,1,\ldots, N_0-1$, $t=0,1,\ldots,\tau-1$.
\end{corollary}
When $\sigma^2$ is big the obtained bound is comparable with bounds for {\tt restarted-RSMD} and {\tt proxBoost}, see Table~\ref{tab:str_cvx_case_comparison}.

\subsection{Proofs}\label{sec:proofs_accelerated}

\subsubsection{Proof of Lemma~\ref{lem:main_opt_lemma_clipped_SSTM}}
Using $z^{k+1} = z^k - \alpha_{k+1}\tnabla f(x^{k+1},\Bxi^k)$ we get that for all $z\in \R^n$
    \begin{eqnarray}
        \alpha_{k+1}\left\la\tnabla f(x^{k+1},\Bxi^k), z^k - z \right\ra &=& \alpha_{k+1}\left\la \tnabla f(x^{k+1},\Bxi^{k}), z^k - z^{k+1}\right\ra\notag\\
        &&\quad+ \alpha_{k+1}\left\la \tnabla f(x^{k+1},\Bxi^{k}), z^{k+1} - z\right\ra\notag\\
        &=& \alpha_{k+1}\left\la \tnabla f(x^{k+1},\Bxi^{k}), z^k - z^{k+1}\right\ra + \left\la z^{k+1}-z^k, z - z^{k+1}\right\ra \notag\\
        &\overset{\eqref{eq:inner_product_representation}}{\le}& \alpha_{k+1}\left\la \tnabla f(x^{k+1},\Bxi^{k}), z^k - z^{k+1}\right\ra - \frac{1}{2}\|z^k - z^{k+1}\|_2^2 \notag\\
        &&\quad+ \frac{1}{2}\|z^k - z\|_2^2 - \frac{1}{2}\|z^{k+1}-z\|_2^2.\label{eq:main_olc_clipped_SSTM_technical_2}
    \end{eqnarray}
    Next, we notice that
    \begin{equation}
        y^{k+1} = \frac{A_k y^k + \alpha_{k+1}z^{k+1}}{A_{k+1}} = \frac{A_k y^k + \alpha_{k+1}z^{k}}{A_{k+1}} + \frac{\alpha_{k+1}}{A_{k+1}}\left(z^{k+1}-z^k\right) = x^{k+1} + \frac{\alpha_{k+1}}{A_{k+1}}\left(z^{k+1}-z^k\right) \label{eq:main_olc_clipped_SSTM_technical_3}
    \end{equation}
    which implies:
    \begin{eqnarray}
        \alpha_{k+1}\left\la\tnabla f(x^{k+1},\Bxi^k), z^k - z \right\ra &\overset{\eqref{eq:theta_k+1_def_clipped_SSTM},\eqref{eq:main_olc_clipped_SSTM_technical_2}}{\le}& \alpha_{k+1}\left\la \nabla f(x^{k+1}), z^{k} - z^{k+1}\right\ra - \frac{1}{2}\|z^k - z^{k+1}\|_2^2 \notag\\
        &&\quad + \alpha_{k+1}\left\la \theta_{k+1}, z^{k} - z^{k+1}\right\ra+ \frac{1}{2}\|z^k - z\|_2^2 - \frac{1}{2}\|z^{k+1}-z\|_2^2\notag\\
        &\overset{\eqref{eq:main_olc_clipped_SSTM_technical_3}}{=}& A_{k+1}\left\la \nabla f(x^{k+1}), x^{k+1} - y^{k+1}\right\ra - \frac{1}{2}\|z^k - z^{k+1}\|_2^2 \notag\\
        &&\quad + \alpha_{k+1}\left\la \theta_{k+1}, z^{k} - z^{k+1}\right\ra + \frac{1}{2}\|z^k - z\|_2^2 - \frac{1}{2}\|z^{k+1}-z\|_2^2\notag\\
        &\overset{\eqref{eq:L_smoothness_cor}}{\le}& A_{k+1}\left(f(x^{k+1}) - f(y^{k+1})\right) + \frac{A_{k+1}L}{2}\|x^{k+1}-y^{k+1}\|_2^2\notag\\
        &&\quad - \frac{1}{2}\|z^k - z^{k+1}\|_2^2 + \alpha_{k+1}\left\la \theta_{k+1}, z^{k} - z^{k+1}\right\ra\notag\\
        &&\quad + \frac{1}{2}\|z^k - z\|_2^2 - \frac{1}{2}\|z^{k+1}-z\|_2^2\notag\\
        &\overset{\eqref{eq:main_olc_clipped_SSTM_technical_3}}{=}& A_{k+1}\left(f(x^{k+1}) - f(y^{k+1})\right) + \frac{1}{2}\left(\frac{\alpha_{k+1}^2L}{A_{k+1}} - 1\right)\|z^k - z^{k+1}\|_2^2\notag\\
        &&\quad + \alpha_{k+1}\left\la \theta_{k+1}, z^{k} - z^{k+1}\right\ra + \frac{1}{2}\|z^k - z\|_2^2 - \frac{1}{2}\|z^{k+1}-z\|_2^2.\notag
    \end{eqnarray}
    Since $A_{k+1} \ge aL\alpha_{k+1}^2$ (see Lemma~\ref{lem:alpha_k}) and $a\ge 1$ we can continue our derivations:
    \begin{eqnarray}
        \alpha_{k+1}\left\la\tnabla f(x^{k+1},\Bxi^k), z^k - z \right\ra &\le& A_{k+1}\left(f(x^{k+1}) - f(y^{k+1})\right) + \alpha_{k+1}\left\la \theta_{k+1}, z^{k} - z^{k+1}\right\ra\notag\\
        &&\quad + \frac{1}{2}\|z^k - z\|_2^2 - \frac{1}{2}\|z^{k+1}-z\|_2^2. \label{eq:main_olc_clipped_SSTM_technical_4}
    \end{eqnarray}
    Next, due to convexity of $f$ we have
    \begin{eqnarray}
        \left\la\tnabla f(x^{k+1},\Bxi^k), y^k - x^{k+1}\right\ra &\overset{\eqref{eq:theta_k+1_def_clipped_SSTM}}{=}& \left\la\nabla f(x^{k+1}), y^k - x^{k+1}\right\ra + \left\la \theta_{k+1}, y^k - x^{k+1}\right\ra\notag\\
        &\le& f(y^k) - f(x^{k+1}) + \left\la \theta_{k+1}, y^k - x^{k+1}\right\ra.\label{eq:main_olc_clipped_SSTM_technical_5}
    \end{eqnarray}
    By definition of $x^{k+1}$ we have $x^{k+1} = \frac{A_k y^k + \alpha_{k+1} z^k}{A_{k+1}}$ which implies
    \begin{equation}
        \alpha_{k+1}\left(x^{k+1} - z^k\right) = A_k\left(y^k - x^{k+1}\right)\label{eq:main_olc_clipped_SSTM_technical_6}
    \end{equation}
    since $A_{k+1} = A_k + \alpha_{k+1}$. Putting all together we derive that
    \begin{eqnarray*}
        \alpha_{k+1}\left\la\tnabla f(x^{k+1},\Bxi^k), x^{k+1} - z \right\ra &=& \alpha_{k+1}\left\la\tnabla f(x^{k+1},\Bxi^k), x^{k+1} - z^k \right\ra \\
        &&\quad+ \alpha_{k+1}\left\la\tnabla f(x^{k+1},\Bxi^k), z^k - z \right\ra\\
        &\overset{\eqref{eq:main_olc_clipped_SSTM_technical_6}}{=}& A_{k}\left\la\tnabla f(x^{k+1},\Bxi^k), y^{k} - x^{k+1} \right\ra\\
        &&\quad+ \alpha_{k+1}\left\la\tnabla f(x^{k+1},\Bxi^k), z^k - z \right\ra\\
        &\overset{\eqref{eq:main_olc_clipped_SSTM_technical_5},\eqref{eq:main_olc_clipped_SSTM_technical_4}}{\le}& A_k\left(f(y^k)-f(x^{k+1})\right) + A_k\left\la \theta_{k+1}, y^k - x^{k+1}\right\ra\\
        &&\quad + A_{k+1}\left(f(x^{k+1}) - f(y^{k+1})\right) + \alpha_{k+1}\left\la \theta_{k+1}, z^{k} - z^{k+1}\right\ra\\
        &&\quad + \frac{1}{2}\|z^k - z\|_2^2 - \frac{1}{2}\|z^{k+1}-z\|_2^2\\
        &\overset{\eqref{eq:main_olc_clipped_SSTM_technical_6}}{=}& A_kf(y^k) - A_{k+1}f(y^{k+1}) + \alpha_{k+1}\left\la \theta_{k+1}, x^{k+1}-z^k\right\ra\\
        &&\quad + \alpha_{k+1}f(x^{k+1}) + \alpha_{k+1}\left\la \theta_{k+1}, z^{k} - z^{k+1}\right\ra\\
        &&\quad + \frac{1}{2}\|z^k - z\|_2^2 - \frac{1}{2}\|z^{k+1}-z\|_2^2\\
        &\le& A_kf(y^k) - A_{k+1}f(y^{k+1}) + \alpha_{k+1}f(x^{k+1})\\
        &&\quad + \alpha_{k+1}\left\la \theta_{k+1}, x^{k+1} - z^{k+1}\right\ra\\
        &&\quad + \frac{1}{2}\|z^k - z\|_2^2 - \frac{1}{2}\|z^{k+1}-z\|_2^2.
    \end{eqnarray*}
    Rearranging the terms we get
    \begin{eqnarray*}
        A_{k+1}f(y^{k+1}) - A_kf(y^k) &\le& \alpha_{k+1}\left(f(x^{k+1}) + \left\la\tnabla f(x^{k+1},\Bxi^k), z-x^{k+1}\right\ra\right) + \frac{1}{2}\|z^k - z\|_2^2\\
        &&\quad - \frac{1}{2}\|z^{k+1} - z\|_2^2 + \alpha_{k+1}\left\la \theta_{k+1}, x^{k+1} - z^{k+1}\right\ra\\
        &\overset{\eqref{eq:theta_k+1_def_clipped_SSTM}}{=}& \alpha_{k+1}\left(f(x^{k+1}) + \left\la\nabla f(x^{k+1}), z-x^{k+1}\right\ra\right)\\
        &&\quad + \alpha_{k+1}\left\la\theta_{k+1}, z-x^{k+1}\right\ra + \frac{1}{2}\|z^k - z\|_2^2 - \frac{1}{2}\|z^{k+1} - z\|_2^2\\
        &&\quad + \alpha_{k+1}\left\la \theta_{k+1}, x^{k+1} - z^{k+1}\right\ra \\
        &\le& \alpha_{k+1}f(z) + \frac{1}{2}\|z^k - z\|_2^2 - \frac{1}{2}\|z^{k+1} - z\|_2^2 + \alpha_{k+1}\left\la \theta_{k+1}, z - z^{k+1}\right\ra
    \end{eqnarray*}
    where in the last inequality we use the convexity of $f$. Taking into account $A_0 = \alpha_0 = 0$ and $A_{N} = \sum_{k=0}^{N-1}\alpha_{k+1}$ we sum up these inequalities for $k=0,\ldots,N-1$ and get
    \begin{eqnarray*}
        A_Nf(y^N) &\le& A_N f(z) + \frac{1}{2}\|z^0 - z\|_2^2 - \frac{1}{2}\|z^{N} - z\|_2^2 + \sum\limits_{k=0}^{N-1}\alpha_{k+1}\left\la \theta_{k+1}, z - z^{k+1}\right\ra\\
        &=& A_N f(z) + \frac{1}{2}\|z^0 - z\|_2^2 - \frac{1}{2}\|z^{N} - z\|_2^2 + \sum\limits_{k=0}^{N-1}\alpha_{k+1}\left\la \theta_{k+1}, z - z^{k}\right\ra\\
        &&\quad + \sum\limits_{k=0}^{N-1}\alpha_{k+1}^2\left\la\theta_{k+1},\tnabla f(x^{k+1},\Bxi^k) \right\ra\\
        &\overset{\eqref{eq:theta_k+1_def_clipped_SSTM}}{=}& A_N f(z) + \frac{1}{2}\|z^0 - z\|_2^2 - \frac{1}{2}\|z^{N} - z\|_2^2 + \sum\limits_{k=0}^{N-1}\alpha_{k+1}\left\la \theta_{k+1}, z - z^{k}\right\ra\\
        &&\quad + \sum\limits_{k=0}^{N-1}\alpha_{k+1}^2\left\|\theta_{k+1}\right\|_2^2 + \sum\limits_{k=0}^{N-1}\alpha_{k+1}^2\left\la\theta_{k+1},\nabla f(x^{k+1})\right\ra
    \end{eqnarray*}
    which concludes the proof.

\subsubsection{Proof of Lemma~\ref{lem:main_stoch_lemma_clipped_SSTM}}
\textbf{Proof of \eqref{eq:magnitude_bound_clipped_SSTM}.} By definition of $\tnabla f(x^{k+1},\Bxi^{k})$ we have that $\|\tnabla f(x^{k+1},\Bxi^{k})\|_2 \le \lambda_{k+1}$ and, as a consequence, $\left\|\EE_{\Bxi^k}[\tnabla f(x^{k+1},\Bxi^{k})]\right\|_2 \le \lambda_{k+1}$. Using this we get
    \begin{equation*}
        \left\|\tnabla f(x^{k+1},\Bxi^{k}) - \EE_{\Bxi^k}\left[\tnabla f(x^{k+1},\Bxi^{k})\right]\right\|_2 \le \left\|\tnabla f(x^{k+1},\Bxi^{k})\right\|_2 + \left\|\EE_{\Bxi^k}\left[\tnabla f(x^{k+1},\Bxi^{k})\right]\right\|_2 \le 2\lambda_{k+1}.
    \end{equation*}
    
    \textbf{Proof of \eqref{eq:bias_bound_clipped_SSTM}.} In order to prove this bound we introduce following indicator random variables:
    \begin{equation}
        \chi_k \eqdef \obf_{\|\nabla f(x^{k+1},\Bxi^k)\|_2 > \lambda_{k+1}},\quad \eta_k \eqdef \obf_{\|\nabla f(x^{k+1},\Bxi^k) - \nabla f(x^{k+1})\|_2 > \frac{1}{2}\lambda_{k+1}}. \label{eq:chi_eta_def_clipped_SSTM}
    \end{equation}
    From the assumptions of the lemma, we have that $\|\nabla f(x^{k+1})\|_2 \le \frac{\lambda_{k+1}}{2}$ which implies
    \begin{eqnarray*}
        \left\|\nabla f(x^{k+1},\Bxi^k)\right\|_2 &\le& \left\|\nabla f(x^{k+1},\Bxi^k) - \nabla f(x^{k+1})\right\|_2 + \left\|\nabla f(x^{k+1})\right\|_2\\
        &\le& \left\|\nabla f(x^{k+1},\Bxi^k) - \nabla f(x^{k+1})\right\|_2 + \frac{\lambda_{k+1}}{2},
    \end{eqnarray*}
    hence
    \begin{equation}
        \chi_k \le \eta_k. \label{eq:chi_smaller_eta_clipped_SSTM}
    \end{equation}
    The introduced notation helps us to rewrite $\tnabla f(x^{k+1},\Bxi^k)$ in the following way:
    \begin{eqnarray}
        \tnabla f(x^{k+1},\Bxi^k) &=& \nabla f(x^{k+1},\Bxi^k)(1-\chi_k) + \frac{\lambda_{k+1}}{\left\|\nabla f(x^{k+1},\Bxi)\right\|_2}\nabla f(x^{k+1},\Bxi^k)\chi_k\label{eq:stoch_grad_representation_1_clipped_SSTM}\\
        &=& \nabla f(x^{k+1},\Bxi^k) + \left(\frac{\lambda_{k+1}}{\left\|\nabla f(x^{k+1},\Bxi^k)\right\|_2} - 1\right)\nabla f(x^{k+1},\Bxi^k)\chi_k. \label{eq:stoch_grad_representation_2_clipped_SSTM}
    \end{eqnarray}
    We use this representation to obtain the following inequality:
    \begin{eqnarray}
        \left\|\EE_{\Bxi^k}\left[\tnabla f(x^{k+1},\Bxi^{k})\right] - \nabla f(x^{k+1})\right\|_2 &\overset{\eqref{eq:mini_batched_unbiasedness_clipped_SSTM},\eqref{eq:stoch_grad_representation_2_clipped_SSTM}}{=}& \left\|\EE_{\Bxi^k}\left[\left(\frac{\lambda_{k+1}}{\left\|\nabla f(x^{k+1},\Bxi^k)\right\|_2} - 1\right)\nabla f(x^{k+1},\Bxi^k)\chi_k\right]\right\|_2\notag\\
        &\le& \EE_{\Bxi^k}\left[\left\|\nabla f(x^{k+1},\Bxi^k)\right\|_2\cdot\left|\frac{\lambda_{k+1}}{\left\|\nabla f(x^{k+1},\Bxi^k)\right\|_2} - 1\right|\chi_k\right]\notag\\
        &\overset{\eqref{eq:chi_eta_def_clipped_SSTM}}{=}& \EE_{\Bxi^k}\left[\left\|\nabla f(x^{k+1},\Bxi^k)\right\|_2\cdot\left(1-\frac{\lambda_{k+1}}{\left\|\nabla f(x^{k+1},\Bxi^k)\right\|_2}\right)\chi_k\right]\notag\\
        &\overset{\eqref{eq:chi_eta_def_clipped_SSTM}}{\le}& \EE_{\Bxi^k}\left[\left\|\nabla f(x^{k+1},\Bxi^k)\right\|_2\chi_k\right]\notag\\
        &\overset{\eqref{eq:chi_smaller_eta_clipped_SSTM}}{\le}& \EE_{\Bxi^k}\left[\left\|\nabla f(x^{k+1},\Bxi^k)\right\|_2\eta_k\right]\notag\\
        &\le& \EE_{\Bxi^k}\left[\left\|\nabla f(x^{k+1},\Bxi^k) - \nabla f(x^{k+1})\right\|_2\eta_k\right]\notag\\
        &&\quad + \left\|\nabla f(x^{k+1})\right\|_2\EE_{\Bxi^k}\left[\eta_k\right]\notag\\
        &\le& \sqrt{\EE_{\Bxi^k}\left[\left\|\nabla f(x^{k+1},\Bxi^k) - \nabla f(x^{k+1})\right\|_2^2\right]\EE_{\Bxi^k}[\eta_k^2]}\notag\\
        &&\quad + \left\|\nabla f(x^{k+1})\right\|_2\EE_{\Bxi^k}\left[\eta_k\right]\notag\\
        &\overset{\eqref{eq:mini_batched_bounded_variance_clipped_SSTM}}{\le}& \frac{\sigma}{\sqrt{m_k}}\sqrt{\EE_{\Bxi^k}\left[\eta_k^2\right]} + \frac{\lambda_{k+1}}{2}\EE_{\Bxi^k}\left[\eta_k\right]. \label{eq:main_stoch_lemma_clipped_SSTM_technical_1}
    \end{eqnarray}
    Next, we derive an upper bound for the expectation of $\eta_k$ using Markov's inequality:
    \begin{eqnarray}
        \EE_{\Bxi^k}\left[\eta_k\right] &=& \EE_{\Bxi^k}\left[\eta_k^2\right] = \PP_{\Bxi^k}\{\eta_k = 1\}\notag\\
        &\overset{\eqref{eq:chi_eta_def_clipped_SSTM}}{=}& \PP_{\Bxi^k}\left\{\left\|\nabla f(x^{k+1},\Bxi^k) - \nabla f(x^{k+1})\right\|_2 > \frac{\lambda_{k+1}}{2}\right\} \notag\\
        &\le& \frac{4\EE_{\Bxi^k}\left[\left\|\nabla f(x^{k+1},\Bxi^k) - \nabla f(x^{k+1})\right\|_2^2\right]}{\lambda_{k+1}^2} \overset{\eqref{eq:mini_batched_bounded_variance_clipped_SSTM}}{\le} \frac{4\sigma^2}{m_k\lambda_{k+1}^2}. \label{eq:eta_expectation_bound_clipped_SSTM}
    \end{eqnarray}
    Putting all together we derive \eqref{eq:bias_bound_clipped_SSTM}:
    \begin{eqnarray*}
        \left\|\EE_{\Bxi^k}\left[\tnabla f(x^{k+1},\Bxi^{k})\right] - \nabla f(x^{k+1})\right\|_2 &\overset{\eqref{eq:main_stoch_lemma_clipped_SSTM_technical_1},\eqref{eq:eta_expectation_bound_clipped_SSTM}}{\le}& \frac{2\sigma^2}{m_k\lambda_{k+1}} + \frac{\lambda_{k+1}}{2}\cdot\frac{4\sigma^2}{m_k\lambda_{k+1}^2} = \frac{4\sigma^2}{m_k\lambda_{k+1}}.
    \end{eqnarray*}
    
    \textbf{Proof of \eqref{eq:distortion_bound_clipped_SSTM}.} Recall that in the space of random variables with finite second moment, i.e.\ in $L_2$, one can introduce a norm as $\sqrt{\EE|X|^2}$ for an arbitrary random variable $X$ from this space. Using triangle inequality for this norm we get
    \begin{eqnarray*}
        \sqrt{\EE_{\Bxi^k}\left[\left\|\nabla f(x^{k+1},\Bxi^{k}) - \nabla f(x^{k+1})\right\|_2^2\right]} &\overset{\eqref{eq:stoch_grad_representation_1_clipped_SSTM}}{\le}& \sqrt{\EE_{\Bxi^k}\left[\left\|\frac{\lambda_{k+1}\nabla f(x^{k+1},\Bxi^k)}{\left\|\nabla f(x^{k+1},\Bxi^k)\right\|_2} - \nabla f(x^{k+1})\right\|_2^2\chi_k^2\right]}\\
        &&\quad +  \sqrt{\EE_{\Bxi^k}\left[\left\|\nabla f(x^{k+1},\Bxi^{k}) - \nabla f(x^{k+1})\right\|_2^2(1-\chi_k)^2\right]}\\
        &\overset{\eqref{eq:squared_norm_sum}}{\le}& \sqrt{\EE_{\Bxi^k}\left[\left(2\left\|\frac{\lambda_{k+1}\nabla f(x^{k+1},\Bxi^k)}{\left\|\nabla f(x^{k+1},\Bxi^k)\right\|_2}\right\|_2^2 + 2\left\|\nabla f(x^{k+1})\right\|_2^2\right)\chi_k^2\right]}\\
        &&\quad +\sqrt{\EE_{\Bxi^k}\left[\left\|\nabla f(x^{k+1},\Bxi^{k}) - \nabla f(x^{k+1})\right\|_2^2\right]}\\
        &\overset{\eqref{eq:mini_batched_bounded_variance_clipped_SSTM}}{\le}& \sqrt{\frac{5}{2}}\lambda_{k+1}\sqrt{\EE_{\Bxi^k}\left[\chi_k^2\right]} + \frac{\sigma}{\sqrt{m_k}}\\
        &\overset{\eqref{eq:chi_smaller_eta_clipped_SSTM},\eqref{eq:eta_expectation_bound_clipped_SSTM}}{\le}& \sqrt{\frac{5}{2}}\lambda_{k+1}\cdot\frac{2\sigma}{\sqrt{m_k}\lambda_{k+1}} + \frac{\sigma}{\sqrt{m_k}} = \left(\sqrt{10}+1\right)\frac{\sigma}{\sqrt{m_k}}\\
        &\le& \frac{\sqrt{18}\sigma}{\sqrt{m_k}}.
    \end{eqnarray*}
    
     \textbf{Proof of \eqref{eq:variance_bound_clipped_SSTM}.} To derive \eqref{eq:variance_bound_clipped_SSTM} we use \eqref{eq:distortion_bound_clipped_SSTM}:
     \begin{eqnarray*}
         \EE_{\Bxi^k}\left[\left\|\tnabla f(x^{k+1},\Bxi^{k}) - \EE_{\Bxi^k}\left[\tnabla f(x^{k+1},\Bxi^{k})\right]\right\|_2^2\right] &\overset{\eqref{eq:variance_decomposition_2}}{\le}& \EE_{\Bxi^k}\left[\left\|\tnabla f(x^{k+1},\Bxi^{k}) - \nabla f(x^{k+1})\right\|_2^2\right]\\
         &\overset{\eqref{eq:distortion_bound_clipped_SSTM}}{\le}& \frac{18\sigma^2}{m_k}.
     \end{eqnarray*}

\subsubsection{Proof of Theorem~\ref{thm:main_result_clipped_SSTM}}
 Lemma~\ref{lem:main_opt_lemma_clipped_SSTM} implies that the inequality
    \begin{eqnarray}
        A_N\left(f(y^N) - f(x^*)\right) &\le& \frac{1}{2}\|z^0 - x^*\|_2^2 - \frac{1}{2}\|z^{N} - x^*\|_2^2 + \sum\limits_{k=0}^{N-1}\alpha_{k+1}\left\la \theta_{k+1}, x^* - z^{k}\right\ra\notag\\
        &&\quad + \sum\limits_{k=0}^{N-1}\alpha_{k+1}^2\left\|\theta_{k+1}\right\|_2^2 + \sum\limits_{k=0}^{N-1}\alpha_{k+1}^2\left\la\theta_{k+1},\nabla f(x^{k+1})\right\ra,\label{eq:main_thm_clipped_SSTM_technical_0}\\
        \theta_{k+1} &\eqdef& \tnabla f(x^{k+1},\Bxi^k) - \nabla f(x^{k+1})\label{eq:main_thm_clipped_SSTM_technical_theta}
    \end{eqnarray}
    holds for all $N\ge 0$. Taking into account that $f(y^N) - f(x^*) \ge 0$ for all $y^N$ and using new notation $R_k \eqdef \|z^k - x^*\|_2$, $\widetilde{R}_0 = R_0$, $\widetilde{R}_{k+1} = \max\{\widetilde{R}_k, R_{k+1}\}$ we derive that for all $k\ge 0$
    \begin{equation}
        R_k^2 \le R_0^2 + 2\sum\limits_{l=0}^{k-1}\alpha_{l+1}\left\la\theta_{l+1}, x^* - z^l\right\ra + 2\sum\limits_{l=0}^{k-1}\alpha_{l+1}^2\left\la\theta_{l+1}, \nabla f(x^{l+1})\right\ra + 2\sum\limits_{l=0}^{k-1}\alpha_{l+1}^2\|\theta_{l+1}\|_2^2. \label{eq:main_thm_clipped_SSTM_technical_1}
    \end{equation}
    
    First of all, we notice that for each $k\ge 0$ iterates $x^{k+1}, z^k, y^k$ lie in the ball $B_{\widetilde{R}_k}(x^*)$. We prove it using induction. Since $y^0 = z^0 = x^0$, $\widetilde{R}_0 = R_0 = \|z^0 - x^*\|_2$ and $x^1 = \frac{A_0y^0 + \alpha_1z^0}{A_1} = z^0$ we have that $x^{1}, z^0, y^0 \in B_{\widetilde{R}_0}(x^*)$. Next, assume that $x^{l}, z^{l-1}, y^{l-1} \in B_{\widetilde{R}_{l-1}}(x^*)$ for some $l\ge 1$. By definitions of $R_l$ and $\widetilde{R}_l$ we have that $z^l \in B_{R_{l}}(x^*)\subseteq B_{\widetilde{R}_{l}}(x^*)$. Since $y^l$ is a convex combination of $y^{l-1}\in B_{\widetilde{R}_{l-1}}(x^*)\subseteq B_{\widetilde{R}_{l}}(x^*)$, $z^l\in B_{\widetilde{R}_{l}}(x^*)$ and $B_{\widetilde{R}_{l}}(x^*)$ is a convex set we conclude that $y^l \in B_{\widetilde{R}_{l}}(x^*)$. Finally, since $x^{l+1}$ is a convex combination of $y^l$ and $z^l$ we have that $x^{l+1}$ lies in $B_{\widetilde{R}_{l}}(x^*)$ as well.
    
    The rest of the proof is based on the refined analysis of inequality \eqref{eq:main_thm_clipped_SSTM_technical_1}. In particular, via induction we prove that for all $k=0,1,\ldots, N$ with probability at least $1 - \frac{k\beta}{N}$ the following statement holds: inequalities
    \begin{eqnarray}
         R_t^2 &\overset{\eqref{eq:main_thm_clipped_SSTM_technical_1}}{\le}& R_0^2 + 2\sum\limits_{l=0}^{t-1}\alpha_{l+1}\left\la\theta_{l+1}, x^* - z^l\right\ra + 2\sum\limits_{l=0}^{t-1}\alpha_{l+1}^2\left\la\theta_{l+1}, \nabla f(x^{l+1})\right\ra + 2\sum\limits_{l=0}^{t-1}\alpha_{k+1}^2\|\theta_{l+1}\|_2^2\notag\\
         &\le& C^2R_0^2 \label{eq:main_thm_clipped_SSTM_technical_2}
    \end{eqnarray}
    hold for $t=0,1,\ldots,k$ simultaneously where $C$ is defined in \eqref{eq:C_definition_clipped_SSTM}. Let us define the probability event when this statement holds as $E_k$. Then, our goal is to show that $\PP\{E_k\} \ge 1 - \frac{k\beta}{N}$ for all $k = 0,1,\ldots,N$. For $t = 0$ inequality \eqref{eq:main_thm_clipped_SSTM_technical_2} holds with probability $1$ since $C \ge 1$, hence $\PP\{E_0\} = 1$. Next, assume that for some $k = T-1 \le N-1$ we have $\PP\{E_k\} = \PP\{E_{T-1}\} \ge 1 - \frac{(T-1)\beta}{N}$. Let us prove that $\PP\{E_{T}\} \ge 1 - \frac{T\beta}{N}$. First of all, probability event $E_{T-1}$ implies that
    \begin{eqnarray}
        f(y^t) - f(x^*) &\overset{\eqref{eq:main_thm_clipped_SSTM_technical_0}}{\le}& \frac{1}{A_t}\left(\frac{1}{2}R_0^2 + \sum\limits_{l=0}^{t-1}\alpha_{l+1}\left\la\theta_{l+1}, x^* - z^l+\alpha_{l+1}\nabla f(x^{l+1})\right\ra + \sum\limits_{l=0}^{t-1}\alpha_{k+1}^2\|\theta_{l+1}\|_2^2\right)\notag\\ &\overset{\eqref{eq:main_thm_clipped_SSTM_technical_2}}{\le}& \frac{C^2R_0^2}{2A_t}\label{eq:main_thm_clipped_SSTM_technical_3}
    \end{eqnarray}
    hold for $t=0,1,\ldots,T-1$. Then, inequalities
    \begin{eqnarray}
        \left\|\nabla f(x^{1})\right\|_2 &=& \left\|\nabla f(z^{0})\right\|_2 \overset{\eqref{eq:L_smoothness}}{\le} L\|z^0 - x^*\|_2 = \frac{1}{a}\cdot \frac{R_0}{\alpha_{1}},\notag\\
        \left\|\nabla f(x^{t+1})\right\|_2 &\le& \left\|\nabla f(x^{t+1}) - \nabla f(y^t)\right\|_2 + \left\|\nabla f(y^t)\right\|_2\notag\\
        &\overset{\eqref{eq:L_smoothness},\eqref{eq:L_smoothness_cor_2}}{\le}& L\|x^{t+1} - y^t\|_2 + \sqrt{2L(f(y^t) - f(x^*))}\notag\\
        &\overset{\eqref{eq:main_olc_clipped_SSTM_technical_6},\eqref{eq:main_thm_clipped_SSTM_technical_3}}{\le}& \frac{\alpha_{t+1}L}{A_t}\|x^{t+1}-z^k\|_2 + \sqrt{\frac{LC^2R_0^2}{A_t}}\notag\\
        &\overset{\eqref{eq:A_k+1_formula}}{\le}& \frac{2L(t+2)}{t(t+3)}\left(\|x^{k+1}-x^*\|_2 + \|x^*-z^k\|_2\right) + \frac{2LCR_0\sqrt{a}}{\sqrt{t(t+3)}}\notag\\
        &\le& \frac{4L(t+2)\widetilde{R}_k}{t(t+3)} + \frac{2LCR_0\sqrt{a}}{\sqrt{t(t+3)}}\notag\\
        &\overset{\eqref{eq:main_thm_clipped_SSTM_technical_2}}{\le}& \frac{2aLCR_0}{t+2}\left(\frac{2(t+2)^2}{at(t+3)} + \frac{t+2}{\sqrt{at(t+3)}}\right)\notag\\
        &\le& \frac{CR_0}{\alpha_{t+1}}\left(\frac{9}{2a} + \frac{3}{2\sqrt{a}}\right)\notag
    \end{eqnarray}
    hold for $t=1,\ldots,T-1$ where the last inequality follows from $\frac{(t+2)^2}{t(t+3)} \le \frac{(1+2)^2}{1(1+3)} = \frac{9}{4}$. Taking $a$ such that
    \begin{equation*}
        a \ge \frac{2R_0}{B}\quad \text{and}\quad \frac{9}{2a} + \frac{3}{2\sqrt{a}} \le \frac{B}{2CR_0}
    \end{equation*}
    we obtain that probability event $E_{T-1}$ implies
    \begin{eqnarray}
        \left\|\nabla f(x^{t+1})\right\|_2 &\le& \frac{B}{2\alpha_{t+1}} = \frac{\lambda_{t+1}}{2} \label{eq:main_thm_clipped_SSTM_technical_4}
    \end{eqnarray}
    for $t=0,\ldots,T-1$. Since $B = \frac{CR_0}{8\ln\frac{4N}{\beta}}$ we have to choose such $a$ that
    \begin{equation*}
        a \ge \frac{16\ln\frac{4N}{\beta}}{C}\quad \text{and}\quad \frac{9}{a} + \frac{3}{\sqrt{a}} \le \frac{1}{8\ln\frac{4N}{\beta}}.
    \end{equation*}
    Solving quadratic inequality 
    \begin{equation*}
        a - 24\sqrt{a}\ln\frac{4N}{\beta} - 72\ln\frac{4N}{\beta} \ge 0
    \end{equation*}
    w.r.t.\ $\sqrt{a}$ we get that $a$ should satisfy
    \begin{equation*}
         a \ge \max\left\{\frac{16\ln\frac{4N}{\beta}}{C},36\left(2\ln\frac{4N}{\beta} + \sqrt{4\ln^2\frac{4N}{\beta} + 2\ln\frac{4N}{\beta}}\right)^2\right\}.
    \end{equation*}
    
    Having inequalities \eqref{eq:main_thm_clipped_SSTM_technical_4} in hand we show in the rest of the proof that \eqref{eq:main_thm_clipped_SSTM_technical_2} holds for $t = T$ with big enough probability. First of all, we introduce new random variables:
    \begin{equation}
        \eta_l = \begin{cases}x^* - z^l,&\text{if } \|x^* - z^l\|_2 \le CR_0,\\ 0,&\text{otherwise,} \end{cases}\quad \text{and}\quad \zeta_l = \begin{cases}\nabla f(x^{l+1}),&\text{if } \|\nabla f(x^{l+1})\|_2 \le \frac{B}{2\alpha_{l+1}},\\ 0,&\text{otherwise,} \end{cases} \label{eq:main_thm_clipped_SSTM_technical_4_1}
    \end{equation}
    for $l=0,1,\ldots T-1$. Note that these random variables are bounded with probability $1$, i.e.\ with probability $1$ we have
    \begin{equation}
        \|\eta_l\|_2 \le CR_0\quad \text{and}\quad \|\zeta_l\|_2 \le \frac{B}{2\alpha_{l+1}}.\label{eq:main_thm_clipped_SSTM_technical_4_2}
    \end{equation}
    Secondly, we use the introduced notation and get that $E_{T-1}$ implies 
    \begin{eqnarray*}
         R_T^2 &\overset{\eqref{eq:main_thm_clipped_SSTM_technical_1},\eqref{eq:main_thm_clipped_SSTM_technical_2},\eqref{eq:main_thm_clipped_SSTM_technical_4},\eqref{eq:main_thm_clipped_SSTM_technical_4_1}}{\le}& R_0^2 + 2\sum\limits_{l=0}^{T-1}\alpha_{l+1}\left\la\theta_{l+1}, \eta_l\right\ra + 2\sum\limits_{l=0}^{T-1}\alpha_{l+1}^2\|\theta_{l+1}\|_2^2 + 2\sum\limits_{l=0}^{T-1}\alpha_{l+1}^2\left\la\theta_{l+1}, \zeta_l\right\ra\notag\\
        &=& R_0^2 + \sum\limits_{l=0}^{T-1}\alpha_{l+1}\left\la\theta_{l+1}, 2\eta_l + 2\alpha_{l+1}\zeta_l\right\ra + 2\sum\limits_{l=0}^{T-1}\alpha_{l+1}^2\|\theta_{l+1}\|_2^2.
    \end{eqnarray*}
    Finally, we do some preliminaries in order to apply Bernstein's inequality (see Lemma~\ref{lem:Bernstein_ineq}) and obtain that $E_{T-1}$ implies
    \begin{eqnarray}
        R_T^2 &\overset{\eqref{eq:squared_norm_sum}}{\le}& R_0^2 + \underbrace{\sum\limits_{l=0}^{T-1}\alpha_{l+1}\left\la\theta_{l+1}^u, 2\eta_l + 2\alpha_{l+1}\zeta_l\right\ra}_{\circledOne}+ \underbrace{\sum\limits_{l=0}^{T-1}\alpha_{l+1}\left\la\theta_{l+1}^b, 2\eta_l + 2\alpha_{l+1}\zeta_l\right\ra}_{\circledTwo}\notag\\
        &&\quad + \underbrace{\sum\limits_{l=0}^{T-1}4\alpha_{l+1}^2\left(\|\theta_{l+1}^u\|_2^2 - \EE_{\Bxi^l}\left[\|\theta_{l+1}^u\|_2^2\right]\right)}_{\circledThree} + \underbrace{\sum\limits_{l=0}^{T-1}4\alpha_{l+1}^2\EE_{\Bxi^l}\left[\|\theta_{l+1}^u\|_2^2\right]}_{\circledFour}\notag\\
        &&\quad + \underbrace{\sum\limits_{l=0}^{T-1}4\alpha_{l+1}^2\|\theta_{l+1}^b\|_2^2}_{\circledFive}\label{eq:main_thm_clipped_SSTM_technical_5}
    \end{eqnarray}
    where we introduce new notations:
    \begin{equation}
        \theta_{l+1}^u \eqdef \tnabla f(x^{l+1},\Bxi^l) - \EE_{\Bxi^l}\left[\tnabla f(x^{l+1},\Bxi^l)\right],\quad \theta_{l+1}^b \eqdef \EE_{\Bxi^l}\left[\tnabla f(x^{l+1},\Bxi^l)\right] - \nabla f(x^{l+1}),\label{eq:main_thm_clipped_SSTM_technical_6}
    \end{equation}
    \begin{equation}
        \theta_{l+1} \overset{\eqref{eq:theta_k+1_def_clipped_SSTM}}{=} \theta_{l+1}^u +\theta_{l+1}^b.\notag
    \end{equation}
    It remains to provide tight upper bounds for $\circledOne$, $\circledTwo$, $\circledThree$, $\circledFour$ and $\circledFive$, i.e.\ in the remaining part of the proof we show that $\circledOne+\circledTwo+\circledThree+\circledFour+\circledFive \le \delta C^2R_0^2$ for some $\delta < 1$.
    
    \textbf{Upper bound for $\circledOne$.} First of all, since $\EE_{\Bxi^l}[\theta_{l+1}^u] = 0$ summands in $\circledOne$ are conditionally unbiased:
    \begin{equation*}
        \EE_{\Bxi^l}\left[\alpha_{l+1}\left\la\theta_{l+1}^u, 2\eta_l + 2\alpha_{l+1}\zeta_l\right\ra\right] = 0.
    \end{equation*}
    Secondly, these summands are bounded with probability $1$:
    \begin{eqnarray*}
        \left|\alpha_{l+1}\left\la\theta_{l+1}^u, 2\eta_l + 2\alpha_{l+1}\zeta_l\right\ra\right| &\le& \alpha_{l+1}\|\theta_{l+1}^u\|_2\left\|2\eta_l + 2\alpha_{l+1}\zeta_l\right\|_2\\
        &\overset{\eqref{eq:magnitude_bound_clipped_SSTM},\eqref{eq:main_thm_clipped_SSTM_technical_4_2}}{\le}& 2\alpha_{l+1}\lambda_{l+1}\left(2CR_0 + B\right) = 2B(2CR_0+B)\\
        &=& \frac{C^2R_0^2}{2\ln\frac{4N}{\beta}} + \frac{C^2R_0^2}{32\ln^2\frac{4N}{\beta}}\\
        &\overset{\eqref{eq:beta_N_condition_clipped_SSTM}}{\le}& \frac{C^2R_0^2}{2\ln\frac{4N}{\beta}} + \frac{C^2R_0^2}{64\ln\frac{4N}{\beta}} \le \frac{33C^2R_0^2}{64\ln\frac{4N}{\beta}}.
    \end{eqnarray*}
    Finally, one can bound conditional variances $\sigma_l^2 \eqdef \EE_{\Bxi^l}\left[\alpha_{l+1}^2\left\la\theta_{l+1}^u, 2\eta_l + 2\alpha_{l+1}\zeta_l\right\ra^2\right]$ in the following way:
    \begin{eqnarray}
         \sigma_l^2 &\le& \EE_{\Bxi^l}\left[\alpha_{l+1}^2\left\|\theta_{l+1}^u\right\|_2^2 \left\|2\eta_l + 2\alpha_{l+1}\zeta_l\right\|_2^2\right]\notag\\
         &\overset{\eqref{eq:main_thm_clipped_SSTM_technical_4_2}}{\le}&\alpha_{l+1}^2\EE_{\Bxi^l}\left[\left\|\theta_{l+1}^u\right\|_2^2\right](2CR_0+B)^2.\label{eq:main_thm_clipped_SSTM_technical_7}
    \end{eqnarray}
    In other words, sequence $\left\{\alpha_{l+1}\left\la\theta_{l+1}^u, 2\eta_l + 2\alpha_{l+1}\zeta_l\right\ra\right\}_{l\ge 0}$ is bounded martingale difference sequence with bounded conditional variances $\{\sigma_l^2\}_{l\ge 0}$. Therefore, we can apply Bernstein's inequality, i.e.\ we apply Lemma~\ref{lem:Bernstein_ineq} with $X_l = \alpha_{l+1}\left\la\theta_{l+1}^u, 2\eta_l + 2\alpha_{l+1}\zeta_l\right\ra$, $c = \frac{33C^2R_0^2}{64\ln\frac{4N}{\beta}}$ and $F = \frac{c^2\ln\frac{4N}{\beta}}{18}$ and get that for all $b > 0$
    \begin{equation*}
        \PP\left\{\left|\sum\limits_{l=0}^{T-1}X_l\right| > b\text{ and } \sum\limits_{l=0}^{T-1}\sigma_l^2 \le F\right\} \le 2\exp\left(-\frac{b^2}{2F + \nicefrac{2cb}{3}}\right)
    \end{equation*}
    or, equivalently, with probability at least $1 - 2\exp\left(-\frac{b^2}{2F + \nicefrac{2cb}{3}}\right)$
    \begin{equation*}
        \text{either } \sum\limits_{l=0}^{T-1}\sigma_l^2 > F \quad \text{or} \quad \underbrace{\left|\sum\limits_{l=0}^{T-1}X_l\right|}_{|\circledOne|} \le b.
    \end{equation*}
    The choice of $F$ will be clarified further, let us now choose $b$ in such a way that $2\exp\left(-\frac{b^2}{2F + \nicefrac{2cb}{3}}\right) = \frac{\beta}{2N}$. This implies that $b$ is the positive root of the quadratic equation \begin{equation*}
        b^2 - \frac{2c\ln\frac{4N}{\beta}}{3}b - 2F\ln\frac{4N}{\beta} = 0,
    \end{equation*}
    hence
    \begin{eqnarray*}
         b &=& \frac{c\ln\frac{4N}{\beta}}{3} + \sqrt{\frac{c^2\ln^2\frac{4N}{\beta}}{9} + 2F\ln\frac{4N}{\beta}}\le\frac{c\ln\frac{4N}{\beta}}{3} + \sqrt{\frac{2c^2\ln^2\frac{4N}{\beta}}{9}}\\
         &=& \frac{1+\sqrt{2}}{3}c\ln\frac{4N}{\beta} \le \frac{33C^2R_0^2}{64}.
    \end{eqnarray*}
    That is, with probability at least $1 - \frac{\beta}{2N}$
     \begin{equation*}
        \underbrace{\text{either } \sum\limits_{l=0}^{T-1}\sigma_l^2 > F \quad \text{or} \quad \left|\circledOne\right| \le \frac{33C^2R_0^2}{64}}_{\text{probability event } E_{\circledOne}}.
    \end{equation*}
    Next, we notice that probability event $E_{T-1}$ implies that
    \begin{eqnarray*}
        \sum\limits_{l=0}^{T-1}\sigma_l^2  &\overset{\eqref{eq:main_thm_clipped_SSTM_technical_7}}{\le}& (2CR_0+B)^2\sum\limits_{l=0}^{T-1}\alpha_{l+1}^2\EE_{\Bxi^l}\left[\left\|\theta_{l+1}^u\right\|_2^2\right]\\
        &\overset{\eqref{eq:variance_bound_clipped_SSTM},\eqref{eq:main_thm_clipped_SSTM_technical_4}}{\le}& 18\sigma^2C^2R_0^2\left(2 + \frac{1}{8\ln\frac{4N}{\beta}}\right)^2\sum\limits_{l=0}^{T-1}\frac{\alpha_{l+1}^2}{m_l}\\
        &\overset{\eqref{eq:beta_N_condition_clipped_SSTM},\eqref{eq:bathces_clipped_SSTM}}{\le}& 18\sigma^2C^2R_0^2 \left(2 + \frac{1}{16}\right)^2\sum\limits_{l=0}^{T-1}\frac{\alpha_{l+1}^2C^2R_0^2}{6000\sigma^2\alpha_{l+1}^2N\ln\frac{4N}{\beta}}\\
        &\overset{T\le N}{\le}& \frac{18\left(2+\frac{1}{16}\right)^2}{6000\ln\frac{4N}{\beta}} C^4R_0^4\sum\limits_{l=0}^{N-1}\frac{1}{N} \le \frac{c^2\ln\frac{4N}{\beta}}{18} = F,
    \end{eqnarray*}
    where the last inequality follows from $c = \frac{33C^2R_0^2}{64\ln\frac{4N}{\beta}}$ and simple arithmetic.
    
    \textbf{Upper bound for $\circledTwo$.} First of all, we notice that probability event $E_{T-1}$ implies
    \begin{eqnarray*}
        \alpha_{l+1}\left\la\theta_{l+1}^b, 2\eta_l + 2\alpha_{l+1}\zeta_l\right\ra &\le& \alpha_{l+1}\left\|\theta_{l+1}^b\right\|_2\left\|2\eta_l + 2\alpha_{l+1}\zeta_l\right\|_2\\
        &\overset{\eqref{eq:bias_bound_clipped_SSTM},\eqref{eq:main_thm_clipped_SSTM_technical_4_2}}{\le}& \alpha_{l+1}\cdot\frac{4\sigma^2}{m_l\lambda_{l+1}}\left(2CR_0 + B\right)\\
        &=& \frac{32\alpha_{l+1}^2\sigma^2\ln\frac{4N}{\beta}}{m_l CR_0}\left(2CR_0 + \frac{CR_0}{8\ln\frac{4N}{\beta}}\right)\\
        &\overset{\eqref{eq:beta_N_condition_clipped_SSTM},\eqref{eq:bathces_clipped_SSTM}}{\le}& \frac{32\alpha_{l+1}^2\sigma^2C^2R_0^2\ln\frac{4N}{\beta}}{6000\alpha_{l+1}^2N\sigma^2\ln\frac{4N}{\beta}}\left(2+ \frac{1}{16}\right)\\
        &=& \frac{11C^2R_0^2}{1000N}.
    \end{eqnarray*}
    This implies that
    \begin{eqnarray*}
        \circledTwo &=& \sum\limits_{l=0}^{T-1}\alpha_{l+1}\left\la\theta_{l+1}^b,2\eta_l + 2\alpha_{l+1}\zeta_l\right\ra \overset{T\le N}{\le} \frac{11C^2R_0^2}{1000}.
    \end{eqnarray*}
    
    \textbf{Upper bound for $\circledThree$.} We derive the upper bound for $\circledThree$ using the same technique as for $\circledOne$. First of all, we notice that the summands in $\circledThree$ are conditionally independent:
    \begin{equation*}
        \EE_{\Bxi^l}\left[4\alpha_{l+1}^2\left(\|\theta_{l+1}^u\|_2^2 - \EE_{\Bxi^l}\left[\|\theta_{l+1}^u\|_2^2\right]\right)\right] = 0.
    \end{equation*}
    Secondly, the summands are bounded with probability $1$:
    \begin{eqnarray}
        \left|4\alpha_{l+1}^2\left(\|\theta_{l+1}^u\|_2^2 - \EE_{\Bxi^l}\left[\|\theta_{l+1}^u\|_2^2\right]\right)\right| &\le& 4\alpha_{l+1}^2\left(\|\theta_{l+1}^u\|_2^2 + \EE_{\Bxi^l}\left[\|\theta_{l+1}^u\|_2^2\right]\right)\notag\\
        &\overset{\eqref{eq:magnitude_bound_clipped_SSTM}}{\le}& 4\alpha_{l+1}^2\left(4\lambda_{l+1}^2 + 4\lambda_{l+1}^2\right)\notag\\
        &=& 32B^2 = \frac{C^2R_0^2}{2\ln^2\frac{4N}{\beta}} \overset{\eqref{eq:beta_N_condition_clipped_SSTM}}{\le} \frac{C^2R_0^2}{4\ln\frac{4N}{\beta}} \eqdef c_1.\label{eq:main_thm_clipped_SSTM_technical_8}
    \end{eqnarray}
    Finally, one can bound conditional variances $\hat \sigma_l^2 \eqdef \EE_{\Bxi^l}\left[\left|4\alpha_{l+1}^2\left(\|\theta_{l+1}^u\|_2^2 - \EE_{\Bxi^l}\left[\|\theta_{l+1}^u\|_2^2\right]\right)\right|^2\right]$ in the following way:
    \begin{eqnarray}
         \hat \sigma_l^2 &\overset{\eqref{eq:main_thm_clipped_SSTM_technical_8}}{\le}& c_1\EE_{\Bxi^l}\left[\left|4\alpha_{l+1}^2\left(\|\theta_{l+1}^u\|_2^2 - \EE_{\Bxi^l}\left[\|\theta_{l+1}^u\|_2^2\right]\right)\right|\right]\notag\\
         &\le& 4c_1\alpha_{l+1}^2\EE_{\Bxi^l}\left[\|\theta_{l+1}^u\|_2^2 + \EE_{\Bxi^l}\left[\|\theta_{l+1}^u\|_2^2\right]\right] = 8c_1\alpha_{l+1}^2\EE_{\Bxi^l}\left[\|\theta_{l+1}^u\|_2^2\right].\label{eq:main_thm_clipped_SSTM_technical_9}
    \end{eqnarray}
    In other words, sequence $\left\{4\alpha_{l+1}^2\left(\|\theta_{l+1}^u\|_2^2 - \EE_{\Bxi^l}\left[\|\theta_{l+1}^u\|_2^2\right]\right)\right\}_{l\ge 0}$ is bounded martingale difference sequence with bounded conditional variances $\{\hat \sigma_l^2\}_{l\ge 0}$. Therefore, we can apply Bernstein's inequality, i.e.\ we apply Lemma~\ref{lem:Bernstein_ineq} with $X_l = \hat X_l = 4\alpha_{l+1}^2\left(\|\theta_{l+1}^u\|_2^2 - \EE_{\Bxi^l}\left[\|\theta_{l+1}^u\|_2^2\right]\right)$, $c = c_1 = \frac{C^2R_0^2}{4\ln\frac{4N}{\beta}}$ and $F = F_1 = \frac{c_1^2\ln\frac{4N}{\beta}}{18}$ and get that for all $b > 0$
    \begin{equation*}
        \PP\left\{\left|\sum\limits_{l=0}^{T-1}\hat X_l\right| > b\text{ and } \sum\limits_{l=0}^{T-1}\hat \sigma_l^2 \le F_1\right\} \le 2\exp\left(-\frac{b^2}{2F_1 + \nicefrac{2c_1b}{3}}\right)
    \end{equation*}
    or, equivalently, with probability at least $1 - 2\exp\left(-\frac{b^2}{2F_1 + \nicefrac{2c_1b}{3}}\right)$
    \begin{equation*}
        \text{either } \sum\limits_{l=0}^{T-1}\hat\sigma_l^2 > F_1 \quad \text{or} \quad \underbrace{\left|\sum\limits_{l=0}^{T-1}\hat X_l\right|}_{|\circledThree|} \le b.
    \end{equation*}
    As in our derivations of the upper bound for $\circledOne$ we choose such $b$ that $2\exp\left(-\frac{b^2}{2F_1 + \nicefrac{2c_1b}{3}}\right) = \frac{\beta}{2N}$, i.e.\
    \begin{eqnarray*}
         b &=& \frac{c_1\ln\frac{4N}{\beta}}{3} + \sqrt{\frac{c_1^2\ln^2\frac{4N}{\beta}}{9} + 2F_1\ln\frac{4N}{\beta}}\le\frac{1+\sqrt{2}}{3}c_1\ln\frac{4N}{\beta} \le \frac{C^2R_0^2}{4}.
    \end{eqnarray*}
    That is, with probability at least $1 - \frac{\beta}{2N}$
     \begin{equation*}
        \underbrace{\text{either } \sum\limits_{l=0}^{T-1}\hat\sigma_l^2 > F_1 \quad \text{or} \quad \left|\circledThree\right| \le \frac{C^2R_0^2}{4}}_{\text{probability event } E_{\circledThree}}.
    \end{equation*}
    Next, we notice that probability event $E_{T-1}$ implies that
    \begin{eqnarray*}
        \sum\limits_{l=0}^{T-1}\hat\sigma_l^2 &\overset{\eqref{eq:main_thm_clipped_SSTM_technical_9}}{\le}& 8c_1\sum\limits_{l=0}^{T-1}\alpha_{l+1}^2\EE_{\Bxi^l}\left[\left\|\theta_{l+1}^u\right\|_2^2\right]\\
        &\overset{\eqref{eq:variance_bound_clipped_SSTM},\eqref{eq:main_thm_clipped_SSTM_technical_4}}{\le}& c_1\sum\limits_{l=0}^{T-1}\frac{144\sigma^2\alpha_{l+1}^2}{m_l}\\
        &\overset{\eqref{eq:bathces_clipped_SSTM}}{\le}& c_1\sum\limits_{l=0}^{T-1}\frac{144\sigma^2\alpha_{l+1}^2C^2R_0^2}{10368\sigma^2\alpha_{l+1}^2N}\\
        &\overset{T\le N}{\le}& c_1\cdot \underbrace{\frac{C^2R_0^2}{4\ln\frac{4N}{\beta}}}_{c_1} \cdot\frac{\ln\frac{4N}{\beta}}{18} = F_1.
    \end{eqnarray*}
    
    \textbf{Upper bound for $\circledFour$.} The probability event $E_{T-1}$ implies
    \begin{eqnarray*}
        \circledFour &=& \sum\limits_{l=0}^{T-1}4\alpha_{l+1}^2\EE_{\Bxi^l}\left[\|\theta_{l+1}^u\|_2^2\right] \overset{\eqref{eq:variance_bound_clipped_SSTM},\eqref{eq:main_thm_clipped_SSTM_technical_4}}{\le} \sum\limits_{l=0}^{T-1}\frac{72\alpha_{l+1}^2\sigma^2}{m_l}\overset{\eqref{eq:bathces_clipped_SSTM}}{\le} \sum\limits_{l=0}^{T-1}\frac{72\alpha_{l+1}^2\sigma^2C^2R_0^2}{10368\alpha_{l+1}^2\sigma^2 N}\\
        &\overset{T\le N}{\le}& \frac{C^2R_0^2}{144}.
    \end{eqnarray*}
    
    \textbf{Upper bound for $\circledFive$.} Again, we use corollaries of probability event $E_{T-1}$:
    \begin{eqnarray*}
        \circledFive &=& \sum\limits_{l=0}^{T-1}4\alpha_{l+1}^2\|\theta_{l+1}^b\|_2^2 \overset{\eqref{eq:bias_bound_clipped_SSTM},\eqref{eq:main_thm_clipped_SSTM_technical_4}}{\le} \sum\limits_{l=0}^{T-1}\frac{64\alpha_{l+1}^2\sigma^4}{m_l^2\lambda_{l+1}^2} = \frac{1}{B^2}\sum\limits_{l=0}^{T-1}\frac{64\alpha_{l+1}^4\sigma^4}{m_l^2}\\
        &\overset{\eqref{eq:bathces_clipped_SSTM}}{\le}& \frac{64\ln^2\frac{4N}{\beta}}{C^2R_0^2}\sum\limits_{l=0}^{T-1}\frac{64\alpha_{l+1}^4\sigma^4C^4R_0^4}{6000^2\sigma^4\alpha_{l+1}^4 N^2\ln^2\frac{4N}{\beta}}\\
        &\overset{T\le N}{\le}& \frac{16C^2R_0^2}{140625}.
    \end{eqnarray*}
    Now we summarize all bound that we have: probability event $E_{T-1}$ implies
    \begin{eqnarray*}
        R_T^2 &\overset{\eqref{eq:main_thm_clipped_SSTM_technical_1}}{\le}& R_0^2 + 2\sum\limits_{l=0}^{T-1}\alpha_{l+1}\left\la\theta_{l+1}, x^* - z^l\right\ra + 2\sum\limits_{l=0}^{k-1}\alpha_{l+1}^2\left\la\theta_{l+1}, \nabla f(x^{l+1})\right\ra + 2\sum\limits_{l=0}^{T-1}\alpha_{l+1}^2\|\theta_{l+1}\|_2^2\\
        &\overset{\eqref{eq:main_thm_clipped_SSTM_technical_5}}{\le}& R_0^2 + \circledOne + \circledTwo + \circledThree + \circledFour + \circledFive,\\
        \circledTwo &\le& \frac{11C^2R_0^2}{1000},\quad \circledFour \le \frac{C^2R_0^2}{144},\quad \circledFive \le \frac{16C^2R_0^2}{140625},\\
        \sum\limits_{l=0}^{T-1}\sigma_l^2 &\le& F,\quad \sum\limits_{l=0}^{T-1}\hat\sigma_l^2 \le F_1
    \end{eqnarray*}
    and
    \begin{equation*}
        \PP\{E_{T-1}\} \ge 1 - \frac{(T-1)\beta}{N},\quad \PP\{E_\circledOne\} \ge 1 - \frac{\beta}{2N},\quad \PP\{E_\circledThree\} \ge 1 - \frac{\beta}{2N},
    \end{equation*}
    where
    \begin{eqnarray*}
        E_{\circledOne} &=& \left\{\text{either } \sum\limits_{l=0}^{T-1}\sigma_l^2 > F \quad \text{or} \quad \left|\circledOne\right| \le \frac{33C^2R_0^2}{64}\right\},\\
        E_{\circledThree} &=& \left\{\text{either } \sum\limits_{l=0}^{T-1}\hat\sigma_l^2 > F_1 \quad \text{or} \quad \left|\circledThree\right| \le \frac{C^2R_0^2}{4}\right\}.
    \end{eqnarray*}
     Taking into account these inequalities we get that probability event $E_{T-1}\cap E_\circledOne \cap E_\circledThree$ implies
   \begin{eqnarray}
         R_T^2 &\overset{\eqref{eq:main_thm_clipped_SSTM_technical_1}}{\le}& R_0^2 + 2\sum\limits_{l=0}^{T-1}\alpha_{l+1}\left\la\theta_{l+1}, x^* - z^l\right\ra + 2\sum\limits_{l=0}^{k-1}\alpha_{l+1}^2\left\la\theta_{l+1}, \nabla f(x^{l+1})\right\ra + 2\sum\limits_{l=0}^{T-1}\alpha_{l+1}^2\|\theta_{l+1}\|_2^2\notag\\
         &\le& R_0^2 + \left(\frac{33}{64} + \frac{11}{1000} + \frac{1}{4} + \frac{1}{144} + \frac{16}{140625}\right)C^2R_0^2\notag\\
         &\le& \left(1 + \frac{4}{5}C^2\right)R_0^2 \overset{\eqref{eq:C_definition_clipped_SSTM}}{\le} C^2R_0^2.\label{eq:main_thm_clipped_SSTM_technical_10}
    \end{eqnarray}
    Moreover, using union bound we derive
    \begin{equation}
        \PP\left\{E_{T-1}\cap E_\circledOne \cap E_\circledThree\right\} = 1 - \PP\left\{\overline{E}_{T-1}\cup\overline{E}_\circledOne\cup\overline{E}_\circledThree\right\} \ge 1 - \frac{T\beta}{N}.\label{eq:main_thm_clipped_SSTM_technical_11}
    \end{equation}
    That is, by definition of $E_T$ and $E_{T-1}$ we have proved that
    \begin{eqnarray*}
        \PP\{E_T\} &\overset{\eqref{eq:main_thm_clipped_SSTM_technical_10}}{\ge}& \PP\left\{E_{T-1}\cap E_\circledOne \cap E_\circledThree\right\} \overset{\eqref{eq:main_thm_clipped_SSTM_technical_11}}{\ge} 1 - \frac{T\beta}{N},
    \end{eqnarray*}
    which implies that for all $k = 0,1,\ldots, N$ we have $\PP\{E_k\} \ge 1 - \frac{k\beta}{N}$. Then, for $k = N$ we have that with probability at least $1-\beta$
    \begin{eqnarray*}
        A_N\left(f(y^N) - f(x^*)\right) &\overset{\eqref{eq:main_thm_clipped_SSTM_technical_0}}{\le}& \frac{1}{2}\|z^0 - z\|_2^2 - \frac{1}{2}\|z^{N} - z\|_2^2 + \sum\limits_{k=0}^{N-1}\alpha_{k+1}\left\la \theta_{k+1}, z - z^{k}\right\ra\notag\\
        &&\quad + \sum\limits_{k=0}^{N-1}\alpha_{k+1}^2\left\|\theta_{k+1}\right\|_2^2 + \sum\limits_{k=0}^{N-1}\alpha_{k+1}^2\left\la\theta_{k+1},\nabla f(x^{k+1})\right\ra\\
        &\overset{\eqref{eq:main_thm_clipped_SSTM_technical_2}}{\le}& \frac{C^2R_0^2}{2}.
    \end{eqnarray*}
    Since $A_{N} = \frac{N(N+3)}{4aL}$ (see Lemma~\ref{lem:alpha_k}) we get that with probability at least $1-\beta$
    \begin{equation*}
        f(y^N) - f(x^*) \le \frac{2aLC^2R_0^2}{N(N+3)}.
    \end{equation*}
    In other words, {\tt clipped-SSTM} with $a = \max\left\{1,\frac{16\ln\frac{4N}{\beta}}{C},36\left(2\ln\frac{4N}{\beta} + \sqrt{4\ln^2\frac{4N}{\beta} + 2\ln\frac{4N}{\beta}}\right)^2\right\} = 36\left(2\ln\frac{4N}{\beta} + \sqrt{4\ln^2\frac{4N}{\beta} + 2\ln\frac{4N}{\beta}}\right)^2$ achieves $f(y^N) - f(x^*) \le \varepsilon$ with probability at least $1-\beta$ after $O\left(\sqrt{\frac{LR_0^2}{\varepsilon}}\ln\frac{LR_0^2}{\varepsilon\beta}\right)$ iterations and requires
    \begin{eqnarray*}
         \sum\limits_{k=0}^{N-1}m_k &\overset{\eqref{eq:bathces_clipped_SSTM}}{=}& \sum\limits_{k=0}^{N-1}O\left(\max\left\{1,\frac{\sigma^2\alpha_{k+1}^2N\ln\frac{N}{\beta}}{R_0^2}\right\}\right)\\
         &=& O\left(\max\left\{N,\sum\limits_{k=0}^{N-1}\frac{\sigma^2(k+2)^2N\ln\frac{N}{\beta}}{a^2L^2R_0^2}\right\}\right)\\
         &\overset{\eqref{eq:B_a_parameters_clipped_SSTM}}{=}& O\left(\max\left\{N,\frac{\sigma^2N^4}{\ln^3\frac{N}{\beta}L^2R_0^2}\right\}\right)\\
         &=& O\left(\max\left\{\sqrt{\frac{LR_0^2}{\varepsilon}},\frac{\sigma^2R_0^2}{\varepsilon^2}\right\}\ln\frac{LR_0^2}{\varepsilon\beta}\right).
    \end{eqnarray*}
    oracle calls.

\subsubsection{Proof of Corollary~\ref{cor:different_stepsizes_clipped_SSTM}}
Theorem~\ref{thm:main_result_clipped_SSTM} implies that with probability at least $1-\beta$
\begin{equation}
    f(y^N) - f(x^*) \overset{\eqref{eq:main_result_clipped_SSTM}}{\le} \frac{2aLC^2R_0^2}{N(N+3)}, \label{eq:cor1_technical1}
\end{equation}
where $a$ satisfies
\begin{equation}
    a \overset{\eqref{eq:B_a_parameters_clipped_SSTM}}{\ge} \max\left\{1,\frac{16\ln\frac{4N}{\beta}}{C},36\left(2\ln\frac{4N}{\beta} + \sqrt{4\ln^2\frac{4N}{\beta} + 2\ln\frac{4N}{\beta}}\right)^2\right\} \eqdef \hat a,\label{eq:cor1_technical2}
\end{equation}
$\alpha_{k+1} = \frac{k+2}{2aL}$ and batchsizes $m_k$ are chosen according to \eqref{eq:bathces_clipped_SSTM}:
\begin{eqnarray}
    m_k &\overset{\eqref{eq:bathces_clipped_SSTM}}{=}& \max\left\{1,\frac{1185\sigma^2 \alpha_{k+1}^2N\ln\frac{4N}{\beta}}{C^2R_0^2}, \frac{10368\sigma^2 \alpha_{k+1}^2N}{C^2R_0^2}\right\}\notag\\
    &=& \max\left\{1,\frac{1185\sigma^2 (k+2)^2N\ln\frac{4N}{\beta}}{4a^2L^2C^2R_0^2}, \frac{10368\sigma^2 (k+2)^2N}{4a^2L^2C^2R_0^2}\right\}.\label{eq:cor1_technical3}
\end{eqnarray}
We consider two different options for $a$.
\begin{enumerate}
    \item If $N\ln\frac{4N}{\beta}$ is bigger than $\hat a$, then we take $a = N\ln\frac{4N}{\beta}$ which implies that
    \begin{equation*}
        m_k = \max\left\{1,\frac{1185\sigma^2 (k+2)^2}{4L^2NC^2R_0^2\ln\frac{4N}{\beta}}, \frac{10368\sigma^2 (k+2)^2}{4L^2C^2R_0^2N\ln^2\frac{4N}{\beta}}\right\} = O\left(\max\left\{1,\frac{\sigma^2(k+2)^2}{L^2R_0^2N\ln\frac{4N}{\beta}}\right\}\right)
    \end{equation*}
    and with probability at least $1-\beta$
    \begin{equation}
         f(y^N) - f(x^*) \le \frac{2LC^2R_0^2\ln\frac{4N}{\beta}}{N+3}.\label{eq:cor1_technical4}
    \end{equation}
    That is, if $\varepsilon$ is small enough to satisfy $\frac{LR_0^2}{\varepsilon}\ln\frac{LR_0^2}{\varepsilon\beta} \ge C_1\ln^2\frac{LR_0^2}{\varepsilon\beta}$ for some constant $C_1$, then due to \eqref{eq:cor1_technical4} we have that after
    \begin{equation*}
        N = O\left(\frac{LR_0^2}{\varepsilon}\ln\frac{LR_0^2}{\varepsilon\beta}\right)\text{ iterations}
    \end{equation*}
    of {\tt clipped-SSTM} we obtain such point $y^N$ that with probability at least $1-\beta$ inequality $f(y^N) - f(x^*) \le \varepsilon$ holds and the method requires
    \begin{eqnarray*}
        \sum\limits_{k=0}^{N-1}m_k &=& \sum\limits_{k=0}^{N-1}O\left(\max\left\{1,\frac{\sigma^2(k+2)^2}{L^2R_0^2N\ln\frac{4N}{\beta}}\right\}\right)\\
        &=&O\left(\max\left\{N, \frac{\sigma^2N^2}{L^2R_0^2\ln\frac{4N}{\beta}}\right\}\right) = O\left(\max\left\{\frac{LR_0^2}{\varepsilon}, \frac{\sigma^2 R_0^2}{\varepsilon^2}\right\}\ln\frac{LR_0^2}{\varepsilon\beta}\right)
    \end{eqnarray*}
    stochastic first-order oracle calls.
    \item If $a_0N^{\nicefrac{3}{2}}\sqrt{\ln\frac{4N}{\beta}}$ is bigger than $\hat a$ for some $a_0 > 0$, then we take $a = a_0N^{\nicefrac{3}{2}}\sqrt{\ln\frac{4N}{\beta}}$ which implies that
    \begin{equation*}
        m_k = \max\left\{1,\frac{1185\sigma^2 (k+2)^2}{4a_0^2L^2N^2C^2R_0^2}, \frac{10368\sigma^2 (k+2)^2}{4a_0^2L^2C^2R_0^2N^2\sqrt{\ln\frac{4N}{\beta}}}\right\} = O\left(\max\left\{1,\frac{\sigma^2(k+2)^2}{a_0^2L^2R_0^2N^2}\right\}\right)
    \end{equation*}
    and with probability at least $1-\beta$
    \begin{equation}
         f(y^N) - f(x^*) \le \frac{2a_0LC^2R_0^2\sqrt{N\ln\frac{4N}{\beta}}}{N+3}.\label{eq:cor1_technical5}
    \end{equation}
    That is, if $\varepsilon$ is small enough to satisfy $\frac{a_0^3L^3R_0^6}{\varepsilon^3}\left(\ln\frac{LR_0^2}{\varepsilon\beta}\right)^{\nicefrac{3}{2}} \ge C_2\ln^2\frac{LR_0^2}{\varepsilon\beta}$ for some constant $C_2$, then due to \eqref{eq:cor1_technical5} we have that after
    \begin{equation*}
        N = O\left(\frac{a_0^2L^2R_0^4}{\varepsilon^2}\ln\frac{a_0^2L^2R_0^4}{\varepsilon^2\beta}\right) = O\left(\frac{a_0^2L^2R_0^4}{\varepsilon^2}\ln\frac{a_0LR_0^2}{\varepsilon\beta}\right)\text{ iterations}
    \end{equation*}
    of {\tt clipped-SSTM} we obtain such point $y^N$ that with probability at least $1-\beta$ inequality $f(y^N) - f(x^*) \le \varepsilon$ holds and the method requires
    \begin{eqnarray*}
        \sum\limits_{k=0}^{N-1}m_k &=& \sum\limits_{k=0}^{N-1}O\left(\max\left\{1,\frac{\sigma^2(k+2)^2}{a_0^2L^2R_0^2N^2}\right\}\right)\\
        &=&O\left(\max\left\{N, \frac{\sigma^2N}{a_0^2L^2R_0^2}\right\}\right) = O\left(\max\left\{\frac{a_0^2L^2R_0^4}{\varepsilon^2}, \frac{\sigma^2 R_0^2}{\varepsilon^2}\right\}\ln\frac{a_0LR_0^2}{\varepsilon\beta}\right)
    \end{eqnarray*}
    stochastic first-order oracle calls. Finally, if all assumptions on $N$, $\beta$ and $\varepsilon$ hold for $a_0 = \frac{\sigma}{LR_0}$, then for all $k=0,1,\ldots,N-1$
    \begin{equation*}
        m_k = O\left(\max\left\{1,\frac{\sigma^2(k+2)^2}{a_0^2L^2R_0^2N^2}\right\}\right) = O\left(\max\left\{1,\frac{(k+2)^2}{N^2}\right\}\right) = O(1),
    \end{equation*}
    i.e.\ one iteration of {\tt clipped-SSTM} requires $O(1)$ oracle calls, and $f(y^N) - f(x^*) \le \varepsilon$ with probability at least $1-\beta$ after
    \begin{equation*}
        N = O\left(\frac{\sigma^2R_0^2}{\varepsilon^2}\ln\frac{\sigma R_0}{\varepsilon\beta}\right)\text{ iterations.}
    \end{equation*}
\end{enumerate}

\subsubsection{Proof of Corollary~\ref{cor:clipped_sstm_small_stepsize_const_batch}}
Recall that
\begin{eqnarray*}
    a' &=& \max\left\{1,\frac{16\ln\frac{4N}{\beta}}{C},36\left(2\ln\frac{4N}{\beta} + \sqrt{4\ln^2\frac{4N}{\beta} + 2\ln\frac{4N}{\beta}}\right)^2\right\},\notag\\
    a &=& \max\left\{a', \frac{\sigma N^{\nicefrac{3}{2}}}{LR_0}\sqrt{\ln\frac{4N}{\beta}}\right\},\quad \alpha_{k+1} = \frac{k+2}{2aL},\\
    m_k &=& \max\left\{1,\frac{6000\sigma^2 \alpha_{k+1}^2N\ln\frac{4N}{\beta}}{C^2R_0^2}, \frac{10368\sigma^2 \alpha_{k+1}^2N}{C^2R_0^2}\right\}.
\end{eqnarray*}
Since $a \ge \frac{\sigma N^{\nicefrac{3}{2}}}{LR_0}$ we have that $m_k = O(1)$. Next, there are two possible situations.
\begin{enumerate}
    \item If $a = a'$, then we are in the settings of Theorem~\ref{thm:main_result_clipped_SSTM}. This means that {\tt clipped-SSTM} achieves $f(y^N) - f(x^*) \le \varepsilon$ with probability at least $1-\beta$ after 
    \begin{equation}
        O\left(\max\left\{\sqrt{\frac{LR_0^2}{\varepsilon}}, \frac{\sigma^2R_0^2}{\varepsilon^2}\right\}\ln\frac{LR_0^2}{\varepsilon\beta}\right)\text{ oracle calls.} \notag
    \end{equation}
    \item If $a = \frac{\sigma N^{\nicefrac{3}{2}}}{LR_0}\sqrt{\ln\frac{4N}{\beta}}$, then we are in the settings of Corollary~\ref{cor:different_stepsizes_clipped_SSTM} which implies that {\tt clipped-SSTM} achieves $f(y^N) - f(x^*) \le \varepsilon$ with probability at least $1-\beta$ after 
    \begin{equation}
       O\left(\frac{\sigma^2R_0^2}{\varepsilon^2}\ln\frac{\sigma R_0}{\varepsilon\beta}\right)\text{ oracle calls.} \notag
    \end{equation}
\end{enumerate}
Finally, we combine these two cases and obtain that with $a = \max\left\{a', \frac{\sigma N^{\nicefrac{3}{2}}}{LR_0}\sqrt{\ln\frac{4N}{\beta}}\right\}$ {\tt clipped-SSTM} guarantees $f(y^N) - f(x^*) \le \varepsilon$ with probability at least $1-\beta$ after
\begin{eqnarray*}
    O\left(\max\left\{\max\left\{\sqrt{\frac{LR_0^2}{\varepsilon}}, \frac{\sigma^2R_0^2}{\varepsilon^2}\right\}\ln\frac{LR_0^2}{\varepsilon\beta},\frac{\sigma^2R_0^2}{\varepsilon^2}\ln\frac{\sigma R_0}{\varepsilon\beta}\right\}\right)&\\
    &\hspace{-2cm}= O\left(\max\left\{\sqrt{\frac{LR_0^2}{\varepsilon}}, \frac{\sigma^2R_0^2}{\varepsilon^2}\right\}\ln\frac{LR_0^2+\sigma R_0}{\varepsilon\beta}\right)
\end{eqnarray*}
iterations/oracle calls.

\subsubsection{Proof of Theorem~\ref{thm:main_result_R_clipped_SSTM}}\label{sec:proof_R_clipped_SSTM}
First of all, consider behavior of {\tt clipped-SSTM} during the first run in {\tt R-clipped-SSTM}. We notice that the proof of Theorem~\ref{thm:main_result_clipped_SSTM} will be valid if we substitute $R_0$ everywhere by its upper bound $R$. From $\mu$-strong convexity of $f$ we have
\begin{equation*}
    R_0^2 = \|x^0 - x^*\|_2^2 \overset{\eqref{eq:str_cvx_cor}}{\le} \frac{2}{\mu}\left(f(x^0) - f(x^*)\right),
\end{equation*}
therefore, one can choose $R = \sqrt{\frac{2}{\mu}\left(f(x^0) - f(x^*)\right)}$. It implies that after $N_0$ iterations of {\tt clipped-SSTM} we have
\begin{equation*}
    f(y^{N_0}) - f(x^*) \le \frac{2aC^2LR^2}{N_0(N_0+3)} = \frac{4aC^2L}{N_0^2\mu} (f(x^0) - f(x^*)).
\end{equation*}
with probability at least $1-\frac{\beta}{\tau}$, hence with the same probability $f(y^{N_0}) - f(x^*) \le \frac{1}{2}(f(x^0) - f(x^*))$ since $N_0 \ge C\sqrt{\frac{8aL}{\mu}}$. In other words, with probability at least $1-\frac{\beta}{\tau}$
\begin{equation*}
    f(\hat x^1) - f(x^*) \le \frac{1}{2}\left(f(x^0) - f(x^*)\right) = \frac{1}{4}\mu R^2.
\end{equation*}
Then, by induction one can show that for arbitrary $k\in\{0,1,\ldots,\tau-1\}$ the inequality
\begin{equation*}
    f(\hat x^{k+1}) - f(x^*) \le \frac{1}{2}\left(f(\hat x^k) - f(x^*)\right)
\end{equation*}
holds with probability at least $1-\frac{\beta}{\tau}$. Therefore, these inequalities hold simultaneously with probability at least $1 - \beta$. Using this we derive that inequality
\begin{equation*}
    f(\hat x^{\tau}) - f(x^*) \le \frac{1}{2}\left(f(\hat x^{\tau-1}) - f(x^*)\right) \le \frac{1}{2^2}\left(f(\hat x^{\tau-2}) - f(x^*)\right) \le \ldots \le \frac{1}{2^\tau}\left(f(x^0) - f(x^*)\right) = \frac{\mu R^2}{2^{\tau+1}}
\end{equation*}
holds with probability $\ge 1- \beta$. That is, after $\tau = \left\lceil\log_2\frac{\mu R^2}{2\varepsilon}\right\rceil$ restarts {\tt R-clipped-SSTM} generates such a point $\hat x^{\tau}$ that $f(\hat x^{\tau}) - f(x^*) \le \varepsilon$ with probability at least $1-\beta$. Moreover, if $a$ equals the maximum from \eqref{eq:a_parameter_R_clipped_SSTM} and $N_0 \le C_1\sqrt{\frac{8aL}{\mu}}$ with some numerical constant $C_1 \ge C$, then $a \sim \left(\ln\frac{N_0\tau}{\beta}\right)^2$, the total number of iterations of {\tt clipped-SSTM} equals
\begin{equation*}
    N_0\tau = O\left(\sqrt{\frac{L}{\mu}}\ln\left(\frac{\mu R^2}{\varepsilon}\right)\ln\left(\frac{L}{\mu\beta}\ln\frac{\mu R^2}{\varepsilon}\right)\right)
\end{equation*}
and the overall number of stochastic first-order oracle calls is
\begin{eqnarray*}
    \sum\limits_{t=0}^{\tau-1}\sum\limits_{k=0}^{N_0-1}m_k^t &=& \sum\limits_{t=0}^{\tau-1}\sum\limits_{k=0}^{N_0-1}O\left(\max\left\{1, \frac{2^{t}\sigma^2 \alpha_{k+1}^2N_0\ln\frac{4N_0\tau}{\beta}}{R^2}\right\}\right)\\
    &=& \sum\limits_{t=0}^{\tau-1}\sum\limits_{k=0}^{N_0-1}O\left(\max\left\{1, \frac{2^{t}\sigma^2 (k+2)^2N_0}{\ln^3\frac{4N_0\tau}{\beta}L^2R^2}\right\}\right)\\
    &=& O\left(\max\left\{N_0\tau, \frac{\sigma^2 2^\tau N_0^4}{\ln^3\frac{4N_0\tau}{\beta}L^2R^2}\right\}\right)\\
    &=& O\left(\max\left\{\sqrt{\frac{L}{\mu}}\ln\left(\frac{\mu R^2}{\varepsilon}\right), \frac{\sigma^2}{\mu\varepsilon}\right\} \ln\left(\frac{L}{\mu\beta}\ln\frac{\mu R^2}{\varepsilon}\right)\right).
\end{eqnarray*}

\subsubsection{Proof of Corollary~\ref{cor:stepsizes_r_clipped_SSTM}}
Similarly to the proof of Theorem~\ref{thm:main_result_R_clipped_SSTM} (see the previous subsection) we derive that under assumptions of the corollary after $\tau = \left\lceil\log_2\frac{\mu R^2}{2\varepsilon}\right\rceil$ restarts {\tt R-clipped-SSTM} generates such a point $\hat x^\tau$ that $f(\hat x^\tau) - f(x^*) \le \varepsilon$ with probability at least $1 - \beta$. Moreover, $a$ and $N_0$ satisfy the following system of inequalities
\begin{equation}
    a = \Theta\left(\frac{\sigma^4\ln^2\frac{N_0\tau}{\beta}}{L\mu\varepsilon^2}\right),\quad N_0 = \Theta\left(\sqrt{\frac{aL}{\mu}}\right)\label{eq:cor_R_clipped_SSTM_tech1}
\end{equation}
which is consistent and implies that
\begin{equation}
    a = \Theta\left(\frac{\sigma^4}{L\mu\varepsilon}\ln^2\left(\frac{\sigma^2}{\mu\varepsilon\beta}\ln\frac{\mu R^2}{\varepsilon}\right)\right),\quad N_0 = \Theta\left(\frac{\sigma^2}{\mu\varepsilon}\ln\left(\frac{\sigma^2}{\mu\varepsilon\beta}\ln\frac{\mu R^2}{\varepsilon}\right)\right).\label{eq:cor_R_clipped_SSTM_tech2}
\end{equation}
Then, for all $k= 0,1,\ldots, N_0-1$ and $t=0,1,\ldots,\tau-1$ batchsizes satisfy
\begin{eqnarray*}
    m_k^t \le m_{N_0-1}^{\tau-1} &=& O\left(\max\left\{1, \frac{2^\tau \sigma^2\alpha_{N_0}^2N_0\ln\frac{N_0\tau}{\beta}}{R^2}\right\}\right)\\
    &=& O\left(\max\left\{1, \frac{\mu R^2\sigma^2 N_0^3\ln\frac{N_0\tau}{\beta}}{a^2L^2\varepsilon R^2}\right\}\right) \overset{\eqref{eq:cor_R_clipped_SSTM_tech1},\eqref{eq:cor_R_clipped_SSTM_tech2}}{=} O(1),
\end{eqnarray*}
i.e.\ the algorithm requires $O(1)$ oracle calls per iteration. Finally, the total number of iterations is
\begin{equation*}
    N_0\tau = O\left(\frac{\sigma^2}{\mu\varepsilon}\ln\left(\frac{\mu R^2}{\varepsilon}\right)\ln\left(\frac{\sigma^2}{\mu\varepsilon\beta}\ln\frac{\mu R^2}{\varepsilon}\right)\right).
\end{equation*}

\clearpage

\section{{\tt SGD} with Clipping: Exact Formulations and Missing Proofs}\label{sec:missing_proofs_SGD}
In this section we provide exact formulations of all the results that we have for {\tt clipped-SGD} and {\tt R-clipped-SGD} together with the full proofs.
\subsection{Convex Case}\label{sec:cvx_case_clipped_SGD}
We start with the case when $f(x)$ is convex and $L$-smooth and, as before, we assume that at each point $x\in \R^n$ function $f$ is accessible only via stochastic gradients $\nabla f(x,\xi)$ such that \eqref{eq:bounded_variance_clipped_SSTM} holds. Next theorem summarizes the main convergence result for {\tt clipped-SGD} in this case.
\begin{theorem}\label{thm:main_result_clipped_SGD}
    Assume that function $f$ is convex and $L$-smooth. Then for all $\beta \in (0,1)$ and $N\ge 1$ such that
    \begin{equation}
        \ln\frac{4N}{\beta} \ge 2 \label{eq:beta_N_condition_clipped_SGD}
    \end{equation}
    we have that after $N$ iterations of {\tt clipped-SGD} with
    \begin{equation}
        \lambda = 2LCR_0,\quad m_k = m = \max\left\{1,\frac{27N\sigma^2}{2(CR_0)^2L^2\ln\frac{4N}{\beta}}\right\},\label{eq:bathces_clipped_SGD}
    \end{equation}
    where $R_0 = \|x^0 - x^*\|_2$ and stepsize
    \begin{equation}
        \gamma = \frac{1}{80L\ln\frac{4N}{\beta}},\label{eq:step_size_clipped_SGD}
    \end{equation}
    that with probability at least $1-\beta$
    \begin{equation}
        f(\Bar{x}^N) - f(x^*) \le \frac{80LC^2R_0^2\ln\frac{4N}{\beta}}{N}, \label{eq:main_result_clipped_SGD}
    \end{equation}
    where $\Bar{x}^N = \frac{1}{N}\sum_{k=0}^{N-1}x^k$ and
    \begin{equation}
        C = \sqrt{2}. \label{eq:C_definition_clipped_SGD}
    \end{equation}
    In other words, the method achieves $f(\Bar{x}^N) - f(x^*) \le \varepsilon$ with probability at least $1-\beta$ after $O\left(\frac{LR_0^2}{\varepsilon}\ln\frac{LR_0^2}{\varepsilon\beta}\right)$ iterations and requires
    \begin{equation}
        O\left(\max\left\{\frac{LR_0^2}{\varepsilon}, \frac{\sigma^2R_0^2}{\varepsilon^2}\right\}\ln\frac{LR_0^2}{\varepsilon\beta}\right)\text{ oracle calls.} \label{eq:clipped_SGD_oracle_complexity}
    \end{equation}
\end{theorem}
To the best of our knowledge, it is the first result for {\tt clipped-SGD} establishing non-trivial complexity guarantees for the convergence with high probability. One can find the full proof in Section~\ref{sec:proof_cvx_SGD}.

\subsection{Strongly Convex Case}\label{sec:str_cvx_case_clipped_SGD}
Next, we consider the situation when $f$ is additionally $\mu$-strongly convex and propose a restarted version of {\tt clipped-SGD} ({\tt R-clipped-SGD}), see Algorithm~\ref{alg:R-clipped-SGD}.
\begin{algorithm}[h]
\caption{Restarted Clipped Stochastic Gradient Descent ({\tt R-clipped-SGD})}
\label{alg:R-clipped-SGD}   
\begin{algorithmic}[1]
\Require starting point $x^0$, number of iterations $N_0$ of {\tt clipped-SGD}, number $\tau$ of {\tt clipped-SGD} runs, batchsizes $m^0$, $m^1$, \ldots, $m^\tau$
\State Set $\hat x^0 = x^0$, stepsize $\gamma > 0$
\For{$t = 0,1,\ldots, \tau-1$}
\State Run {\tt clipped-SGD} (Algorithm~\ref{alg:clipped-SGD}) for $N_0$ iterations with constant batchsizes $m^t$, stepsize $\gamma$ and starting point $\hat x^t$. Define the output of {\tt clipped-SGD} by $\hat x^{t+1}$.
\EndFor
\Ensure $\hat x^{\tau}$ 
\end{algorithmic}
\end{algorithm}
For this method we prove the following result.
\begin{theorem}\label{thm:main_result_R_clipped_SGD}
    Assume that $f$ is $\mu$-strongly convex and $L$-smooth. If we choose $\beta \in (0,1)$, $\tau$ and $N_0\ge 1$ such that
    \begin{equation}
        \ln\frac{4N_0\tau}{\beta} \ge 2,\quad \frac{N_0}{\ln \frac{4N_0\tau}{\beta}} \ge \frac{320C^2L}{\mu}, \label{eq:beta_N_tau_condition_R_clipped_SGD}
    \end{equation}
    and
     \begin{equation}
        m^t = \max\left\{1,\frac{27 \cdot 2^tN_0\sigma^2}{2(CR)^2L^2\ln\frac{4N_0\tau}{\beta}}\right\},\label{eq:bathces_R_clipped_SGD}
    \end{equation}
    where $R = \sqrt{\frac{2(f(x^0) - f(x^*))}{\mu}}$ and $C = \sqrt{2}$, then we have that after $\tau$ runs of {\tt clipped-SGD} in {\tt R-clipped-SGD} the inequality
    \begin{equation}
        f(\hat x^\tau) - f(x^*) \le 2^{-\tau} \left(f(x^0)-f(x^*)\right) \label{eq:main_result_R_clipped_SGD}
    \end{equation}
    holds with probability at least $1-\beta$. That is, if we choose $\frac{N_0}{\ln \frac{4N_0\tau}{\beta}} \le \frac{C_1L}{\mu}$ with some numerical constant $C_1 \ge 320C^2$, then the method achieves $f(\hat x^{\tau}) - f(x^*) \le \varepsilon$ with probability at least $1-\beta$ after
    \begin{equation}
        O\left(\frac{L}{\mu}\ln\left(\frac{\mu R^2}{\varepsilon}\right)\ln\left(\frac{L}{\mu\beta}\ln\frac{\mu R^2}{\varepsilon}\right)\right)\text{ iterations (in total)}
    \end{equation}
    of {\tt clipped-SGD} and requires
    \begin{equation}
        O\left(\max\left\{\frac{L}{\mu}\ln\frac{\mu R^2}{\varepsilon}, \frac{\sigma^2}{\mu\varepsilon}\right\}\ln\left(\frac{L}{\mu\beta}\ln\frac{\mu R^2}{\varepsilon}\right)\right)\text{ oracle calls.} \label{eq:R_clipped_SGD_oracle_complexity}
    \end{equation}
\end{theorem}
This theorem implies that {\tt R-clipped-SGD} has the same complexity as the restarted version of {\tt RSMD} from \cite{nazin2019algorithms} up to the difference in logarithmical factors. We notice that the main difference between our result and one from \cite{nazin2019algorithms} is that we do not need to assume that the optimization problem is considered on the bounded set. 

However, in order to get \eqref{eq:R_clipped_SGD_oracle_complexity} {\tt R-clipped-SGD} requires to know strong convexity parameter $\mu$. In order to remove this drawback we analyse {\tt clipped-SGD} for the strongly convex case and get the following result.
\begin{theorem}\label{thm:main_result_clipped_SGD_2}
    Assume that function $f$ is $\mu$-strongly convex and $L$-smooth. Then for all $\beta \in (0,1)$ and $N\ge 1$ such that
    \begin{equation}
        \ln\frac{4N}{\beta} \ge 2 \label{eq:beta_N_condition_clipped_SGD_2}
    \end{equation}
    we have that after $N$ iterations of {\tt clipped-SGD} with
    \begin{equation}
        \lambda_l = 4 \sqrt{L (1-\gamma\mu)^l r_0},\quad m_k = \max\left\{1,\frac{27N\sigma^2}{16Lr_0(1-\gamma\mu)^k\ln\frac{4N}{\beta}}\right\},\label{eq:bathces_clipped_SGD_2}
    \end{equation}
    where $r_0=f(x^0)-f(x^*)$ and stepsize
    \begin{equation}
        \gamma = \frac{1}{81L\ln\frac{4N}{\beta}},\label{eq:step_size_clipped_SGD_2}
    \end{equation}
    that with probability at least $1-\beta$
    \begin{equation}
        f(x^N) - f(x^*) \le 2(1-\gamma\mu)^N(f(x^0)-f(x^*)). \label{eq:main_result_clipped_SGD_2}
    \end{equation}
    In other words, the method achieves $f(x^N) - f(x^*) \le \varepsilon$ with probability at least $1-\beta$ after $ O\left(\frac{L}{\mu}\ln\left(\frac{r_0}{\varepsilon}\right)\ln\left(\frac{L}{\mu\beta}\ln\frac{r_0}{\varepsilon}\right)\right)$ iterations and requires
    \begin{equation}
        O\left(\max\left\{\frac{L}{\mu}, \frac{\sigma^2}{\mu\varepsilon}\cdot \frac{L}{\mu}\right\}\ln\left(\frac{r_0}{\varepsilon}\right)\ln\left(\frac{L}{\mu\beta}\ln\frac{r_0}{\varepsilon}\right)\right)\text{ oracle calls.} \label{eq:clipped_SGD_oracle_complexity_2}
    \end{equation}
\end{theorem}
Unfortunately, our approach leads to worse complexity bound than we have for {\tt R-clipped-SGD}: in the second term of the maximum in \eqref{eq:clipped_SGD_oracle_complexity_2} we get an extra factor $\nicefrac{L}{\mu}$ that can be large. Nevertheless, to the best of our knowledge it is the first non-trivial complexity result for {\tt clipped-SGD} that guarantees convergence with high probability. One can find the full proof of Theorem~\ref{thm:main_result_clipped_SGD_2} in Section~\ref{sec:proof_str_cvx_clipped_SGD}.

\subsection{Proofs}
\subsubsection{Proof of Theorem~\ref{thm:main_result_clipped_SGD}}\label{sec:proof_cvx_SGD}
 Since $f(x)$ is convex and $L$-smooth, we get the following inequality:
    \begin{eqnarray*}
        \|x^{k+1} - x^*\|_2^2 &=& \|x^k - \gamma \tnabla f(x^k,\Bxi^k) - x^*\|_2^2 = \|x^k - x^*\|^2_2 + \gamma^2\|\tnabla f(x^k,\Bxi^k)\|^2_2 - 2\gamma \left\la x^k-x^*, g^k\right\ra\\
        &=& 
        \|x^k-x^*\|^2_2 + \gamma^2 \|\nabla f(x^k) + \theta_k\|^2_2 - 2\gamma \left\la x^k-x^*, \nabla f(x^k) + \theta_k \right\ra\\
        &\overset{\eqref{eq:squared_norm_sum}}{\le}& \|x^k-x^*\|^2_2 + 2 \gamma^2 \|\nabla f(x^k)\|^2_2 + 2\gamma^2 \|\theta_k\|^2_2 - 2\gamma \left\la x^k-x^*, \nabla f(x^k) + \theta_k\right\ra\\
        &\overset{\eqref{eq:L_smoothness_cor_2}}{\le}& \|x^k-x^*\|^2_2 + 4 \gamma^2 L \left( f(x^k) - f(x^*) \right) + 2\gamma^2 \|\theta_k\|^2_2 - 2\gamma \left\la x^k-x^*, \nabla f(x^k) + \theta_k\right\ra\\
        &\le& \|x^k-x^*\|^2_2 + (4 \gamma^2 L - 2\gamma)\left( f(x^k) - f(x^*) \right) + 2\gamma^2 \|\theta_k\|^2_2 - 2\gamma \left\la x^k-x^*, \theta_k \right\ra,
    \end{eqnarray*}
    where $\theta_k = \tnabla f(x^k, \Bxi^k) - \nabla f(x^k)$ and the last inequality follows from the convexity of $f$. Using notation $R_k \eqdef \|x^k - x^*\|_2$ we derive that for all $k\ge 0$
    \begin{equation*}
        R_{k+1}^2 \le R_k^2 + (4 \gamma^2 L - 2\gamma)\left( f(x^k) - f(x^*) \right) + 2\gamma^2 \|\theta_k\|^2_2 - 2\gamma \left\la x^k-x^*, \theta_k \right\ra.
    \end{equation*}
    Let us define $A = \left( 2\gamma - 4 \gamma^2 L \right)$, then
    \begin{equation*}
        A \left( f(x^k) - f(x^*) \right) \le R_k^2 - R_{k+1}^2 + 2\gamma^2 \|\theta_k\|^2_2 - 2\gamma \left\la x^k-x^*, \theta_k \right\ra.
    \end{equation*}
    Summing up these inequalities for $k=0, \dots, N-1 $ we obtain
    \begin{eqnarray*}
        \frac{A}{N} \sum\limits_{k=0}^{N-1}\left[ f(x^k) - f(x^*) \right] &\le& \frac{1}{N} \sum\limits_{k=0}^{N-1}\left( R_k^2 - R_{k+1}^2\right) + \frac{2\gamma^2}{N} \sum\limits_{k=0}^{N-1} \|\theta_k\|^2_2 - \frac{2\gamma^2}{N} \sum\limits_{k=0}^{N-1} \left\la x^k-x^*, \theta_k \right\ra\\
        &=& \frac{1}{N}\left( R_0^2 - R_{N}^2\right) +  \frac{2\gamma^2}{N} \sum\limits_{k=0}^{N-1} \|\theta_k\|^2_2 - \frac{2\gamma^2}{N} \sum\limits_{k=0}^{N-1} \left\la x^k-x^*, \theta_k \right\ra.
    \end{eqnarray*}
    Noticing that for $\Bar{x}^N = \frac{1}{N}\sum\limits_{k=0}^{N-1}x^k$ Jensen's inequality gives $f(\Bar{x}^N) = f\left(\frac{1}{N}\sum\limits_{k=0}^{N-1}x^k\right) \le \frac{1}{N}\sum\limits_{k=0}^{N-1}f(x^k)$ we have
    \begin{equation}
        A N \left( f(\Bar{x}^N) - f(x^*) \right) \le R_0^2 - R_{N}^2 + 2\gamma^2 \sum\limits_{k=0}^{N-1}\|\theta_k\|^2_2 - 2\gamma \sum\limits_{k=0}^{N-1} \left\la x^k-x^*, \theta_k \right\ra.
        \label{eq:main_thm_clipped_SGD_technical_0}
    \end{equation}
    Taking into account that $f(\Bar{x}^N) - f(x^*) \ge 0$ and changing the indices we get that for all $k\ge 0$
    \begin{equation}
         R_k^2 \le R_{0}^2 + 2\gamma^2 \sum\limits_{l=0}^{k-1}\|\theta_l\|^2_2 - 2\gamma \sum\limits_{l=0}^{k-1} \left\la x^l-x^*, \theta_k \right\ra.
         \label{eq:main_thm_clipped_SGD_technical_1}
    \end{equation}
    The remaining part of the proof is based on the analysis of inequality \eqref{eq:main_thm_clipped_SGD_technical_1}. In particular, via induction we prove that for all $k=0,1,\ldots, N$ with probability at least $1 - \frac{k\beta}{N}$ the following statement holds: inequalities
    \begin{eqnarray}
         R_t^2 \overset{\eqref{eq:main_thm_clipped_SGD_technical_1}}{\le} R_{0}^2 + 2\gamma^2 \sum\limits_{l=0}^{t-1}\|\theta_k\|^2_2 - 2\gamma \sum\limits_{l=0}^{t-1} \left\la x^k-x^*, \theta_k \right\ra \le C^2R_0^2 \label{eq:main_thm_clipped_SGD_technical_2}
    \end{eqnarray}
    hold for $t=0,1,\ldots,k$ simultaneously where $C$ is defined in \eqref{eq:C_definition_clipped_SGD}. Let us define the probability event when this statement holds as $E_k$. Then, our goal is to show that $\PP\{E_k\} \ge 1 - \frac{k\beta}{N}$ for all $k = 0,1,\ldots,N$. For $t = 0$ inequality \eqref{eq:main_thm_clipped_SGD_technical_2} holds with probability $1$ since $C \ge 1$. Next, assume that for some $k = T-1 \le N-1$ we have $\PP\{E_k\} = \PP\{E_{T-1}\} \ge 1 - \frac{(T-1)\beta}{N}$. Let us prove that $\PP\{E_{T}\} \ge 1 - \frac{T\beta}{N}$. First of all, probability event $E_{T-1}$ implies that
    \begin{eqnarray}
        f(\Bar{x}^N) - f(x^*) \overset{\eqref{eq:main_thm_clipped_SGD_technical_0}}{\le} \frac{1}{AN}\left(R_0^2 +  2\gamma^2 \sum\limits_{k=0}^{N-1}\|\theta_k\|^2_2 - 2\gamma \sum\limits_{k=0}^{N-1} \left\la x^k-x^*, \theta_k \right\ra\right)\notag \overset{\eqref{eq:main_thm_clipped_SGD_technical_2}}{\le} \frac{C^2R_0^2}{AN}\label{eq:main_thm_clipped_SGD_technical_3}
    \end{eqnarray}
    hold for $t=0,1,\ldots,T-1$. Since $f$ is $L$-smooth, we have that probability event $E_{T-1}$ implies
    \begin{eqnarray}
        \left\|\nabla f(x^{t})\right\|_2 \le L\|x^t - x^*\|_2 \le LCR_0 = \frac{\lambda}{2} \label{eq:main_thm_clipped_SGD_technical_4}
    \end{eqnarray}
    for $t=0,\ldots,T-1$, where the clipping level is defined as
    \begin{eqnarray}
        \lambda = 2LCR_0. \label{eq:main_thm_clipped_SGD_technical_4_0}
    \end{eqnarray}
    Having inequalities \eqref{eq:main_thm_clipped_SGD_technical_4} in hand we show in the rest of the proof that \eqref{eq:main_thm_clipped_SGD_technical_2} holds for $t = T$ with big enough probability. First of all, we introduce new random variables:
    \begin{equation}
        \eta_l = \begin{cases}x^* - z^l,&\text{if } \|x^* - z^l\|_2 \le CR_0,\\ 0,&\text{otherwise,} \end{cases}\quad \label{eq:main_thm_clipped_SGD_technical_4_1}
    \end{equation}
    for $l=0,1,\ldots T-1$. Note that these random variables are bounded with probability $1$, i.e.\ with probability $1$ we have
    \begin{equation}
        \|\eta_l\|_2 \le CR_0.\label{eq:main_thm_clipped_SGD_technical_4_2}
    \end{equation}
    Secondly, we use the introduced notation and get that $E_{T-1}$ implies 
    \begin{eqnarray*}
         R_T^2 &\overset{\eqref{eq:main_thm_clipped_SGD_technical_1},\eqref{eq:main_thm_clipped_SGD_technical_2},\eqref{eq:main_thm_clipped_SGD_technical_4},\eqref{eq:main_thm_clipped_SGD_technical_4_1}}{\le}& R_0^2 + 2\gamma \sum\limits_{l=0}^{T-1}\left\la\theta_{l}, \eta_l\right\ra + 2\gamma^2\sum\limits_{l=0}^{T-1}\|\theta_{l+1}\|_2^2 \notag.
    \end{eqnarray*}
    Finally, we do some preliminaries in order to apply Bernstein's inequality (see Lemma~\ref{lem:Bernstein_ineq}) and obtain that $E_{T-1}$ implies
    \begin{eqnarray}
        R_T^2 &\overset{\eqref{eq:squared_norm_sum}}{\le}& R_0^2 + \underbrace{2\gamma\sum\limits_{l=0}^{T-1}\left\la\theta_{l}^u, \eta_l\right\ra}_{\circledOne}+ \underbrace{2\gamma\sum\limits_{l=0}^{T-1}\left\la\theta_{l}^b, \eta_l \right\ra}_{\circledTwo}\notag + \underbrace{4\gamma^2\sum\limits_{l=0}^{T-1}\left(\|\theta_{l}^u\|_2^2 - \EE_{\Bxi^l}\left[\|\theta_{l}^u\|_2^2\right]\right)}_{\circledThree}\\
        &&\quad  + \underbrace{4\gamma^2\sum\limits_{l=0}^{T-1}\EE_{\Bxi^l}\left[\|\theta_{l}^u\|_2^2\right]}_{\circledFour} + \underbrace{4\gamma^2\sum\limits_{l=0}^{T-1}\|\theta_{l}^b\|_2^2}_{\circledFive}\label{eq:main_thm_clipped_SGD_technical_5}
    \end{eqnarray}
    where we introduce new notations:
    \begin{equation}
        \theta_{l}^u \eqdef \tnabla f(x^{l},\Bxi^l) - \EE_{\Bxi^l}\left[\tnabla f(x^{l},\Bxi^l)\right],\quad \theta_{l}^b \eqdef \EE_{\Bxi^l}\left[\tnabla f(x^{l},\Bxi^l)\right] - \nabla f(x^{l}),\label{eq:main_thm_clipped_SGD_technical_6}
    \end{equation}
    \begin{equation}
        \theta_{l} = \theta_{l}^u +\theta_{l}^b.\notag
    \end{equation}
    It remains to provide tight upper bounds for $\circledOne$, $\circledTwo$, $\circledThree$, $\circledFour$ and $\circledFive$, i.e.\ in the remaining part of the proof we show that $\circledOne+\circledTwo+\circledThree+\circledFour+\circledFive \le \delta C^2R_0^2$ for some $\delta < 1$.
    
    \textbf{Upper bound for $\circledOne$.} First of all, since $\EE_{\Bxi^l}[\theta_{l}^u] = 0$ summands in $\circledOne$ are conditionally unbiased:
    \begin{equation*}
        \EE_{\Bxi^l}\left[2\gamma \left\la \theta_l^u, \eta_l\right\ra\right] = 0.
    \end{equation*}
    Secondly, these summands are bounded with probability $1$:
    \begin{eqnarray*}
        \left|2\gamma \left\la \theta_l^u, \eta_l\right\ra\right| &\le& 2\gamma\|\theta_{l}^u\|_2\left\|\eta_l\right\|_2 \overset{\eqref{eq:magnitude_bound_clipped_SSTM},\eqref{eq:main_thm_clipped_SGD_technical_4_2}}{\le} 4\gamma\lambda CR_0 \overset{\eqref{eq:main_thm_clipped_SGD_technical_4_0}}{=} 8\gamma (CR_0)^2L.
    \end{eqnarray*}
    Finally, one can bound conditional variances $\sigma_l^2 \eqdef \EE_{\Bxi^l}\left[4\gamma^2\left\la\theta_{l}^u, \eta_l\right\ra^2\right]$ in the following way:
    \begin{eqnarray}
         \sigma_l^2 &\le& \EE_{\Bxi^l}\left[4\gamma^2\left\|\theta_{l}^u\right\|_2^2 \left\|\eta_l\right\|_2^2\right] \overset{\eqref{eq:main_thm_clipped_SGD_technical_4_2}}{\le} 4\gamma^2(CR_0)^2\EE_{\Bxi^l}\left[\left\|\theta_{l}^u\right\|_2^2\right].\notag
    \end{eqnarray}
    In other words, sequence $\left\{2\gamma \left\la \theta_l^u, \eta_l\right\ra\right\}_{l\ge 0}$ is a bounded martingale difference sequence with bounded conditional variances $\{\sigma_l^2\}_{l\ge 0}$.  Therefore, we can apply Bernstein's inequality, i.e.\ we apply Lemma~\ref{lem:Bernstein_ineq} with $X_l = 2\gamma \left\la \theta_l^u, \eta_l\right\ra$, $c = 8\gamma (CR_0)^2L$ and $F = \frac{c^2\ln\frac{4N}{\beta}}{6}$ and get that for all $b > 0$
    \begin{equation*}
        \PP\left\{\left|\sum\limits_{l=0}^{T-1}X_l\right| > b\text{ and } \sum\limits_{l=0}^{T-1}\sigma_l^2 \le F\right\} \le 2\exp\left(-\frac{b^2}{2F + \nicefrac{2cb}{3}}\right)
    \end{equation*}
    or, equivalently, with probability at least $1 - 2\exp\left(-\frac{b^2}{2F + \nicefrac{2cb}{3}}\right)$
    \begin{equation*}
        \text{either } \sum\limits_{l=0}^{T-1}\sigma_l^2 > F \quad \text{or} \quad \underbrace{\left|\sum\limits_{l=0}^{T-1}X_l\right|}_{|\circledOne|} \le b.
    \end{equation*}
    The choice of $F$ will be clarified further, let us now choose $b$ in such a way that $2\exp\left(-\frac{b^2}{2F + \nicefrac{2cb}{3}}\right) = \frac{\beta}{2N}$. This implies that $b$ is the positive root of the quadratic equation \begin{equation*}
        b^2 - \frac{2c\ln\frac{4N}{\beta}}{3}b - 2F\ln\frac{4N}{\beta} = 0,
    \end{equation*}
    hence
    \begin{eqnarray*}
         b &=& \frac{c\ln\frac{4N}{\beta}}{3} + \sqrt{\frac{c^2\ln^2\frac{4N}{\beta}}{9} + 2F\ln\frac{4N}{\beta}}=\frac{c\ln\frac{4N}{\beta}}{3} + \sqrt{\frac{4c^2\ln^2\frac{4N}{\beta}}{9}}\\
         &=& c\ln\frac{4N}{\beta} = 8\gamma (CR_0)^2L\ln\frac{4N}{\beta}.
    \end{eqnarray*}
    That is, with probability at least $1 - \frac{\beta}{2N}$
     \begin{equation*}
        \underbrace{\text{either } \sum\limits_{l=0}^{T-1}\sigma_l^2 > F \quad \text{or} \quad \left|\circledOne\right| \le 8\gamma (CR_0)^2L\ln\frac{4N}{\beta}}_{\text{probability event } E_{\circledOne}}.
    \end{equation*}
    Next, we notice that probability event $E_{T-1}$ implies that
    \begin{eqnarray*}
        \sum\limits_{l=0}^{T-1}\sigma_l^2 &\le& 4\gamma^2(CR_0)^2\sum\limits_{l=0}^{T-1}\EE_{\Bxi^l}\left[\|\theta_l^u\|_2^2\right] \overset{\eqref{eq:variance_bound_clipped_SSTM}}{\le} 72\gamma^2 (CR_0)^2 \sigma^2 \frac{T}{m}\\
        &\overset{\eqref{eq:bathces_clipped_SGD}}{\le}& 72\gamma^2 (CR_0)^2 \sigma^2\frac{2T(CR_0)^2L^2\ln\frac{4N}{\beta}}{27N\sigma^2}\\
        &\overset{T\le N}{\le}& \frac{16}{3}\gamma^2 (CR_0)^4 L^2 \ln\frac{4N}{\beta} \le \frac{c^2\ln\frac{4N}{\beta}}{6} = F,
    \end{eqnarray*}
    where the last inequality follows from $c = 8\gamma (CR_0)^2L$ and simple arithmetic.
    
     \textbf{Upper bound for $\circledTwo$.} First of all, we notice that probability event $E_{T-1}$ implies
     \begin{eqnarray*}
        2\gamma\left\la\theta_{l}^b, \eta_l \right\ra &\le& 2\gamma\left\|\theta_{l}^b\right\|_2\left\|\eta_l\right\|_2\overset{\eqref{eq:bias_bound_clipped_SSTM},\eqref{eq:main_thm_clipped_SGD_technical_4_2}}{\le} 2\gamma\frac{4\sigma^2}{m \lambda}CR_0 \overset{\eqref{eq:main_thm_clipped_SGD_technical_4_0}}{=} \frac{4\gamma \sigma^2}{Lm}.
    \end{eqnarray*}
    This implies that
    \begin{eqnarray*}
        \circledTwo &=& 2\gamma\sum\limits_{l=0}^{T-1}\left\la\theta_{l}^b,\eta_l\right\ra \overset{T\le N}{\le} \frac{4\gamma N \sigma^2}{mL}.
    \end{eqnarray*}
    
     \textbf{Upper bound for $\circledThree$.} We derive the upper bound for $\circledThree$ using the same technique as for $\circledOne$. First of all, we notice that the summands in $\circledThree$ are conditionally independent:
    \begin{equation*}
        \EE_{\Bxi^l}\left[4\gamma^2\left(\|\theta_{l}^u\|_2^2 - \EE_{\Bxi^l}\left[\|\theta_{l}^u\|_2^2\right]\right)\right] = 0.
    \end{equation*}
    Secondly, the summands are bounded with probability $1$:
    \begin{eqnarray}
        \left|4\gamma^2\left(\|\theta_{l}^u\|_2^2 - \EE_{\Bxi^l}\left[\|\theta_{l}^u\|_2^2\right]\right)\right| &\le& 4\gamma^2\left(\|\theta_{l}^u\|_2^2 + \EE_{\Bxi^l}\left[\|\theta_{l}^u\|_2^2\right]\right)\overset{\eqref{eq:magnitude_bound_clipped_SSTM}}{\le} 4\gamma^2\left(4\lambda^2 + 4\lambda^2\right)\notag\\
        &\overset{\eqref{eq:main_thm_clipped_SGD_technical_4_0}}{=}& 128 \gamma^2 (CR_0)^2 L^2  \eqdef c_1.\label{eq:main_thm_clipped_SGD_technical_8}
    \end{eqnarray}
    Finally, one can bound conditional variances $\hat \sigma_l^2 \eqdef \EE_{\Bxi^l}\left[\left|4\gamma^2\left(\|\theta_{l}^u\|_2^2 - \EE_{\Bxi^l}\left[\|\theta_{l}^u\|_2^2\right]\right)\right|^2\right]$ in the following way:
    \begin{eqnarray}
         \hat \sigma_l^2 &\overset{\eqref{eq:main_thm_clipped_SGD_technical_8}}{\le}& c_1\EE_{\Bxi^l}\left[\left|4\gamma^2\left(\|\theta_{l}^u\|_2^2 - \EE_{\Bxi^l}\left[\|\theta_{l}^u\|_2^2\right]\right)\right|\right]\notag\\
         &\le& 4\gamma^2c_1\EE_{\Bxi^l}\left[\|\theta_{l}^u\|_2^2 + \EE_{\Bxi^l}\left[\|\theta_{l}^u\|_2^2\right]\right] = 8\gamma^2c_1\EE_{\Bxi^l}\left[\|\theta_{l}^u\|_2^2\right].\label{eq:main_thm_clipped_SGD_technical_9}
    \end{eqnarray}
    In other words, sequence $\left\{4\gamma^2\left(\|\theta_{l}^u\|_2^2 - \EE_{\Bxi^l}\left[\|\theta_{l}^u\|_2^2\right]\right)\right\}_{l\ge 0}$ is a bounded martingale difference sequence with bounded conditional variances $\{\hat \sigma_l^2\}_{l\ge 0}$. Therefore, we can apply Bernstein's inequality, i.e.\ we apply Lemma~\ref{lem:Bernstein_ineq} with $X_l = \hat X_l = 4\gamma^2\left(\|\theta_{l}^u\|_2^2 - \EE_{\Bxi^l}\left[\|\theta_{l}^u\|_2^2\right]\right)$, $c = c_1 = 128 \gamma^2 (CR_0)^2 L^2$ and $F = F_1 = \frac{c_1^2\ln\frac{4N}{\beta}}{6}$ and get that for all $b > 0$
    \begin{equation*}
        \PP\left\{\left|\sum\limits_{l=0}^{T-1}\hat X_l\right| > b\text{ and } \sum\limits_{l=0}^{T-1}\hat \sigma_l^2 \le F_1\right\} \le 2\exp\left(-\frac{b^2}{2F_1 + \nicefrac{2c_1b}{3}}\right)
    \end{equation*}
    or, equivalently, with probability at least $1 - 2\exp\left(-\frac{b^2}{2F_1 + \nicefrac{2c_1b}{3}}\right)$
    \begin{equation*}
        \text{either } \sum\limits_{l=0}^{T-1}\hat\sigma_l^2 > F_1 \quad \text{or} \quad \underbrace{\left|\sum\limits_{l=0}^{T-1}\hat X_l\right|}_{|\circledThree|} \le b.
    \end{equation*}
    As in our derivations of the upper bound for $\circledOne$ we choose such $b$ that $2\exp\left(-\frac{b^2}{2F_1 + \nicefrac{2c_1b}{3}}\right) = \frac{\beta}{2N}$, i.e.\
    \begin{eqnarray*}
         b &=& \frac{c_1\ln\frac{4N}{\beta}}{3} + \sqrt{\frac{c_1^2\ln^2\frac{4N}{\beta}}{9} + 2F_1\ln\frac{4N}{\beta}} = c_1\ln\frac{4N}{\beta} = 128\gamma^2 (CR_0)^2L^2\ln\frac{4N}{\beta} .
    \end{eqnarray*}
    That is, with probability at least $1 - \frac{\beta}{2N}$
     \begin{equation*}
        \underbrace{\text{either } \sum\limits_{l=0}^{T-1}\hat\sigma_l^2 > F_1 \quad \text{or} \quad \left|\circledThree\right| \le 128\gamma^2 (CR_0)^2L^2\ln\frac{4N}{\beta}}_{\text{probability event } E_{\circledThree}}.
    \end{equation*}
    Next, we notice that probability event $E_{T-1}$ implies that
    \begin{eqnarray*}
        \sum\limits_{l=0}^{T-1}\hat\sigma_l^2 &\overset{\eqref{eq:main_thm_clipped_SGD_technical_9}}{\le}& 8\gamma^2 c_1\sum\limits_{l=0}^{T-1}\EE_{\Bxi^l}\left[\left\|\theta_{l}^u\right\|_2^2\right]\overset{\eqref{eq:variance_bound_clipped_SSTM}}{\le} 144\gamma^2 c_1 \sigma^2\frac{T}{m}\\
        &\overset{\eqref{eq:bathces_clipped_SGD}}{\le}& 
        \frac{32}{3}\gamma^2 c_1 (CR_0)^2L^2 \frac{T}{N}\ln\frac{4N}{\beta}\\
        &\overset{T\le N}{\le}& \frac{c_1^2\ln\frac{4N}{\beta}}{6} \le F_1.
    \end{eqnarray*}
    
    \textbf{Upper bound for $\circledFour$.} The probability event $E_{T-1}$ implies
    \begin{eqnarray*}
        \circledFour &=& 4\gamma^2\sum\limits_{l=0}^{T-1}\EE_{\Bxi^l}\left[\|\theta_{l}^u\|_2^2\right] \overset{\eqref{eq:variance_bound_clipped_SSTM}}{\le} 72\gamma^2\sigma^2\sum\limits_{l=0}^{T-1}\frac{1}{m} \overset{T\le N}{\le} \frac{72\gamma^2N\sigma^2}{m}.
    \end{eqnarray*}
    
    \textbf{Upper bound for $\circledFive$.} Again, we use corollaries of probability event $E_{T-1}$:
    \begin{eqnarray*}
        \circledFive &=& 4\gamma^2\sum\limits_{l=0}^{T-1}\|\theta_{l}^b\|_2^2 \overset{\eqref{eq:bias_bound_clipped_SSTM}}{\le} 64 \gamma^2 \sigma^4\frac{T}{m^2\lambda^2} \overset{\eqref{eq:main_thm_clipped_SGD_technical_4_0}}{=} \frac{64\gamma^2 \sigma^4}{4(CR_0)^2L^2}\cdot\frac{T}{m^2}\overset{T\le N}{\le} \frac{16\gamma^2 N \sigma^4}{(CR_0)^2 L^2 m^2}.
    \end{eqnarray*}
    
    Now we summarize all bound that we have: probability event $E_{T-1}$ implies
    \begin{eqnarray*}
         R_T^2 &\overset{\eqref{eq:main_thm_clipped_SGD_technical_1}}{\le}& R_{0}^2 + 2\gamma^2 \sum\limits_{l=0}^{T-1}\|\theta_l\|^2_2 - 2\gamma \sum\limits_{l=0}^{T-1} \left\la x^l-x^*, \theta_l \right\ra\\
        &\overset{\eqref{eq:main_thm_clipped_SGD_technical_5}}{\le}& R_0^2 + \circledOne + \circledTwo + \circledThree + \circledFour + \circledFive,\\
        \circledTwo &\le& \frac{4\gamma N \sigma^2}{mL},\quad \circledFour \le \frac{72\gamma^2N\sigma^2}{m},\quad \circledFive \le \frac{16\gamma^2 N \sigma^4}{(CR_0)^2 L^2 m^2},\\
        \sum\limits_{l=0}^{T-1}\sigma_l^2 &\le& F,\quad \sum\limits_{l=0}^{T-1}\hat\sigma_l^2 \le F_1
    \end{eqnarray*}
    and
    \begin{equation*}
        \PP\{E_{T-1}\} \ge 1 - \frac{(T-1)\beta}{N},\quad \PP\{E_\circledOne\} \ge 1 - \frac{\beta}{2N},\quad \PP\{E_\circledThree\} \ge 1 - \frac{\beta}{2N},
    \end{equation*}
    where
    \begin{eqnarray*}
        E_{\circledOne} &=& \left\{\text{either } \sum\limits_{l=0}^{T-1}\sigma_l^2 > F \quad \text{or} \quad \left|\circledOne\right| \le 8\gamma (CR_0)^2L\ln\frac{4N}{\beta}\right\},\\
        E_{\circledThree} &=& \left\{\text{either } \sum\limits_{l=0}^{T-1}\hat\sigma_l^2 > F_1 \quad \text{or} \quad \left|\circledThree\right| \le 128\gamma^2 (CR_0)^2L^2\ln\frac{4N}{\beta}\right\}.
    \end{eqnarray*}
    Taking into account these inequalities and our assumptions on $m$ and $\gamma$ (see \eqref{eq:bathces_clipped_SGD} and \eqref{eq:step_size_clipped_SGD}) we get that probability event $E_{T-1}\cap E_\circledOne \cap E_\circledThree$ implies
    \begin{eqnarray}
         R_T^2 &\overset{\eqref{eq:main_thm_clipped_SGD_technical_1}}{\le}& R_{0}^2 + 2\gamma^2 \sum\limits_{l=0}^{T-1}\|\theta_l\|^2_2 - 2\gamma \sum\limits_{l=0}^{T-1} \left\la x^l-x^*, \theta_l \right\ra\notag\\
         &\le& R_0^2 + \left(\frac{1}{10} + \frac{1}{10} + \frac{1}{10} + \frac{1}{10} + \frac{1}{10}\right)C^2R_0^2 \le \left(1 + \frac{1}{2}C^2\right)R_0^2 \overset{\eqref{eq:C_definition_clipped_SGD}}{\le} C^2R_0^2.\label{eq:main_thm_clipped_SGD_technical_10}
    \end{eqnarray}
    Moreover, using union bound we derive
    \begin{equation}
        \PP\left\{E_{T-1}\cap E_\circledOne \cap E_\circledThree\right\} = 1 - \PP\left\{\overline{E}_{T-1}\cup\overline{E}_\circledOne\cup\overline{E}_\circledThree\right\} \ge 1 - \frac{T\beta}{N}.\label{eq:main_thm_clipped_SGD_technical_11}
    \end{equation}
    That is, by definition of $E_T$ and $E_{T-1}$ we have proved that
    \begin{eqnarray*}
        \PP\{E_T\} &\overset{\eqref{eq:main_thm_clipped_SGD_technical_10}}{\ge}& \PP\left\{E_{T-1}\cap E_\circledOne \cap E_\circledThree\right\} \overset{\eqref{eq:main_thm_clipped_SGD_technical_11}}{\ge} 1 - \frac{T\beta}{N},
    \end{eqnarray*}
    which implies that for all $k = 0,1,\ldots, N$ we have $\PP\{E_k\} \ge 1 - \frac{k\beta}{N}$. Then, for $k = N$ we have that with probability at least $1-\beta$
    \begin{eqnarray*}
        ANf(\Bar{x}^N) - f(x^*) \overset{\eqref{eq:main_thm_clipped_SGD_technical_0}}{\le} R_0^2 +  2\gamma^2 \sum\limits_{k=0}^{N-1}\|\theta_k\|^2_2 - 2\gamma \sum\limits_{k=0}^{N-1} \left\la x^k-x^*, \theta_k \right\ra \overset{\eqref{eq:main_thm_clipped_SGD_technical_2}}{\le} C^2R_0^2.
    \end{eqnarray*}
    Since $A = 2\gamma\left( 1 - 2 \gamma L \right)$ and $1-\gamma L \ge \frac{1}{2}$ we get that with probability at least $1-\beta$
    \begin{eqnarray*}
        f(\Bar{x}^N) - f(x^*) &\le& \frac{C^2R_0^2}{AN} \le \frac{C^2R_0^2}{ \gamma N} \overset{\eqref{eq:step_size_clipped_SGD}}{\le} \frac{80 C^2R_0^2L \ln\frac{4N}{\beta}}{N}.
    \end{eqnarray*}
    In other words, {\tt clipped-SGD} achieves $f(\Bar{x}^N) - f(x^*) \le \varepsilon$ with probability at least $1-\beta$ after $O\left(\frac{LR_0^2}{\varepsilon}\ln\frac{LR_0^2}{\varepsilon\beta}\right)$ iterations and requires
    \begin{eqnarray*}
         \sum\limits_{k=0}^{N-1}m_k &\overset{\eqref{eq:bathces_clipped_SGD}}{=}& \sum\limits_{k=0}^{N-1}O\left(\max\left\{1,\frac{N\sigma^2}{C^2R_0^2L^2\ln\frac{N}{\beta}}\right\}\right) = O\left(\max\left\{N,\frac{N^2\sigma^2}{C^2R_0^2L^2\ln\frac{N}{\beta}}\right\}\right)\\
         &=& O\left(\max\left\{\frac{LR_0^2}{\varepsilon},\frac{\sigma^2R_0^2}{\varepsilon^2}\right\}\ln\frac{LR_0^2}{\varepsilon\beta}\right)
    \end{eqnarray*}
    oracle calls.

\subsubsection{Proof of Theorem~\ref{thm:main_result_R_clipped_SGD}}
First of all, consider behavior of {\tt clipped-SGD} during the first run in {\tt R-clipped-SGD}. We notice that the proof of Theorem~\ref{thm:main_result_clipped_SGD} will be valid if we substitute $R_0$ everywhere by its upper bound $R$. From $\mu$-strong convexity of $f$ we have
\begin{equation*}
    R_0^2 = \|x^0 - x^*\|_2^2 \overset{\eqref{eq:str_cvx_cor}}{\le} \frac{2}{\mu}\left(f(x^0) - f(x^*)\right),
\end{equation*}
therefore, one can choose $R = \sqrt{\frac{2}{\mu}\left(f(x^0) - f(x^*)\right)}$. It implies that after $N_0$ iterations of {\tt clipped-SGD} we have
\begin{equation*}
    f(\Bar{x}^{N_0}) - f(x^*) \le \frac{80LC^2R^2\ln{\frac{4N_0\tau}{\beta}}}{N_0} = \frac{160LC^2R^2\ln{\frac{4N_0\tau}{\beta}}}{N_0\mu} (f(x^0) - f(x^*)).
\end{equation*}
with probability at least $1-\frac{\beta}{\tau}$, hence with the same probability $f(\Bar{x}^{N_0}) - f(x^*) \le \frac{1}{2}(f(x^0) - f(x^*))$ since $\frac{N_0}{\ln \frac{4N_0\tau}{\beta}} \ge \frac{320C^2L}{\mu}$. In other words, with probability at least $1-\frac{\beta}{\tau}$
\begin{equation*}
    f(\hat x^1) - f(x^*) \le \frac{1}{2}\left(f(x^0) - f(x^*)\right) = \frac{1}{4}\mu R^2.
\end{equation*}
Then, by induction one can show that for arbitrary $k\in\{0,1,\ldots,\tau-1\}$ the inequality
\begin{equation*}
    f(\hat x^{k+1}) - f(x^*) \le \frac{1}{2}\left(f(\hat x^k) - f(x^*)\right)
\end{equation*}
holds with probability at least $1-\frac{\beta}{\tau}$. Therefore, these inequalities hold simultaneously with probability at least $1 - \beta$. Using this we derive that inequality
\begin{eqnarray*}
    f(\hat x^{\tau}) - f(x^*) &\le& \frac{1}{2}\left(f(\hat x^{\tau-1}) - f(x^*)\right) \le \frac{1}{2^2}\left(f(\hat x^{\tau-2}) - f(x^*)\right) \le \ldots \le \frac{1}{2^\tau}\left(f(x^0) - f(x^*)\right)\\
    &=& \frac{\mu R^2}{2^{\tau+1}}
\end{eqnarray*}
holds with probability $\ge 1- \beta$. That is, after $\tau = \left\lceil\log_2\frac{\mu R^2}{2\varepsilon}\right\rceil$ restarts {\tt R-clipped-SGD} generates such point $\hat x^{\tau}$ that $f(\hat x^{\tau}) - f(x^*) \le \varepsilon$ with probability at least $1-\beta$. Moreover, if $\frac{N_0}{\ln \frac{4N_0\tau}{\beta}} \le \frac{C_1L}{\mu}$ with some numerical constant $C_1 \ge 320C^2$, then the total number of iterations of {\tt clipped-SGD} equals
\begin{equation*}
    N_0\tau = O\left(\frac{L}{\mu}\ln\left(\frac{\mu R^2}{\varepsilon}\right)\ln\left(\frac{L}{\mu\beta}\ln\frac{\mu R^2}{\varepsilon}\right)\right)
\end{equation*}
and the overall number of stochastic first-order oracle calls is
\begin{eqnarray*}
    \sum\limits_{t=0}^{\tau-1}N_0m^t &=& \sum\limits_{t=0}^{\tau-1}O\left(\max\left\{N_0,\frac{  2^tN_0^2\sigma^2}{R^2L^2\ln\frac{4N_0\tau}{\beta}}\right\}\right)\\
    &=& O\left(\max\left\{N_0\tau,\sum\limits_{t=0}^{\tau-1}\frac{  2^tN_0^2\sigma^2}{R^2L^2\ln\frac{4N_0\tau}{\beta}}\right\}\right)\\
    &=& O\left(\max\left\{\frac{L}{\mu}\ln\left(\frac{\mu R^2}{\varepsilon}\right), \frac{\sigma^2}{\mu\varepsilon}\right\} \ln\left(\frac{L}{\mu\beta}\ln\frac{\mu R^2}{\varepsilon}\right)\right).
\end{eqnarray*}

\subsubsection{Proof of Theorem~\ref{thm:main_result_clipped_SGD_2}}\label{sec:proof_str_cvx_clipped_SGD}
Since $f$ is $L$-smooth we have
    \begin{eqnarray}
        f(x^{k+1}) &\le& f(x^k) - \gamma\left\la \nabla f(x^k), \tnabla f(x^{k},\Bxi^k)\right\ra + \frac{L\gamma^2}{2}\|\tnabla f(x^{k},\Bxi^k)\|_2^2\notag\\
        &\overset{\eqref{eq:squared_norm_sum}}{\le}& f(x^k) - \gamma\|\nabla f(x^k)\|^2_2 - \gamma \left\la \nabla f(x^k), \theta_k\right\ra + L\gamma^2 \|\nabla f(x^k)\|^2_2 + L\gamma^2\|\theta_k\|^2_2\notag\\
        &=& f(x^k) - \gamma(1-L\gamma)\|\nabla f(x^k)\|_2^2 - \gamma \left\la \nabla f(x^k), \theta_k\right\ra + L\gamma^2\|\theta_k\|^2_2\notag\\
        &\le& f(x^k) - \frac{\gamma}{2}\|\nabla f(x^k)\|_2^2 - \gamma \left\la \nabla f(x^k), \theta_k\right\ra + L\gamma^2\|\theta_k\|_2^2,\notag\\
        \theta_{k} &\eqdef& \tnabla f(x^{k},\Bxi^k) - \nabla f(x^{k})\label{eq:main_thm_clipped_SGD_technical_theta_2_2}
    \end{eqnarray}
where in the last inequality we use $1-\gamma L \ge \frac{1}{2}$. Next, $\mu$-strong convexity of $f$ implies $\|\nabla f(x^k)\|_2^2 \ge 2\mu (f(x^k)-f(x^*))$ and
    \begin{eqnarray*}
        f(x^{k+1})-f(x^*)) &\le& f(x^k) - f(x^*) - \gamma\mu (f(x^k)-f(x^*)) - \gamma \left\la \nabla f(x^k), \theta_k \right\ra + L\gamma^2 \|\theta_k\|^2_2\\
        &=& (1-\gamma\mu)(f(x^k)-f(x^*)) - \gamma \left\la \nabla f(x^k), \theta_k\right\ra + L\gamma^2\|\theta_k\|_2^2.
    \end{eqnarray*}
Unrolling the recurrence we obtain
\begin{eqnarray}
        f(x^N)-f(x^*)) 
        &\le& (1-\gamma\mu)^N(f(x^0)-f(x^*)) + \gamma \sum\limits_{l=0}^{N-1} (1-\gamma\mu)^{N-1-l} \left\la-\nabla f(x^l), \theta_l \right\ra\notag\\
        && + L\gamma^2 \sum\limits_{l=0}^{N-1} (1-\gamma\mu)^{N-1-l}\|\theta_l\|_2^2,\label{eq:main_thm_clipped_SGD_technical_0_2}
\end{eqnarray}
for all $N\ge 0$. Using notation $r_k \eqdef f(x^k)-f(x^*)$ we rewrite this inequality in the following form:
    \begin{equation}
        r_k \le (1-\gamma\mu)^k r_0 + \gamma \sum\limits_{l=0}^{k-1} (1-\gamma\mu)^{k-1-l} \left\la-\nabla f(x^l), \theta_l \right\ra + L\gamma^2 \sum\limits_{l=0}^{k-1} (1-\gamma\mu)^{k-1-l}\|\theta_l\|_2^2. \label{eq:main_thm_clipped_SGD_technical_1_2}
    \end{equation}
    The rest of the proof is based on the refined analysis of inequality \eqref{eq:main_thm_clipped_SGD_technical_1_2}. In particular, via induction we prove that for all $k=0,1,\ldots, N$ with probability at least $1 - \frac{k\beta}{N}$ the following statement holds: inequalities
    \begin{eqnarray}
         r_t &\overset{\eqref{eq:main_thm_clipped_SGD_technical_1_2}}{\le}& (1-\gamma\mu)^t r_0 + \gamma \sum\limits_{l=0}^{t-1} (1-\gamma\mu)^{t-1-l} \left\la-\nabla f(x^l), \theta_l \right\ra + L\gamma^2 \sum\limits_{l=0}^{t-1} (1-\gamma\mu)^{t-1-l}\|\theta_l\|_2^2\notag\\
         &\le& 2(1-\gamma\mu)^t r_0 \label{eq:main_thm_clipped_SGD_technical_2_2}
    \end{eqnarray}
    hold for $t=0,1,\ldots,k$ simultaneously. Let us define the probability event when this statement holds as $E_k$. Then, our goal is to show that $\PP\{E_k\} \ge 1 - \frac{k\beta}{N}$ for all $k = 0,1,\ldots,N$. For $t = 0$ inequality \eqref{eq:main_thm_clipped_SGD_technical_2_2} holds with probability $1$ since $2(1-\gamma\mu)^0 \ge 1$, hence $\PP\{E_0\} = 1$. Next, assume that for some $k = T-1 \le N-1$ we have $\PP\{E_k\} = \PP\{E_{T-1}\} \ge 1 - \frac{(T-1)\beta}{N}$. Let us prove that $\PP\{E_{T}\} \ge 1 - \frac{T\beta}{N}$. First of all, probability event $E_{T-1}$ implies that
     \begin{eqnarray}
        f(x^t) - f(x^*) &\overset{\eqref{eq:main_thm_clipped_SGD_technical_2_2}}{\le}& 2(1-\gamma\mu)^t r_0\label{eq:main_thm_clipped_SGD_technical_3_str}
    \end{eqnarray}
     hold for $t=0,1,\ldots,T-1$. Since $f$ is $L$-smooth, we have that probability event $E_{T-1}$ implies
    \begin{eqnarray}
        \left\|\nabla f(x^{l})\right\|_2 \le \sqrt{2L(f(x^l)-f(x^*))} \le \sqrt{4L(1-\gamma\mu)^lr_0} = \frac{\lambda_l}{2} \label{eq:main_thm_clipped_SGD_technical_4_str}
    \end{eqnarray}
    for $t=0,\ldots,T-1$ and 
    \begin{eqnarray}
        \lambda_l = 4 \sqrt{L (1-\gamma\mu)^l r_0}. \label{eq:main_thm_clipped_SGD_technical_4_0_str}
    \end{eqnarray}
    Having inequalities \eqref{eq:main_thm_clipped_SGD_technical_4_str} in hand we show in the rest of the proof that \eqref{eq:main_thm_clipped_SGD_technical_2_2} holds for $t = T$ with big enough probability. First of all, we introduce new random variables:
    \begin{equation}
        \zeta_l = \begin{cases}-\nabla f(x^{l+1}),&\text{if } \|\nabla f(x^{l+1})\|_2 \le \frac{\lambda_l}{2},\\ 0,&\text{otherwise,} \end{cases} \label{eq:main_thm_clipped_SGD_technical_4_1_str}
    \end{equation}
    for $l=0,1,\ldots T-1$. Note that these random variables are bounded with probability $1$, i.e.\ with probability $1$ we have
    \begin{equation}
        \|\zeta_l\|_2 \le \frac{\lambda_l}{2}.\label{eq:main_thm_clipped_SGD_technical_4_2_str}
    \end{equation}
    Secondly, we use the introduced notation and get that $E_{T-1}$ implies 
    \begin{eqnarray*}
         r_T &\overset{\eqref{eq:main_thm_clipped_SGD_technical_1_2},\eqref{eq:main_thm_clipped_SGD_technical_2_2},\eqref{eq:main_thm_clipped_SGD_technical_4_str},\eqref{eq:main_thm_clipped_SGD_technical_4_1_str}}{\le}& (1-\gamma\mu)^Tr_0 + \gamma \sum\limits_{l=0}^{T-1}(1-\gamma\mu)^{T-1-l}\left\la \zeta_l, \theta_l \right\ra\\
         &&+ L \gamma^2 \sum\limits_{l=0}^{T-1} (1-\gamma\mu)^{T-1-l}\|\theta_l\|_2^2\notag.
    \end{eqnarray*}
    Finally, we do some preliminaries in order to apply Bernstein's inequality (see Lemma~\ref{lem:Bernstein_ineq}) and obtain that $E_{T-1}$ implies
    \begin{eqnarray}
        r_T &\overset{\eqref{eq:squared_norm_sum}}{\le}& (1-\gamma\mu)^Tr_0 + \underbrace{\gamma\sum\limits_{l=0}^{T-1}(1-\gamma\mu)^{T-1-l}\left\la\theta_{l}^u, \zeta_l\right\ra}_{\circledOne}+ \underbrace{\gamma\sum\limits_{l=0}^{T-1}(1-\gamma\mu)^{T-1-l}\left\la\theta_{l}^b, \zeta_l \right\ra}_{\circledTwo}\notag\\
        &&+ \underbrace{2L\gamma^2\sum\limits_{l=0}^{T-1}(1-\gamma\mu)^{T-1-l}\left(\|\theta_{l}^u\|_2^2 - \EE_{\Bxi^l}\left[\|\theta_{l}^u\|_2^2\right]\right)}_{\circledThree}\notag\\
        &&+ \underbrace{2L\gamma^2\sum\limits_{l=0}^{T-1}(1-\gamma\mu)^{T-1-l}\EE_{\Bxi^l}\left[\|\theta_{l}^u\|_2^2\right]}_{\circledFour} + \underbrace{2L\gamma^2\sum\limits_{l=0}^{T-1}(1-\gamma\mu)^{T-1-l}\|\theta_{l}^b\|_2^2}_{\circledFive}\label{eq:main_thm_clipped_SGD_technical_5_str}
    \end{eqnarray}
    where we introduce new notations:
    \begin{equation}
        \theta_{l}^u \eqdef \tnabla f(x^{l},\Bxi^l) - \EE_{\Bxi^l}\left[\tnabla f(x^{l},\Bxi^l)\right],\quad \theta_{l}^b \eqdef \EE_{\Bxi^l}\left[\tnabla f(x^{l},\Bxi^l)\right] - \nabla f(x^{l}),\label{eq:main_thm_clipped_SGD_technical_6_str}
    \end{equation}
    \begin{equation}
        \theta_{l} = \theta_{l}^u +\theta_{l}^b.\notag
    \end{equation}
    It remains to provide tight upper bounds for $\circledOne$, $\circledTwo$, $\circledThree$, $\circledFour$ and $\circledFive$, i.e.\ in the remaining part of the proof we show that $\circledOne+\circledTwo+\circledThree+\circledFour+\circledFive \le (1-\gamma\mu)^Tr_0$.
    
     \textbf{Upper bound for $\circledOne$.} First of all, since $\EE_{\Bxi^l}[\theta_{l}^u] = 0$ summands in $\circledOne$ are conditionally unbiased:
    \begin{equation*}
        \EE_{\Bxi^l}\left[\gamma (1-\gamma\mu)^{T-1-l}\left\la \theta_l^u, \zeta_l\right\ra\right] = 0.
    \end{equation*}
    Secondly, these summands are bounded with probability $1$:
    \begin{eqnarray*}
        \left|\gamma (1-\gamma\mu)^{T-1-l}\left\la \theta_l^u, \zeta_l\right\ra\right| &\le& \gamma(1-\gamma\mu)^{T-1-l}\|\theta_{l}^u\|_2\left\|\zeta_l\right\|_2\\
        &\overset{\eqref{eq:magnitude_bound_clipped_SSTM},\eqref{eq:main_thm_clipped_SGD_technical_4_2_str}}{\le}& \gamma (1-\gamma\mu)^{T-1-l} \lambda_l^2 \overset{\eqref{eq:main_thm_clipped_SGD_technical_4_0_str}}{=} 16\gamma Lr_0(1-\gamma\mu)^{T-1}.
    \end{eqnarray*}
     Finally, one can bound conditional variances $\sigma_l^2 \eqdef \EE_{\Bxi^l}\left[\gamma^2 (1-\gamma\mu)^{2(T-1-l)}\left\la \theta_l^u, \zeta_l\right\ra^2\right]$ in the following way:
    \begin{eqnarray}
         \sigma_l^2 &\le& \EE_{\Bxi^l}\left[\gamma^2 (1-\gamma\mu)^{2(T-1-l)} \|\theta_l^u\|_2^2, \|\zeta_l\|_2^2\right]\overset{\eqref{eq:main_thm_clipped_SGD_technical_4_2_str}}{\le}\gamma^2 (1-\gamma\mu)^{2(T-1-l)}\frac{\lambda^2}{4}\EE_{\Bxi^l}\left[\left\|\theta_{l}^u\right\|_2^2\right]\notag\\
         &\overset{\eqref{eq:main_thm_clipped_SGD_technical_4_0_str}}{\le}& 4\gamma^2 Lr_0 (1-\gamma\mu)^{2(T-1)-l}\EE_{\Bxi^l}\left[\left\|\theta_{l}^u\right\|_2^2\right].\label{eq:main_thm_clipped_SGD_technical_7_str}
    \end{eqnarray}
    In other words, sequence $\left\{\gamma (1-\gamma\mu)^{T-1-l}\left\la \theta_l^u, \zeta_l\right\ra\right\}_{l\ge 0}$ is a bounded martingale difference sequence with bounded conditional variances $\{\sigma_l^2\}_{l\ge 0}$. Therefore, we can apply Bernstein's inequality, i.e.\ we apply Lemma~\ref{lem:Bernstein_ineq} with $X_l = \gamma (1-\gamma\mu)^{T-1-l}\left\la \theta_l^u, \zeta_l\right\ra$, $c = 16\gamma Lr_0(1-\gamma\mu)^{T-1}$ and $F = \frac{c^2\ln\frac{4N}{\beta}}{6}$ and get that for all $b > 0$
    \begin{equation*}
        \PP\left\{\left|\sum\limits_{l=0}^{T-1}X_l\right| > b\text{ and } \sum\limits_{l=0}^{T-1}\sigma_l^2 \le F\right\} \le 2\exp\left(-\frac{b^2}{2F + \nicefrac{2cb}{3}}\right)
    \end{equation*}
    or, equivalently, with probability at least $1 - 2\exp\left(-\frac{b^2}{2F + \nicefrac{2cb}{3}}\right)$
    \begin{equation*}
        \text{either } \sum\limits_{l=0}^{T-1}\sigma_l^2 > F \quad \text{or} \quad \underbrace{\left|\sum\limits_{l=0}^{T-1}X_l\right|}_{|\circledOne|} \le b.
    \end{equation*}
    The choice of $F$ will be clarified further, let us now choose $b$ in such a way that $2\exp\left(-\frac{b^2}{2F + \nicefrac{2cb}{3}}\right) = \frac{\beta}{2N}$. This implies that $b$ is the positive root of the quadratic equation \begin{equation*}
        b^2 - \frac{2c\ln\frac{4N}{\beta}}{3}b - 2F\ln\frac{4N}{\beta} = 0,
    \end{equation*}
    hence
    \begin{eqnarray*}
         b &=& \frac{c\ln\frac{4N}{\beta}}{3} + \sqrt{\frac{c^2\ln^2\frac{4N}{\beta}}{9} + 2F\ln\frac{4N}{\beta}}=\frac{c\ln\frac{4N}{\beta}}{3} + \sqrt{\frac{4c^2\ln^2\frac{4N}{\beta}}{9}}\\
         &=& c\ln\frac{4N}{\beta} = 16\gamma Lr_0(1-\gamma\mu)^{T-1}\ln\frac{4N}{\beta}.
    \end{eqnarray*}
    That is, with probability at least $1 - \frac{\beta}{2N}$
     \begin{equation*}
        \underbrace{\text{either } \sum\limits_{l=0}^{T-1}\sigma_l^2 > F \quad \text{or} \quad \left|\circledOne\right| \le 16\gamma Lr_0(1-\gamma\mu)^{T-1}\ln\frac{4N}{\beta}}_{\text{probability event } E_{\circledOne}}.
    \end{equation*}
    Next, we notice that probability event $E_{T-1}$ implies that
    \begin{eqnarray*}
        \sum\limits_{l=0}^{T-1}\sigma_l^2  &\overset{\eqref{eq:main_thm_clipped_SGD_technical_7_str}}{\le}& 4\gamma^2 Lr_0 \sigma^2 (1-\gamma \mu)^{2(T-1)} \sum\limits_{l=0}^{T-1}\EE_{\Bxi^l}\left[\left\|\theta_l^u\right\|_2^2\right]\\
        &\overset{\eqref{eq:variance_bound_clipped_SSTM}}{\le}&72\gamma^2 Lr_0 \sigma^2 (1-\gamma \mu)^{2(T-1)} \sum\limits_{l=0}^{T-1}\frac{1}{m_l (1-\gamma\mu)^l}\\
        &\overset{\eqref{eq:bathces_clipped_SGD_2}}{\le}& \frac{128}{3} \gamma^2L^2r_0^2(1-\gamma\mu)^{2(T-1)}\ln\frac{4N}{\beta} = \frac{c^2\ln\frac{4N}{\beta}}{6} = F,
    \end{eqnarray*}
    where the last inequality follows from $c = 16\gamma Lr_0(1-\gamma\mu)^{T-1}$ and simple arithmetic.
    
    \textbf{Upper bound for $\circledTwo$.} First of all, we notice that probability event $E_{T-1}$ implies
     \begin{eqnarray*}
        \gamma(1-\gamma\mu)^{T-1-l}\left\la\theta_{l}^b, \zeta_l \right\ra &\le& \gamma(1-\gamma\mu)^{T-1-l}\left\|\theta_{l}^b\right\|_2\left\|\zeta_l\right\|_2\\
        &\overset{\eqref{eq:bias_bound_clipped_SSTM},\eqref{eq:main_thm_clipped_SGD_technical_4_2_str}}{\le}&  \gamma(1-\gamma\mu)^{T-1-l}\frac{4\sigma^2}{m_l \lambda_{l}}\frac{\lambda_l}{2}\\
        &=& \frac{2\sigma^2 \gamma (1-\gamma\mu)^{T-1-l}\sigma^2}{m_l}\\
        &\overset{\eqref{eq:bathces_clipped_SGD_2}}{=}&
        \frac{64}{27}\frac{\gamma L r_0 (1-\gamma\mu)^{T-1}\ln \frac{4N}{\beta}}{N}.
    \end{eqnarray*}
    This implies that
    \begin{eqnarray*}
        \circledTwo &=& \sum\limits_{l=0}^{T-1}\gamma(1-\gamma\mu)^{T-1-l}\left\la\theta_{l}^b,\zeta_l\right\ra \overset{T\le N}{\le}\frac{64}{27}\gamma L r_0 (1-\gamma\mu)^{T-1}\ln \frac{4N}{\beta}.
    \end{eqnarray*}
    
    \textbf{Upper bound for $\circledThree$.} We derive the upper bound for $\circledThree$ using the same technique as for $\circledOne$. First of all, we notice that the summands in $\circledThree$ are conditionally independent:
    \begin{equation*}
        \EE_{\Bxi^l}\left[2L\gamma^2(1-\gamma\mu)^{T-1-l}\left(\|\theta_{l}^u\|_2^2 - \EE_{\Bxi^l}\left[\|\theta_{l}^u\|_2^2\right]\right)\right] = 0.
    \end{equation*}
     Secondly, the summands are bounded with probability $1$:
    \begin{eqnarray}
        \left|2L\gamma^2(1-\gamma\mu)^{T-1-l}\left(\|\theta_{l}^u\|_2^2 - \EE_{\Bxi^l}\left[\|\theta_{l}^u\|_2^2\right]\right)\right| &\le& 2L\gamma^2(1-\gamma\mu)^{T-1-l}\left(\|\theta_{l}^u\|_2^2 + \EE_{\Bxi^l}\left[\|\theta_{l}^u\|_2^2\right]\right)\notag\\
        &\overset{\eqref{eq:magnitude_bound_clipped_SSTM}}{\le}& 2L\gamma^2(1-\gamma\mu)^{T-1-l}\left(4\lambda_{l}^2 + 4\lambda_{l}^2\right)\notag\\
        &\overset{\eqref{eq:main_thm_clipped_SGD_technical_4_0_str}}{=}& 256 \gamma^2 L^2 r_0 (1-\gamma\mu)^{T-1} \eqdef c_1.\label{eq:main_thm_clipped_SGD_technical_8_str}
    \end{eqnarray}
    Finally, one can bound conditional variances $\hat \sigma_l^2 \eqdef \EE_{\Bxi^l}\left[\left|2L\gamma^2(1-\gamma\mu)^{T-1-l}\left(\|\theta_{l}^u\|_2^2 - \EE_{\Bxi^l}\left[\|\theta_{l}^u\|_2^2\right]\right)\right|^2\right]$ in the following way:
    \begin{eqnarray}
         \hat \sigma_l^2 &\overset{\eqref{eq:main_thm_clipped_SGD_technical_8_str}}{\le}& c_1\EE_{\Bxi^l}\left[\left|2L\gamma^2(1-\gamma\mu)^{T-1-l}\left(\|\theta_{l}^u\|_2^2 - \EE_{\Bxi^l}\left[\|\theta_{l}^u\|_2^2\right]\right)\right|\right]\notag\\
         &\le& 2L\gamma^2(1-\gamma\mu)^{T-1-l}c_1\EE_{\Bxi^l}\left[\|\theta_{l}^u\|_2^2 + \EE_{\Bxi^l}\left[\|\theta_{l}^u\|_2^2\right]\right] \notag\\
         &=& 4L\gamma^2(1-\gamma\mu)^{T-1-l}c_1\EE_{\Bxi^l}\left[\|\theta_{l}^u\|_2^2\right].\label{eq:main_thm_clipped_SGD_technical_9_str}
    \end{eqnarray}
    In other words, sequence $\left\{2L\gamma^2(1-\gamma\mu)^{T-1-l}\left(\|\theta_{l}^u\|_2^2 - \EE_{\Bxi^l}\left[\|\theta_{l}^u\|_2^2\right]\right)\right\}_{l\ge 0}$ is a bounded martingale difference sequence with bounded conditional variances $\{\hat \sigma_l^2\}_{l\ge 0}$. Therefore, we can apply Bernstein's inequality, i.e.\ we apply Lemma~\ref{lem:Bernstein_ineq} with $X_l = \hat X_l = 2L\gamma^2(1-\gamma\mu)^{T-1-l}\left(\|\theta_{l}^u\|_2^2 - \EE_{\Bxi^l}\left[\|\theta_{l}^u\|_2^2\right]\right)$, $c = c_1 = 256 \gamma^2 L^2 r_0 (1-\gamma\mu)^{T-1}$ and $F = F_1 = \frac{c_1^2\ln\frac{4N}{\beta}}{6}$ and get that for all $b > 0$
    \begin{equation*}
        \PP\left\{\left|\sum\limits_{l=0}^{T-1}\hat X_l\right| > b\text{ and } \sum\limits_{l=0}^{T-1}\hat \sigma_l^2 \le F_1\right\} \le 2\exp\left(-\frac{b^2}{2F_1 + \nicefrac{2c_1b}{3}}\right)
    \end{equation*}
    or, equivalently, with probability at least $1 - 2\exp\left(-\frac{b^2}{2F_1 + \nicefrac{2c_1b}{3}}\right)$
    \begin{equation*}
        \text{either } \sum\limits_{l=0}^{T-1}\hat\sigma_l^2 > F_1 \quad \text{or} \quad \underbrace{\left|\sum\limits_{l=0}^{T-1}\hat X_l\right|}_{|\circledThree|} \le b.
    \end{equation*}
    As in our derivations of the upper bound for $\circledOne$ we choose such $b$ that $2\exp\left(-\frac{b^2}{2F_1 + \nicefrac{2c_1b}{3}}\right) = \frac{\beta}{2N}$, i.e.\
    \begin{eqnarray*}
         b &=& \frac{c_1\ln\frac{4N}{\beta}}{3} + \sqrt{\frac{c_1^2\ln^2\frac{4N}{\beta}}{9} + 2F_1\ln\frac{4N}{\beta}} = c_1\ln\frac{4N}{\beta} = 256 \gamma^2 L^2 r_0 (1-\gamma\mu)^{T-1}\ln\frac{4N}{\beta}.
    \end{eqnarray*}
    That is, with probability at least $1 - \frac{\beta}{2N}$
     \begin{equation*}
        \underbrace{\text{either } \sum\limits_{l=0}^{T-1}\hat\sigma_l^2 > F_1 \quad \text{or} \quad \left|\circledThree\right| \le 256 \gamma^2 L^2 r_0 (1-\gamma\mu)^{T-1}\ln\frac{4N}{\beta}}_{\text{probability event } E_{\circledThree}}.
    \end{equation*}
     Next, we notice that probability event $E_{T-1}$ implies that
    \begin{eqnarray*}
        \sum\limits_{l=0}^{T-1}\hat\sigma_l^2 &\overset{\eqref{eq:main_thm_clipped_SGD_technical_9_str}}{\le}& 4L\gamma^2 (1-\gamma\mu)^{T-1}c_1 \sum\limits_{l=0}^{T-1}\frac{1}{(1-\gamma\mu)^{l}}\EE_{\Bxi^l}\left[\left\|\theta_{l}^u\right\|_2^2\right]\\
        &\overset{\eqref{eq:variance_bound_clipped_SSTM}}{\le}& 72L\gamma^2 (1-\gamma\mu)^{T-1}c_1\sigma^2 \sum\limits_{l=0}^{T-1}\frac{1}{(1-\gamma\mu)^{l}}\frac{1}{m_l}\overset{\eqref{eq:bathces_clipped_SGD_2},T\le N}{\le} \frac{c_1^2\ln\frac{4N}{\beta}}{6} = F_1.
    \end{eqnarray*}
    
    \textbf{Upper bound for $\circledFour$.} The probability event $E_{T-1}$ implies
    \begin{eqnarray*}
        \circledFour &=& 2L\gamma^2\sum\limits_{l=0}^{T-1}(1-\gamma\mu)^{T-1-l}\EE_{\Bxi^l}\left[\|\theta_{l}^u\|_2^2\right] \overset{\eqref{eq:variance_bound_clipped_SSTM}}{\le} 2L\gamma^2(1-\gamma\mu)^{T-1} \sum\limits_{l=0}^{T-1} \frac{1}{(1-\gamma\mu)^l}\frac{18\sigma^2}{m_l} \\
        &\overset{\eqref{eq:bathces_clipped_SGD_2},T\le N}{\le}& \frac{64}{3} \gamma^2 L^2 r_0 (1-\gamma\mu)^{T-1}\ln\frac{4N}{\beta}.
    \end{eqnarray*}
    
     \textbf{Upper bound for $\circledFive$.} Again, we use corollaries of probability event $E_{T-1}$:
    \begin{eqnarray*}
        \circledFive &=& 2L\gamma^2\sum\limits_{l=0}^{T-1}(1-\gamma\mu)^{T-1-l}\|\theta_{l}^b\|_2^2 \overset{\eqref{eq:bias_bound_clipped_SSTM}}{\le} 2L\gamma^2 (1-\gamma\mu)^{T-1} \sum\limits_{l=0}^{T-1}\frac{1}{(1-\gamma\mu)^l}\frac{16\sigma^4}{m_l^2\lambda_{l}^2} \\
        &\overset{\eqref{eq:main_thm_clipped_SGD_technical_4_0_str},\eqref{eq:bathces_clipped_SGD_2}}{=}& \frac{512}{729} \gamma^2L^2r_0 (1-\gamma\mu)^{T-1}\ln^2\frac{4N}{\beta}\sum\limits_{l=0}^{T-1}\frac{1}{N^2}\\
        &\overset{T\le N}{\le}& \frac{512}{729} \frac{\gamma^2L^2r_0(1-\gamma\mu)^{T-1}\ln^2\frac{4N}{\beta}}{N}.
    \end{eqnarray*}
     Now we summarize all bounds that we have: probability event $E_{T-1}$ implies
      \begin{eqnarray*}
        r_T &\overset{\eqref{eq:main_thm_clipped_SGD_technical_1_2}}{\le}& (1-\gamma\mu)^T r_0 + \gamma \sum\limits_{l=0}^{T-1} (1-\gamma\mu)^{T-1-l} \left\la-\nabla f(x^l), \theta_l \right\ra + L\gamma^2 \sum\limits_{l=0}^{T-1} (1-\gamma\mu)^{T-1-l}\|\theta_l\|_2^2\\
        &\overset{\eqref{eq:main_thm_clipped_SGD_technical_5_str}}{\le}& (1-\gamma\mu)^T r_0 + \circledOne + \circledTwo + \circledThree + \circledFour + \circledFive,\\
        \circledTwo &\le& \frac{32}{27}\gamma L r_0 (1-\gamma\mu)^{T-1}\ln \frac{4N}{\beta},\quad \circledFour \le \frac{64}{3} \gamma^2 L^2 r_0 (1-\gamma\mu)^{T-1}\ln\frac{4N}{\beta},\\ \circledFive &\le& \frac{512}{729} \frac{\gamma^2L^2r_0(1-\gamma\mu)^{T-1}\ln^2\frac{4N}{\beta}}{N},\quad
        \sum\limits_{l=0}^{T-1}\sigma_l^2 \le F,\quad \sum\limits_{l=0}^{T-1}\hat\sigma_l^2 \le F_1
    \end{eqnarray*}
    and
    \begin{equation*}
        \PP\{E_{T-1}\} \ge 1 - \frac{(T-1)\beta}{N},\quad \PP\{E_\circledOne\} \ge 1 - \frac{\beta}{2N},\quad \PP\{E_\circledThree\} \ge 1 - \frac{\beta}{2N},
    \end{equation*}
    where
    \begin{eqnarray*}
        E_{\circledOne} &=& \left\{\text{either } \sum\limits_{l=0}^{T-1}\sigma_l^2 > F \quad \text{or} \quad \left|\circledOne\right| \le 16\gamma Lr_0(1-\gamma\mu)^{T-1}\ln\frac{4N}{\beta}\right\},\\
        E_{\circledThree} &=& \left\{\text{either } \sum\limits_{l=0}^{T-1}\hat\sigma_l^2 > F_1 \quad \text{or} \quad \left|\circledThree\right| \le 256 \gamma^2 L^2 r_0 (1-\gamma\mu)^{T-1}\ln\frac{4N}{\beta}\right\}.
    \end{eqnarray*}
    Taking into account these inequalities and our assumptions on $m_k$ and $\gamma$ (see \eqref{eq:bathces_clipped_SGD_2} and \eqref{eq:step_size_clipped_SGD_2}) we get that probability event $E_{T-1}\cap E_\circledOne \cap E_\circledThree$ implies
    \begin{eqnarray}
         r_T &\overset{\eqref{eq:main_thm_clipped_SGD_technical_1_2}}{\le}& (1-\gamma\mu)^T r_0 + \gamma \sum\limits_{l=0}^{T-1} (1-\gamma\mu)^{T-1-l} \left\la-\nabla f(x^l), \theta_l \right\ra + L\gamma^2 \sum\limits_{l=0}^{T-1} (1-\gamma\mu)^{T-1-l}\|\theta_l\|_2^2\notag\\
         &\le& (1-\gamma\mu)^T r_0 + \left(\frac{1}{5} + \frac{1}{5} + \frac{1}{5} + \frac{1}{5} + \frac{1}{5}\right)(1-\gamma\mu)^T r_0 = 2(1-\gamma\mu)^T r_0 .\label{eq:main_thm_clipped_SGD_technical_10_str}
    \end{eqnarray}
     Moreover, using union bound we derive
    \begin{equation}
        \PP\left\{E_{T-1}\cap E_\circledOne \cap E_\circledThree\right\} = 1 - \PP\left\{\overline{E}_{T-1}\cup\overline{E}_\circledOne\cup\overline{E}_\circledThree\right\} \ge 1 - \frac{T\beta}{N}.\label{eq:main_thm_clipped_SGD_technical_11_str}
    \end{equation}
    That is, by definition of $E_T$ and $E_{T-1}$ we have proved that
    \begin{eqnarray*}
        \PP\{E_T\} &\overset{\eqref{eq:main_thm_clipped_SGD_technical_10_str}}{\ge}& \PP\left\{E_{T-1}\cap E_\circledOne \cap E_\circledThree\right\} \overset{\eqref{eq:main_thm_clipped_SGD_technical_11_str}}{\ge} 1 - \frac{T\beta}{N},
    \end{eqnarray*}
    which implies that for all $k = 0,1,\ldots, N$ we have $\PP\{E_k\} \ge 1 - \frac{k\beta}{N}$. Then, for $k = N$ we have that with probability at least $1-\beta$
    \begin{eqnarray}
        f(x^N)-f(x^*)) 
        &\overset{\eqref{eq:main_thm_clipped_SGD_technical_0_2}}{\le}& (1-\gamma\mu)^N(f(x^0)-f(x^*)) + \gamma \sum\limits_{l=0}^{N-1} (1-\gamma\mu)^{N-1-l} \left\la-\nabla f(x^l), \theta_l \right\ra\notag\\
        && + L\gamma^2 \sum\limits_{l=0}^{N-1} (1-\gamma\mu)^{N-1-l}\|\theta_l\|_2^2 \overset{\eqref{eq:main_thm_clipped_SGD_technical_2_2}}{\le} 2(1-\gamma\mu)^N(f(x^0)-f(x^*)).
    \end{eqnarray}
    As a result, we get that with probability at least $1-\beta$
    \begin{eqnarray*}
        f(x^N) - f(x^*) &\le& 2(1-\gamma\mu)^N(f(x^0)-f(x^*)) \le 2\exp\left(-\gamma\mu N\right)(f(x^0)-f(x^*))\\
        &\overset{\eqref{eq:step_size_clipped_SGD_2}}{\le}& 2\exp\left(-\frac{\mu N}{80L\ln\frac{4N}{\beta}}\right)(f(x^0)-f(x^*)).
    \end{eqnarray*}
    In other words, {\tt clipped-SGD} achieves $
    f(x^N) - f(x^*) \le \varepsilon$ with probability at least $1-\beta$ after
    \begin{eqnarray*}
    O\left(\frac{L}{\mu}\ln\left(\frac{r_0}{\varepsilon}\right)\ln\left(\frac{L}{\mu \beta}\ln\left(\frac{r_0}{\varepsilon}\right)\right)\right)\end{eqnarray*} 
    iterations, where $r_0=f(x^0)-f(x^*)$ and requires
    \begin{eqnarray*}
         \sum\limits_{k=0}^{N-1}m_k &\overset{\eqref{eq:bathces_clipped_SGD_2}}{=}& \sum\limits_{k=0}^{N-1}O\left(\max\left\{1,\frac{N\sigma^2}{Lr_0(1-\gamma\mu)^k\ln\frac{4N}{\beta}}\right\}\right)\\
         &\overset{\eqref{eq:step_size_clipped_SGD_2}}{=}& O\left(\max\left\{N,\frac{N\sigma^2}{\mu r_0 (1-\gamma\mu)^{N-1}}\right\}\right) = O\left(\max\left\{N,\frac{N\sigma^2}{\mu \varepsilon}\right\}\right)\\
         &=& O\left(\max\left\{\frac{L}{\mu},\frac{\sigma^2}{\mu\varepsilon}\cdot\frac{L}{\mu}\right\}\ln\left(\frac{r_0}{\varepsilon}\right)\ln\left(\frac{L}{\mu\beta}\ln\left(\frac{r_0}{\varepsilon}\right)\right)\right).
    \end{eqnarray*}
    oracle calls.

\clearpage

\section{Extra Experiments}\label{sec:extra_exp}

\subsection{Detailed Description of Experiments from Section~\ref{sec:motivation}}\label{sec:toy_details}
In this section we provide a detailed description of experiments from Section~\ref{sec:motivation} together with additional experiments. In these experiments we consider the following problem:
\begin{equation}
    \min\limits_{x\in\R^n} f(x),\quad f(x) = \nicefrac{\|x\|_2^2}{2} = \EE_\xi\left[f(x,\xi)\right],\quad f(x,\xi) = \nicefrac{\|x\|_2^2}{2} + \la \xi, x\ra\label{eq:toy_problem}
\end{equation}
where $\xi$ is a random vector with zero mean and bounded variance. Clearly, $f(x)$ is $\mu$-strongly convex and $L$-smooth with $\mu = L =1$. We assume that $\EE\left[\|\xi\|_2^2\right] \le \sigma^2$ for some non-negative number $\sigma$. Then, the stochastic gradient $\nabla f(x,\xi) = x + \xi$ satisfies conditions \eqref{eq:bounded_variance_clipped_SSTM} and the state-of-the-art theory (e.g.\ \cite{gorbunov2019unified, gower2019sgd}) says that after $k$ iterations of {\tt SGD} with constant stepsize $\gamma \le \nicefrac{1}{L} = 1$ we have $\EE\left[\|x^k - x^*\|_2^2\right] \le (1 - \gamma\mu)^k\|x^0 - x^*\|_2^2 + \nicefrac{\gamma\sigma^2}{\mu}.$
Taking into account that for our problem $x^* = 0$, $f(x) = \frac{1}{2}\|x\|_2^2$, $f(x^*) = 0$ and $\mu = 1$ we derive
\begin{equation}
    \EE\left[f(x^k) - f(x^*)\right] \le (1 - \gamma)^k\left(f(x^0) - f(x^*)\right) + \nicefrac{\gamma\sigma^2}{2}.\label{eq:conv_in_exp_toy}
\end{equation}
That is, for given $k$ the r.h.s.\ of the formula above depends only on the stepsize $\gamma$, initial suboptimality $f(x^0) - f(x^*)$ and the variance $\sigma$.

We emphasize that the obtained bound and the convergence in expectation itself does not imply non-trivial upper bound for $f(x^k) - f(x^*)$ with high-probability without additional assumptions on the distribution of random vector $\xi$. In fact, the trajectory of {\tt SGD} significantly depends on the distribution of $\xi$. To illustrate this we consider $3$ different distributions of $\xi$ with the same $\sigma$.
\begin{enumerate}
    \item In the first case we consider $\xi$ from standard normal distribution, i.e.\ $\xi$ is a Gaussian random vector with zero mean and covariance matrix $I$. Clearly, in this situation $\sigma^2 = n$.
    \item Next, we consider a random vector $\xi$ with i.i.d.\ components having Weibull distribution \cite{weibull1951statistical}. The cumulative distribution function (CDF) for Weibull distribution with parameters $c>0$ and $\alpha>0$ is
    \begin{equation}
        \text{CDF}_{W}(x) = \begin{cases}1 - \exp\left(-\left(\frac{x}{\alpha}\right)^c\right),& \text{if } x\ge 0,\\
        0,&\text{if } x < 0.\end{cases}\label{eq:weibull_distribution}
    \end{equation}
    There are explicit formulas for mean and variance for Weibull distribution:
    \begin{equation*}
        \text{mean} = \alpha\Gamma\left(1+\frac{1}{c}\right),\quad \text{variance} = \alpha^2\left(\Gamma\left(1 + \frac{2}{c}\right) - \left(\Gamma\left(1 + \frac{1}{c}\right)\right)^2\right),
    \end{equation*}
    where $\Gamma$ denotes the gamma function. Having these formulas one can easily shift and scale the distribution in order to get a random variable with zero mean and the variance equal $1$. 
    
    In our experiments, we take $c = 0.2$, 
    \begin{equation*}
        \alpha = \frac{1}{\sqrt{\Gamma\left(1 + \frac{2}{c}\right) - \left(\Gamma\left(1 + \frac{1}{c}\right)\right)^2}},
    \end{equation*}
    shift the distribution by $-\alpha\Gamma\left(1+\frac{1}{c}\right)$ and sample from the obtained distribution $n$ i.i.d.\ random variables to form $\xi$. Such a choice of parameters implies that $\EE[\xi] = 0$ and $\EE[\|\xi\|_2^2] = n$.
    \item Finally, we consider a random vector $\xi$ with i.i.d.\ components having Burr Type XII distribution \cite{burr1942cumulative} having the following cumulative distribution function
    \begin{equation}
        \text{CDF}_{B}(x) = \begin{cases}1 - \left(1+x^c\right)^{-d},& \text{if } x> 0,\\
        0,&\text{if } x \le 0,\end{cases}\label{eq:burr_distribution}
    \end{equation}
    where $c>0$ and $d > 0$ are the positive parameters. There are explicit formulas for mean and variance for Burr distribution:
    \begin{equation*}
        \text{mean} = \mu_1,\quad \text{variance} = -\mu_1^2 + \mu_2,
    \end{equation*}
    where the $r$-th moment (if exists) is defined as follows \cite{mclaughlin2001compendium}:
    \begin{equation*}
        \mu_r = d\text{B}\left(\frac{cd - r}{c}, \frac{c+r}{c}\right),
    \end{equation*}
    where $\text{B}$ denotes the beta function.
    
    In our experiments, we take $c = 1$ and $d=2.3$ and then apply shifts and scales similarly to the case with Weibull distribution. Again, such a choice of parameters implies that $\EE[\xi] = 0$ and $\EE[\|\xi\|_2^2] = n$.
\end{enumerate}

For all experiments we considered the dimension $n=100$, the stepsize $\gamma = 0.001$ and for {\tt clipped-SGD} we set $\lambda = 100$. The result of $10$ independent runs of {\tt SGD} and {\tt clipped-SGD} are presented in Figures~\ref{fig:tou_runs1}-\ref{fig:tou_runs5}. These numerical tests show that for Weibull and Burr Type XII distributions {\tt SGD} have significantly larger oscillations than for Gaussian distribution in all $10$ tests. In contrast, {\tt clipped-SGD} behaves much more robust in all $3$ cases during all $10$ runs without significant oscillations.

\begin{figure}[h]
    \centering
    \includegraphics[width=0.32\textwidth]{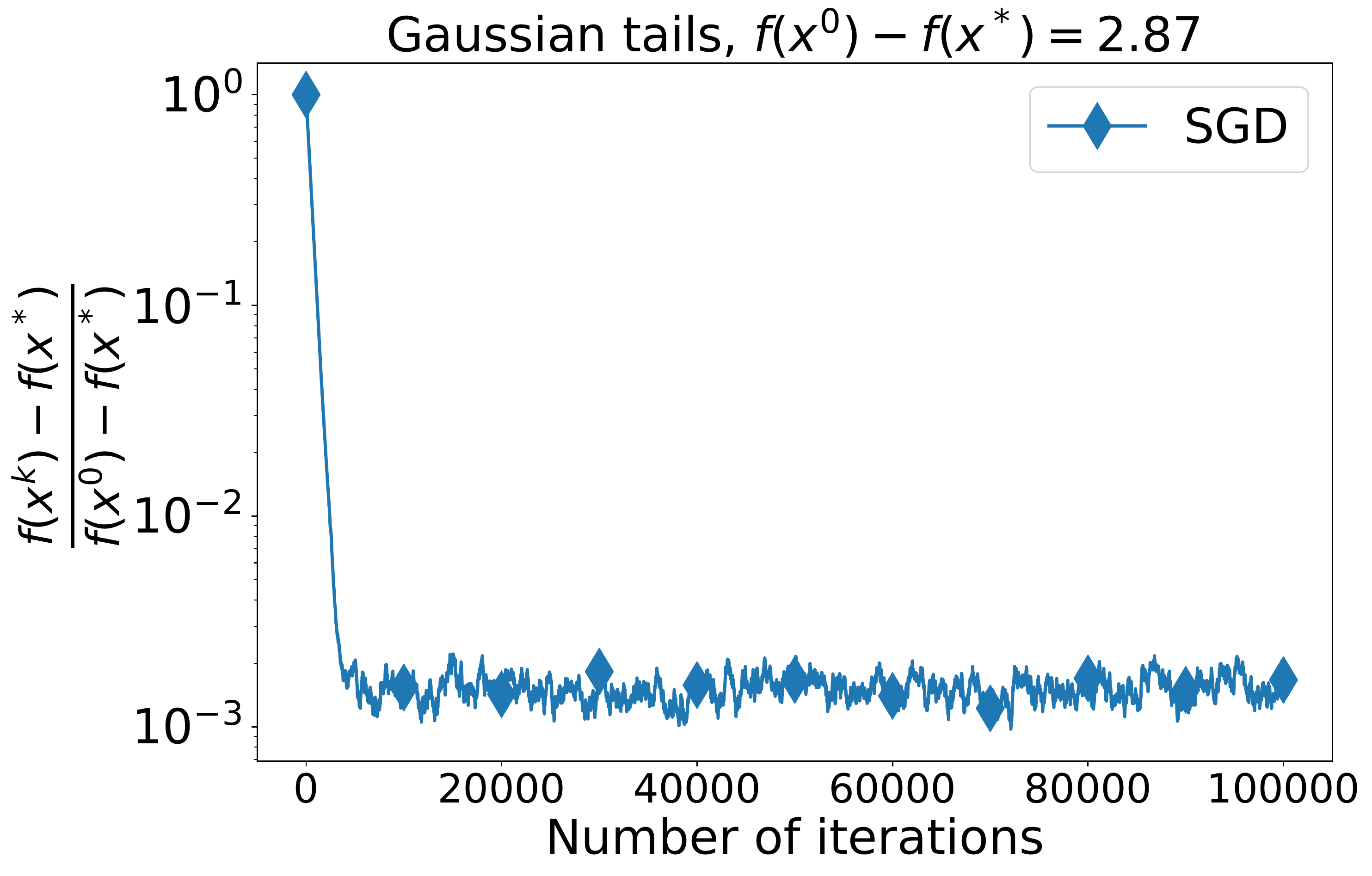}
    \includegraphics[width=0.32\textwidth]{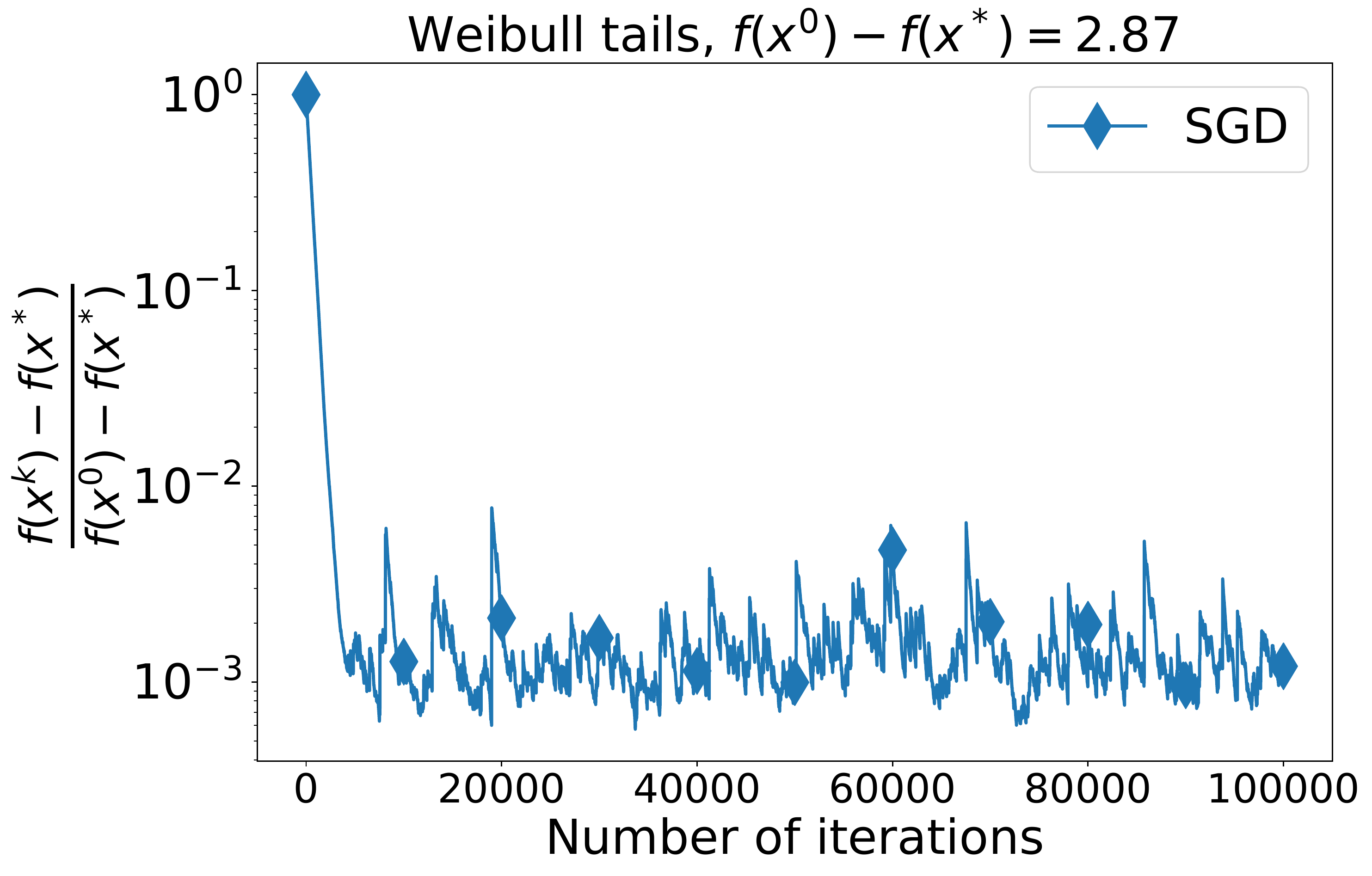}
    \includegraphics[width=0.32\textwidth]{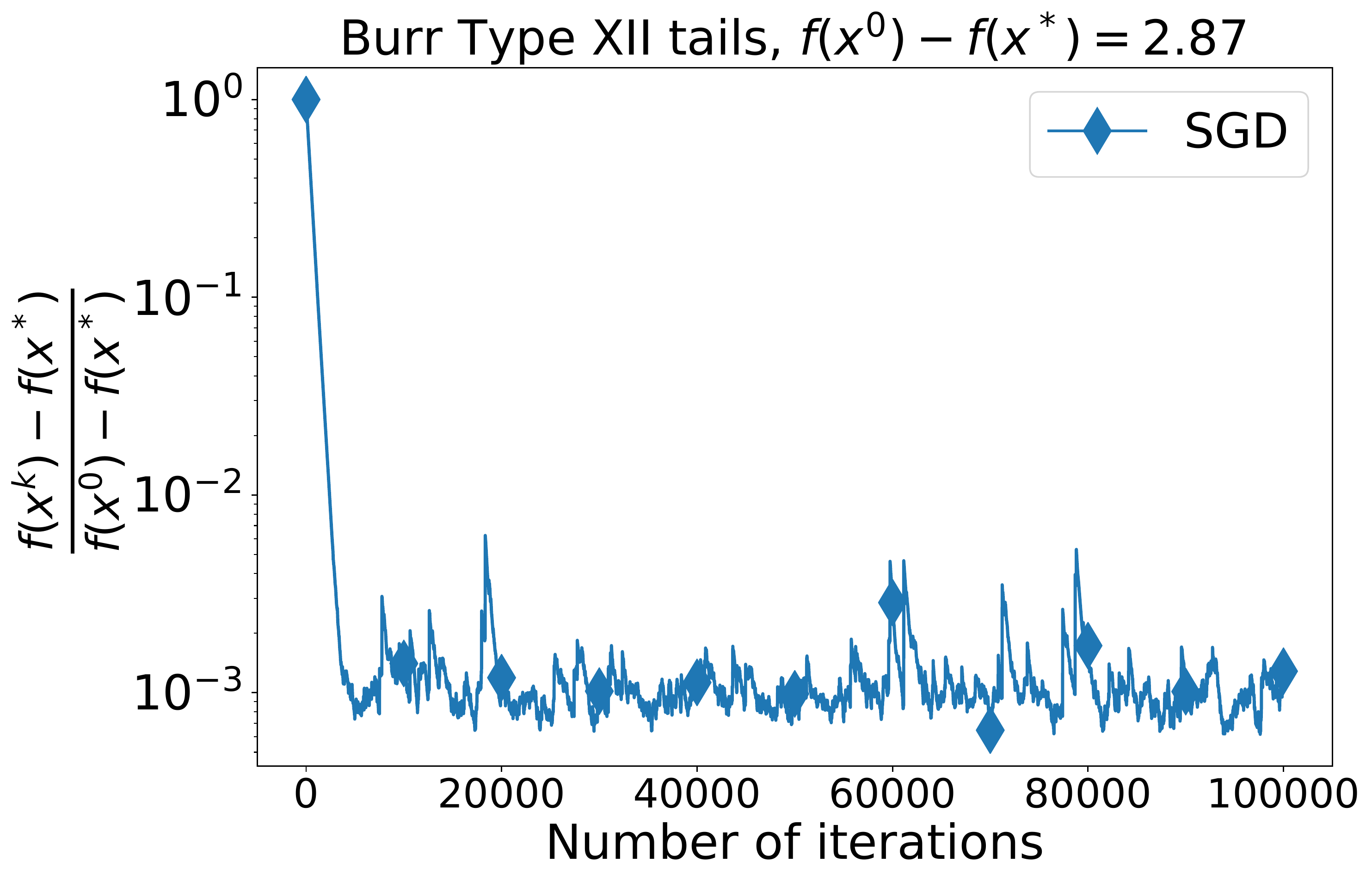}
    \includegraphics[width=0.32\textwidth]{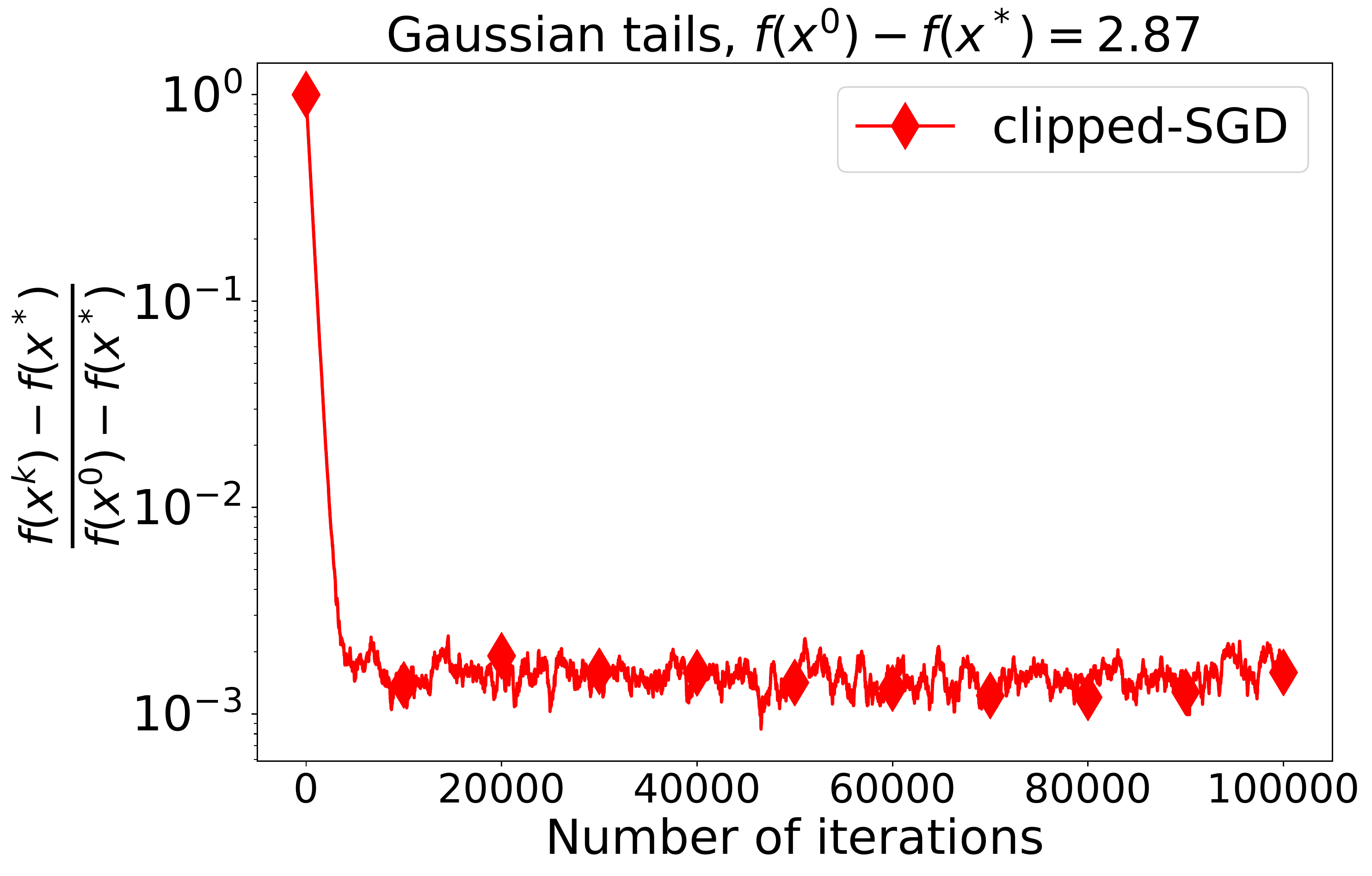}
    \includegraphics[width=0.32\textwidth]{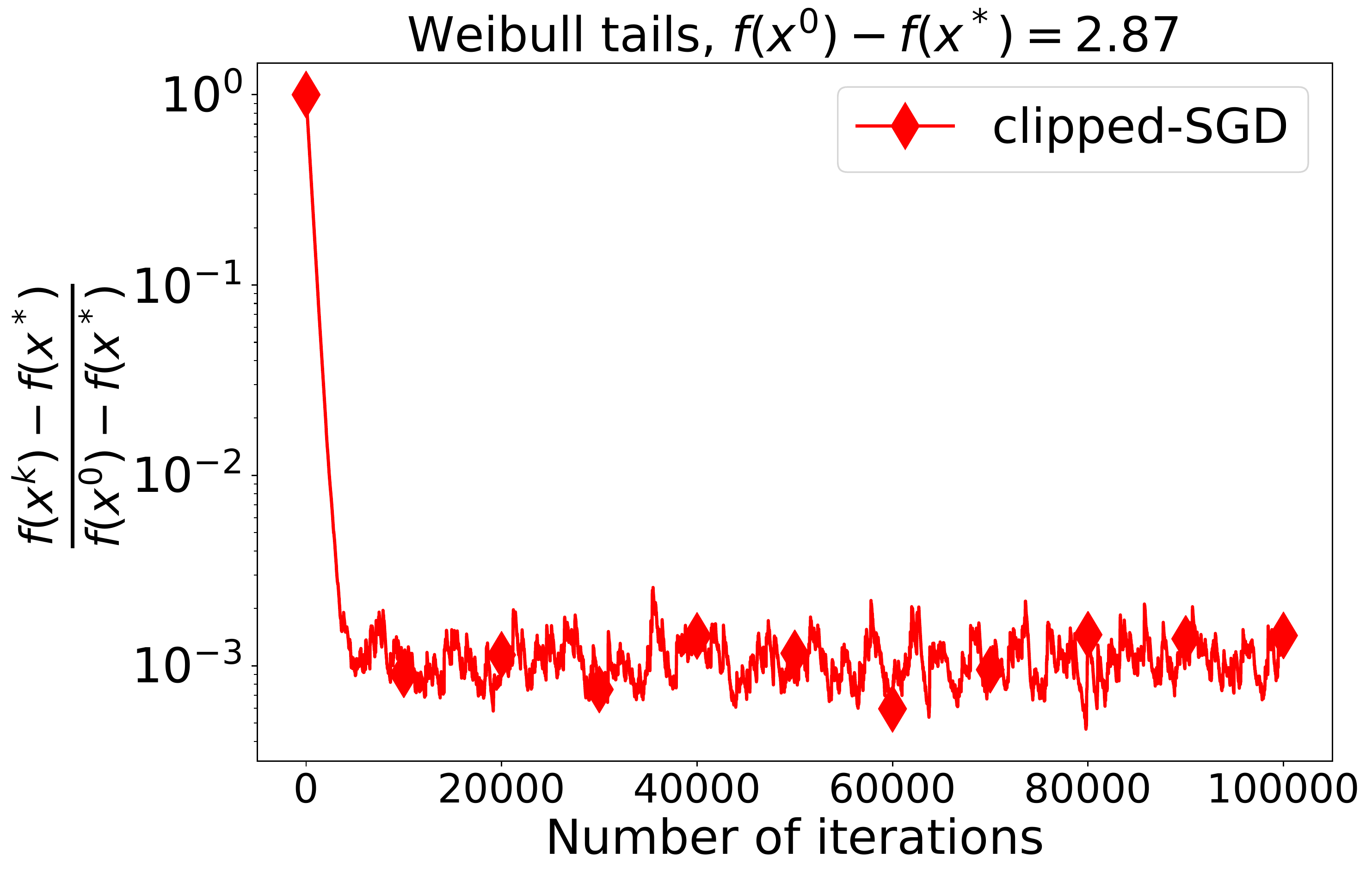}
    \includegraphics[width=0.32\textwidth]{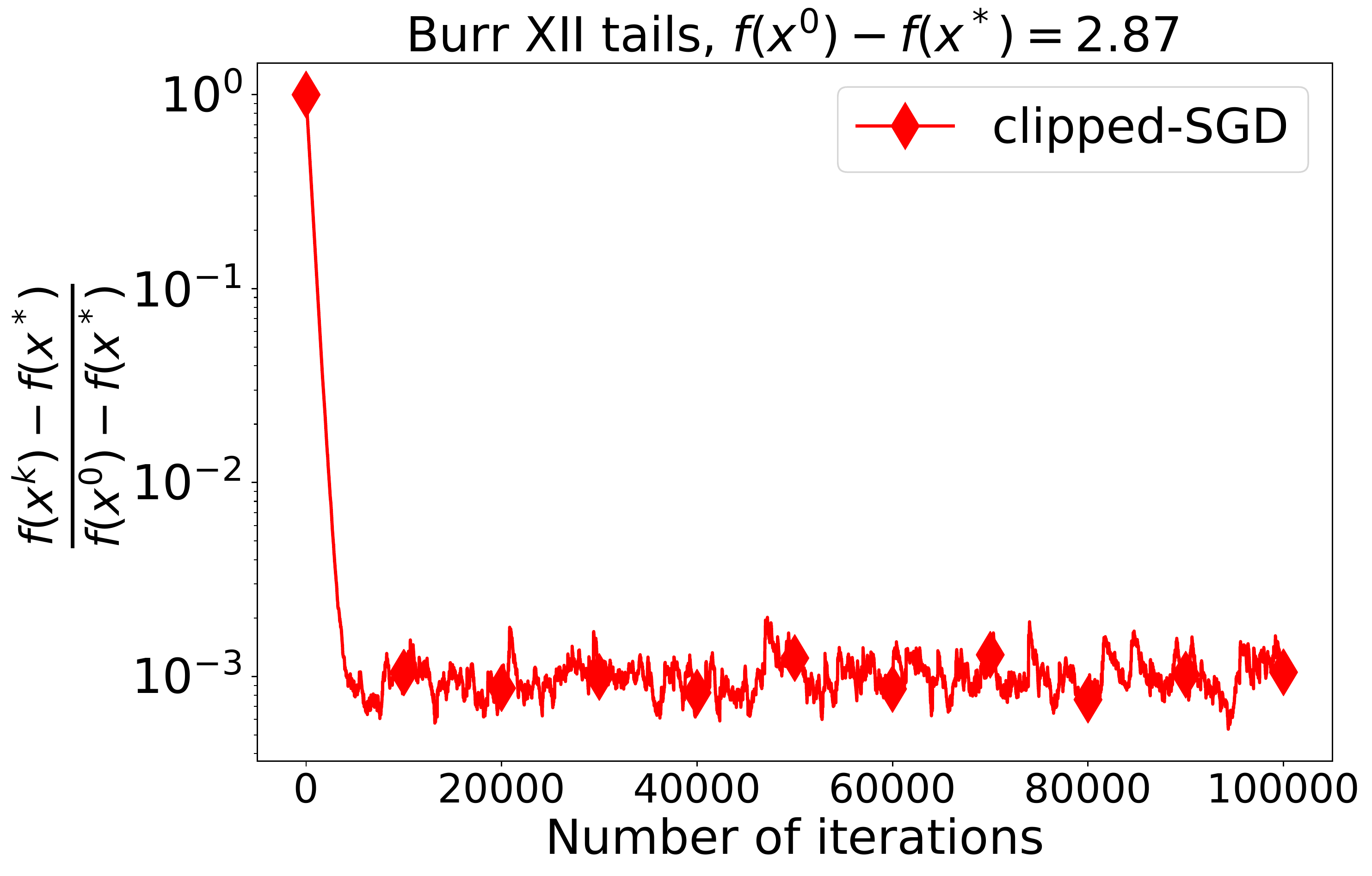}
    \includegraphics[width=0.32\textwidth]{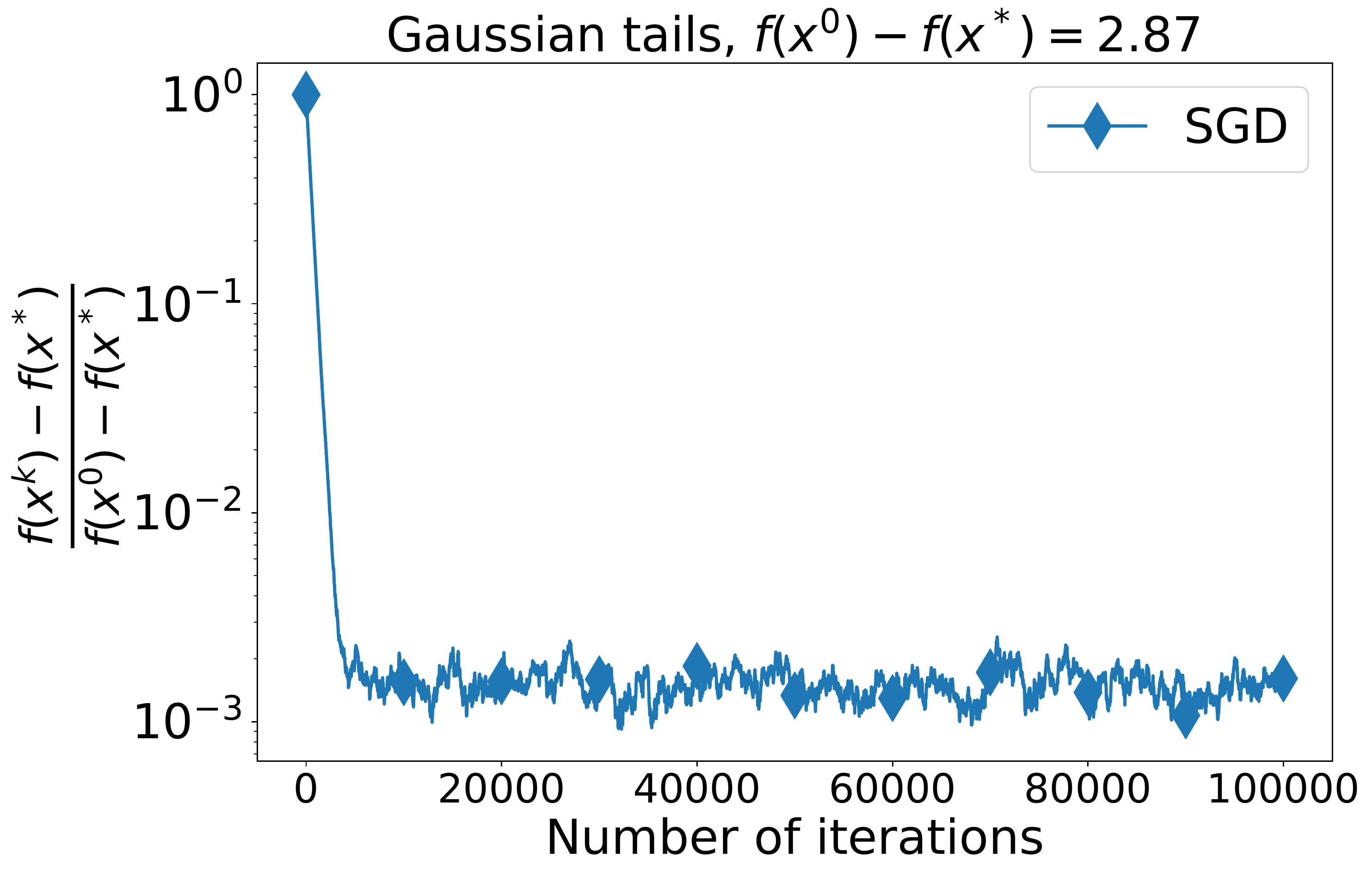}
    \includegraphics[width=0.32\textwidth]{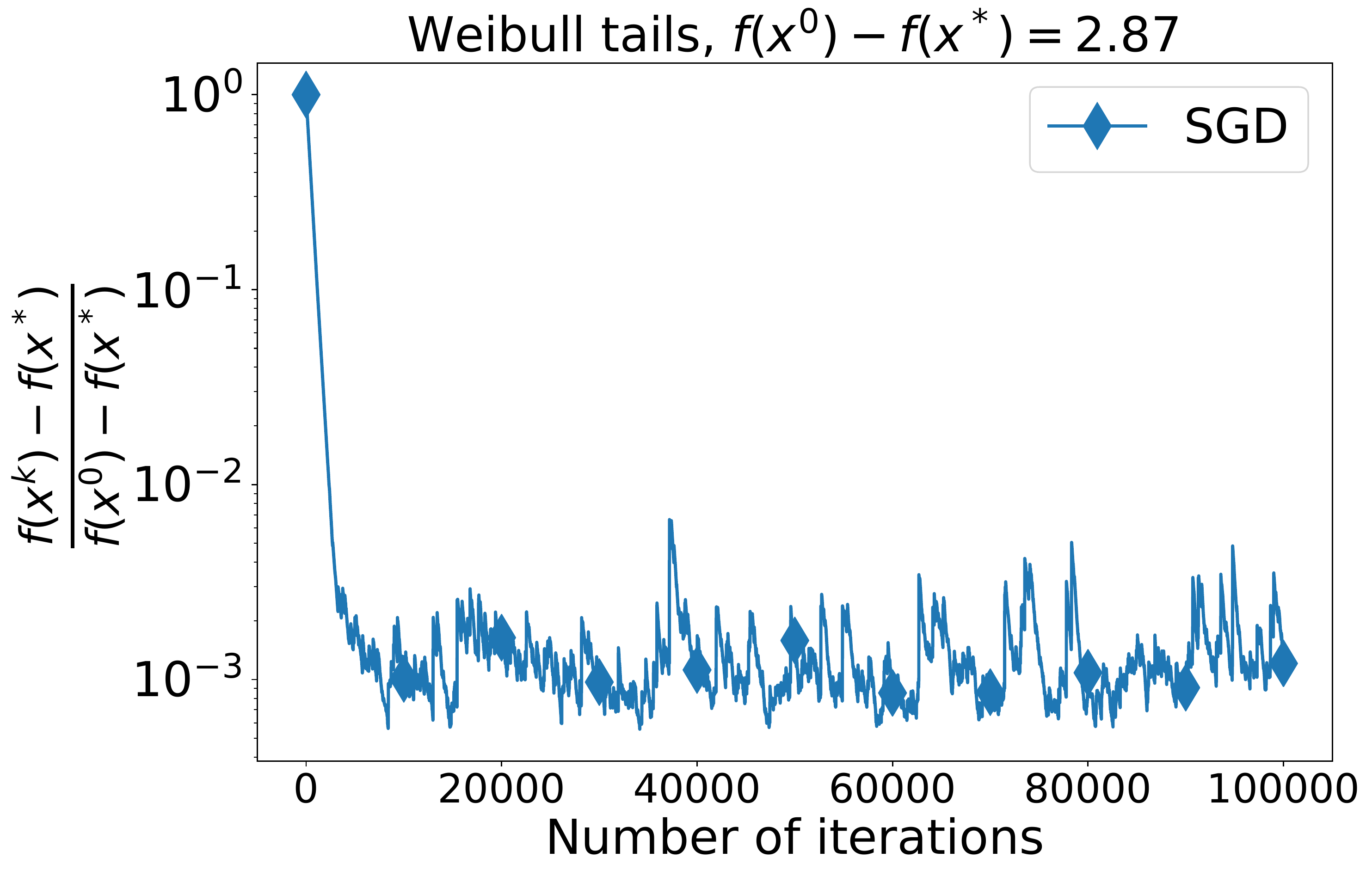}
    \includegraphics[width=0.32\textwidth]{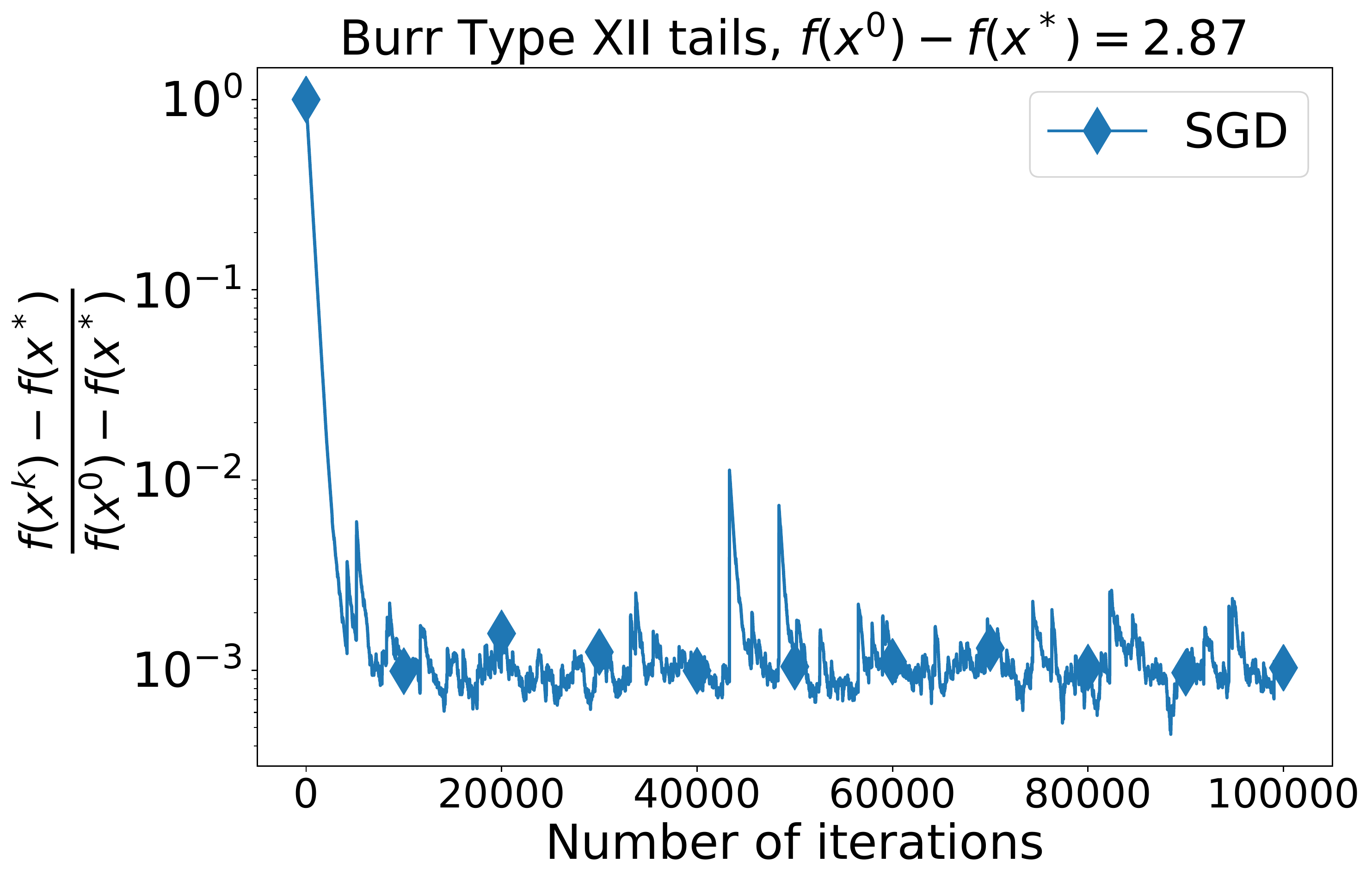}
    \includegraphics[width=0.32\textwidth]{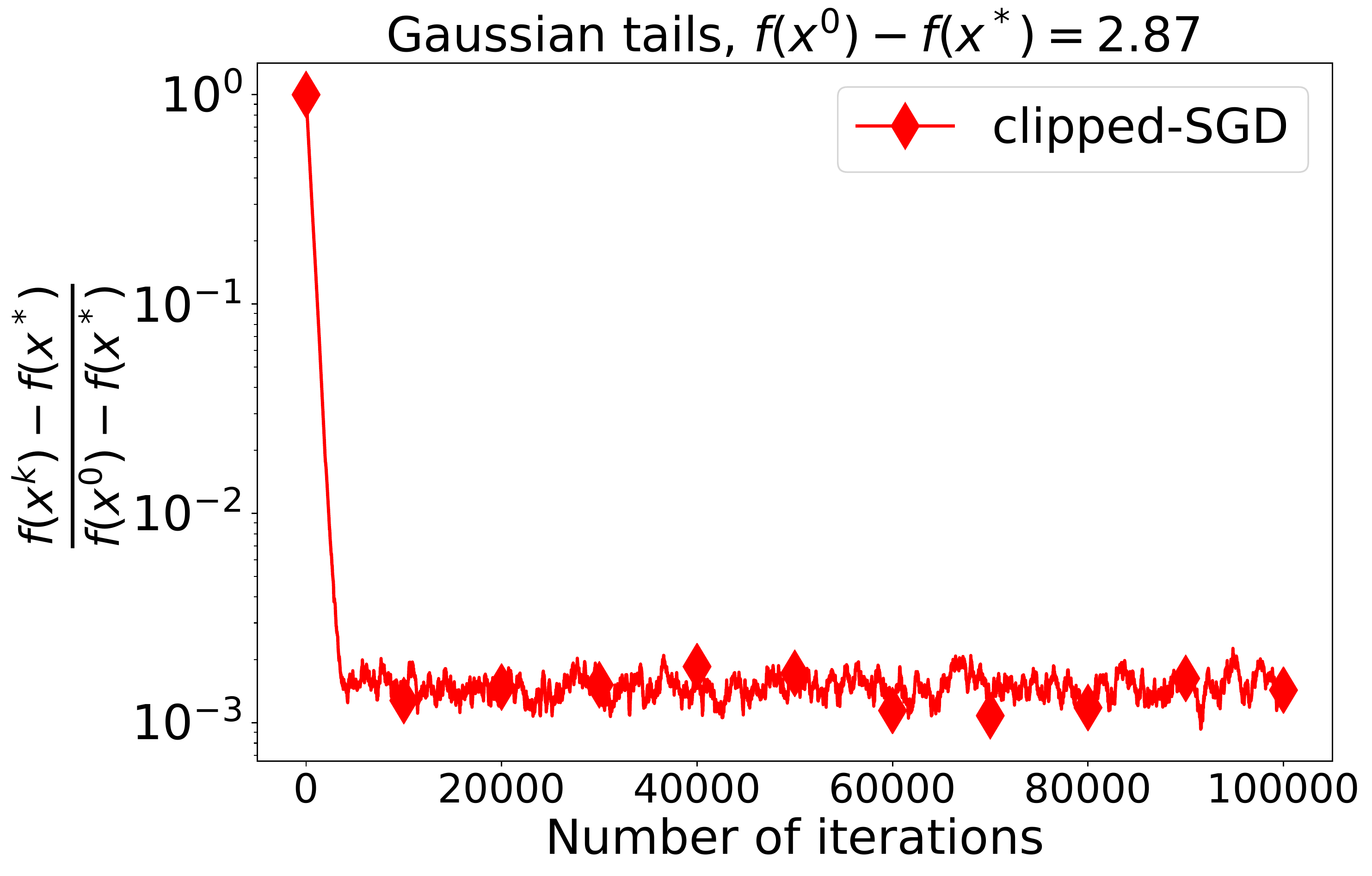}
    \includegraphics[width=0.32\textwidth]{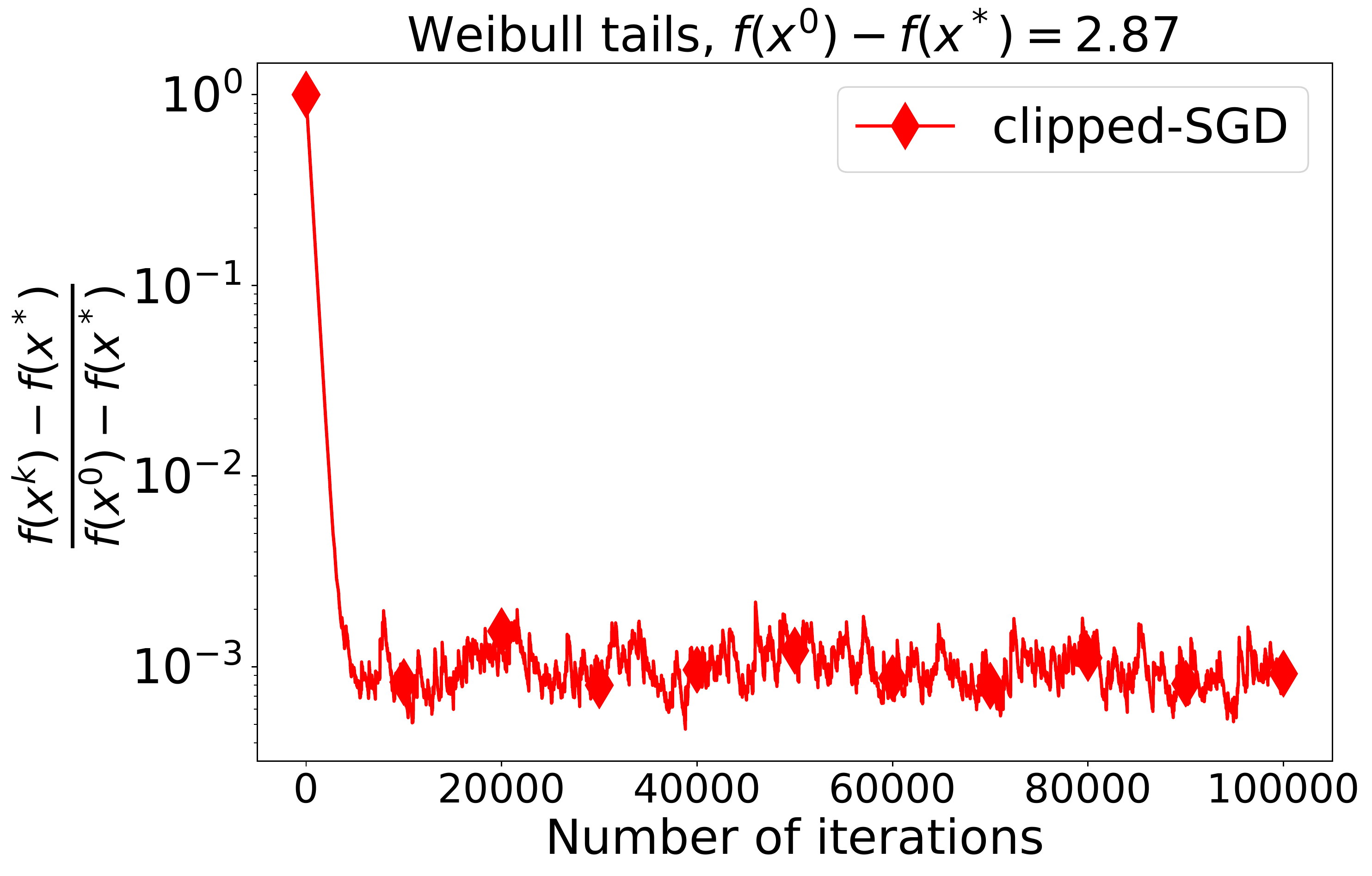}
    \includegraphics[width=0.32\textwidth]{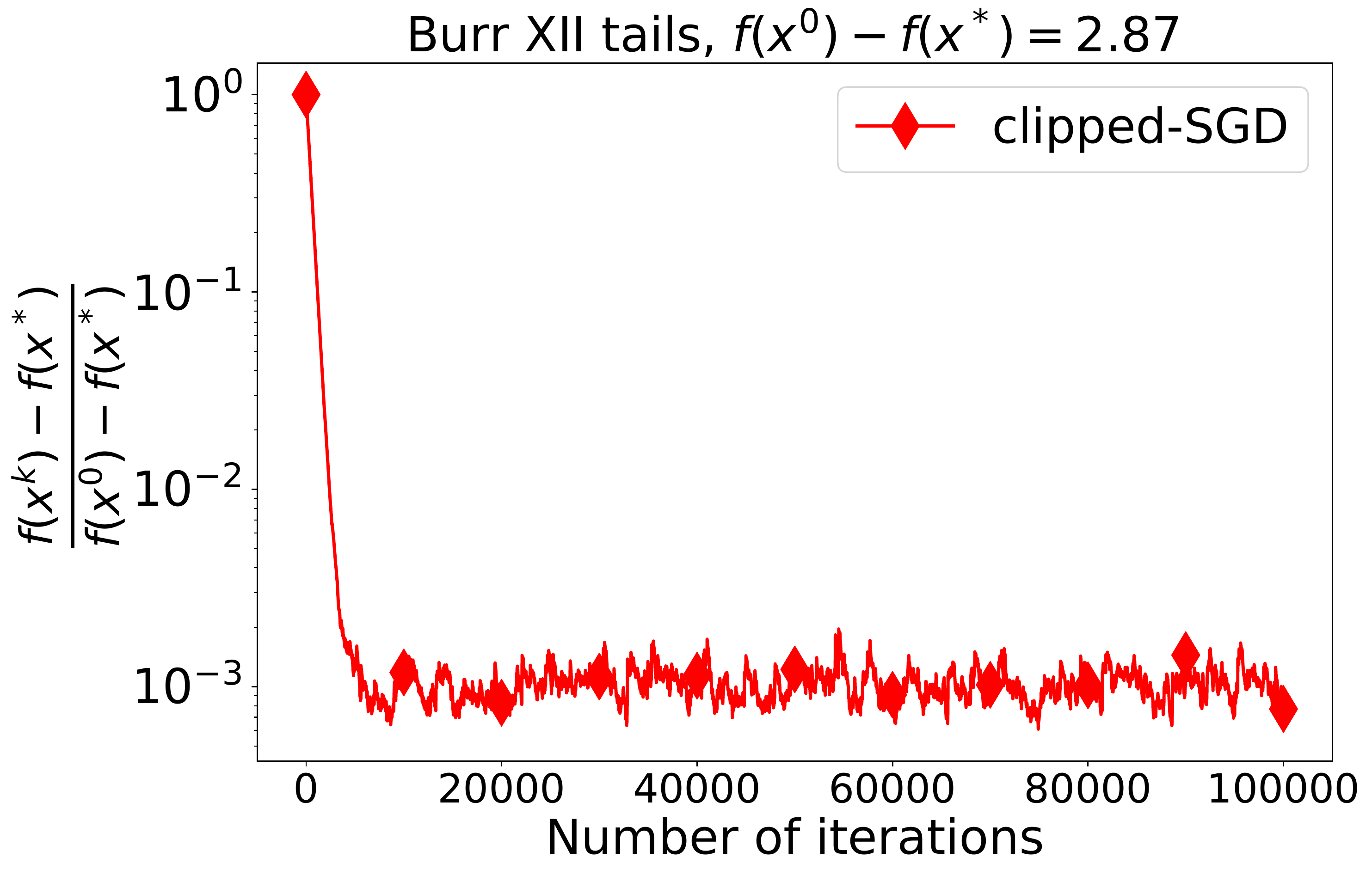}
    \caption{$2$ independent runs of {\tt SGD} (blue) and {\tt clipped-SGD} (red) applied to solve \eqref{eq:toy_problem} with $\xi$ having Gaussian (left column), Weibull (central column) and Burr Type XII (right column) tails.}
    \label{fig:tou_runs1}
\end{figure}

\begin{figure}[h]
    \centering
    \includegraphics[width=0.32\textwidth]{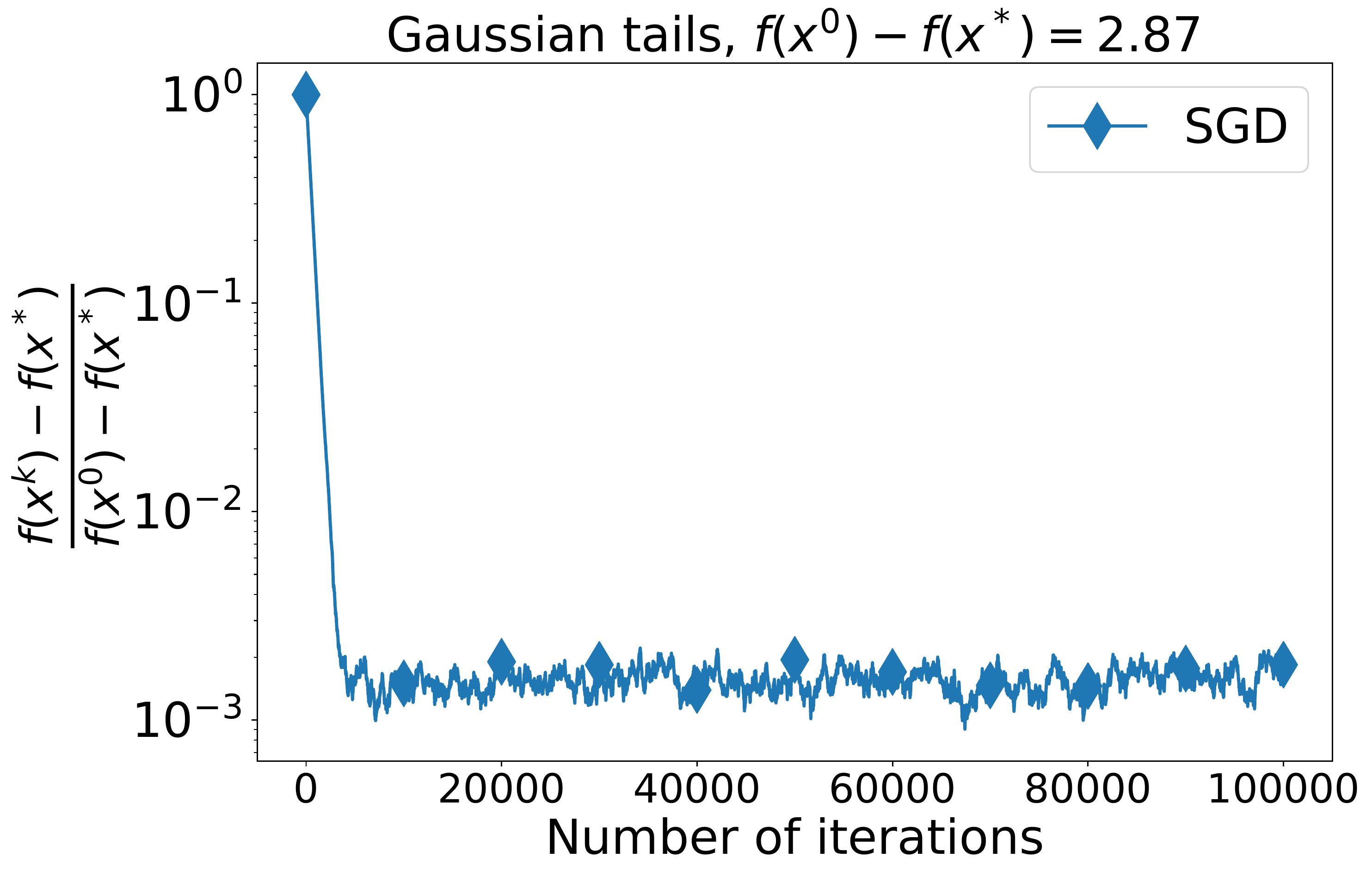}
    \includegraphics[width=0.32\textwidth]{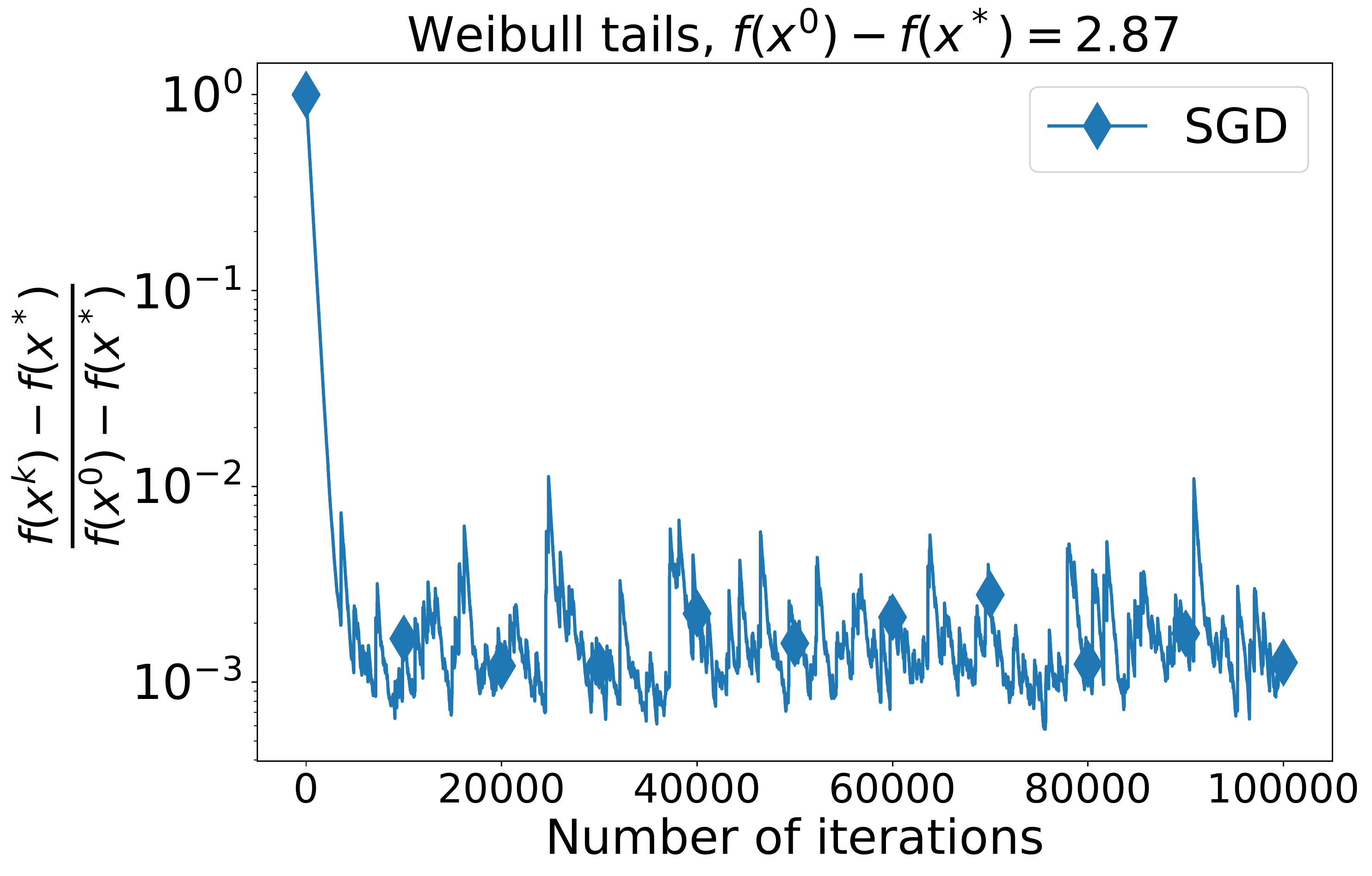}
    \includegraphics[width=0.32\textwidth]{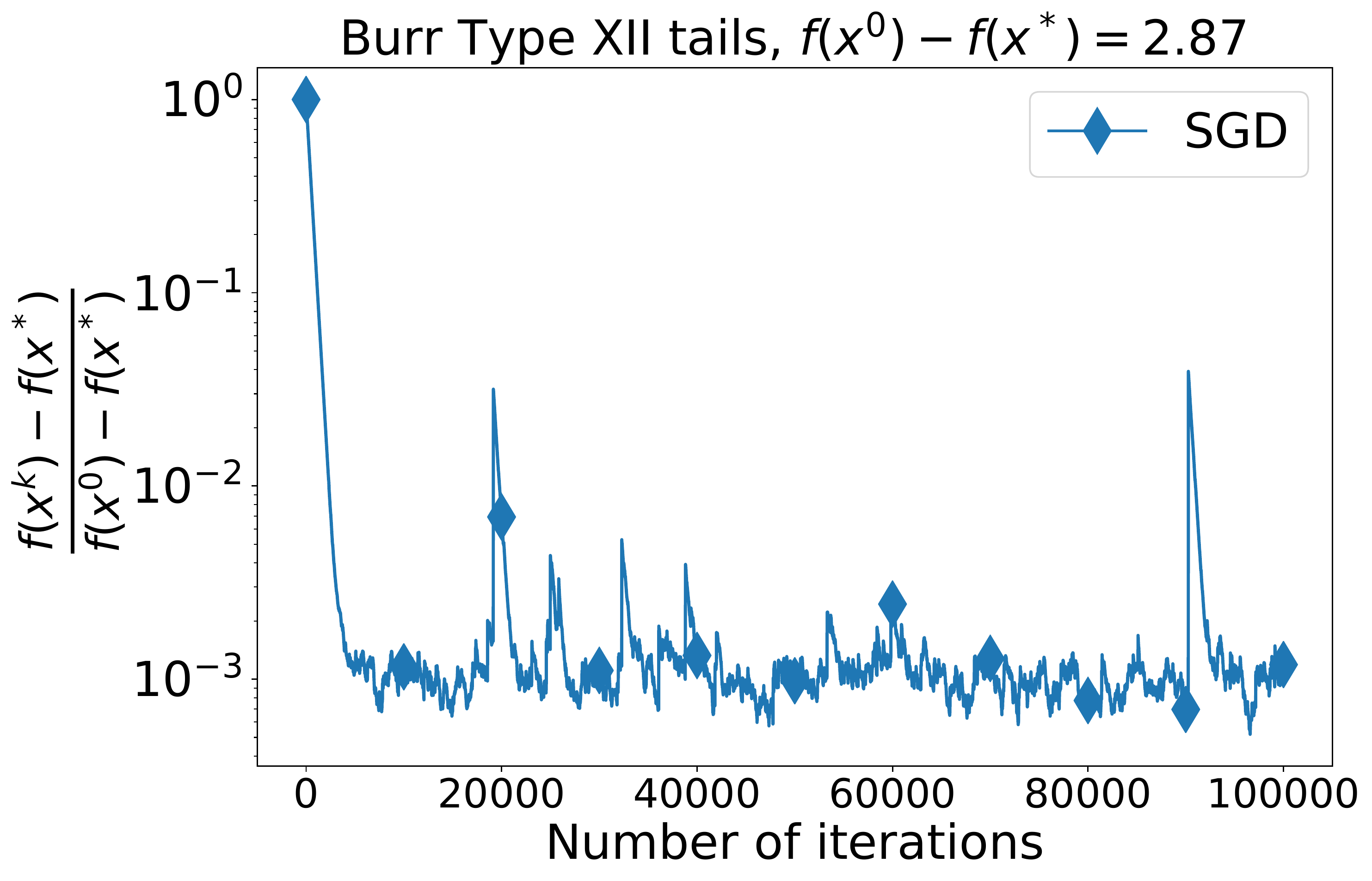}
    \includegraphics[width=0.32\textwidth]{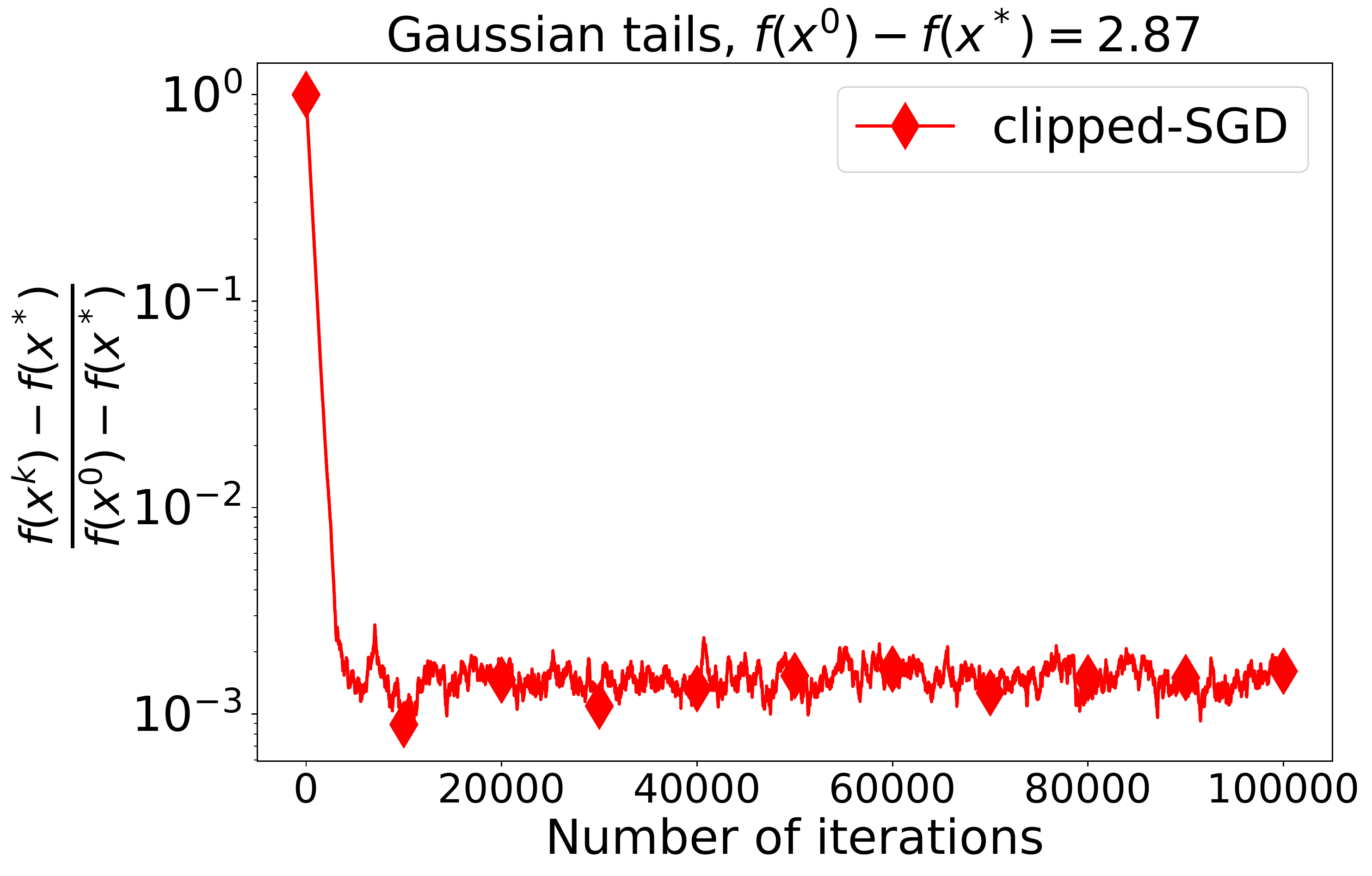}
    \includegraphics[width=0.32\textwidth]{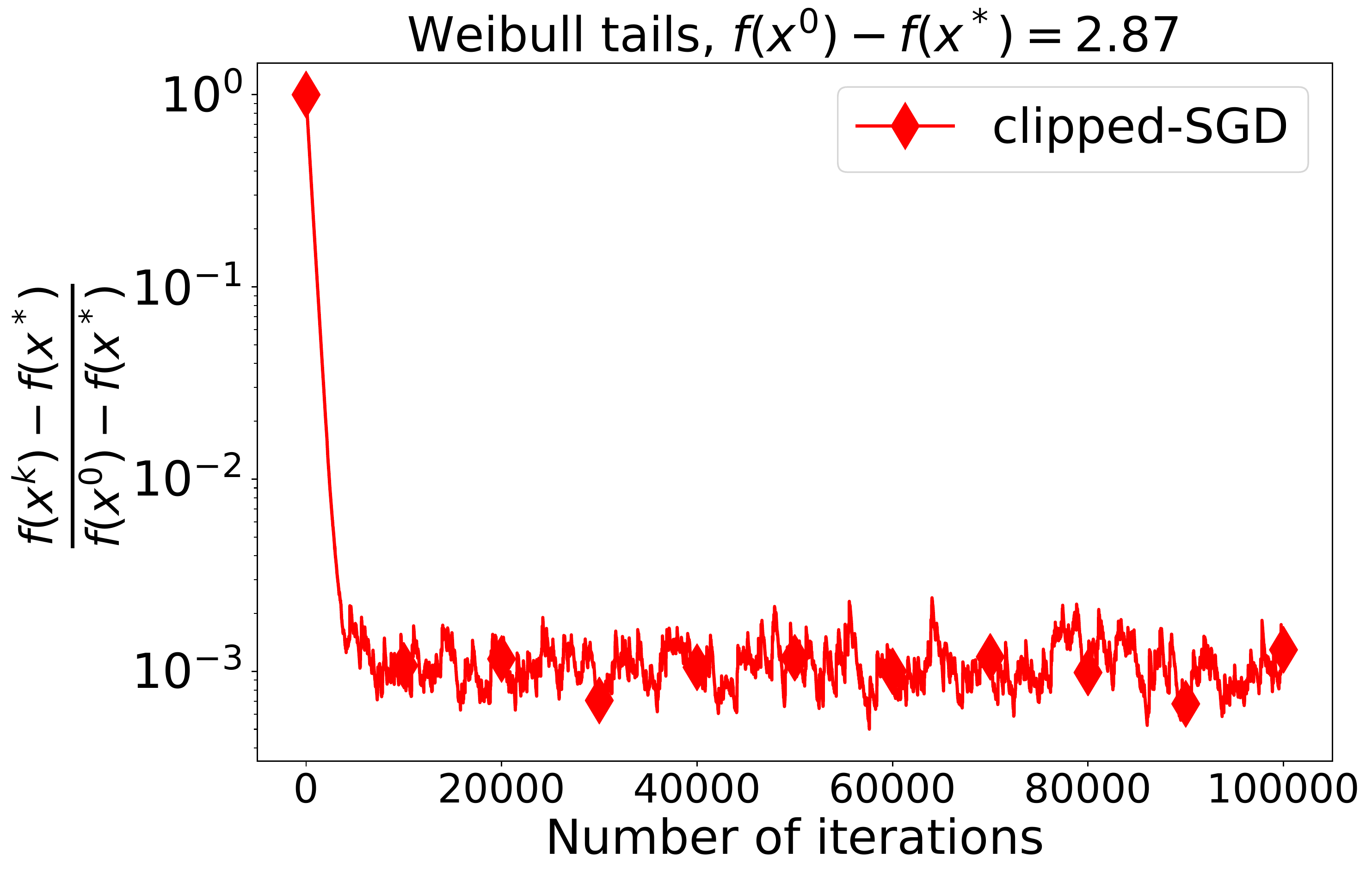}
    \includegraphics[width=0.32\textwidth]{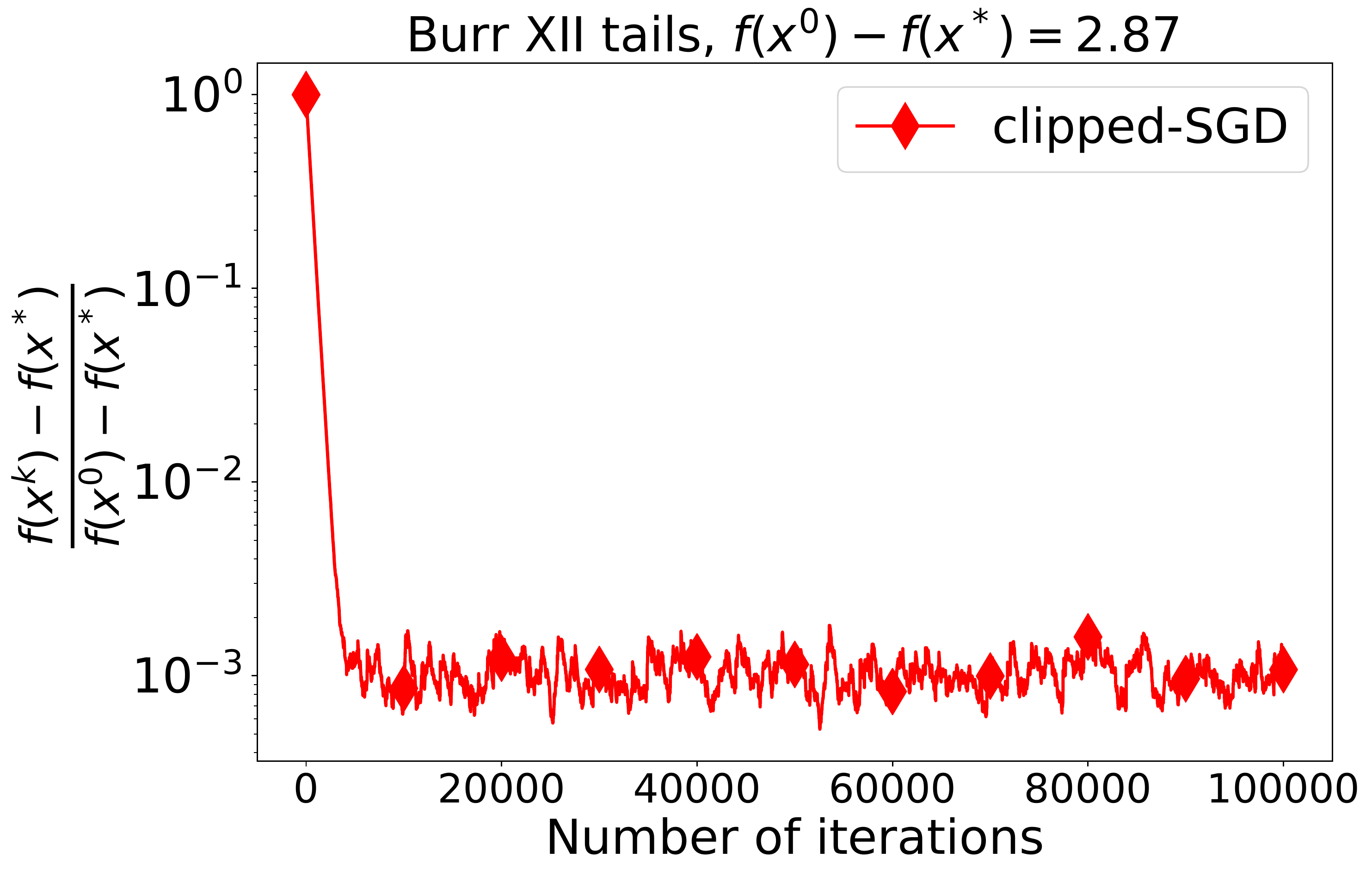}
    \includegraphics[width=0.32\textwidth]{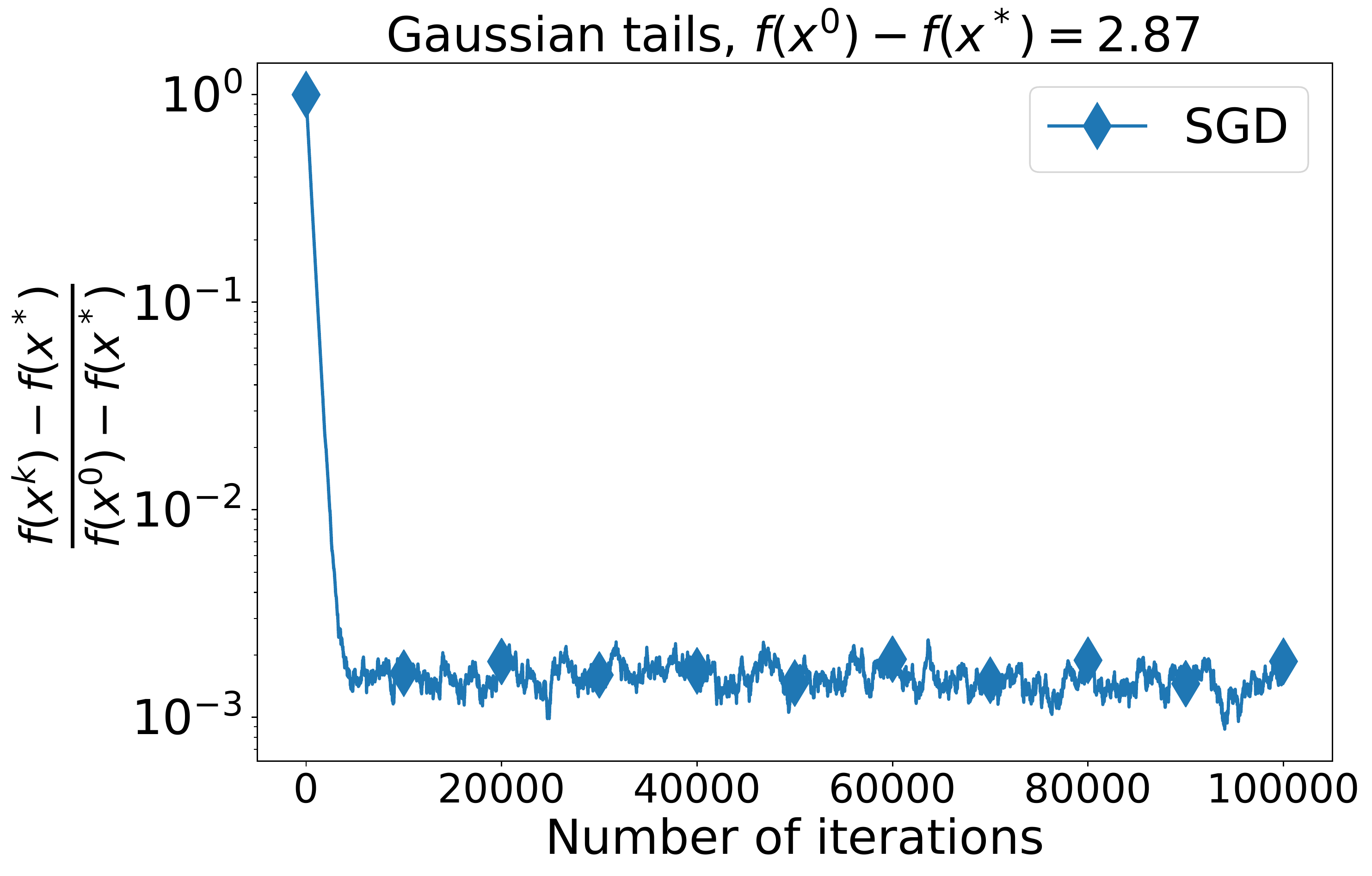}
    \includegraphics[width=0.32\textwidth]{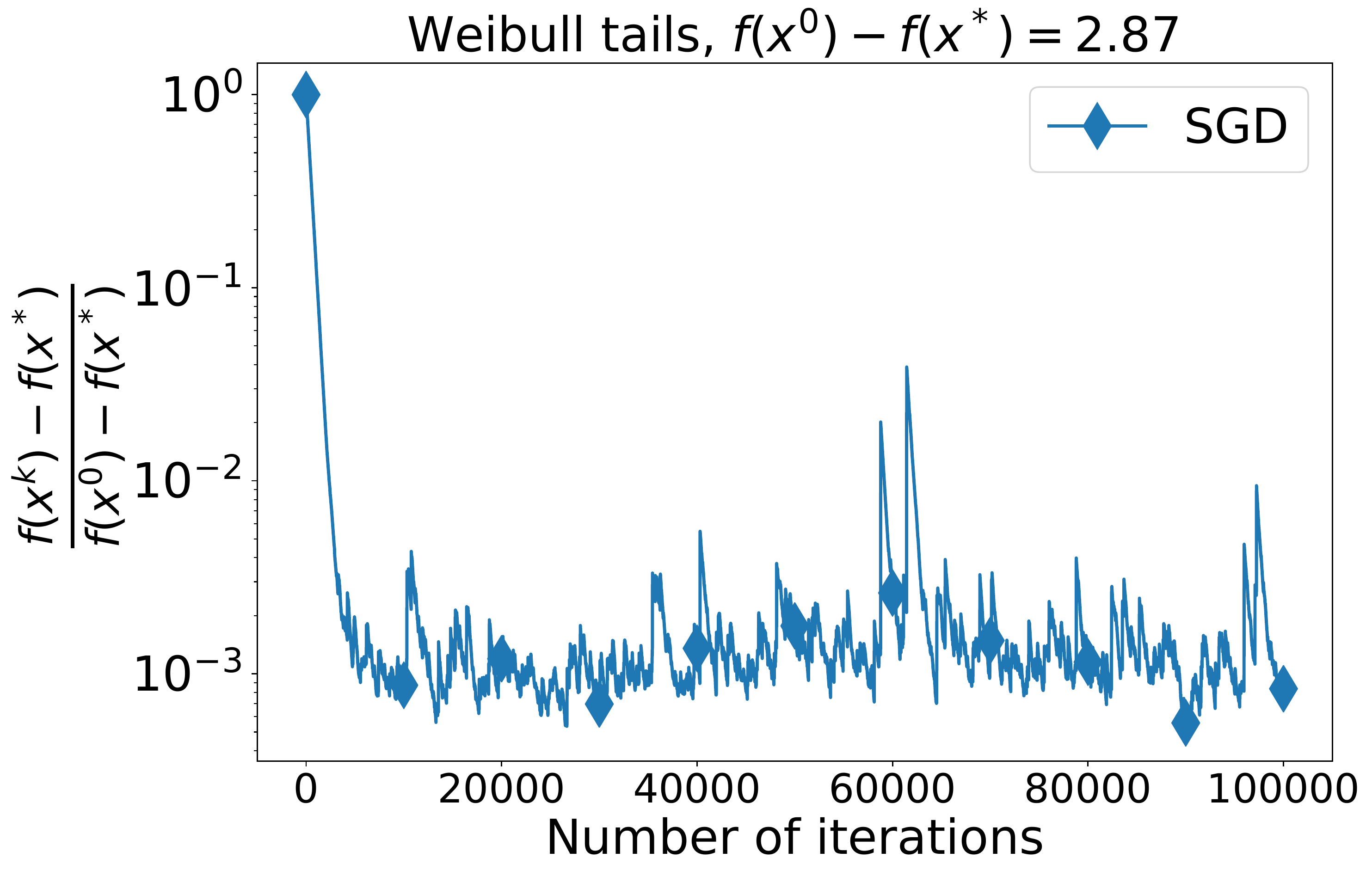}
    \includegraphics[width=0.32\textwidth]{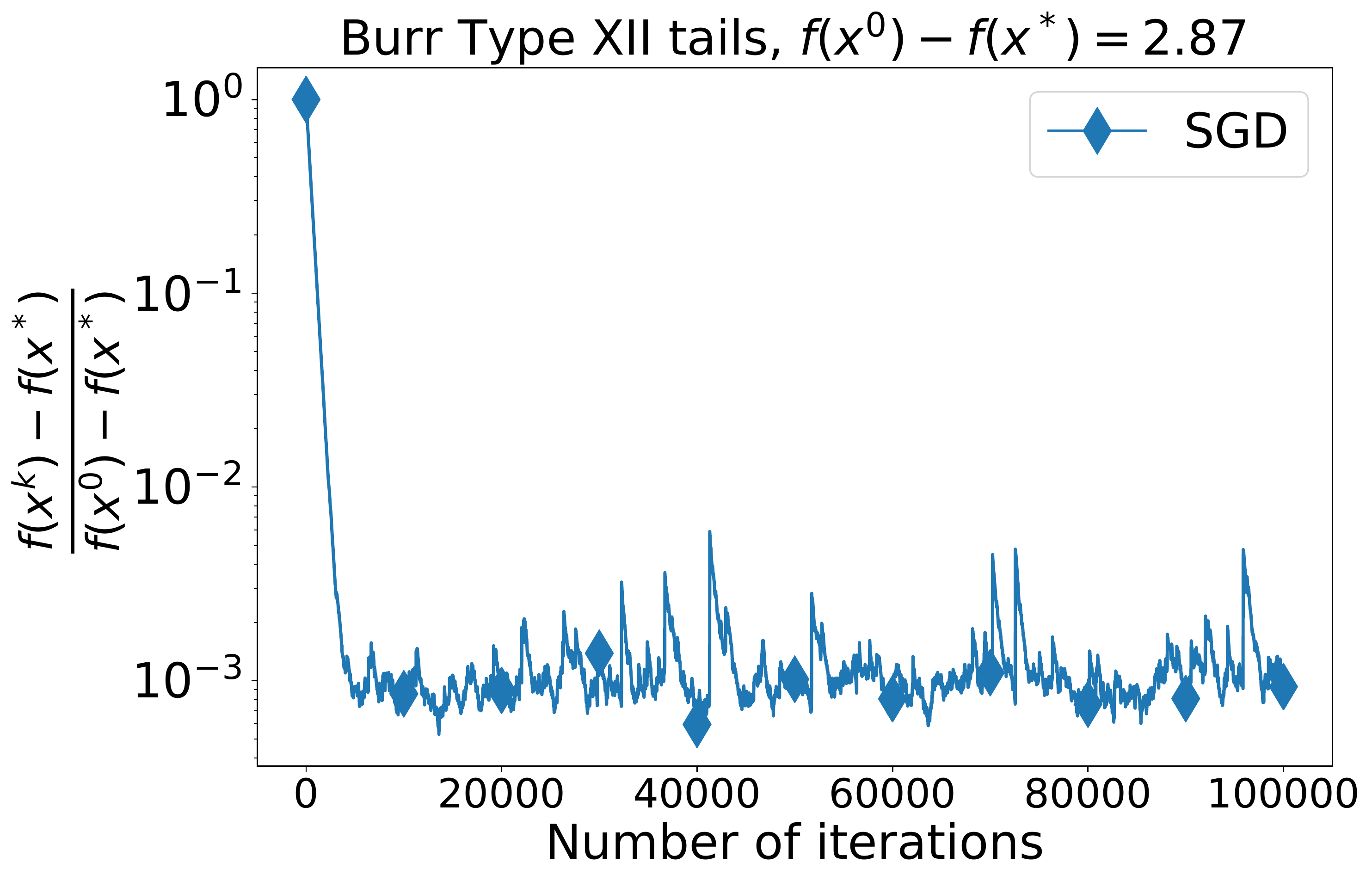}
    \includegraphics[width=0.32\textwidth]{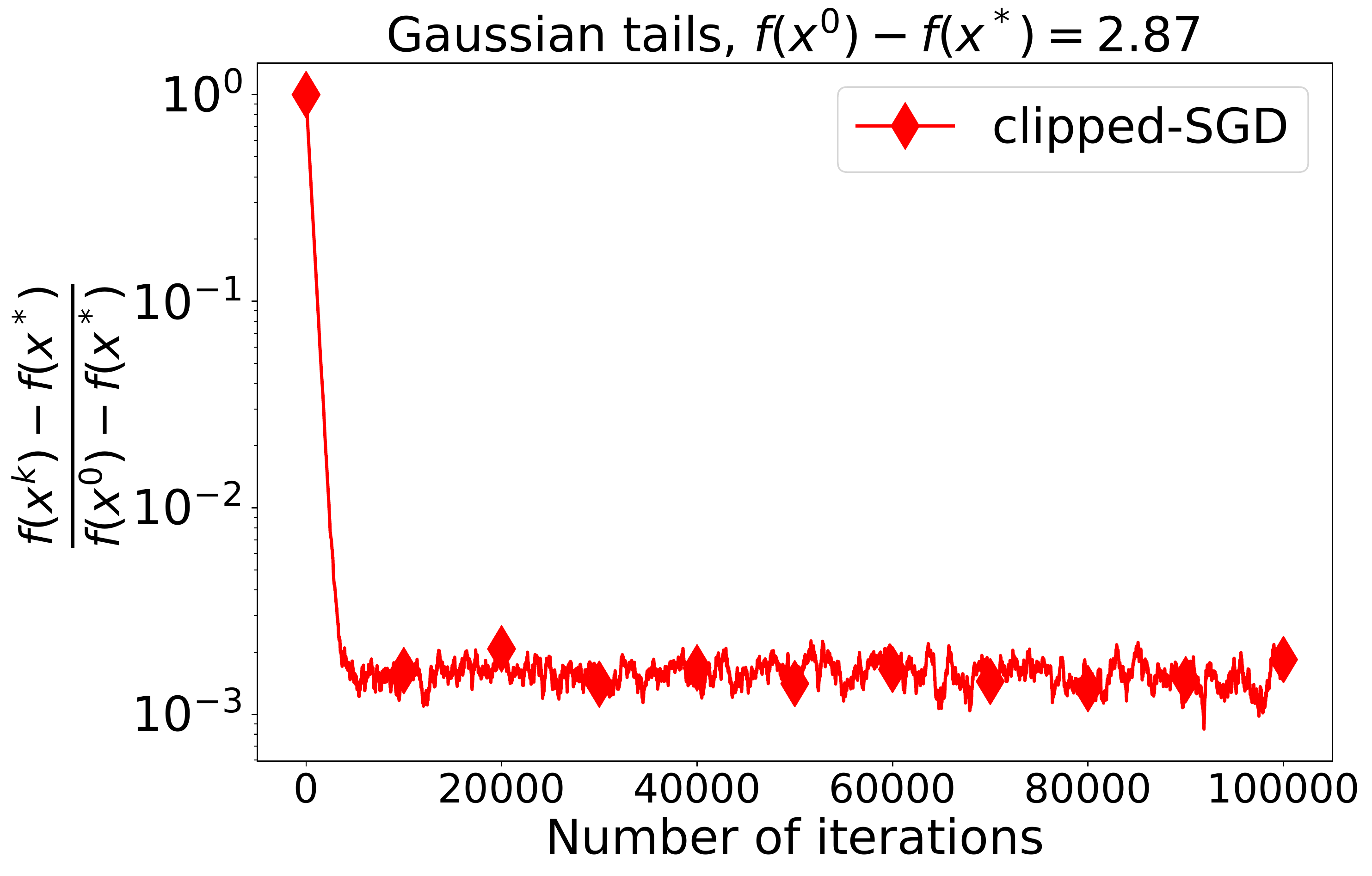}
    \includegraphics[width=0.32\textwidth]{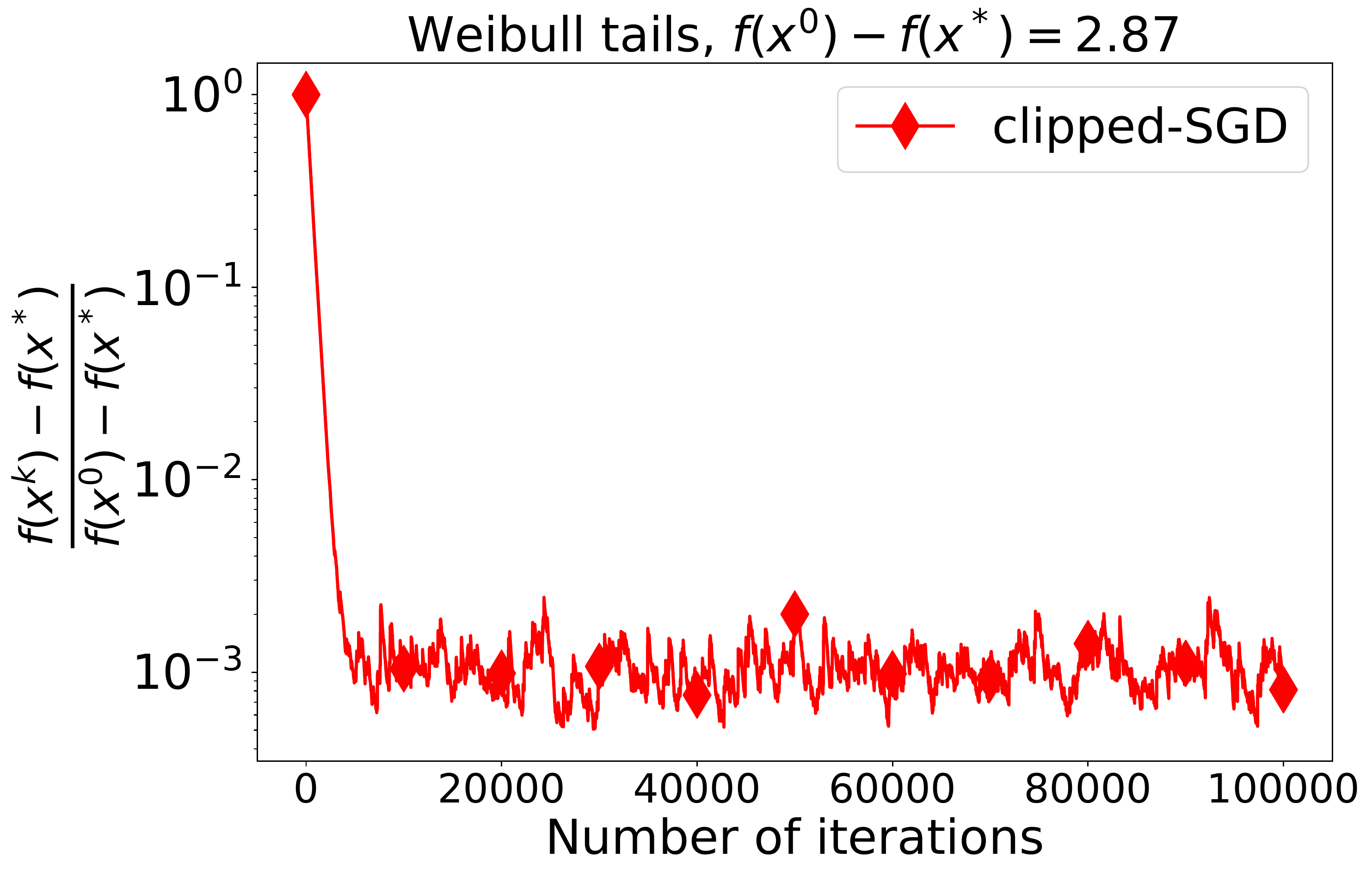}
    \includegraphics[width=0.32\textwidth]{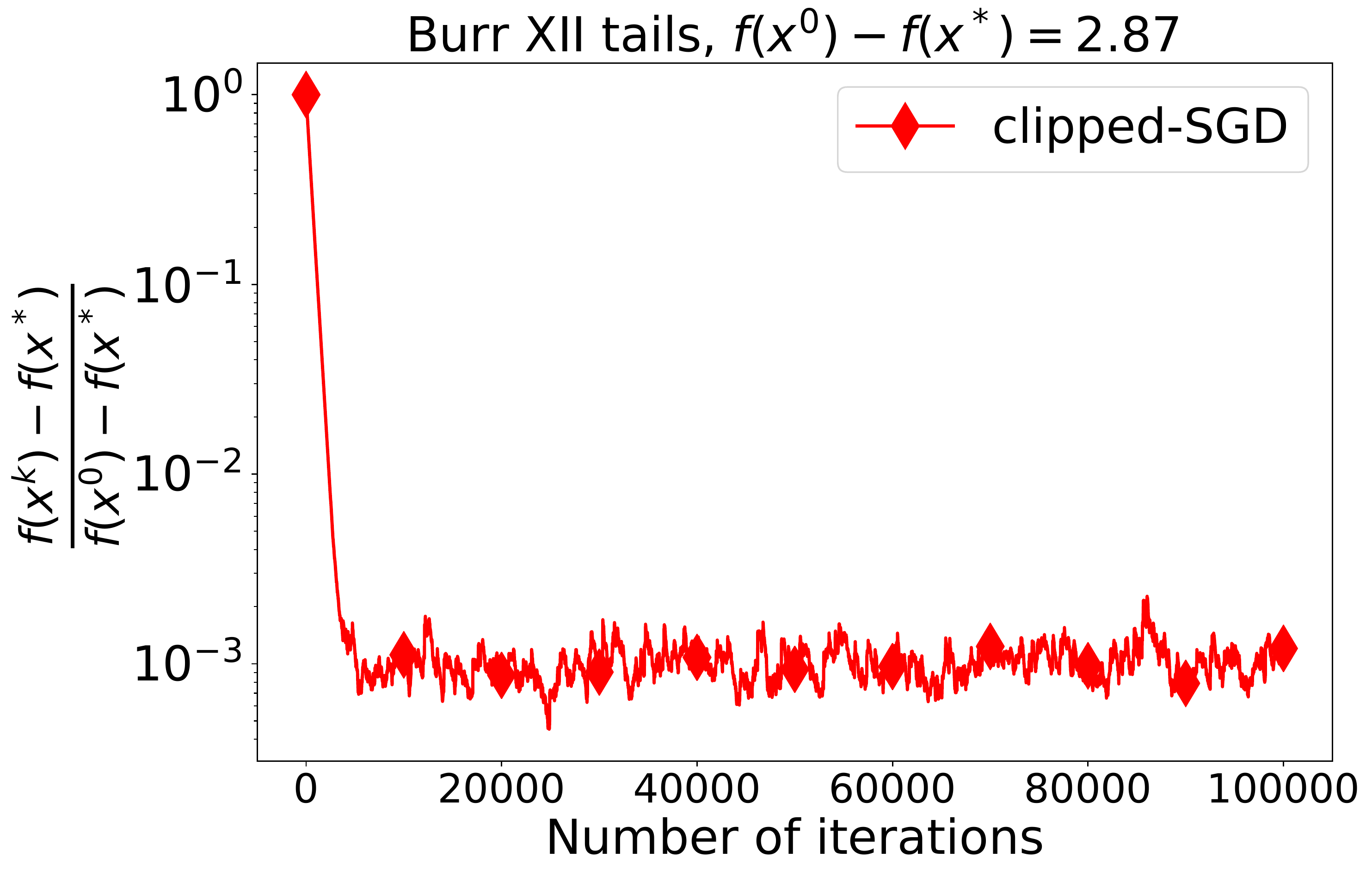}
    \caption{$2$ independent runs of {\tt SGD} (blue) and {\tt clipped-SGD} (red) applied to solve \eqref{eq:toy_problem} with $\xi$ having Gaussian (left column), Weibull (central column) and Burr Type XII (right column) tails.}
    \label{fig:tou_runs2}
\end{figure}

\begin{figure}[h]
    \centering
    \includegraphics[width=0.32\textwidth]{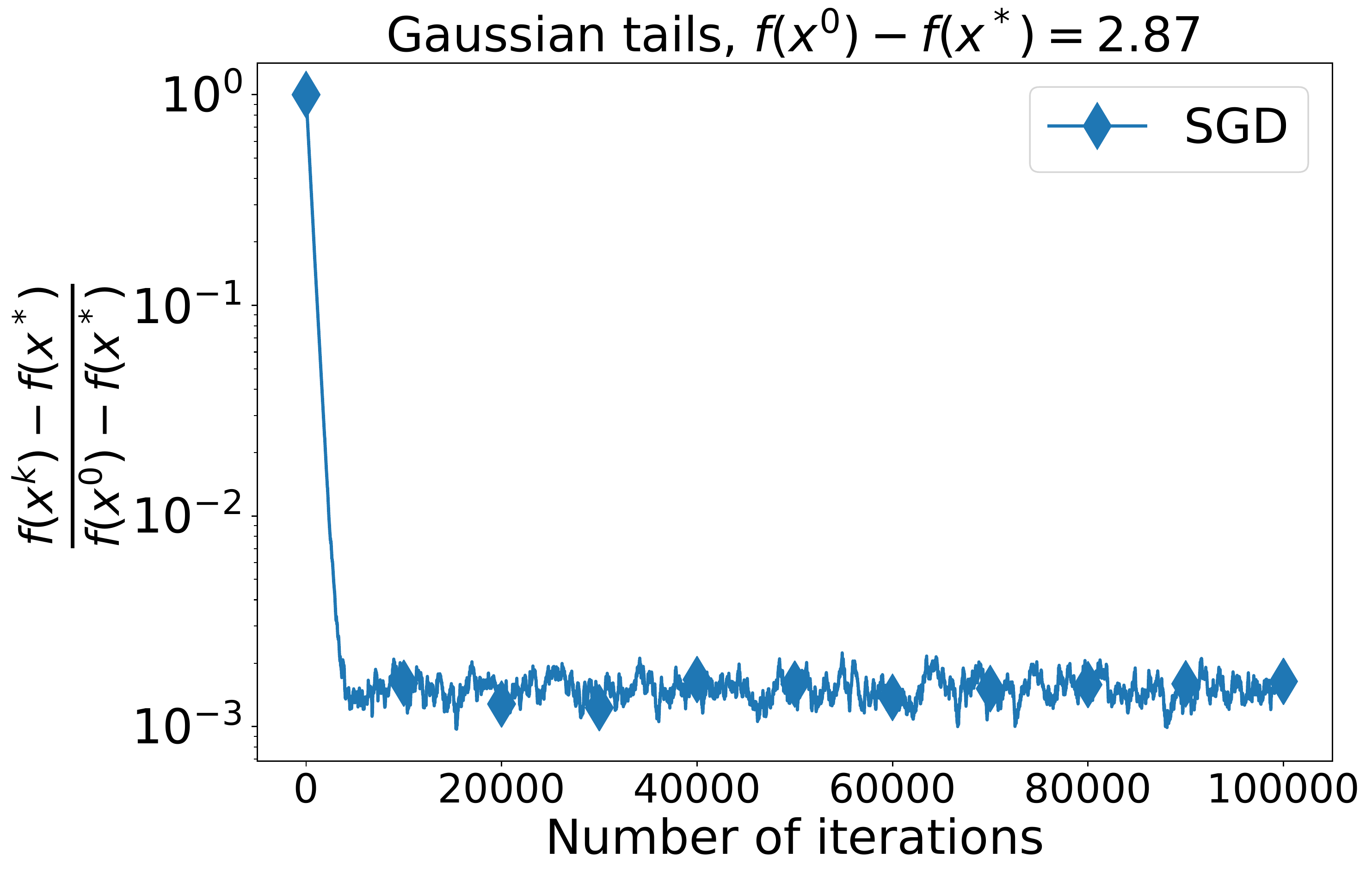}
    \includegraphics[width=0.32\textwidth]{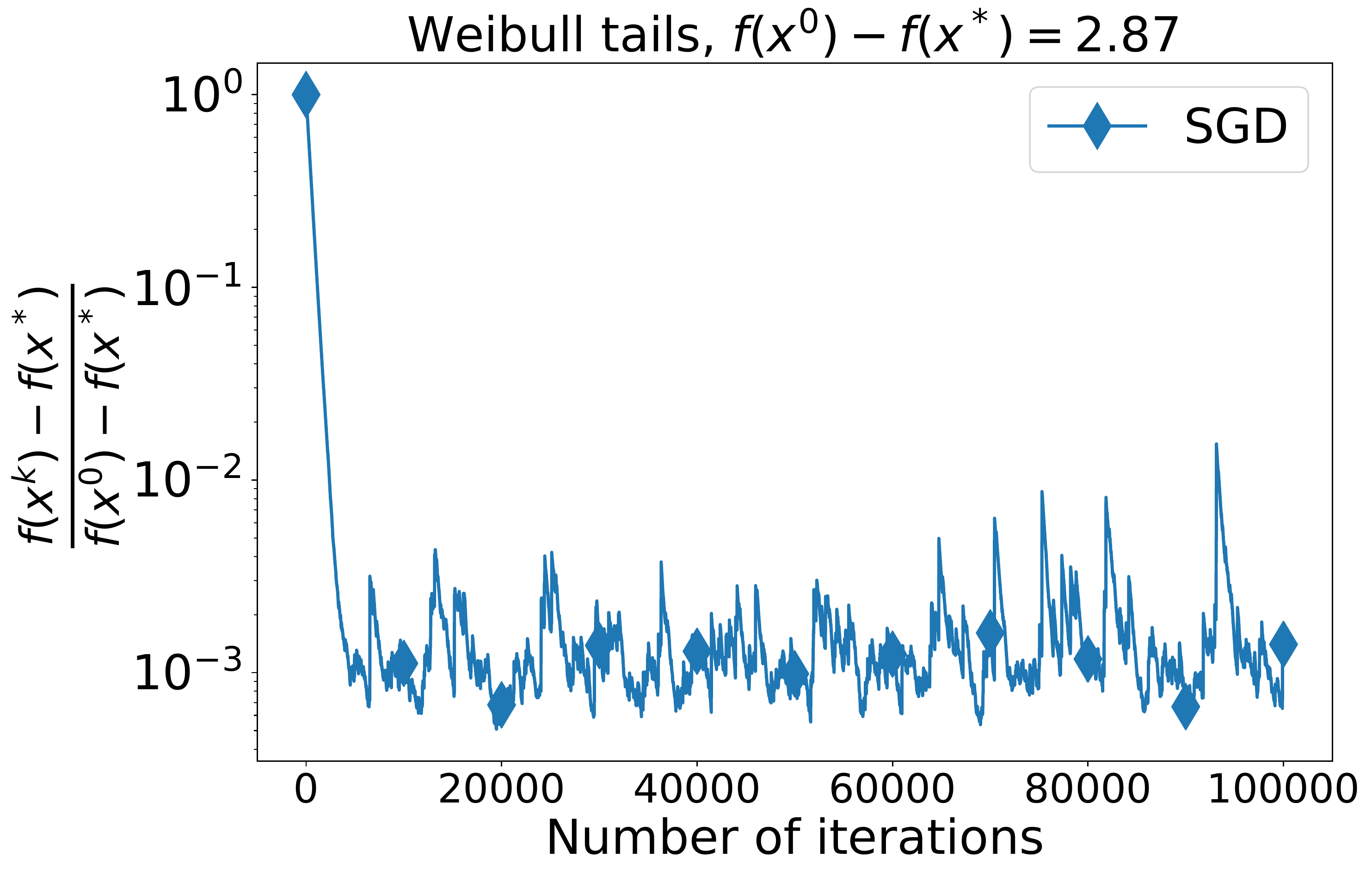}
    \includegraphics[width=0.32\textwidth]{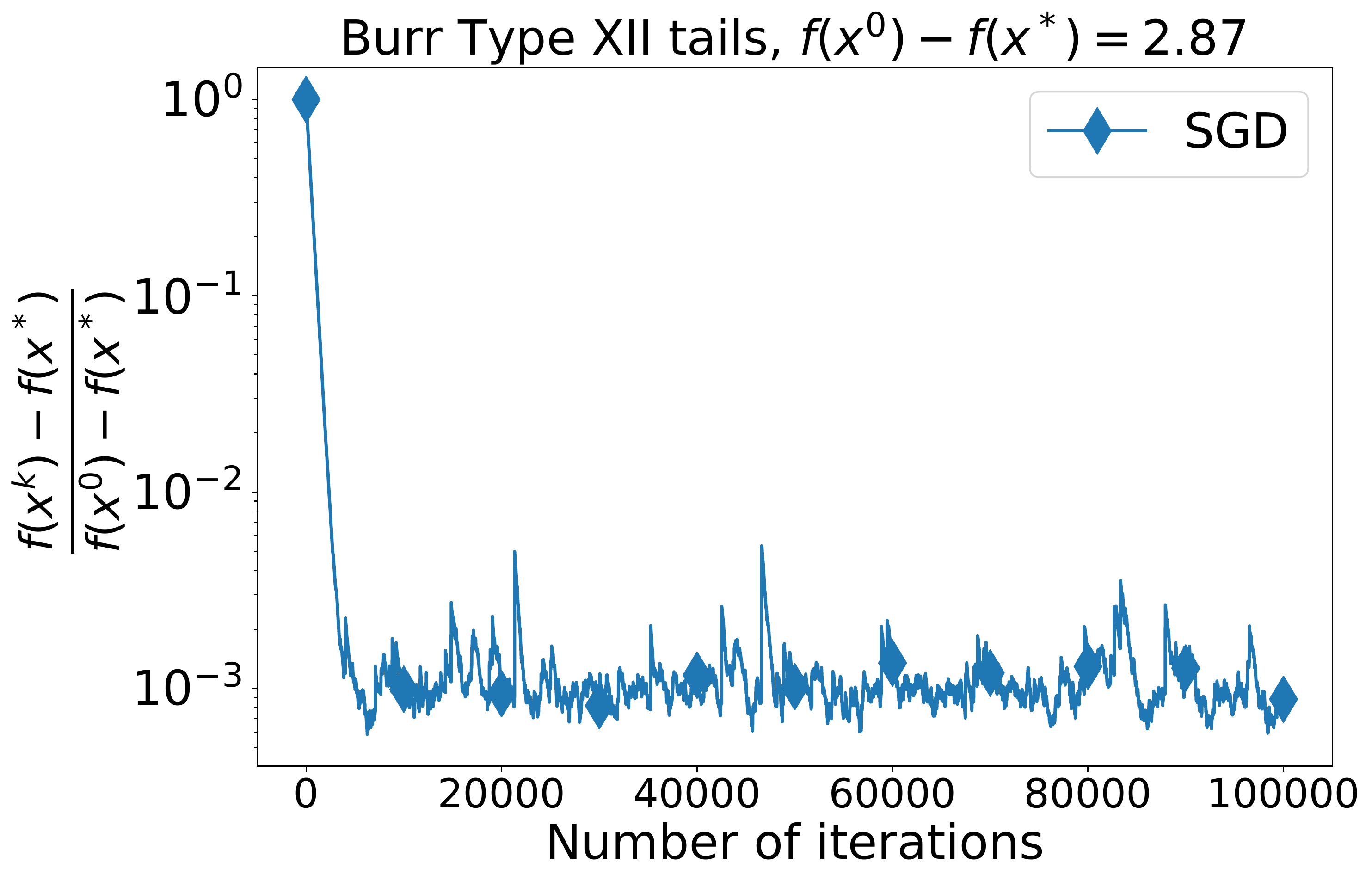}
    \includegraphics[width=0.32\textwidth]{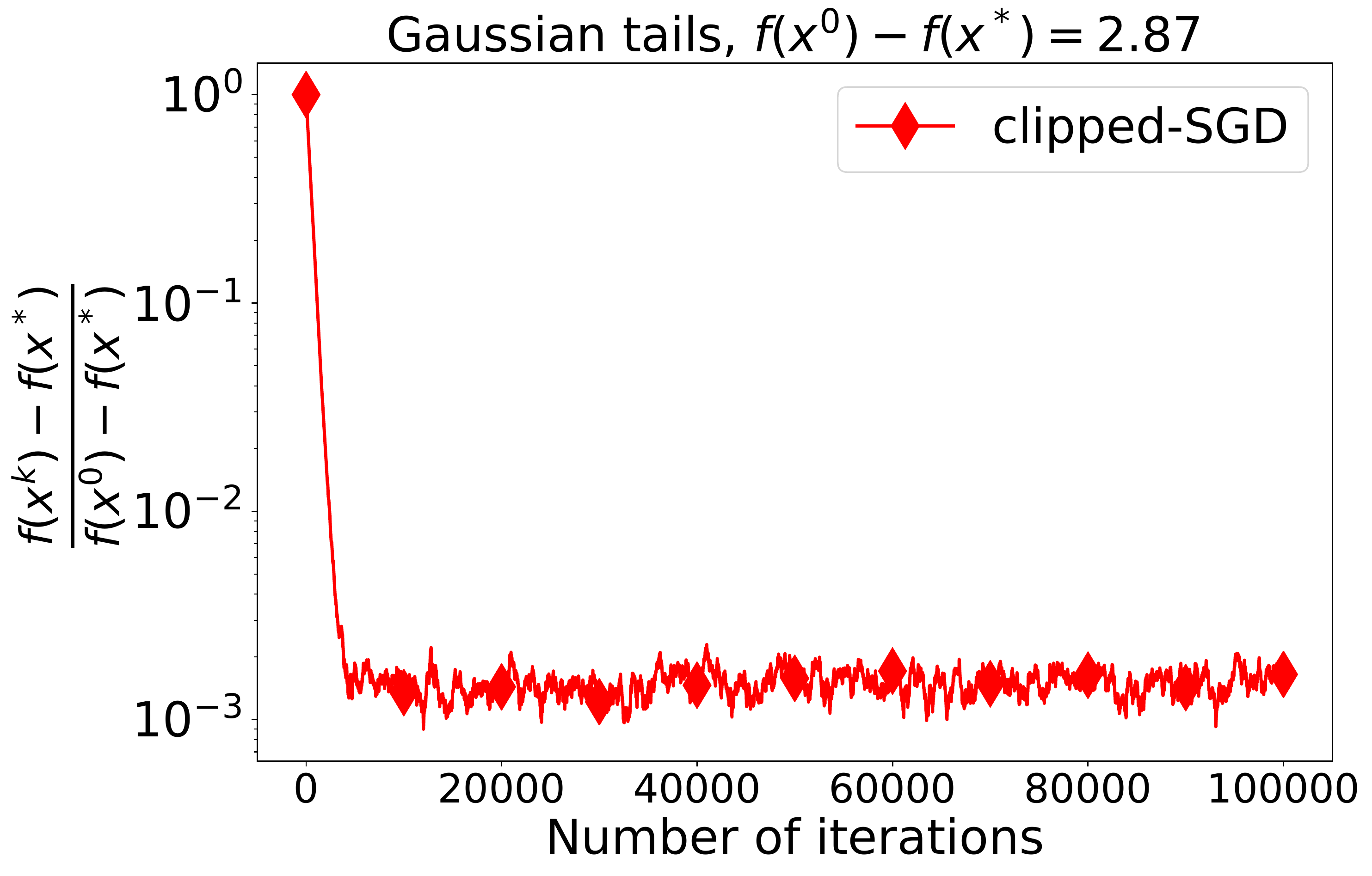}
    \includegraphics[width=0.32\textwidth]{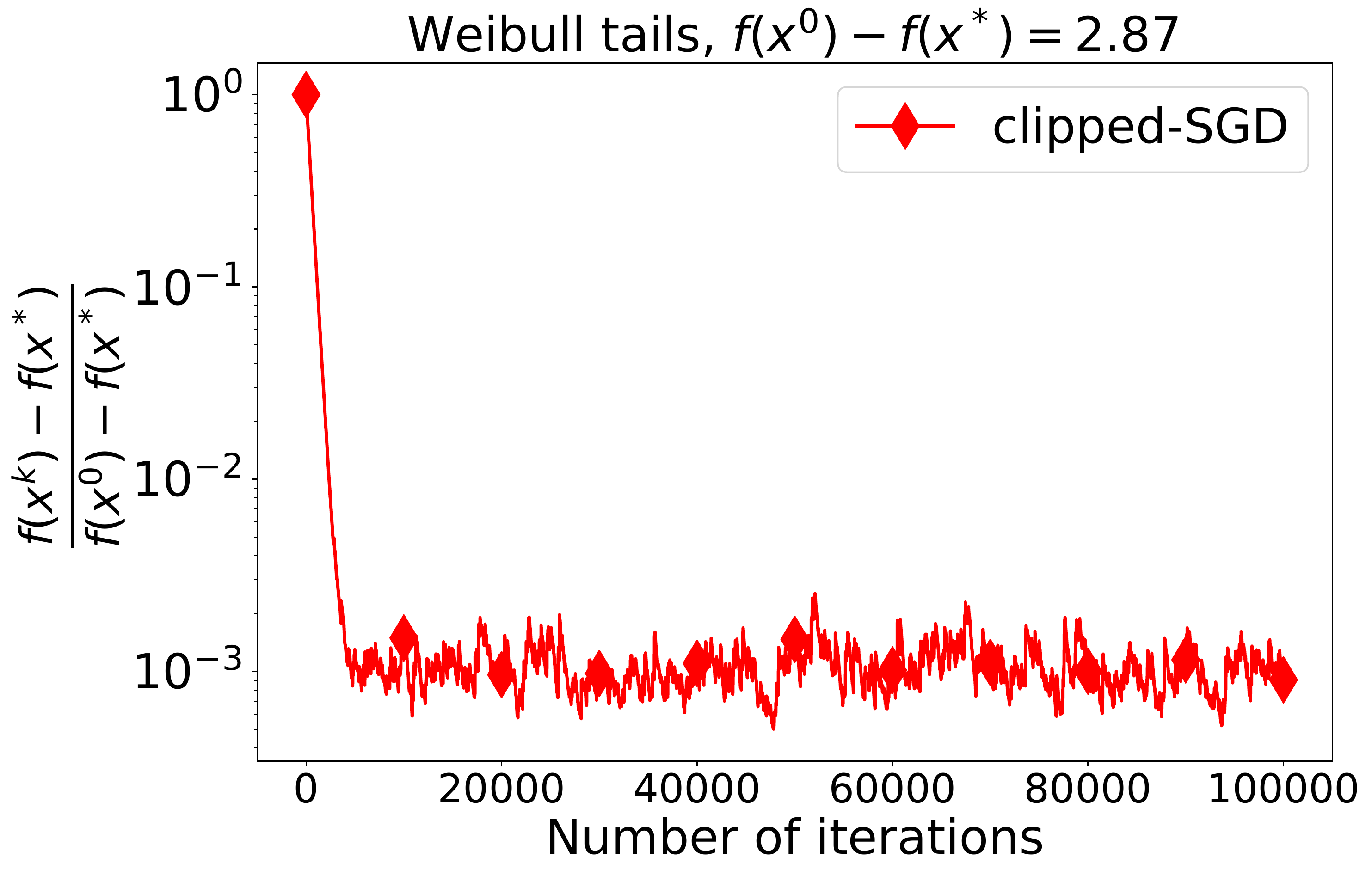}
    \includegraphics[width=0.32\textwidth]{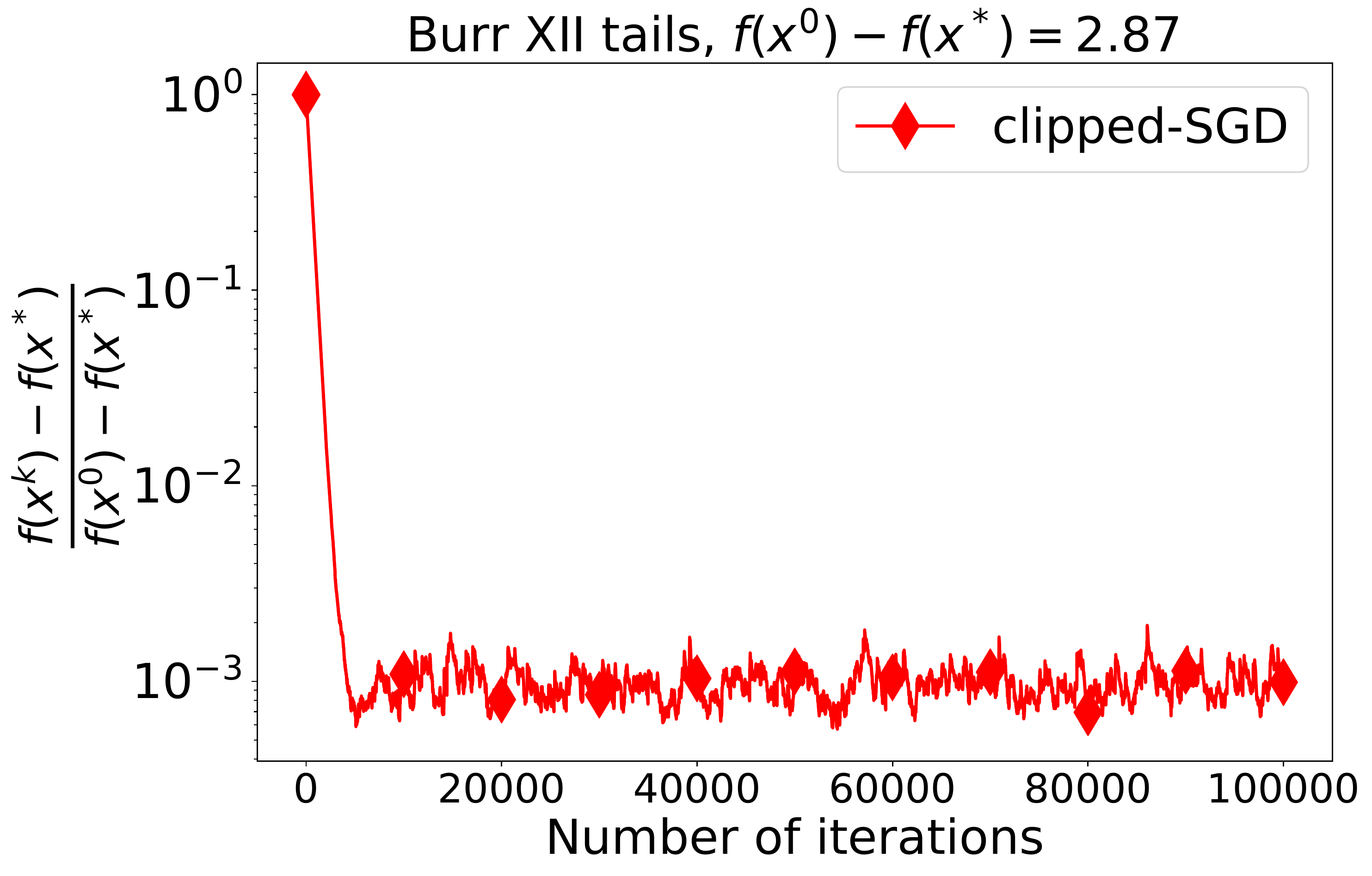}
    \includegraphics[width=0.32\textwidth]{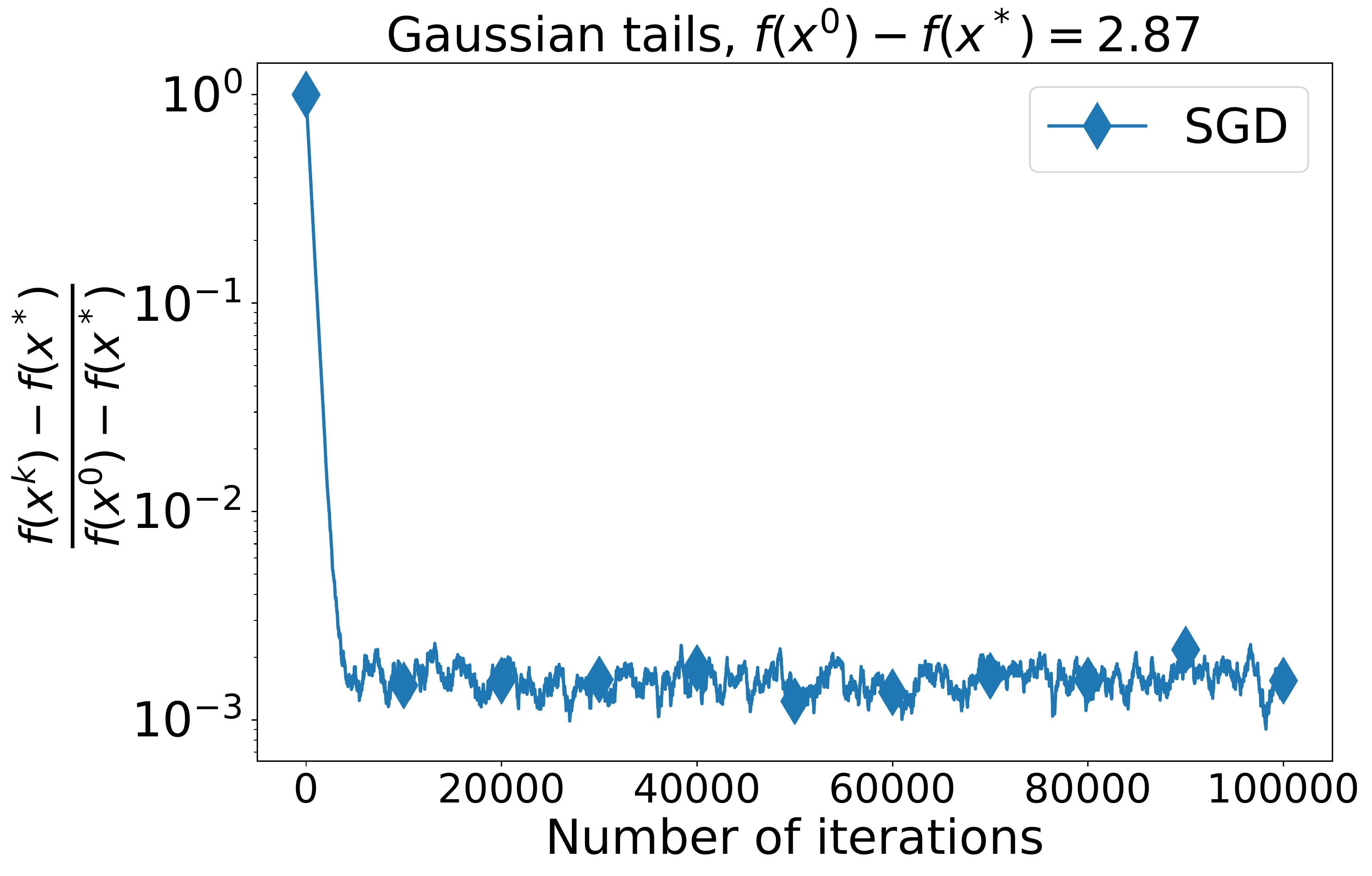}
    \includegraphics[width=0.32\textwidth]{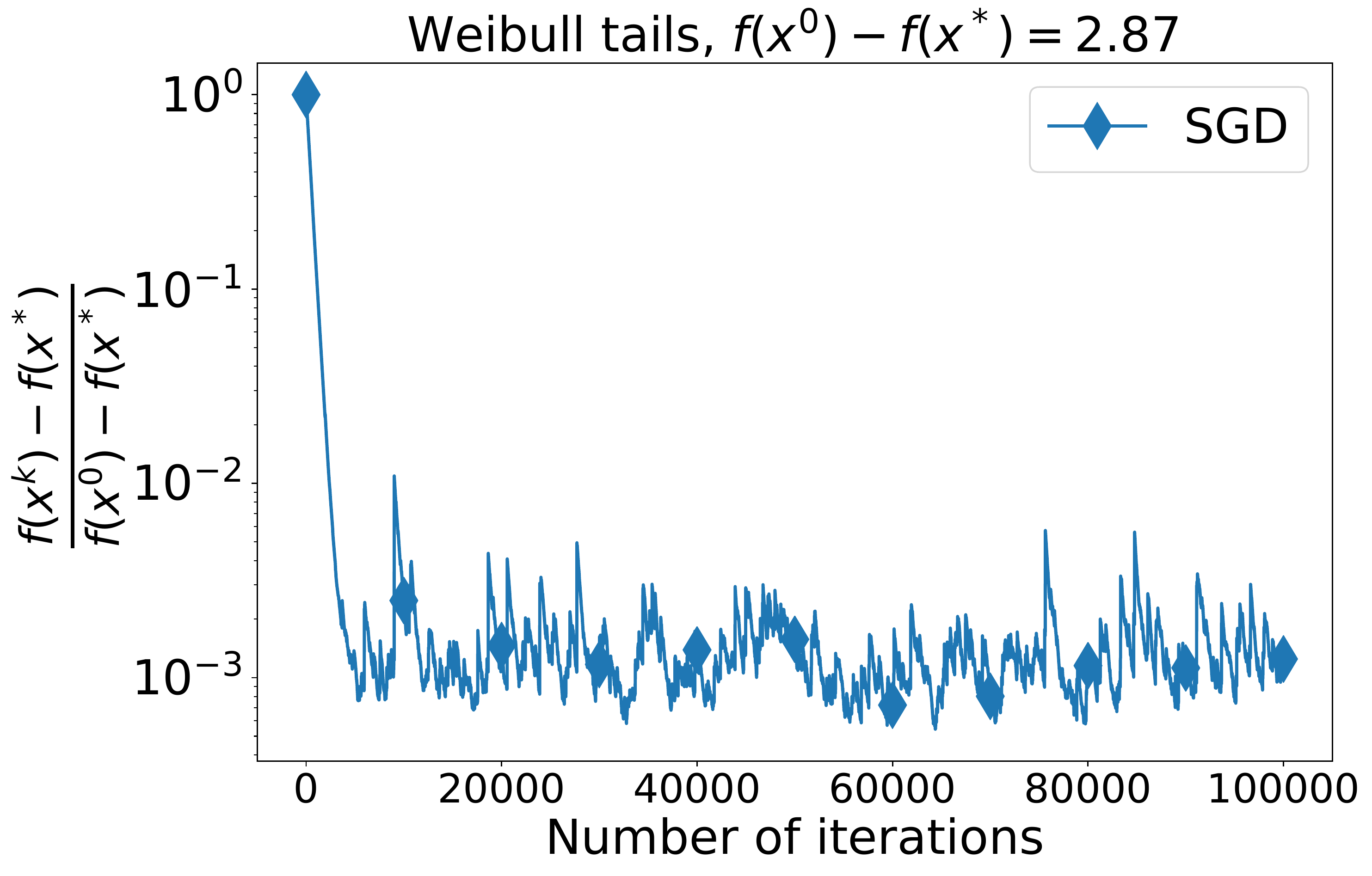}
    \includegraphics[width=0.32\textwidth]{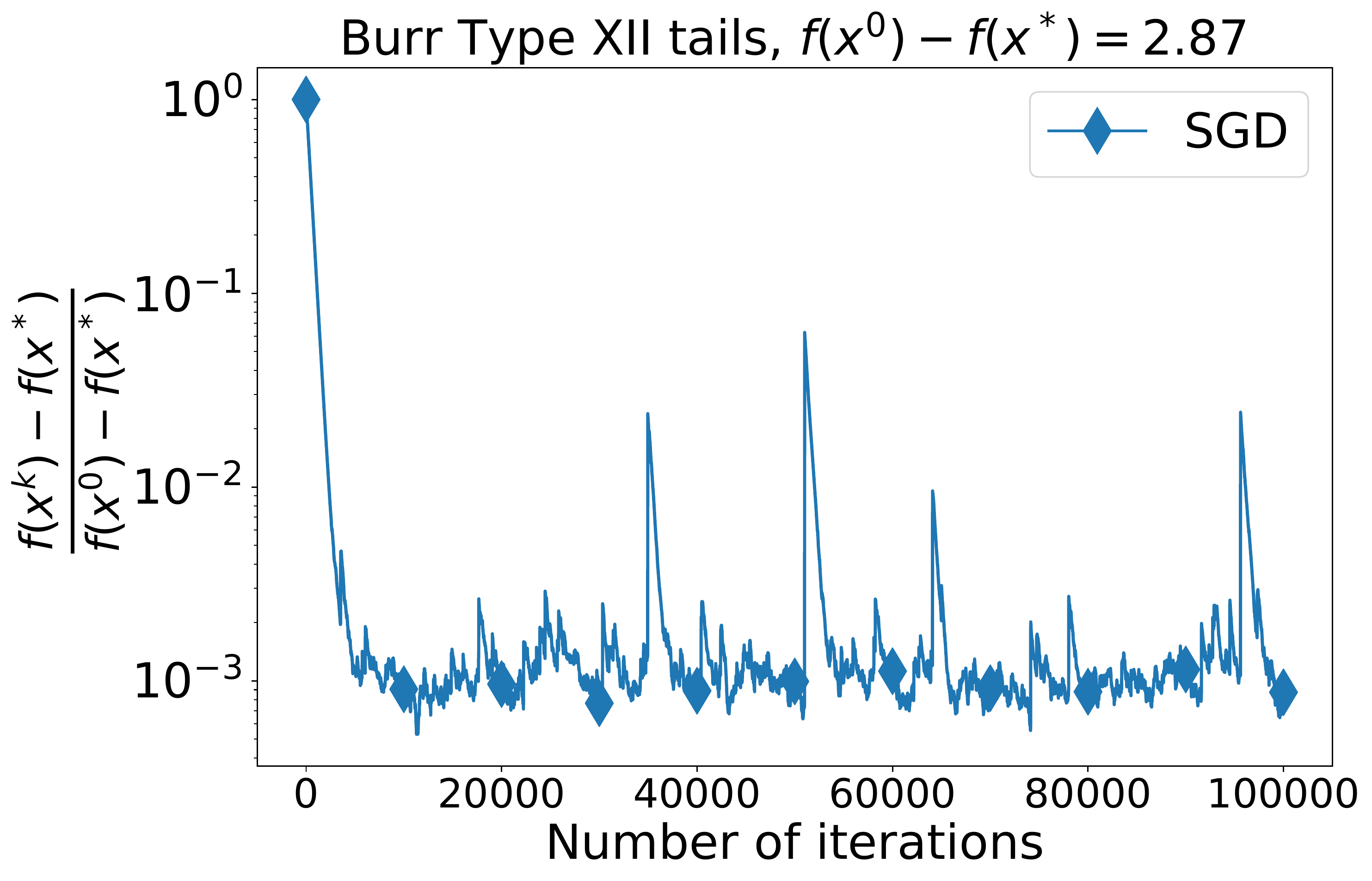}
    \includegraphics[width=0.32\textwidth]{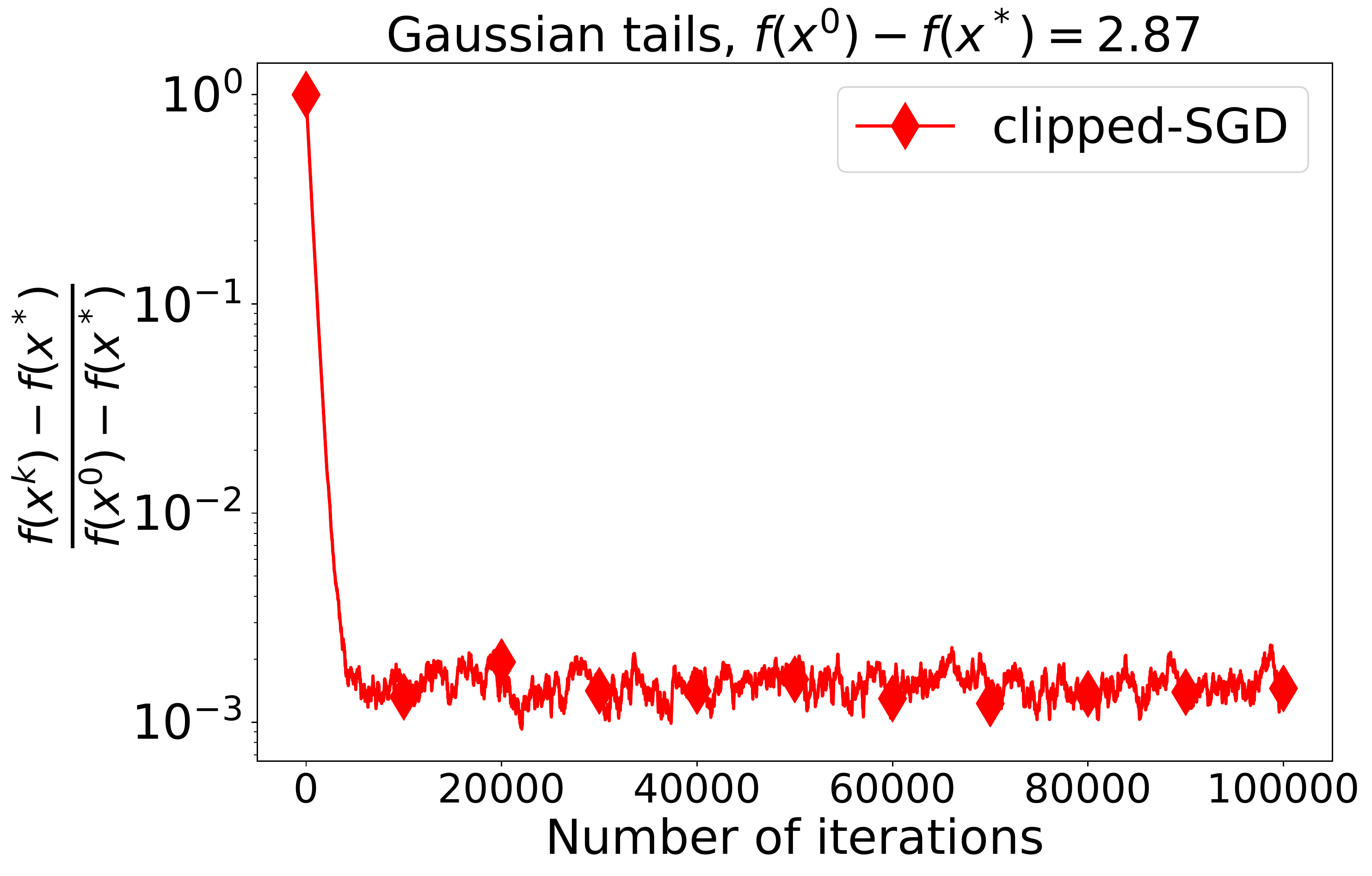}
    \includegraphics[width=0.32\textwidth]{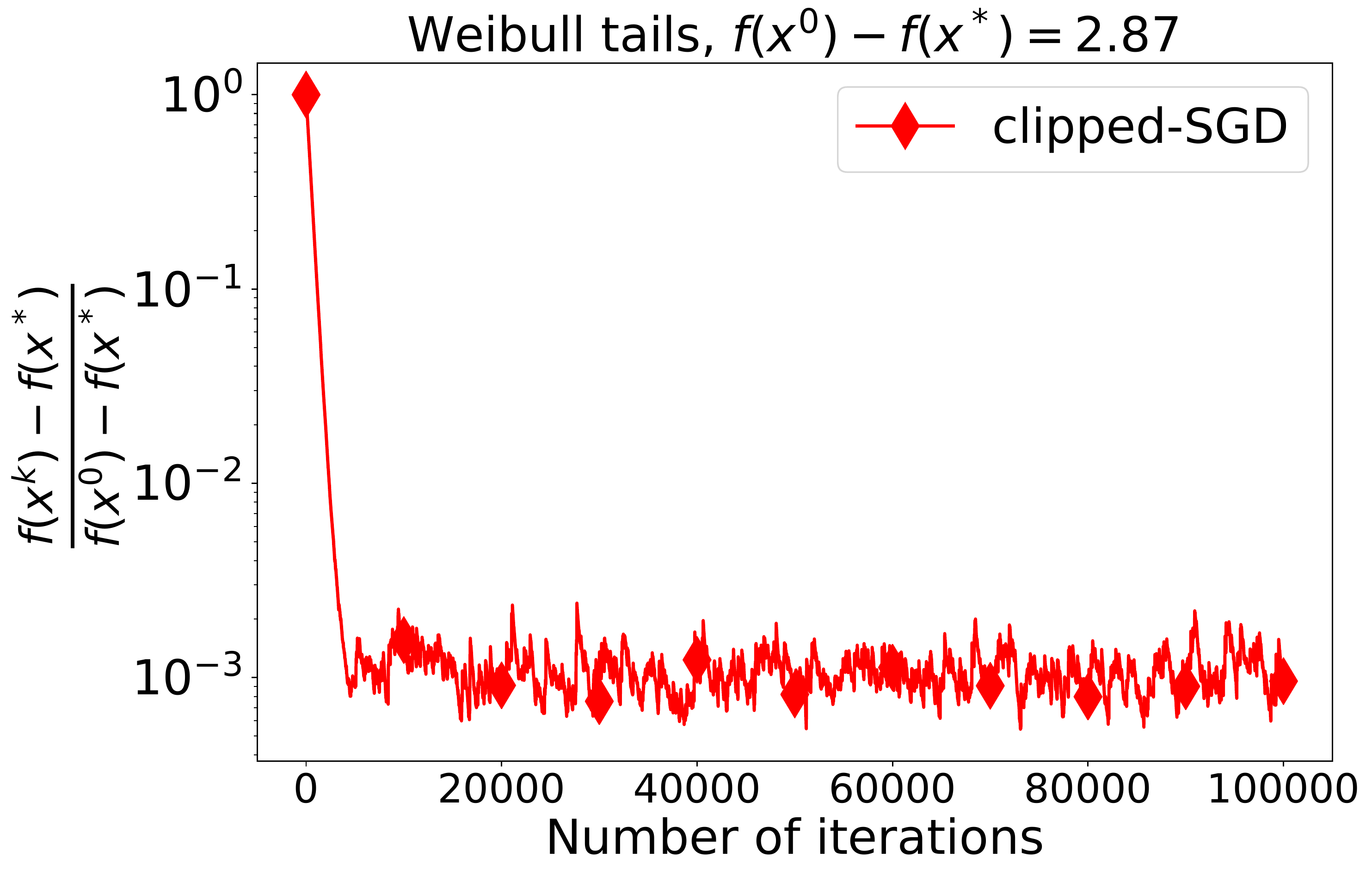}
    \includegraphics[width=0.32\textwidth]{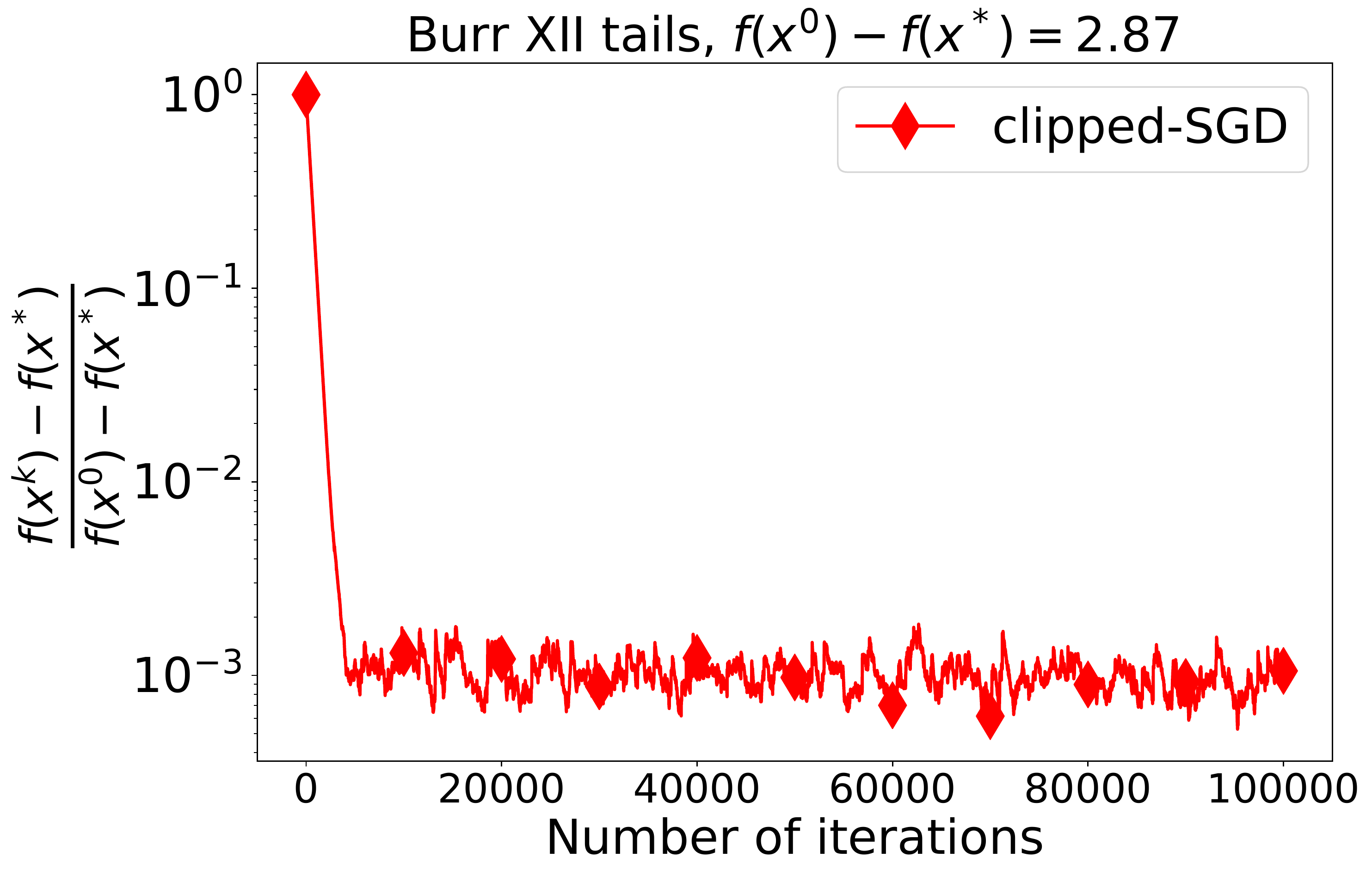}
    \caption{$2$ independent runs of {\tt SGD} (blue) and {\tt clipped-SGD} (red) applied to solve \eqref{eq:toy_problem} with $\xi$ having Gaussian (left column), Weibull (central column) and Burr Type XII (right column) tails.}
    \label{fig:tou_runs3}
\end{figure}

\begin{figure}[h]
    \centering
    \includegraphics[width=0.32\textwidth]{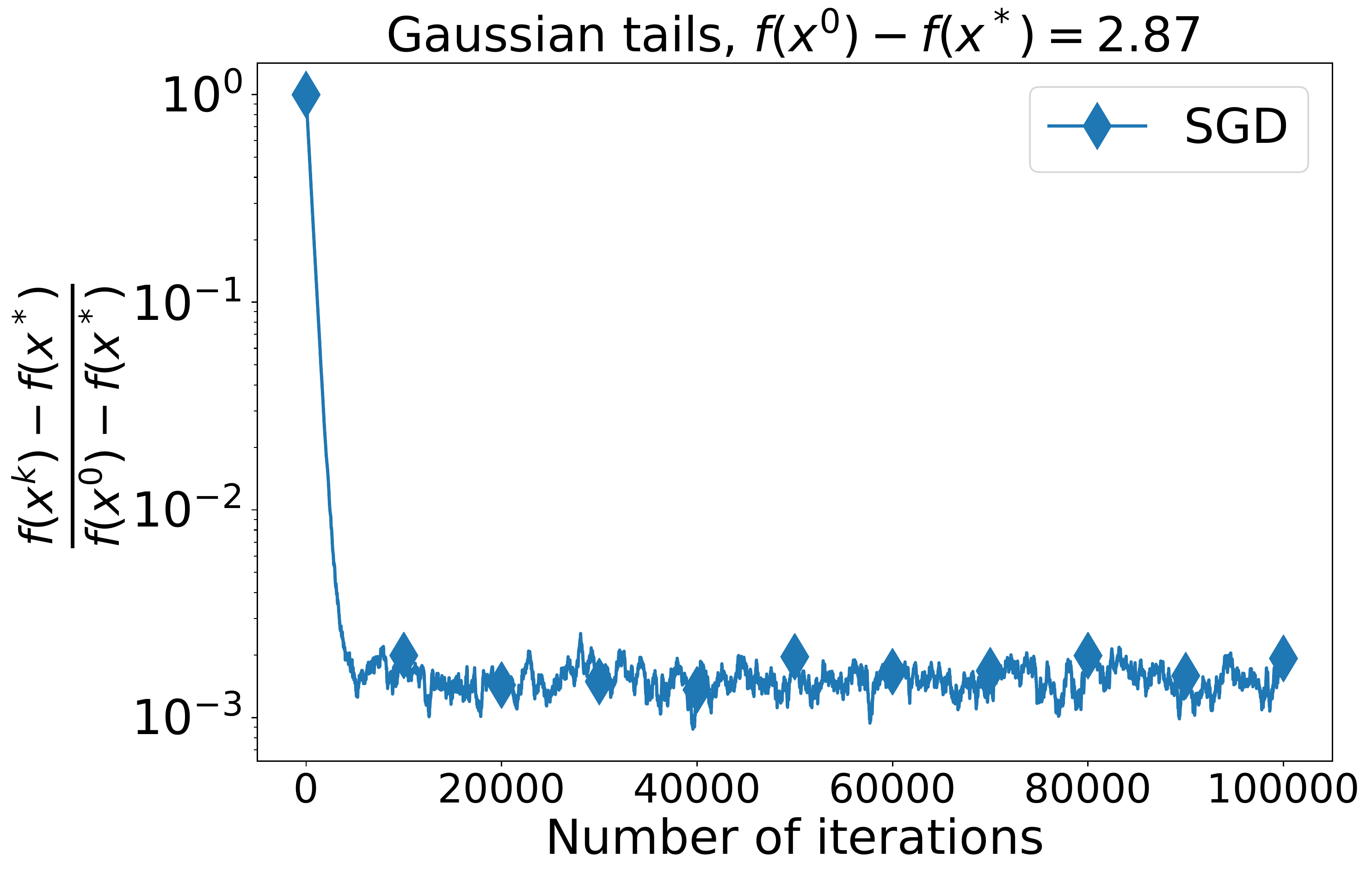}
    \includegraphics[width=0.32\textwidth]{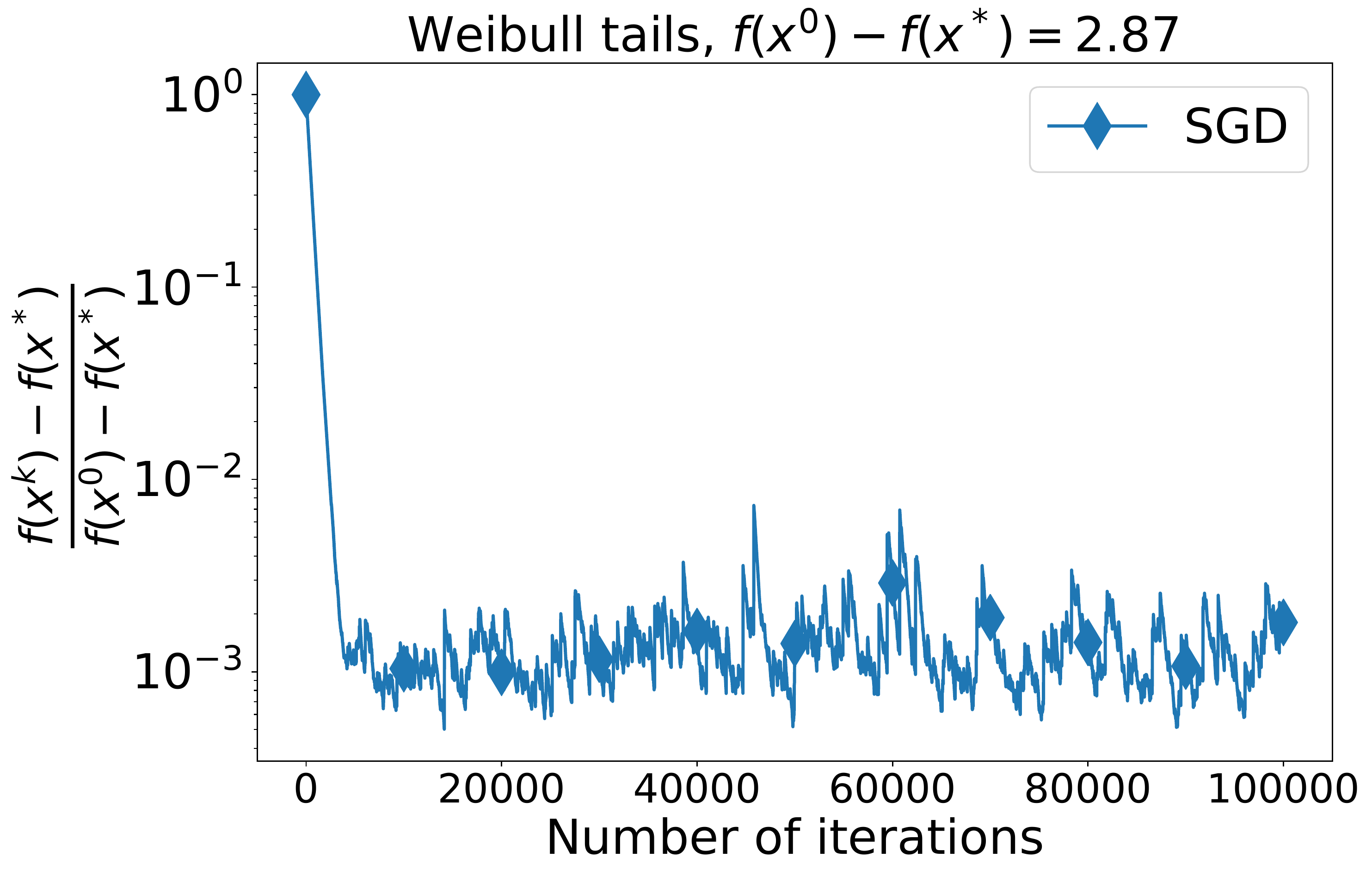}
    \includegraphics[width=0.32\textwidth]{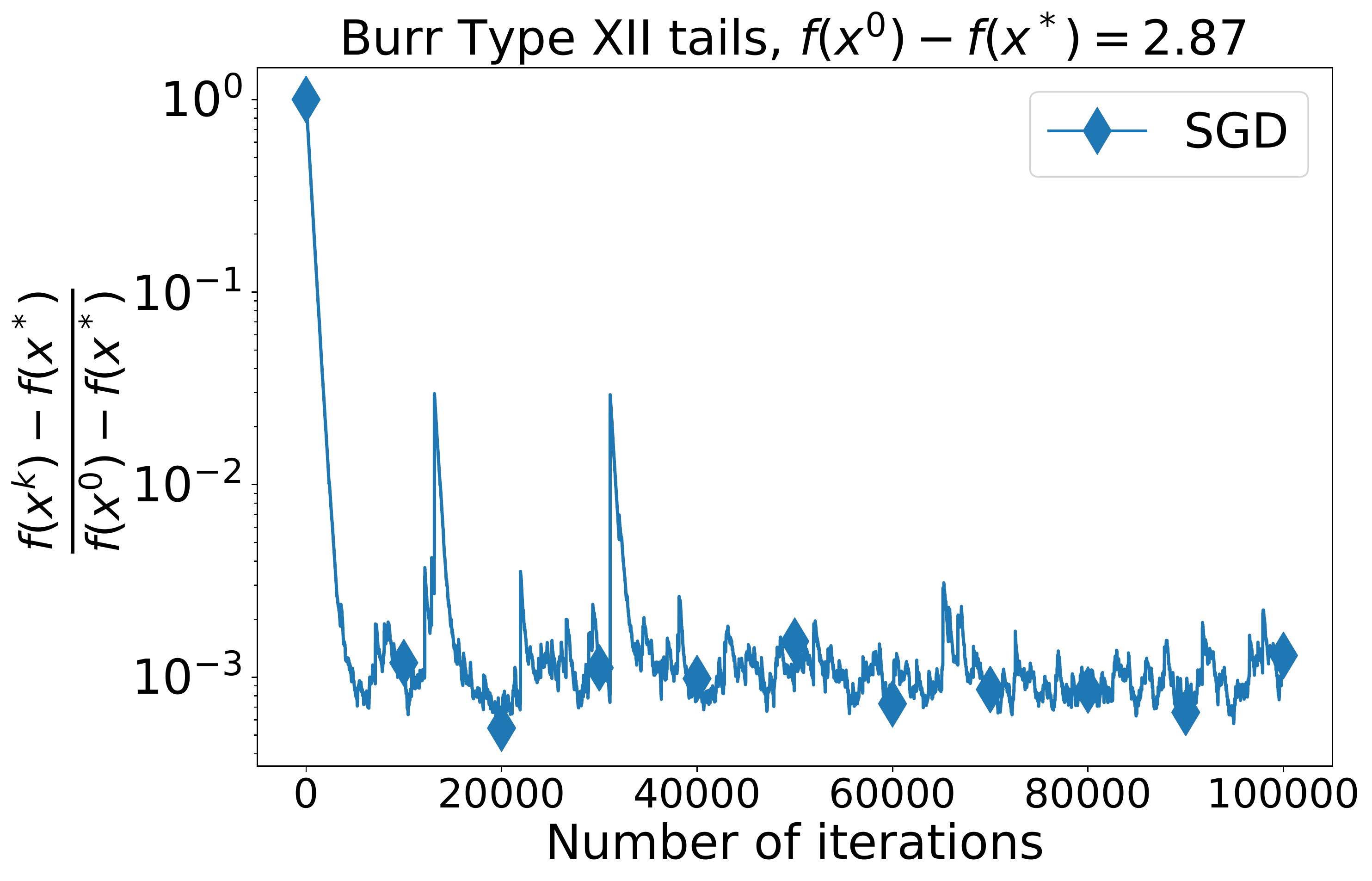}
    \includegraphics[width=0.32\textwidth]{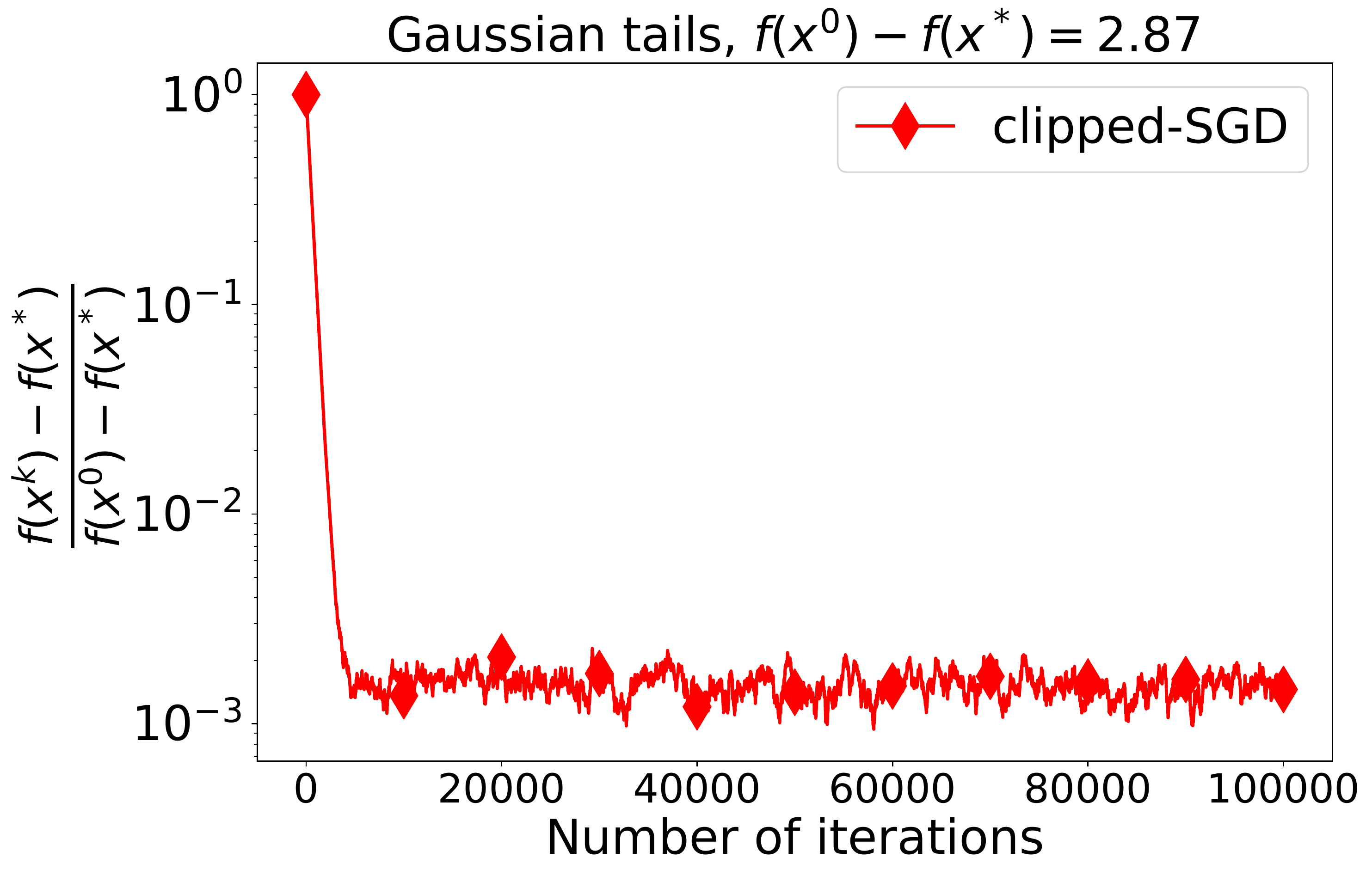}
    \includegraphics[width=0.32\textwidth]{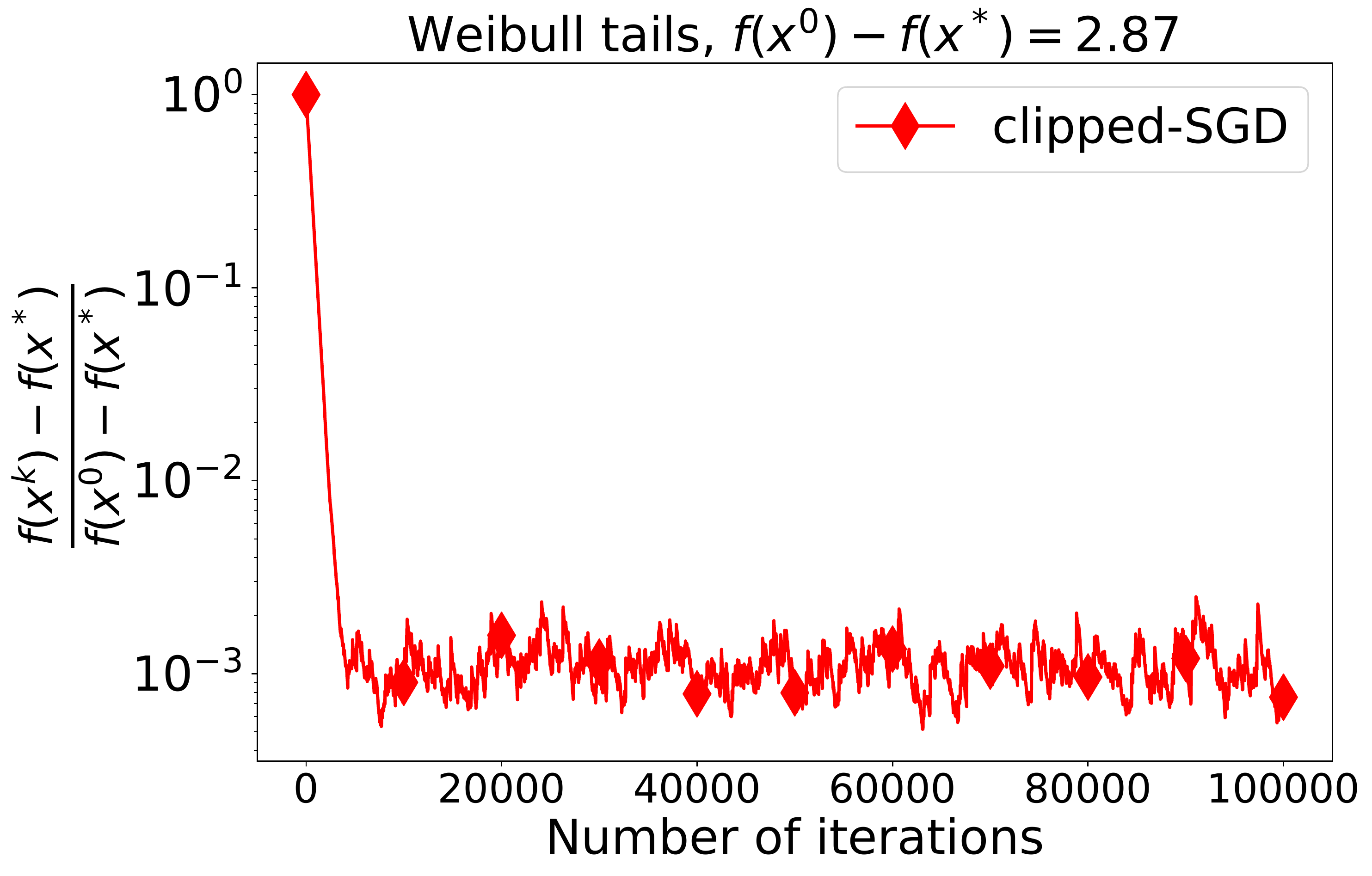}
    \includegraphics[width=0.32\textwidth]{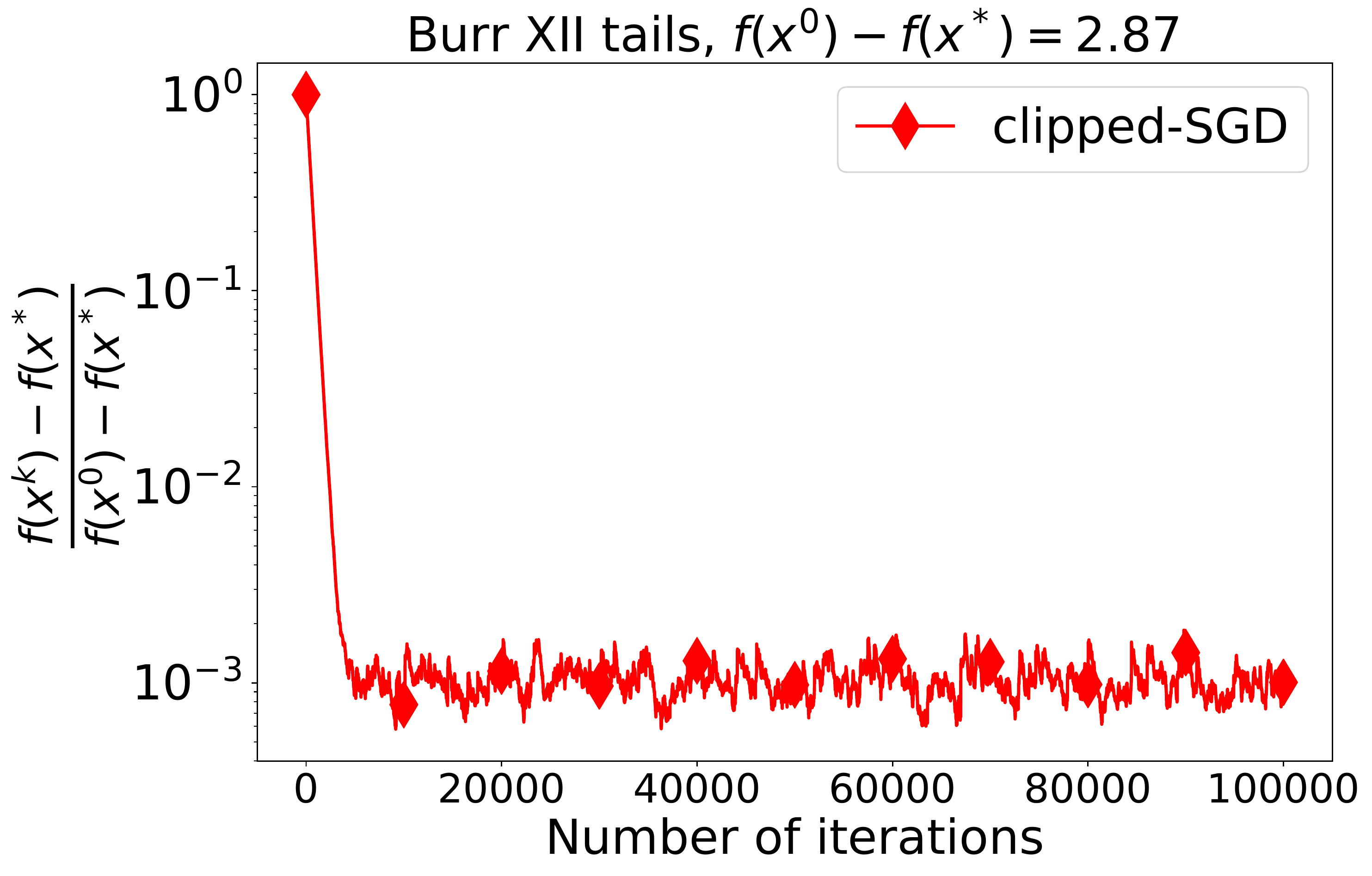}
    \includegraphics[width=0.32\textwidth]{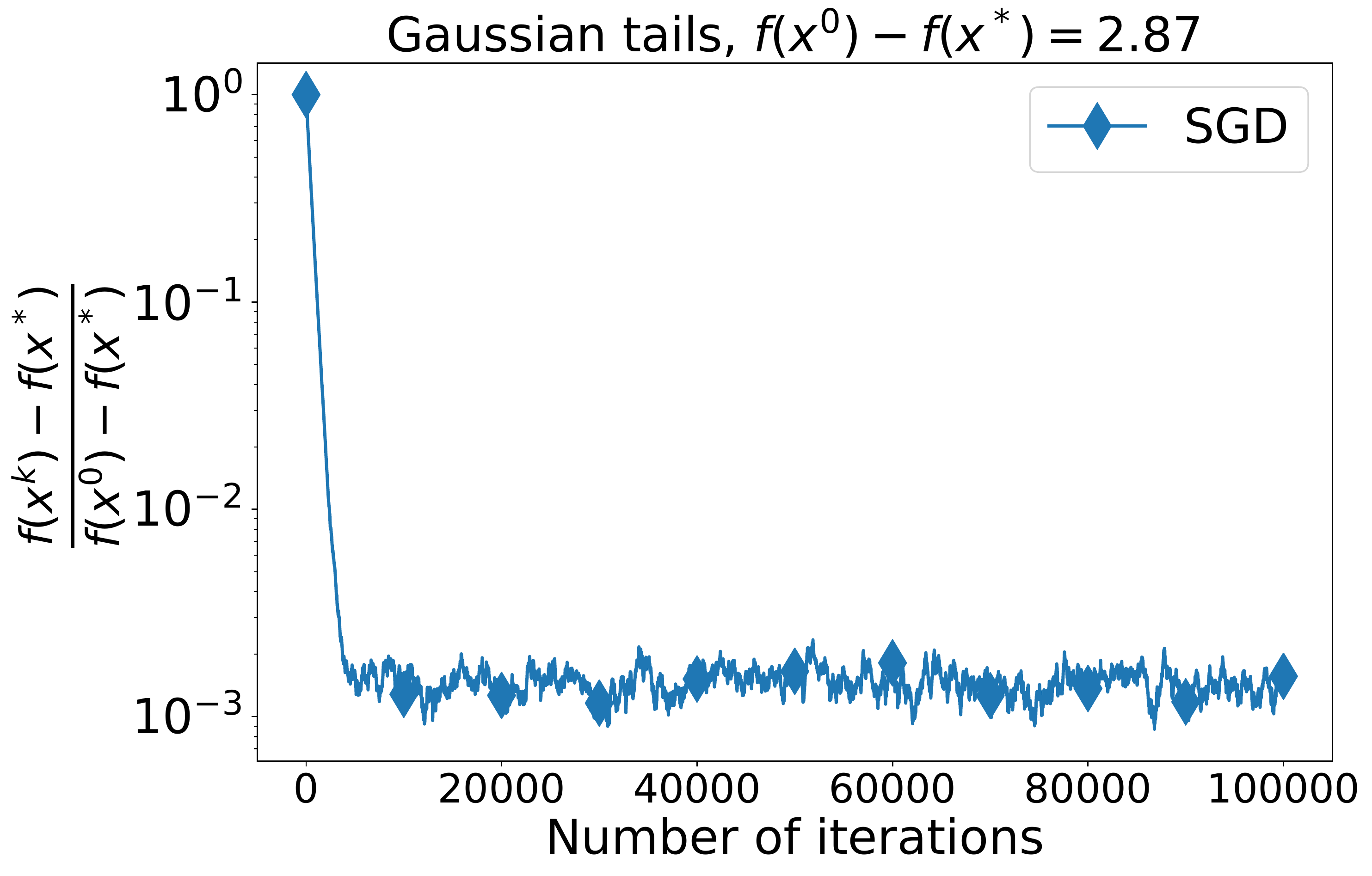}
    \includegraphics[width=0.32\textwidth]{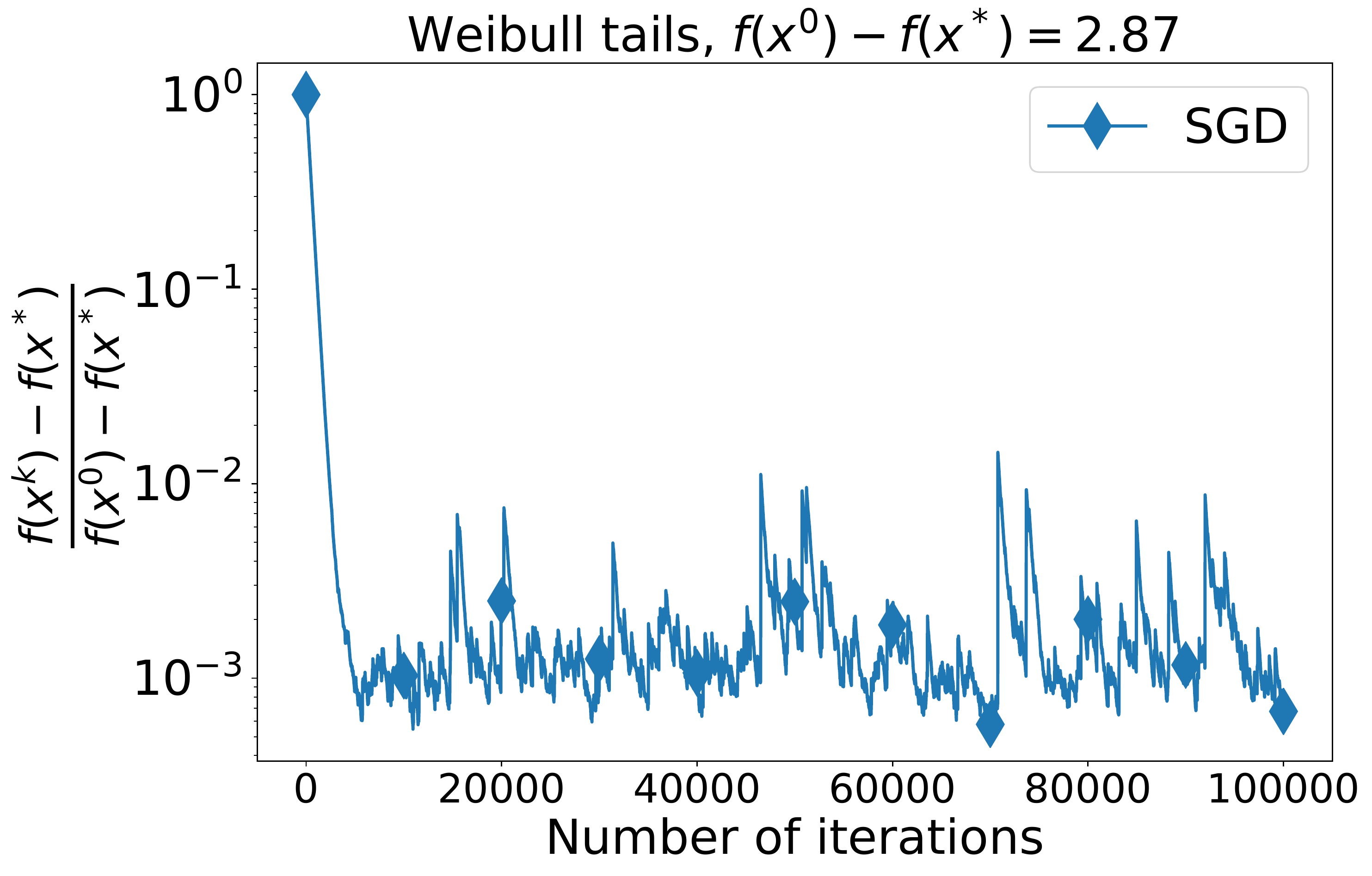}
    \includegraphics[width=0.32\textwidth]{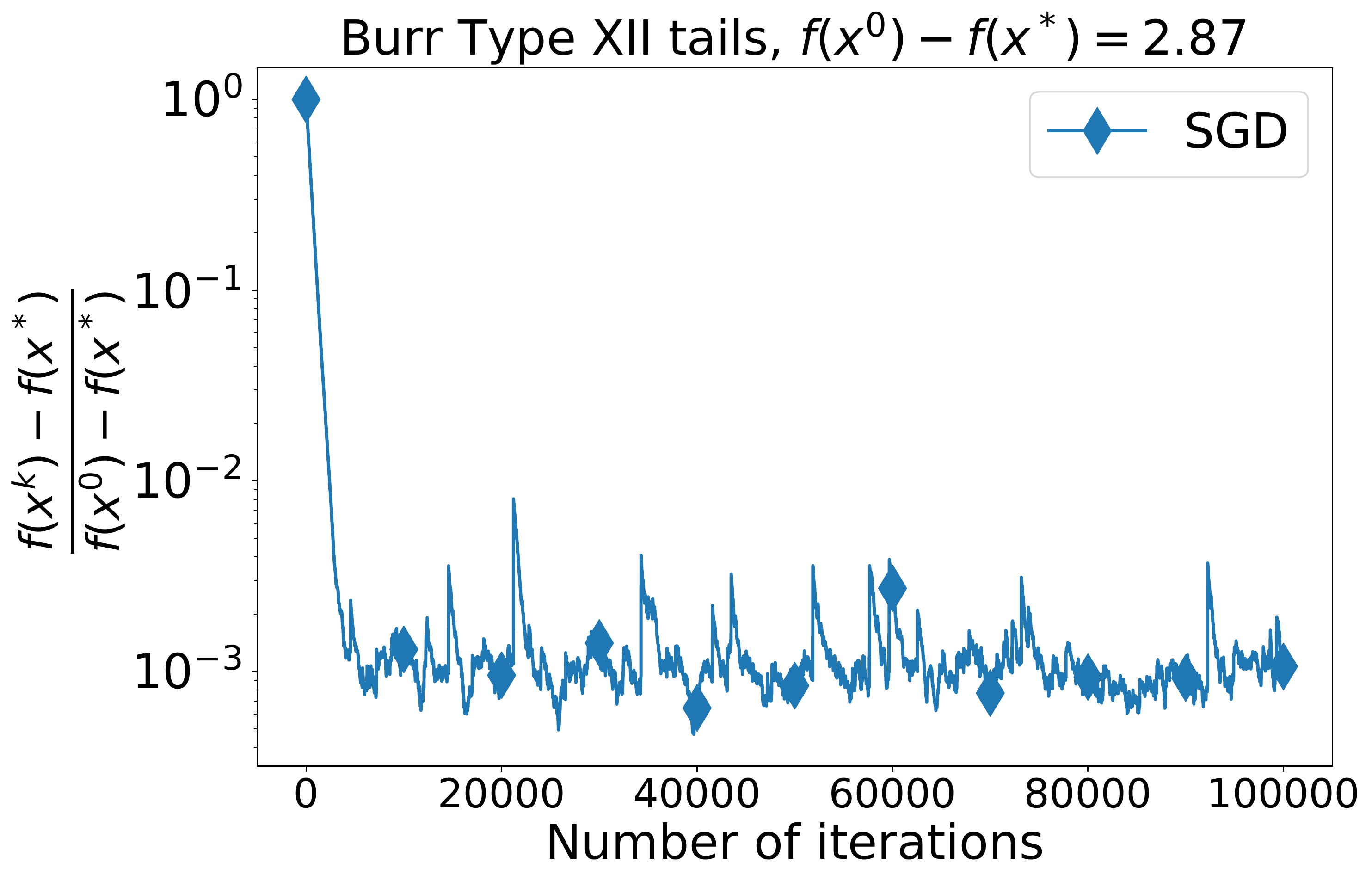}
    \includegraphics[width=0.32\textwidth]{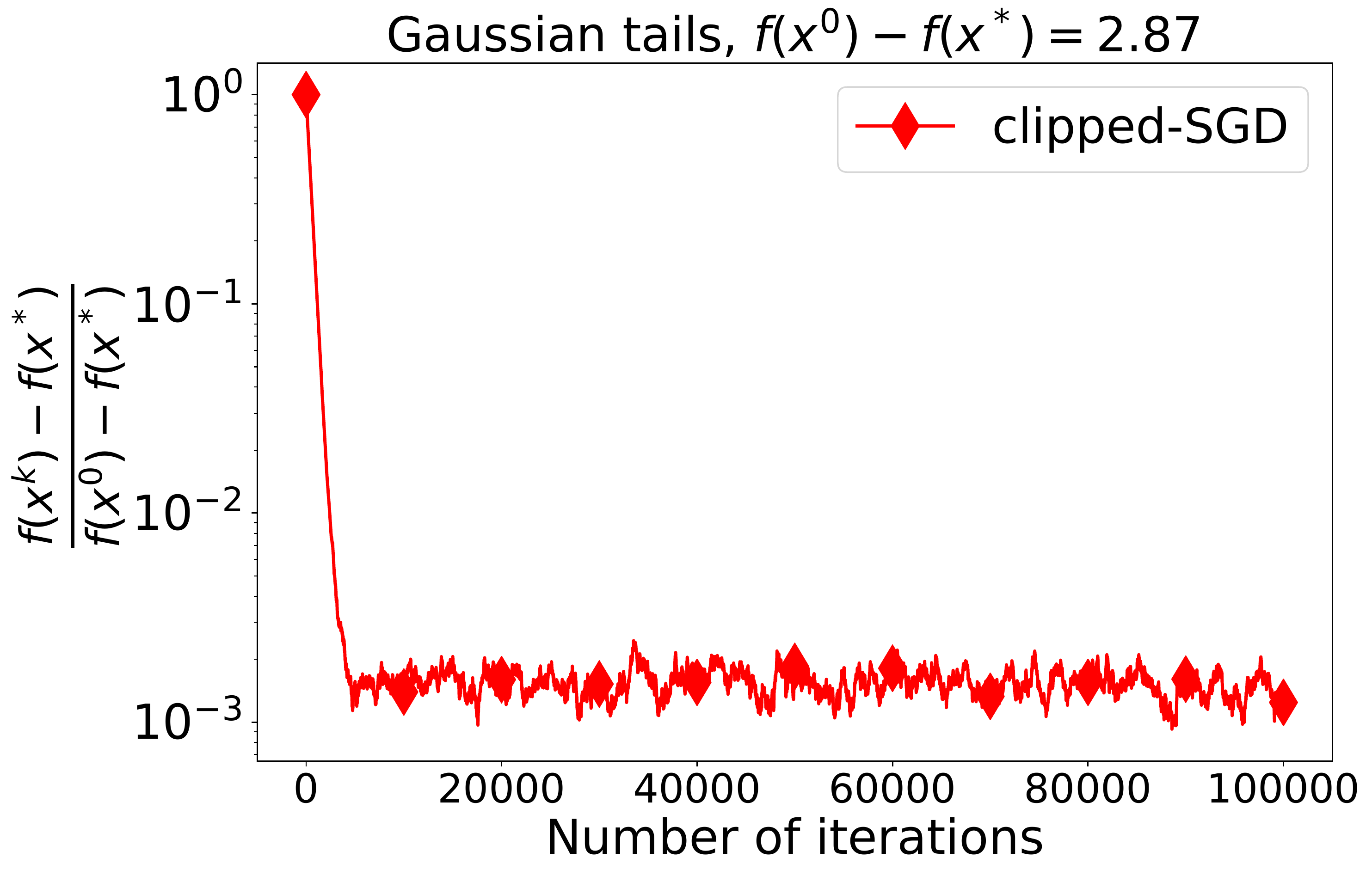}
    \includegraphics[width=0.32\textwidth]{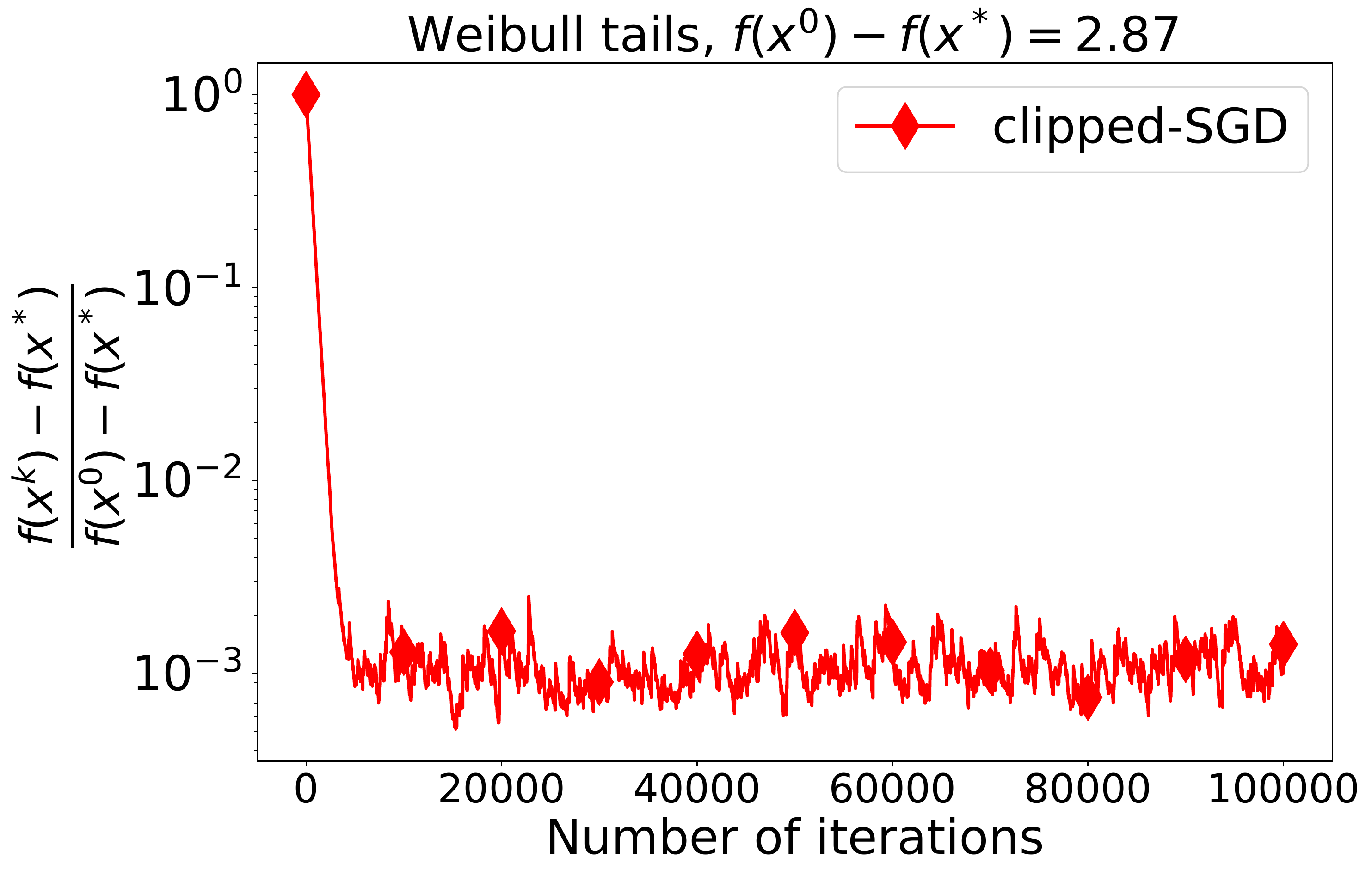}
    \includegraphics[width=0.32\textwidth]{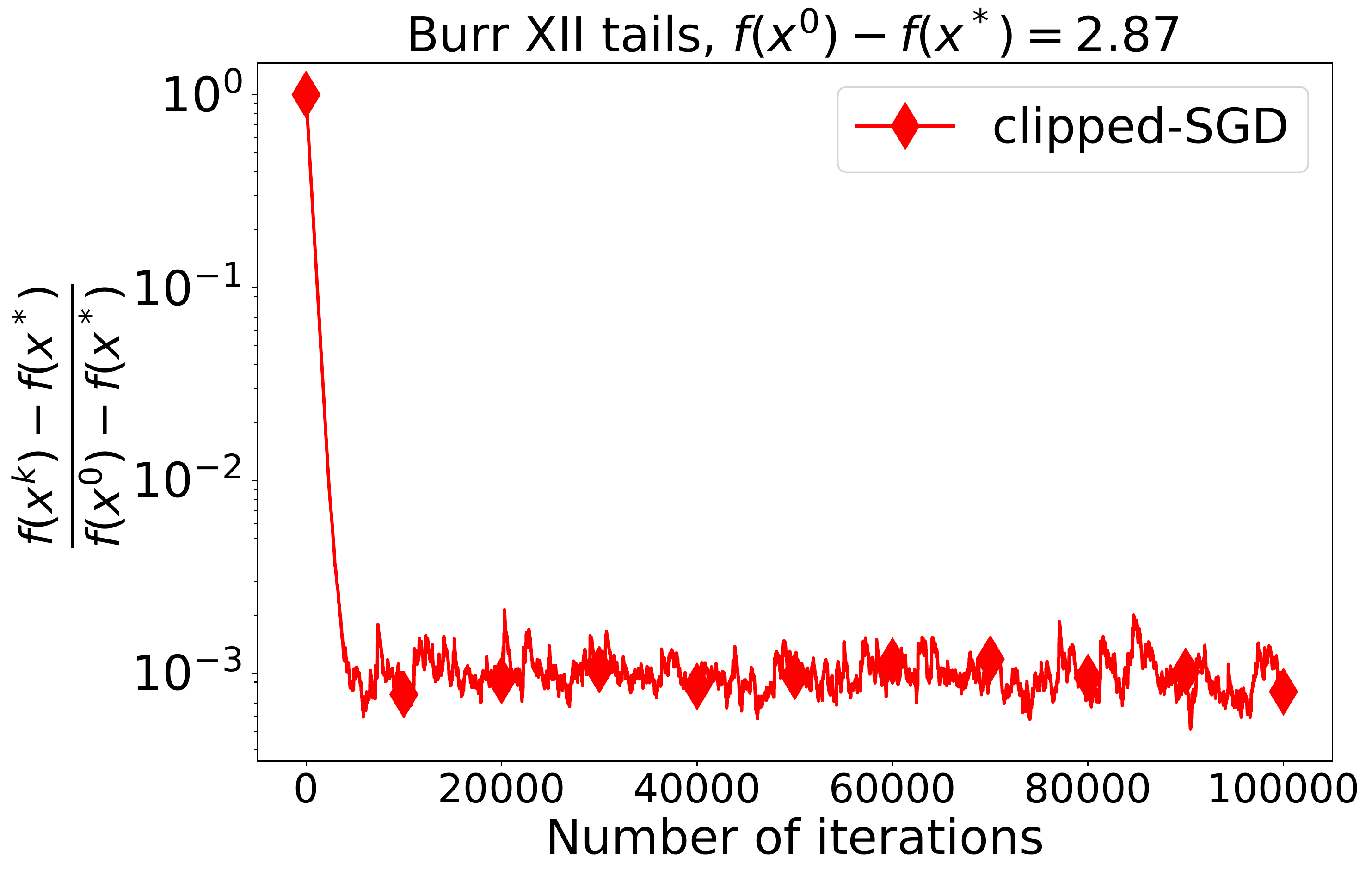}
    \caption{$2$ independent runs of {\tt SGD} (blue) and {\tt clipped-SGD} (red) applied to solve \eqref{eq:toy_problem} with $\xi$ having Gaussian (left column), Weibull (central column) and Burr Type XII (right column) tails.}
    \label{fig:tou_runs4}
\end{figure}

\begin{figure}[h]
    \centering
    \includegraphics[width=0.32\textwidth]{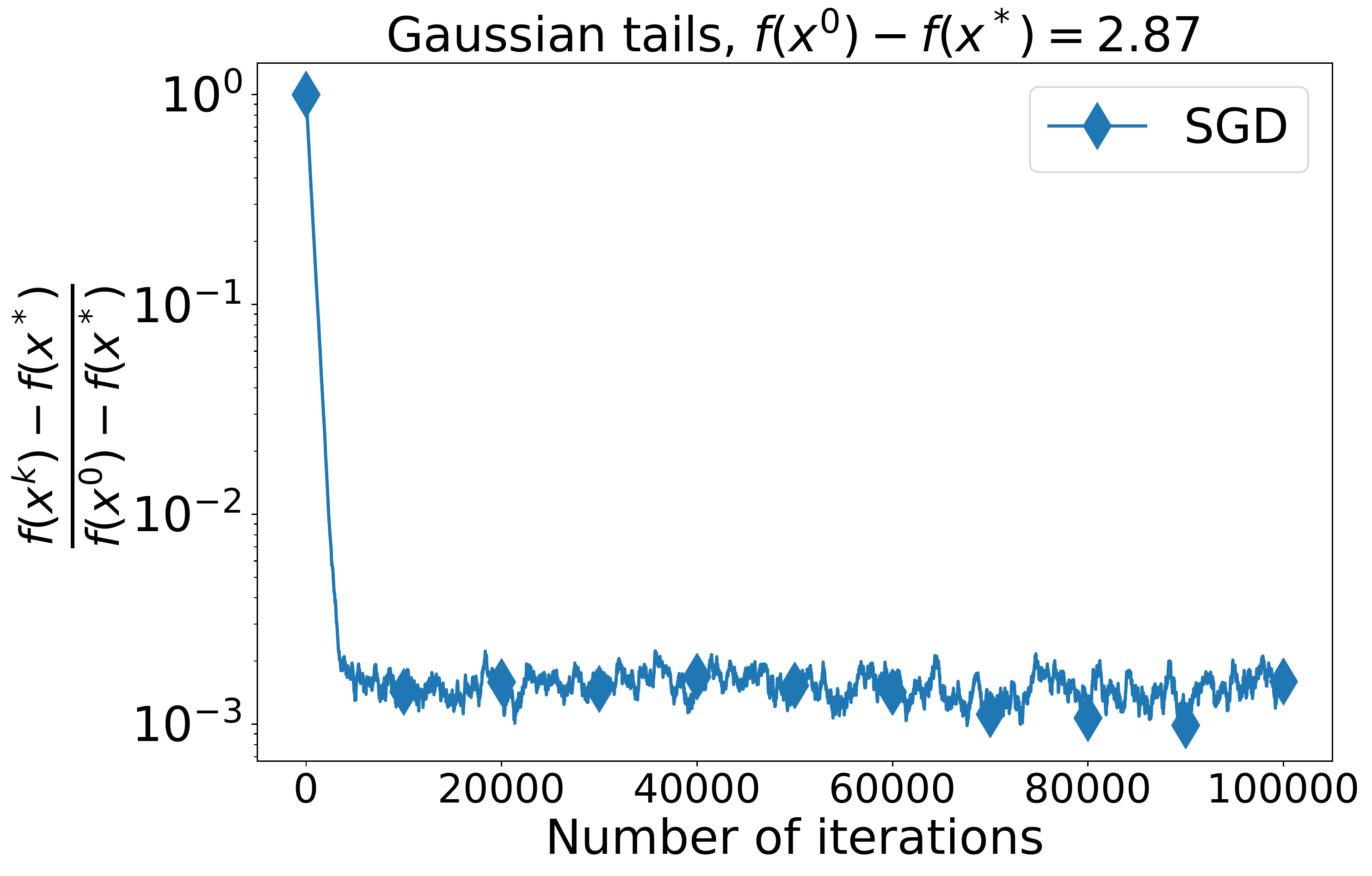}
    \includegraphics[width=0.32\textwidth]{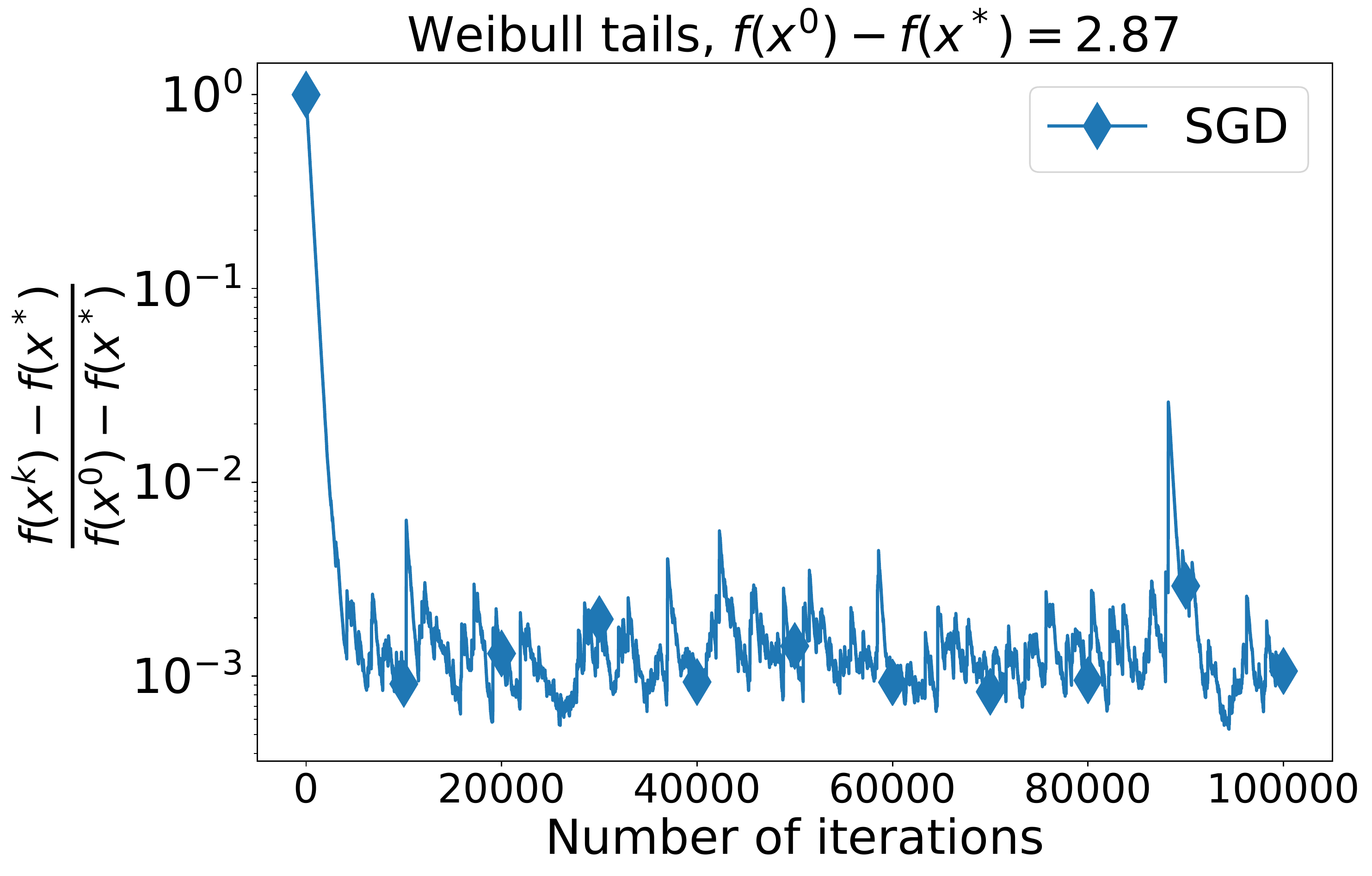}
    \includegraphics[width=0.32\textwidth]{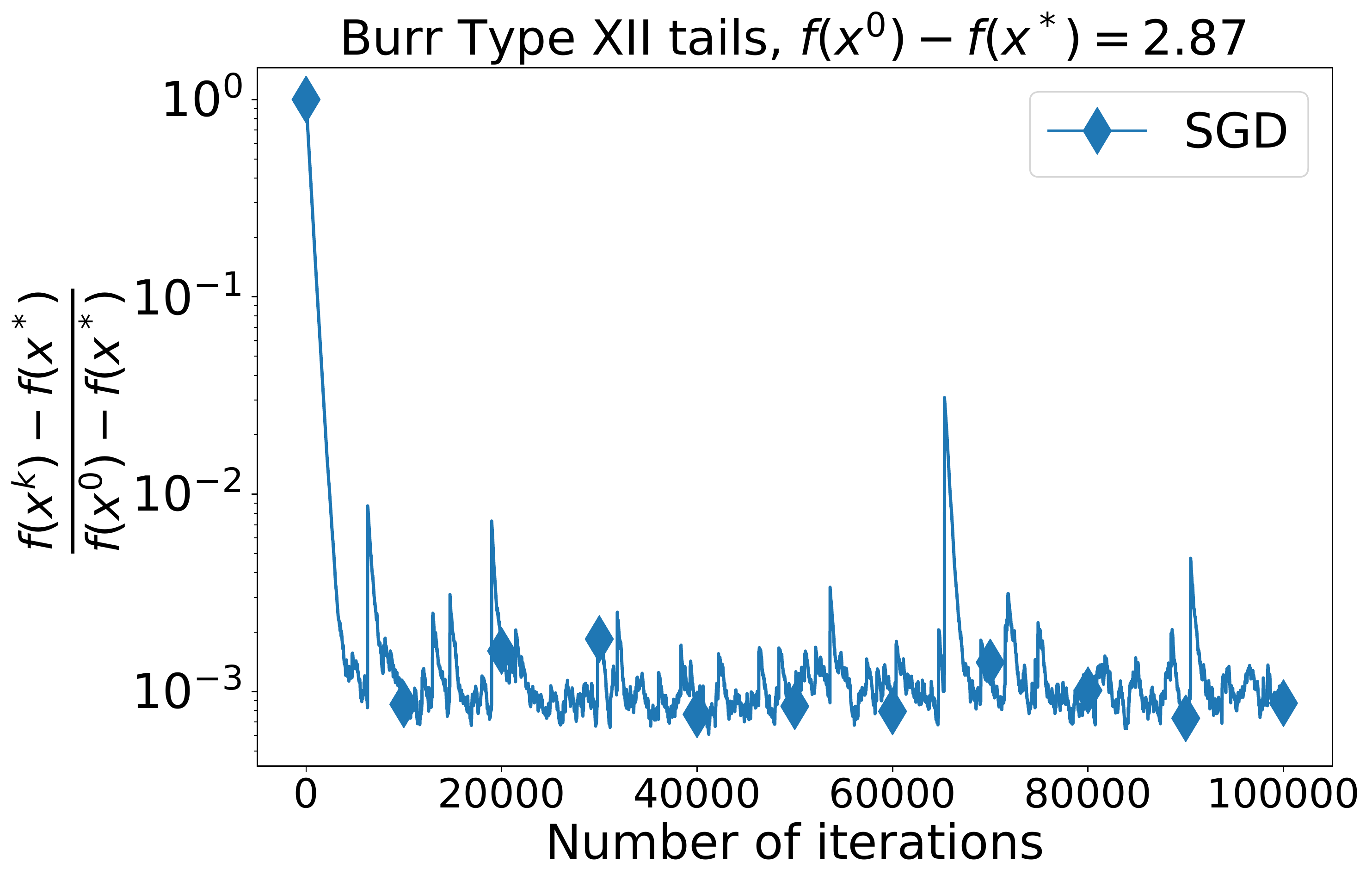}
    \includegraphics[width=0.32\textwidth]{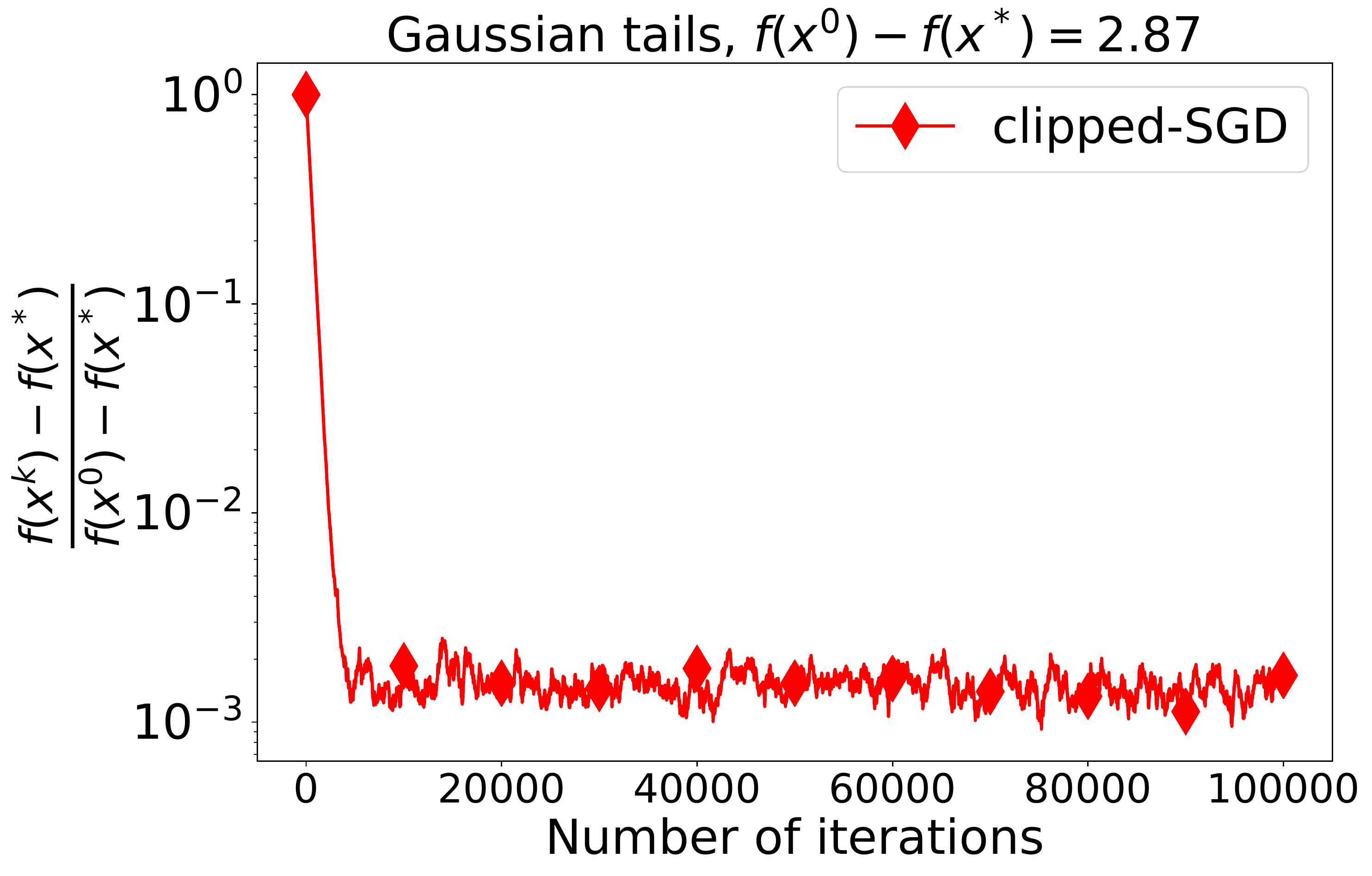}
    \includegraphics[width=0.32\textwidth]{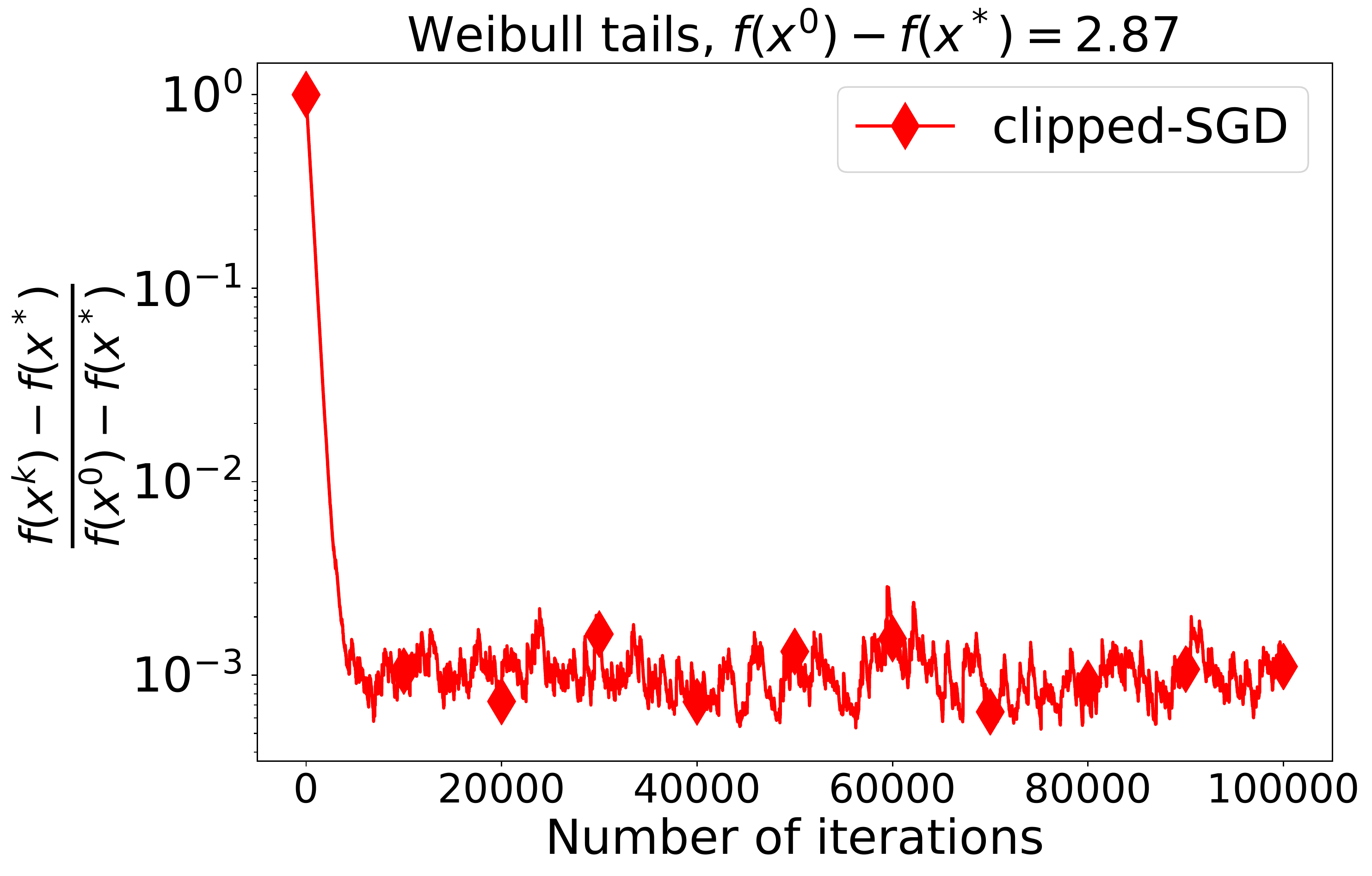}
    \includegraphics[width=0.32\textwidth]{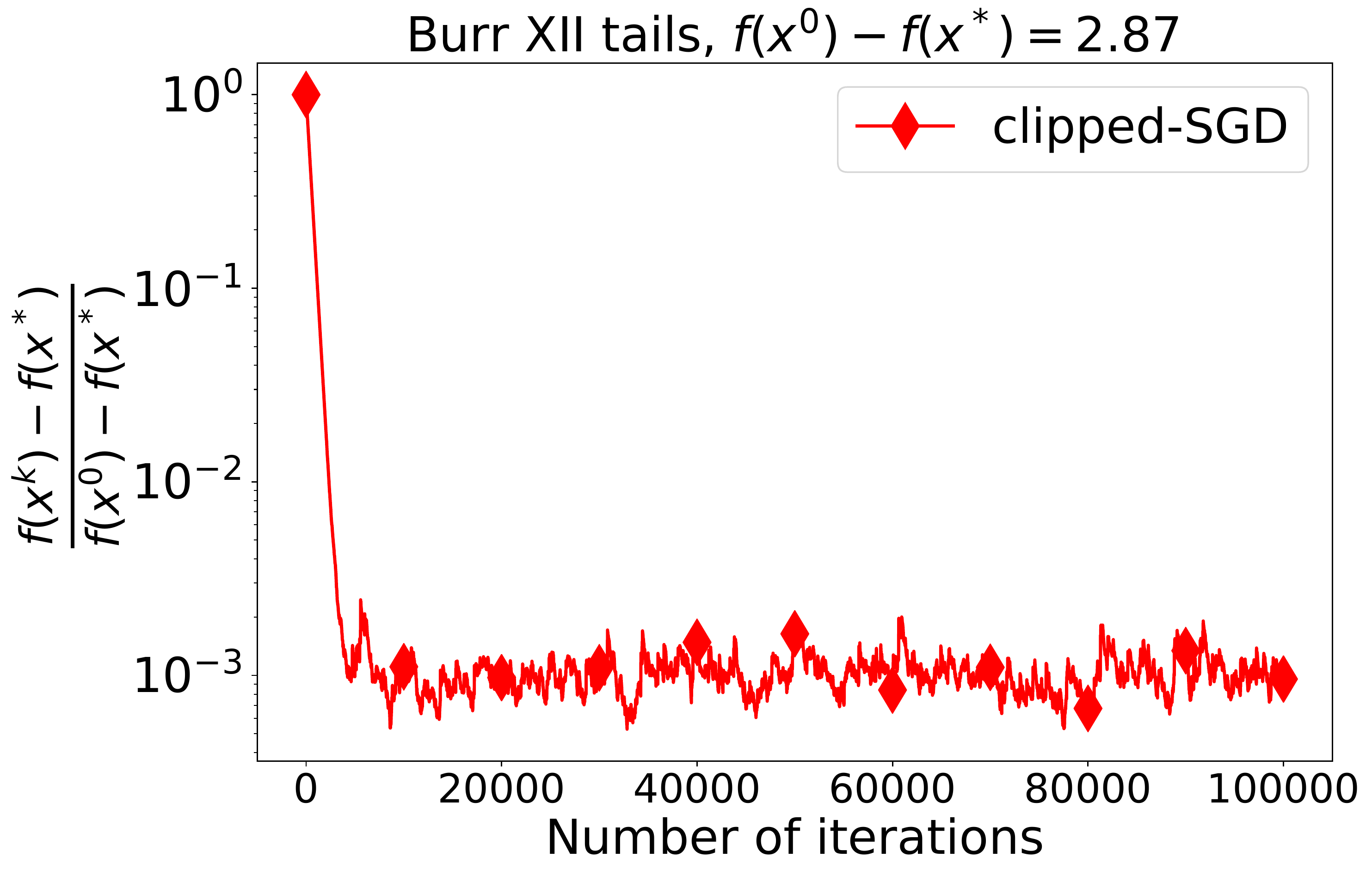}
    \includegraphics[width=0.32\textwidth]{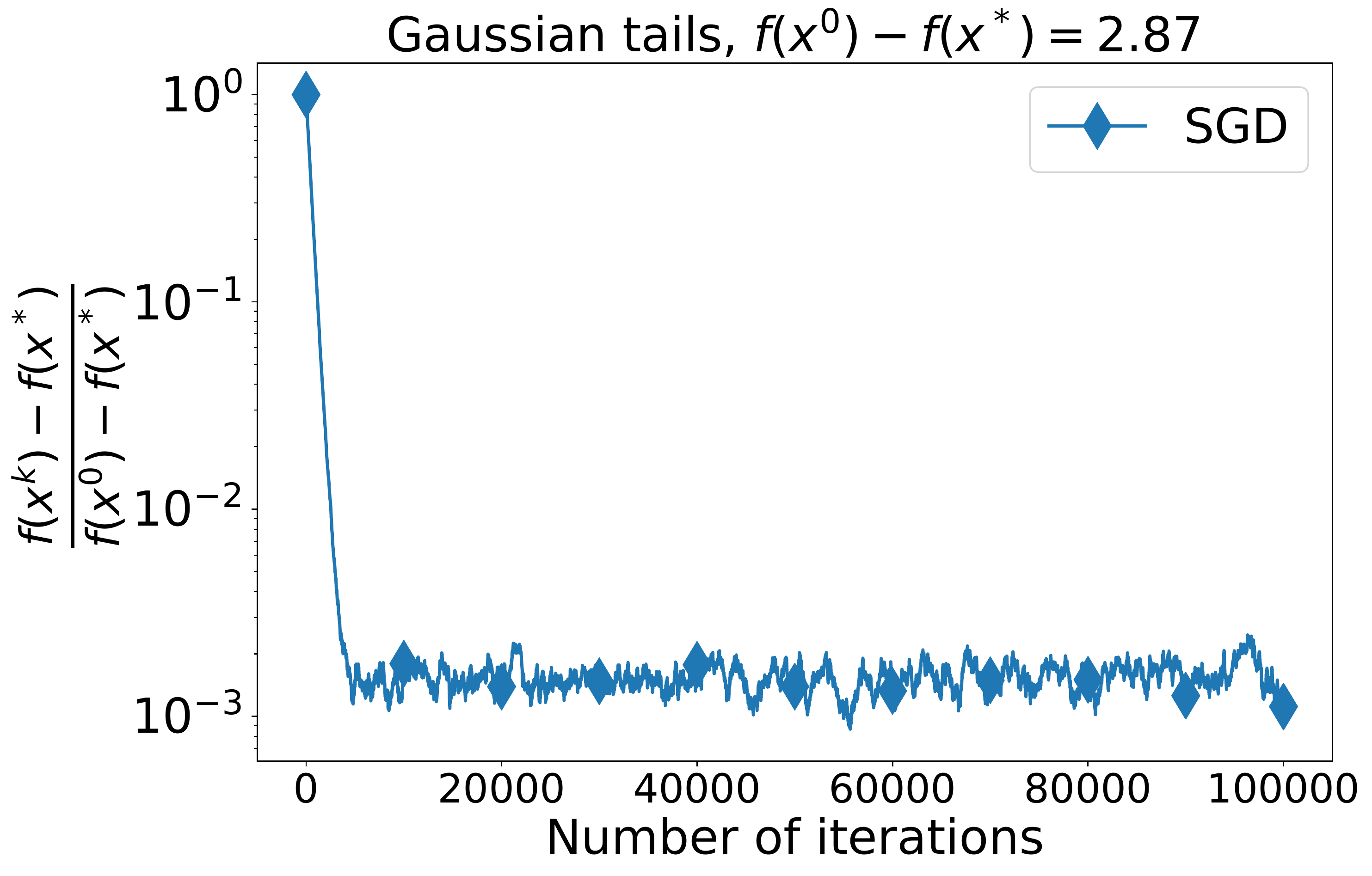}
    \includegraphics[width=0.32\textwidth]{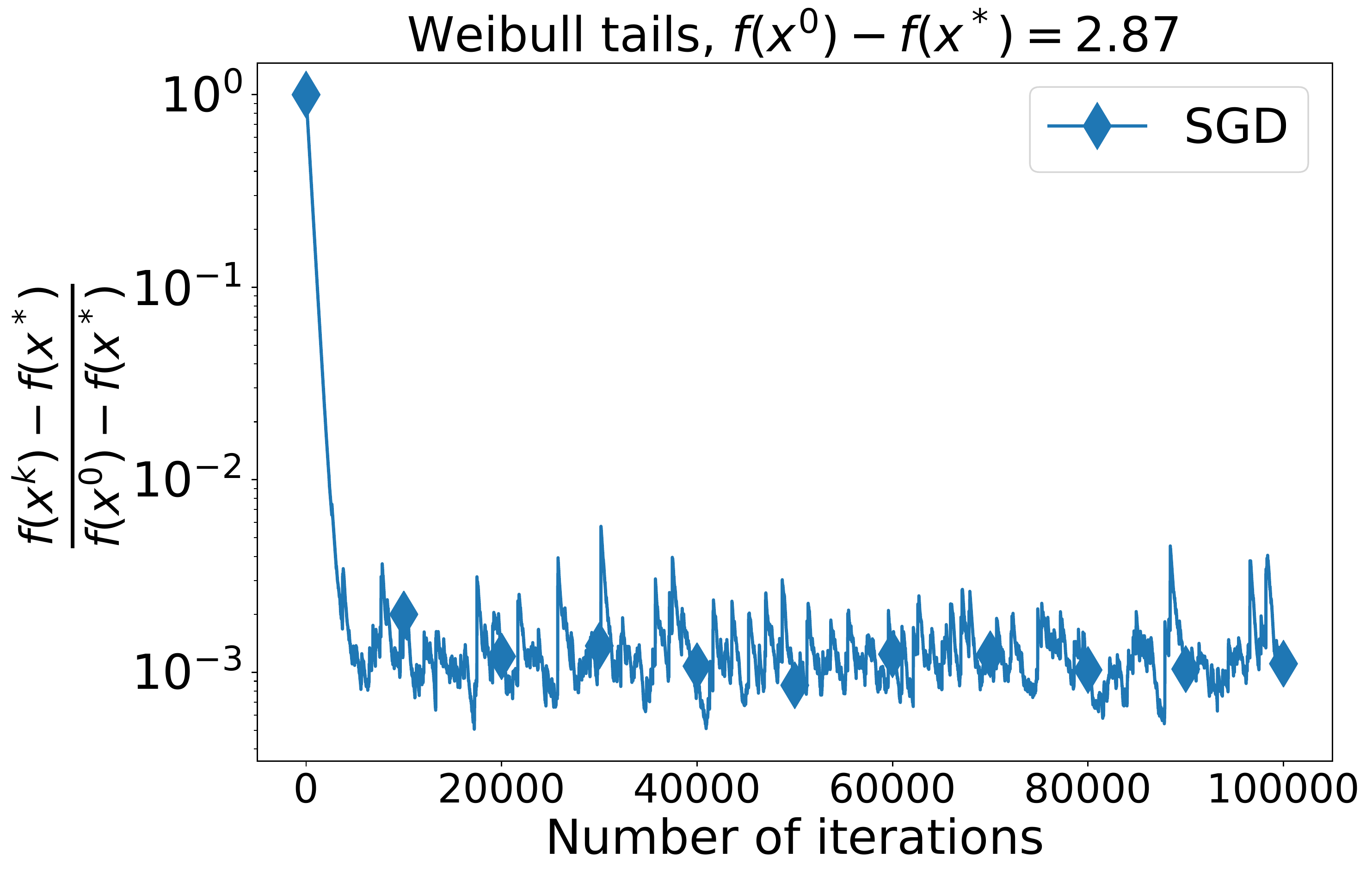}
    \includegraphics[width=0.32\textwidth]{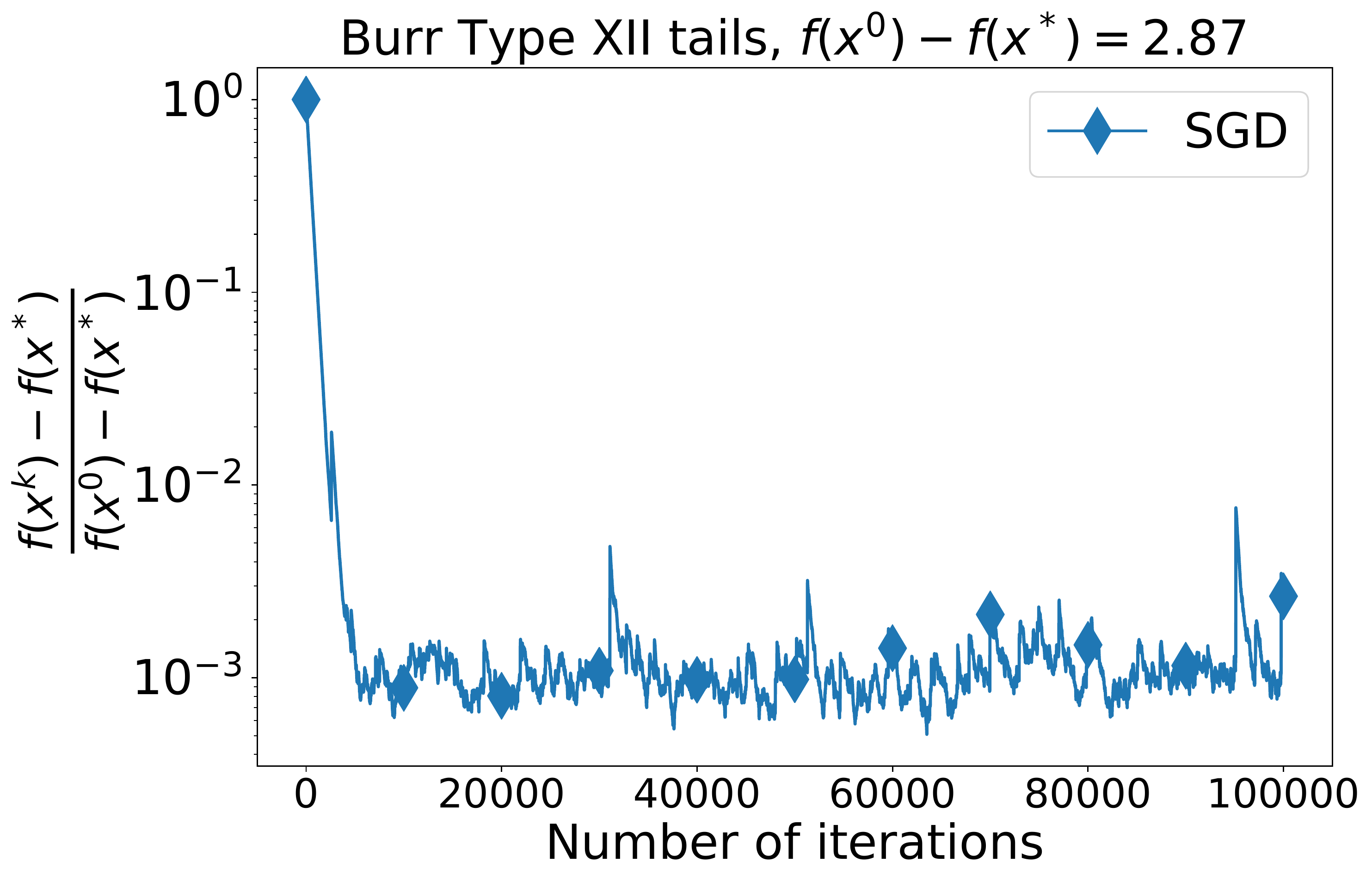}
    \includegraphics[width=0.32\textwidth]{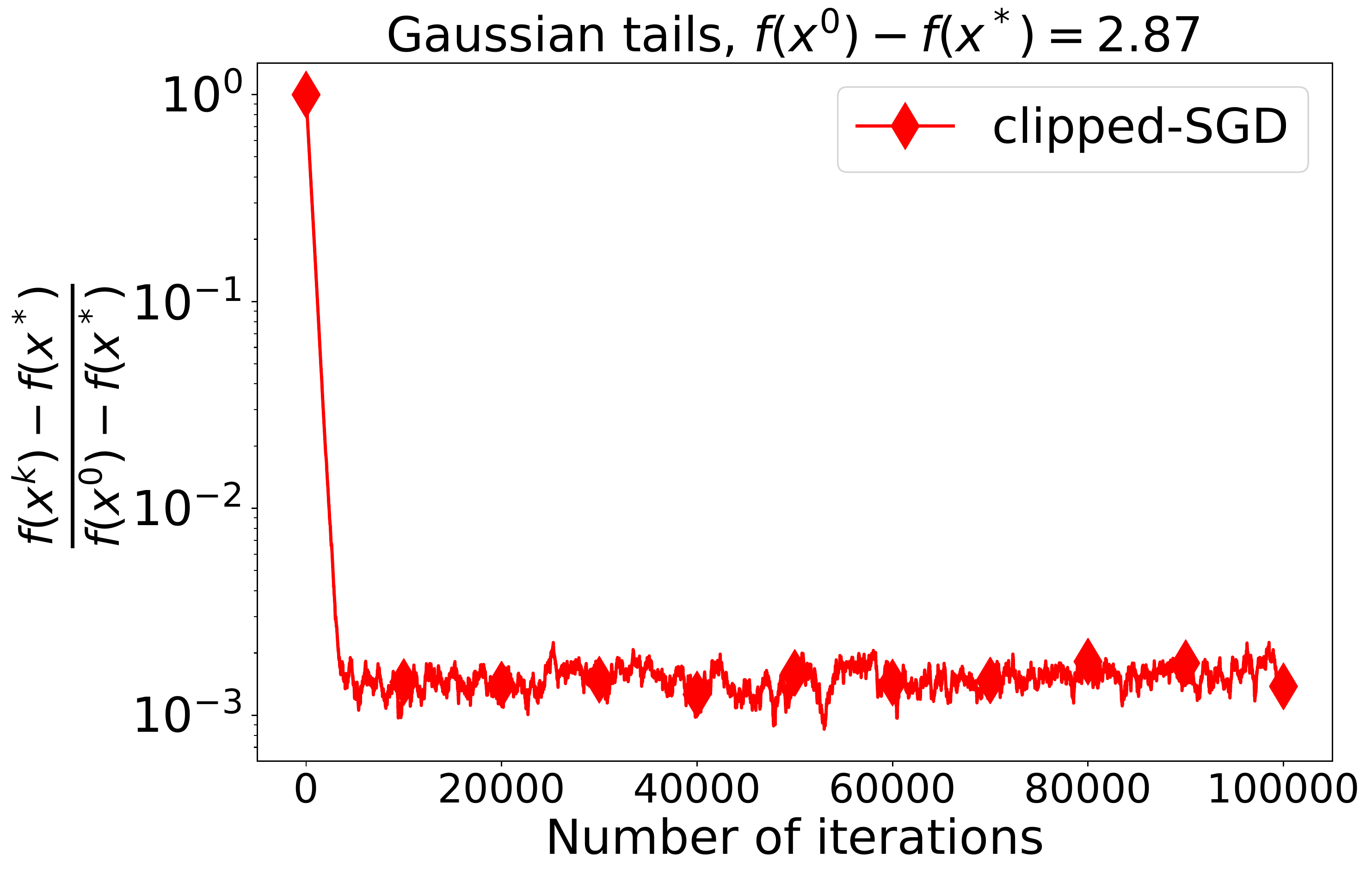}
    \includegraphics[width=0.32\textwidth]{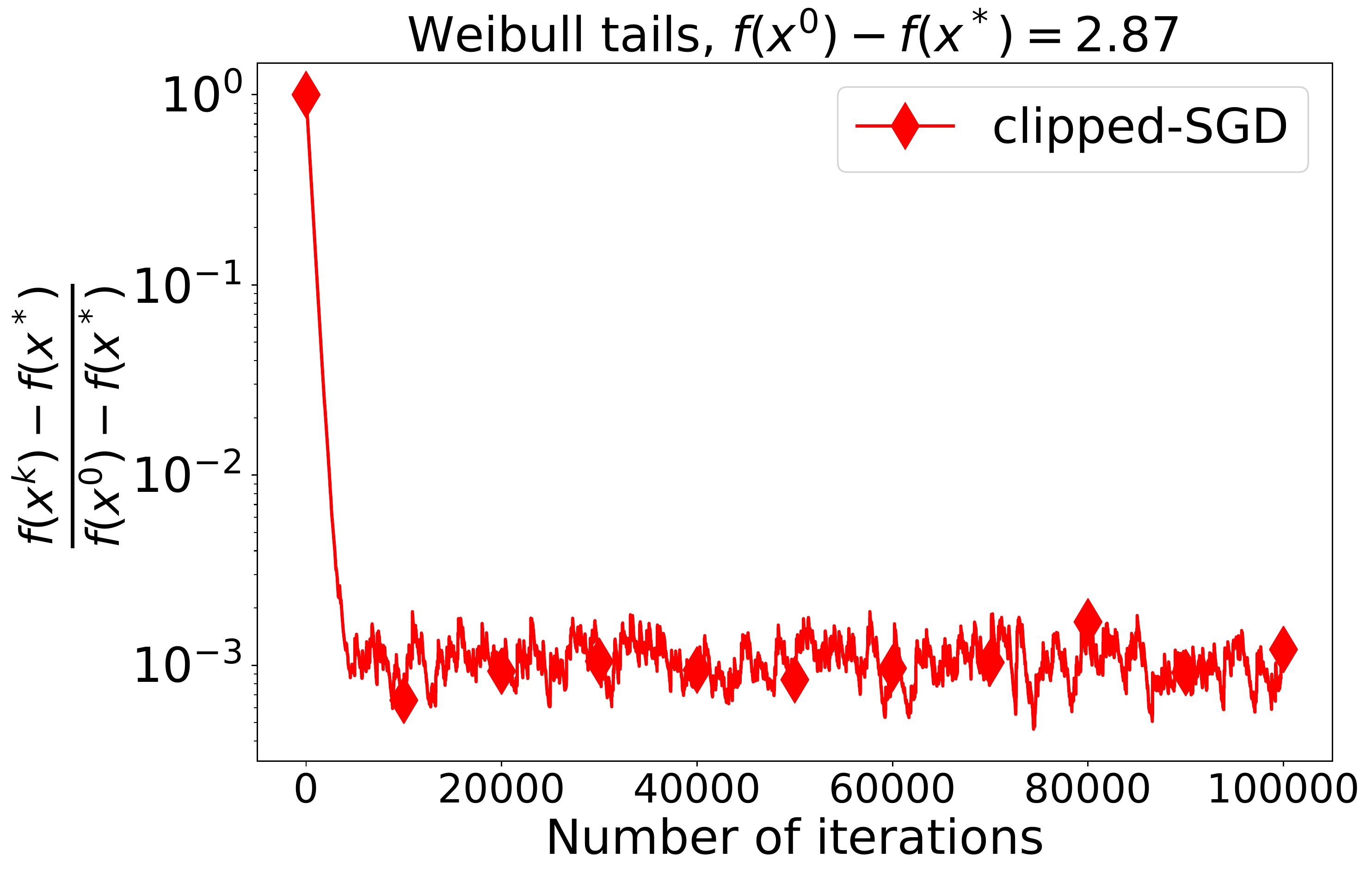}
    \includegraphics[width=0.32\textwidth]{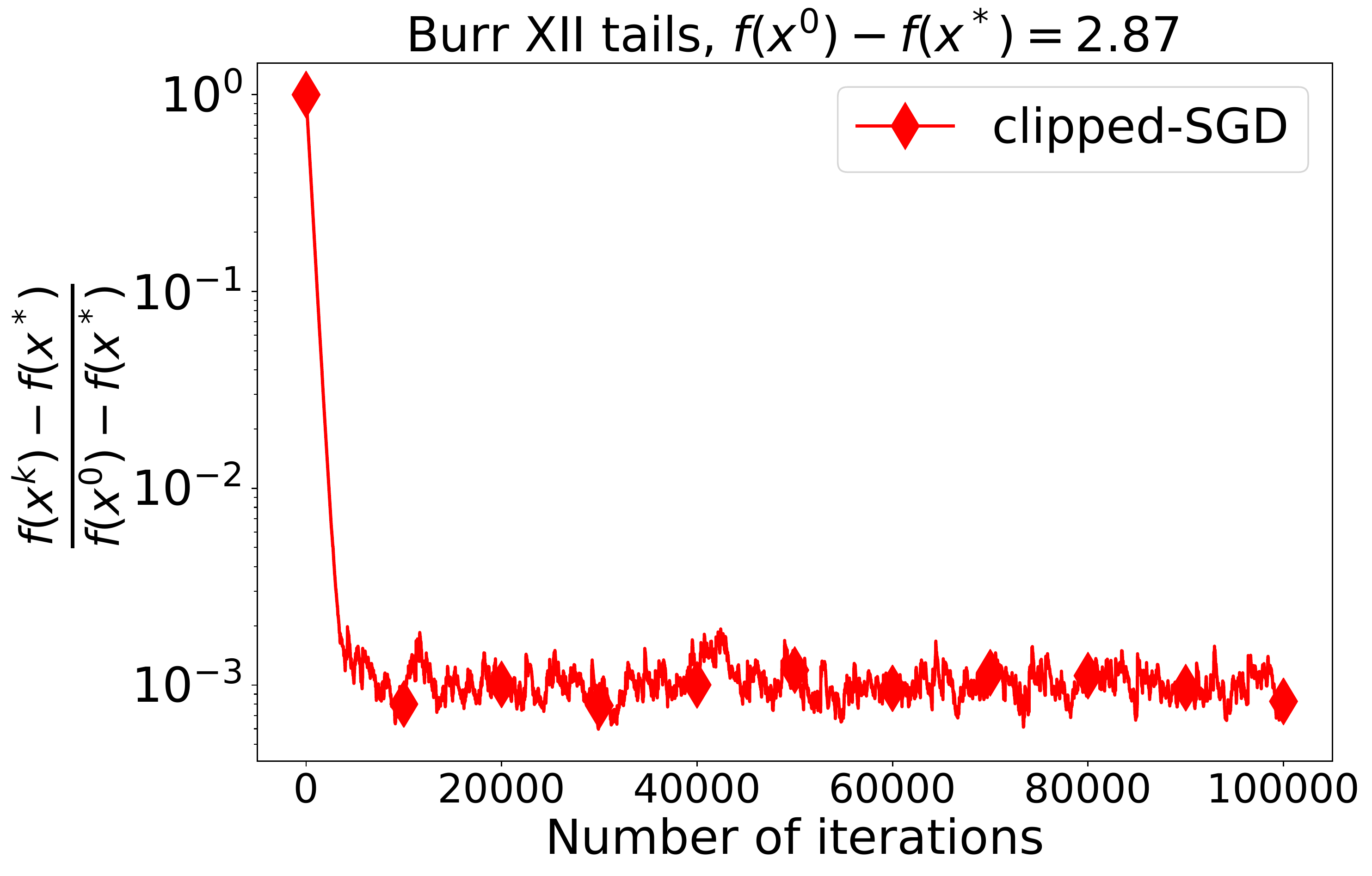}
    \caption{$2$ independent runs of {\tt SGD} (blue) and {\tt clipped-SGD} (red) applied to solve \eqref{eq:toy_problem} with $\xi$ having Gaussian (left column), Weibull (central column) and Burr Type XII (right column) tails.}
    \label{fig:tou_runs5}
\end{figure}

\subsection{Additional Details and Experiments with Logistic Regression}\label{sec:additional_logreg}
In this section, we provide additional details of the experiments presented in Section~\ref{sec:numerical_exp} together with extra numerical results. In particular, we consider the logistic regression problem:
\begin{equation}
    \min\limits_{x\in\R^n} f(x) = \frac{1}{r}\sum\limits_{i=1}^r \underbrace{\log\left(1 + \exp\left(-y_i\cdot (Ax)_i\right)\right)}_{f_i(x)} \label{eq:log_regr}
\end{equation}
where $A\in\R^{r\times n}$ is matrix of instances and $y\in\{0,1\}^r$ is vector of labels. It is well-known that $f(x)$ from \eqref{eq:log_regr} is convex and $L$-smooth with $L = \nicefrac{\lambda_{\max}(A^\top A)}{4r}$ where $\lambda_{\max}(A^\top A)$ denotes the maximal eigenvalue of $A^\top A$. One can consider problem \eqref{eq:log_regr} as a special case of \eqref{eq:main_problem} where $\xi$ is a random index uniformly distributed on $\{1,\ldots,r\}$ and $f(x,\xi) = f_{\xi}(x)$. We take the datasets from LIBSVM library \cite{chang2011libsvm}: see Table~\ref{tab:logreg_datasets} with the summary of the datasets we used.
\begin{table}[h]
    \centering
    \caption{Summary of used datasets.}
    \label{tab:logreg_datasets}
    \begin{tabular}{|c|c|c|c|c|c|}
        \hline
         & {\tt heart} & {\tt diabetes} & {\tt australian} & {\tt a9a} & {\tt w8a} \\
        \hline
        Size & $270$ & $768$ & $690$ & $32561$ & $49749$ \\
        \hline
        Dimension & $13$ & $8$ & $13$ & $123$ & $300$ \\
        \hline
    \end{tabular}
\end{table}

We notice that in all experiments that we did with logistic regression the initial suboptimality $f(x^0)-f(x^*)$ was of order $10$. Moreover, as it was mentioned in the main part of the paper the parameters for the methods were tuned. One can find parameters that we used in the experiments from Section~\ref{sec:numerical_exp} in Table~\ref{tab:logreg_params}.
\begin{table}[h]
    \centering
    \scriptsize
     \caption{Parameters that are used to produce plots presented in Figures~\ref{fig:heart_logreg}-\ref{fig:australian_logreg}. In the first contains the name of the dataset and the batchsize $m$ that was used for all methods tested on the dataset. For {\tt d-clipped-SGD} $\lambda_0$ is an initial clipping level, $l$ is a period (in terms of epochs) of decreasing the clipping level and $\alpha$ is a coefficient of decrease, i.e.\ every $l$ epochs the clipping level is multiplied by $\alpha$. For {\tt SSTM} parameter $a$ was picked the same as for {\tt clipped-SSTM} in order to emphasize the effect of clipping.}
    \label{tab:logreg_params}
    \begin{tabular}{|c|c|c|c|c|c|}
        \hline
         & {\tt SGD} & {\tt clipped-SGD} & {\tt d-clipped-SGD} & {\tt SSTM} & {\tt clipped-SSTM} \\
        \hline
        \makecell{{\tt heart}\\ $m = 20$} & $\gamma = \frac{1}{2L}$ & $\gamma = \frac{1}{2L}$, $\lambda = 2.72$ & \makecell{$\gamma = \frac{1}{2L}$, $\lambda_0 = 2.72$,\\ $l=10^3$, $\alpha = 0.9$} & $a = 10^4$ & \makecell{$a = 10^4$,\\ $B = 2\cdot 10^{-4}$}\\
        \hline
        \makecell{{\tt diabetes}\\ $m = 100$} & $\gamma = \frac{1}{10L}$ & $\gamma = \frac{1}{10L}$, $\lambda = 68.86$ & \makecell{$\gamma = \frac{1}{10L}$, $\lambda_0 = 68.86$,\\ $l=10^3$, $\alpha = 0.7$} & $a = 5\cdot10^3$ & \makecell{$a = 5\cdot10^3$,\\ $B = 7\cdot 10^{-4}$} \\
        \hline
        \makecell{{\tt australian}\\ $m = 50$} & $\gamma = \frac{1}{L}$ & $\gamma = \frac{1}{L}$, $\lambda = 74.47$ & \makecell{$\gamma = \frac{1}{L}$, $\lambda_0 = 74.47$,\\$l=1000$, $\alpha = 0.9$} & $a = 10^3$ & \makecell{$a = 5\cdot10^3$,\\ $B = 2\cdot 10^{-4}$} \\
        \hline
        \makecell{{\tt a9a}\\ $m = 100$} & $\gamma = \frac{1}{2L}$ & $\gamma = \frac{1}{2L}$, $\lambda = 0.025$ & \makecell{$\gamma = \frac{1}{L}$, $\lambda_0 = 4.9$,\\$l=5$, $\alpha = 0.5$} & $a = 1$ & \makecell{$a = 1$,\\ $B = 3\cdot 10^{-2}$} \\
        \hline
        \makecell{{\tt w8a}\\ $m = 1000$} & $\gamma = \frac{1}{L}$ & $\gamma = \frac{1}{L}$, $\lambda = 1.3$ & \makecell{$\gamma = \frac{1}{L}$, $\lambda_0 = 64.78$,\\$l=50$, $\alpha = 0.9$} & $a = 1$ & \makecell{$a = 1$,\\ $B = 19\cdot 10^{-2}$} \\
        \hline
    \end{tabular}
\end{table}

Next, we provide our numerical study of the distribution of $\|\nabla f_i(x^k) - \nabla f(x^k)\|_2$, where $x^k$ is the last iterate produced by {\tt SGD} in experiments presented in Section~\ref{sec:numerical_exp}, see Figure~\ref{fig:noise_distrib_last_sgd}.
\begin{figure}[h]
    \centering
    \includegraphics[width=0.32\textwidth]{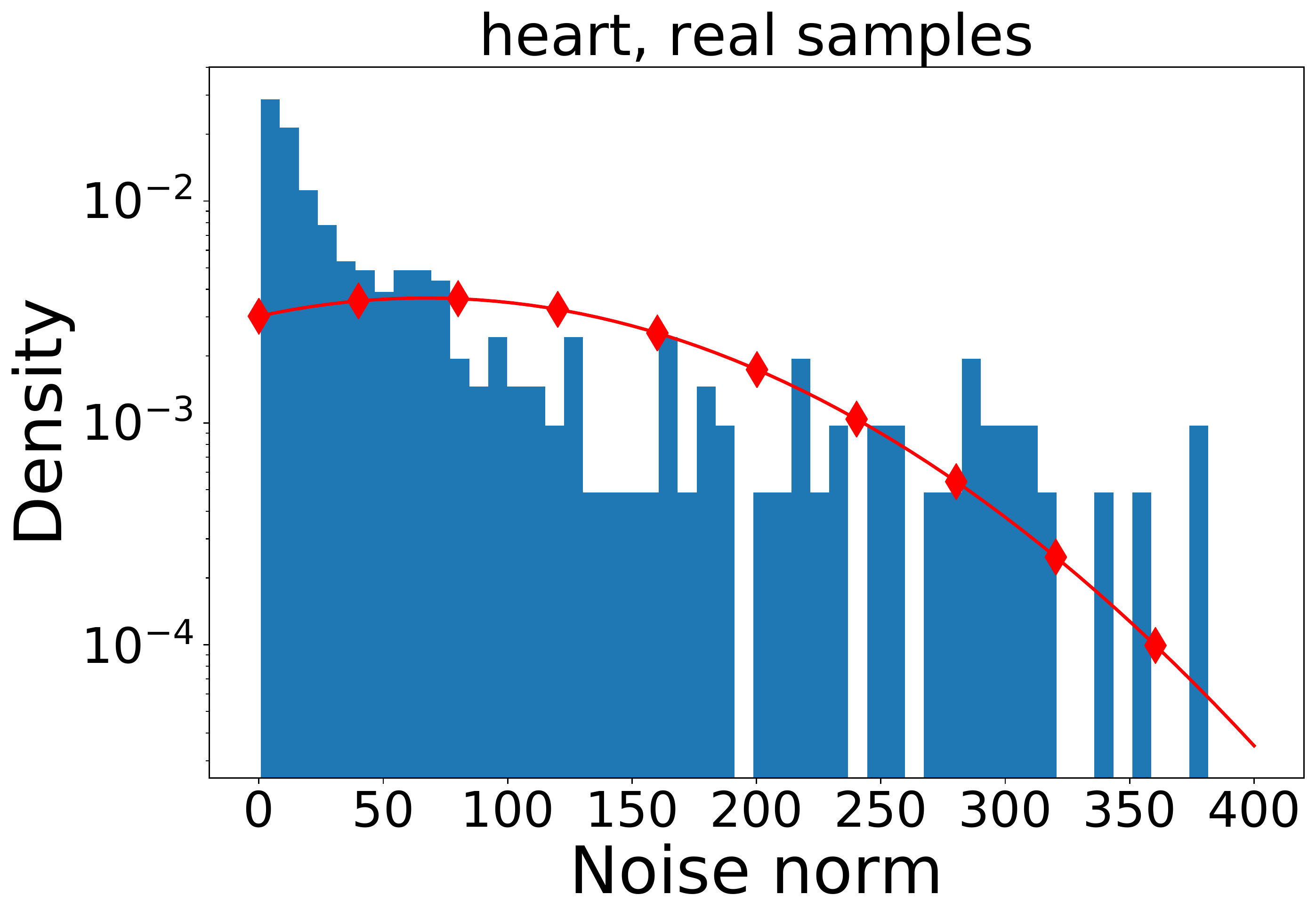}
    \includegraphics[width=0.32\textwidth]{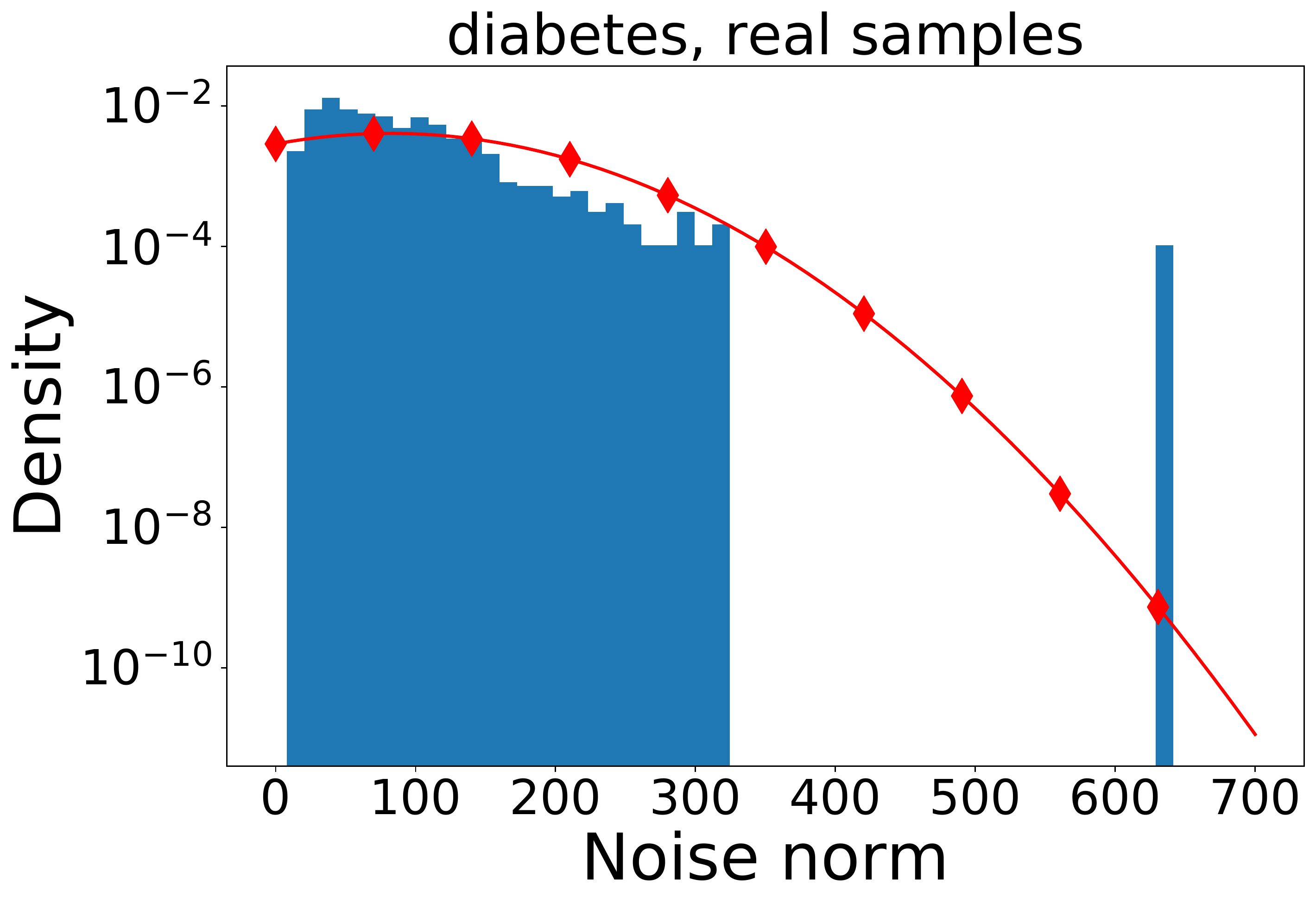}
    \includegraphics[width=0.32\textwidth]{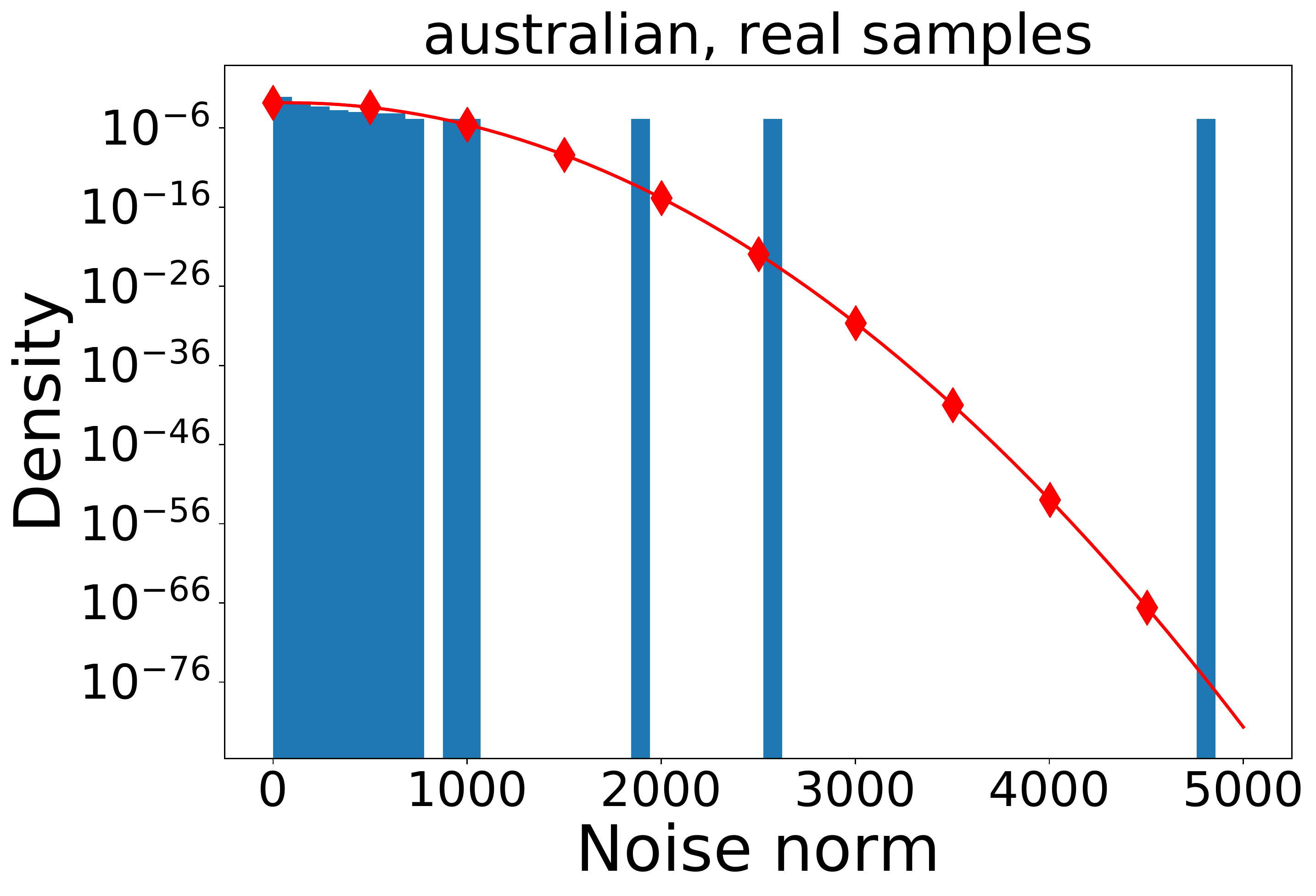}
    \includegraphics[width=0.32\textwidth]{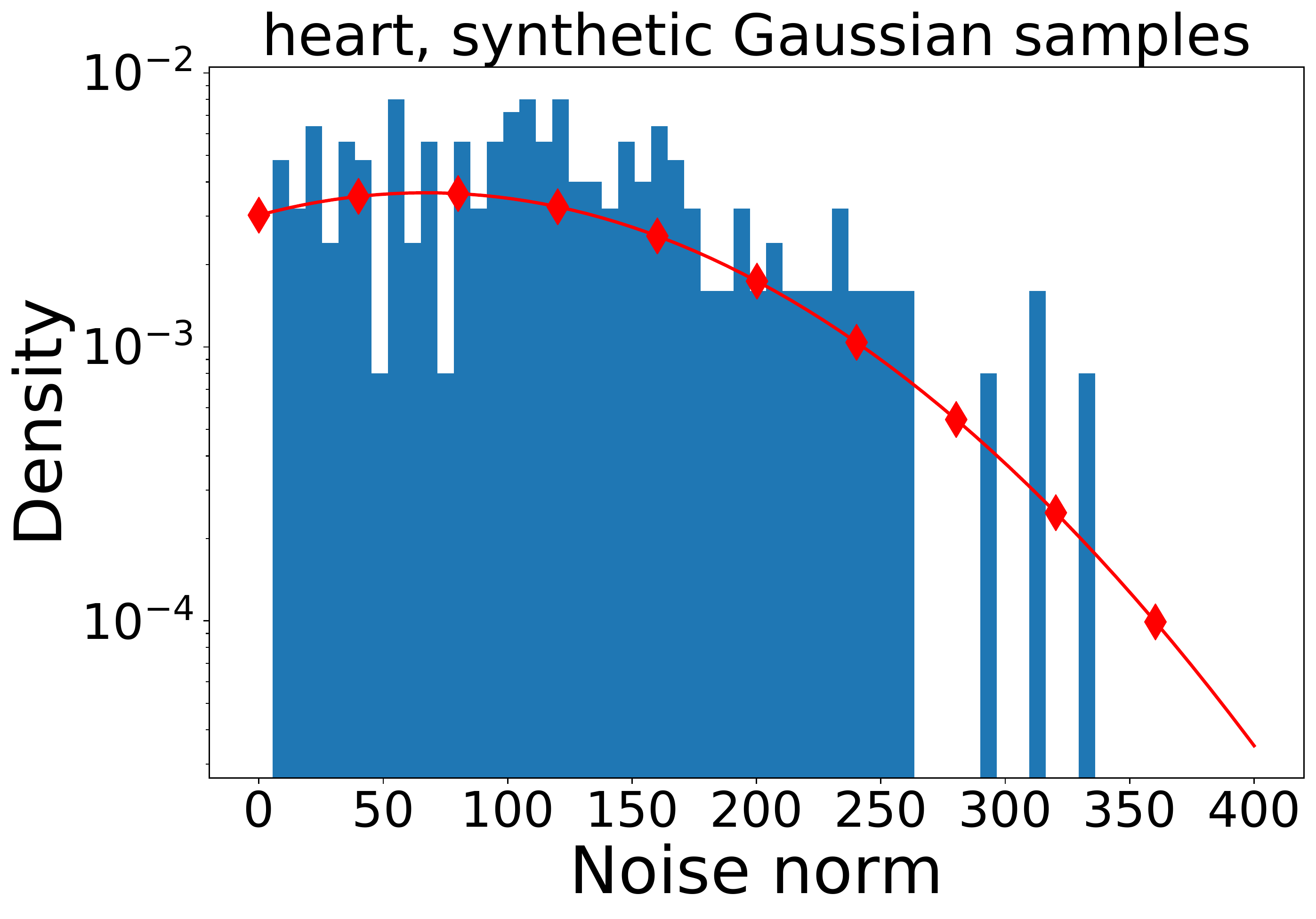}
    \includegraphics[width=0.32\textwidth]{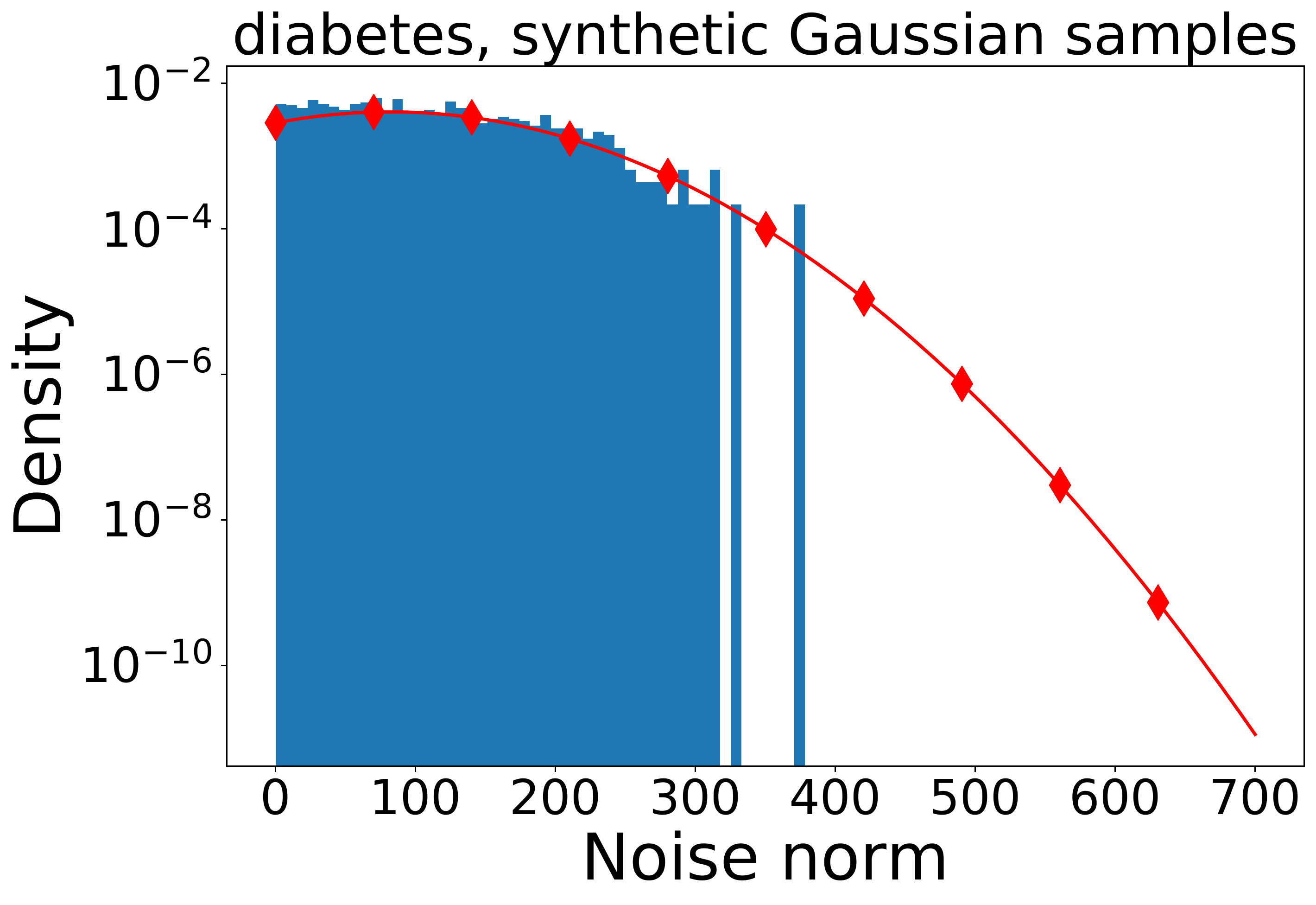}
    \includegraphics[width=0.32\textwidth]{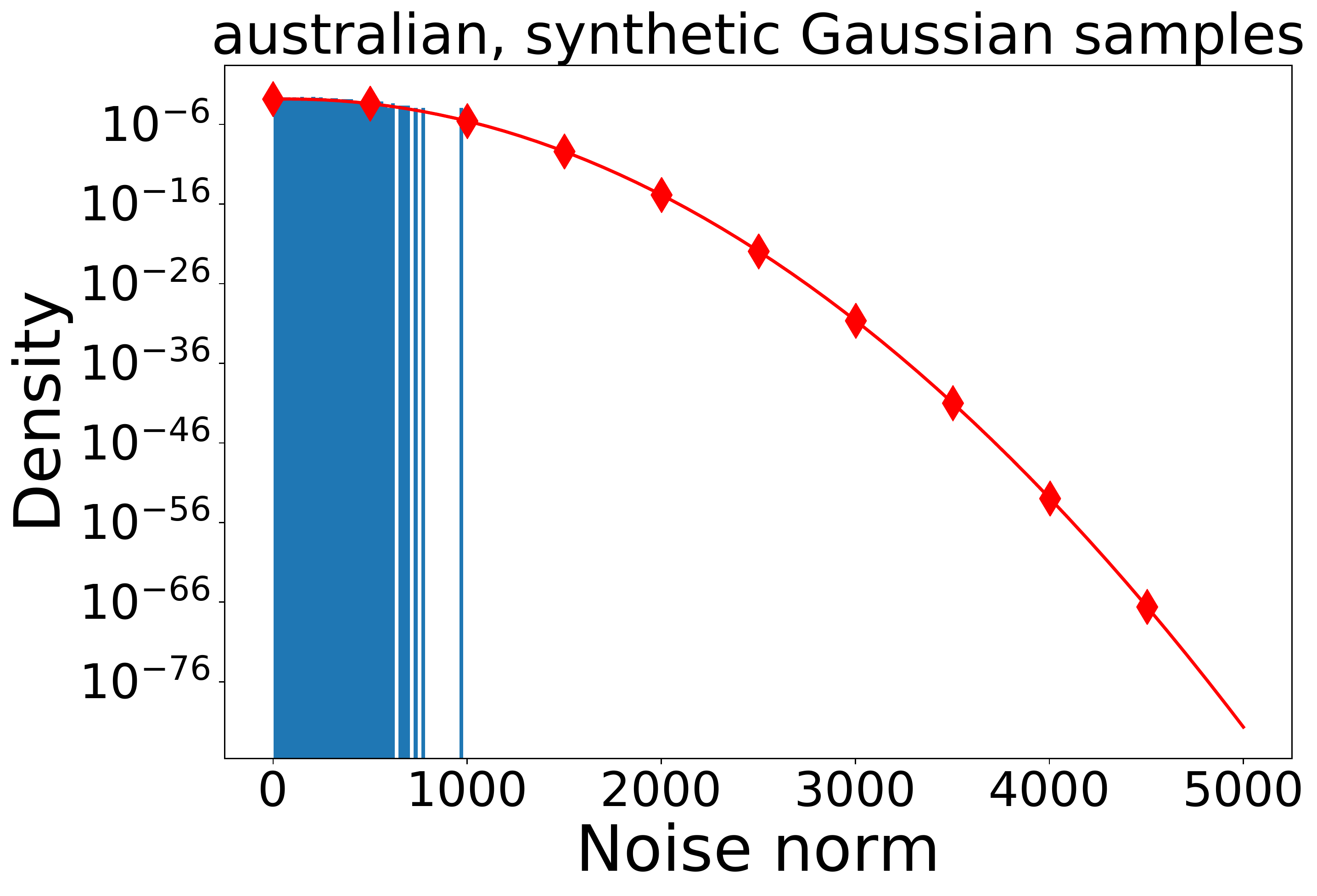}
    \caption{Histograms of $\|\nabla f_i(x^k) - \nabla f(x^k)\|_2$ for different datasets (the first row) and synthetic Gaussian samples with mean and variance estimated via empirical mean and variance of real samples $\|\nabla f_1(x^k) - \nabla f(x^k)\|_2,\ldots,\|\nabla f_r(x^k) - \nabla f(x^k)\|_2$ (the second row) where $x^k$ is the last point produced by {\tt SGD}. Red lines correspond to probability density functions of normal distributions with empirically estimated means and variances.}
    \label{fig:noise_distrib_last_sgd}
\end{figure}
As we mentioned in the main part of the paper these histograms are very similar to ones presented in Figure~\ref{fig:noise_distrib}, so, the insights that we got from Figure~\ref{fig:noise_distrib} are right. However, in our experiments with {\tt australian} dataset {\tt SGD} with the stepsize $\gamma = \nicefrac{1}{L}$ did not reach needed suboptimality in order to oscillate.

Therefore, we run {\tt SGD} along with its clipped variants with the same batchsize $m=50$ for bigger number of epochs and also tuned their parameters. One can find the results of these runs in Figure~\ref{fig:australian_logreg_bigger_steps}.
\begin{figure}[h]
    \centering
    \includegraphics[width=0.8\textwidth]{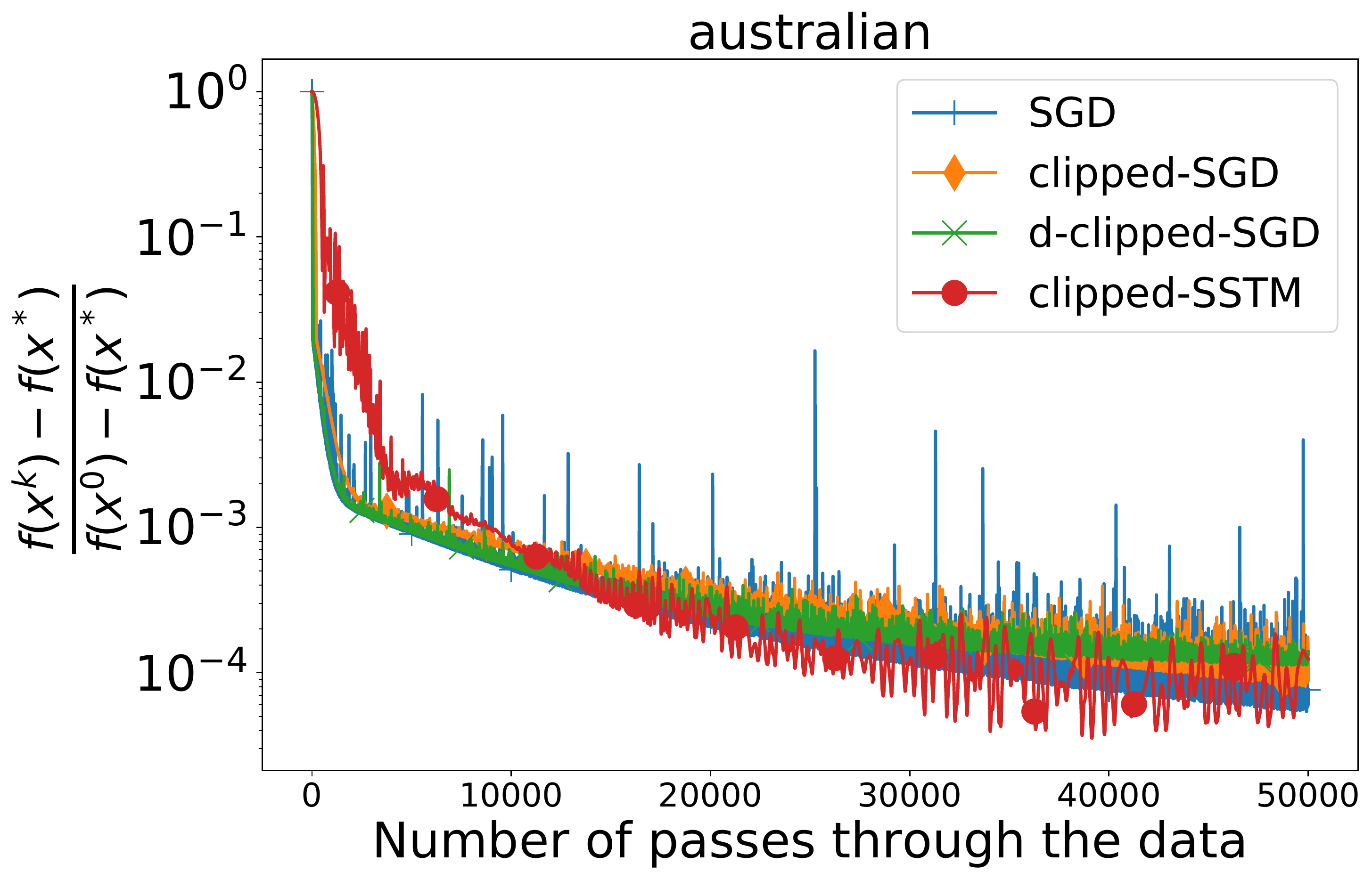}
    \caption{Trajectories of {\tt SGD}, {\tt clipped-SGD}, {\tt d-clipped-SGD} and {\tt clipped-SSTM} applied to solve logistic regression problem on {\tt australian} dataset. For {\tt SGD} and its clipped variants stepsize $\gamma = \frac{20}{L}$ was used. For {\tt clipped-SGD} we used $\lambda = 18.62$ and for {\tt d-clipped-SGD} the parameters are as follows: $\lambda_0 = 74.47$, $l = 1500$, $\alpha = 0.9$. Parameters for {\tt clipped-SSTM} are the same as in the corresponding cell in Table~\ref{tab:logreg_params}.}
    \label{fig:australian_logreg_bigger_steps}
\end{figure}
We see that {\tt SGD} with this stepsize achieves better suboptimality but it also oscillates significantly more. In contrast, {\tt clipped-SGD} and {\tt d-clipped-SGD} do not have significant oscillations and converge with the same rate as {\tt SGD}. Moreover, {\tt clipped-SSTM} shows slightly better performance in this case. Finally, we numerically studied the distribution of $\|\nabla f_i(x^k) - \nabla f(x^k)\|_2$, where $x^k$ is the last iterate produced by {\tt SGD}, see Figure~\ref{fig:noise_distrib_last_sgd_big_step}.
\begin{figure}[h]
    \centering
    \includegraphics[width=0.49\textwidth]{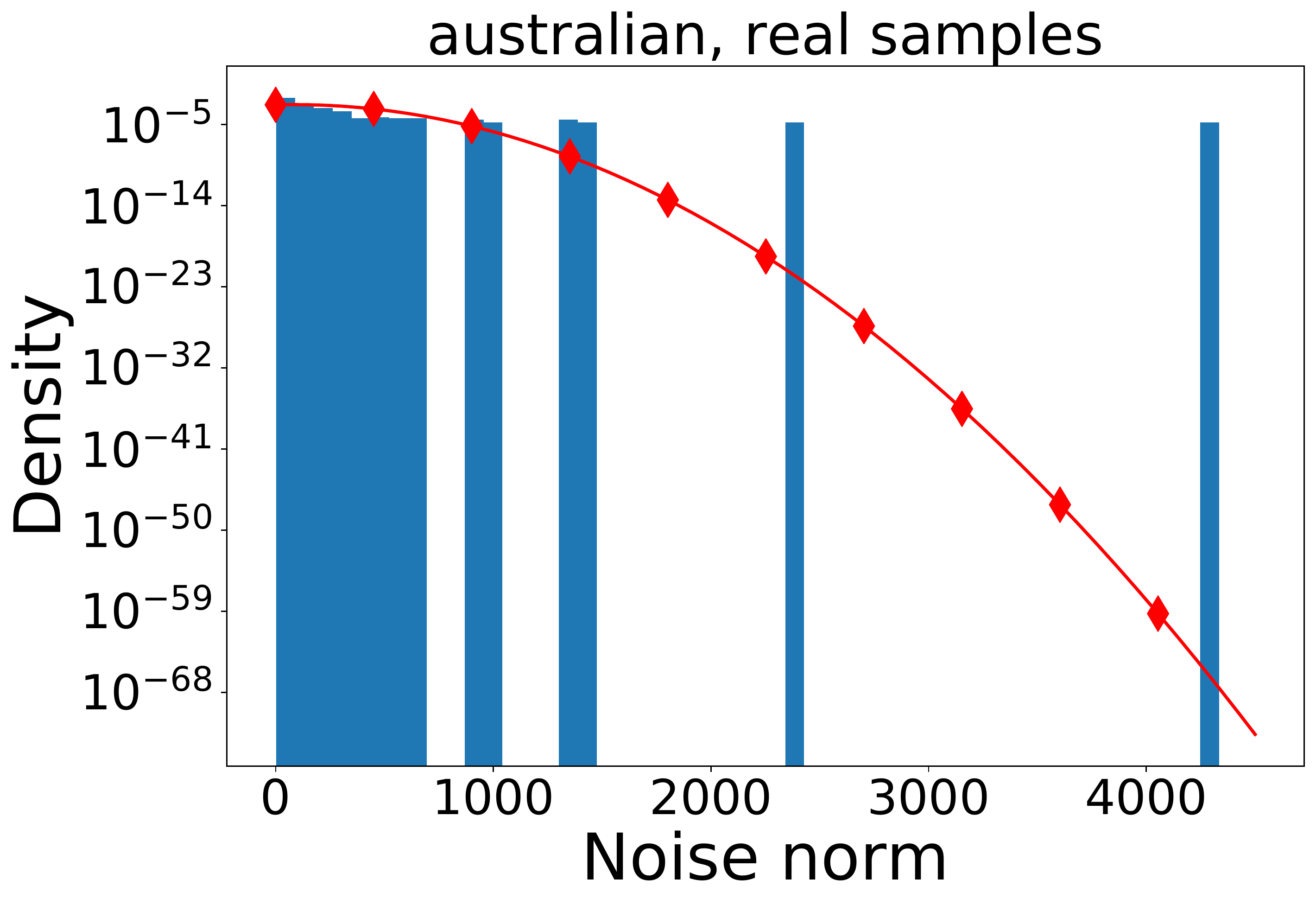}
    \includegraphics[width=0.49\textwidth]{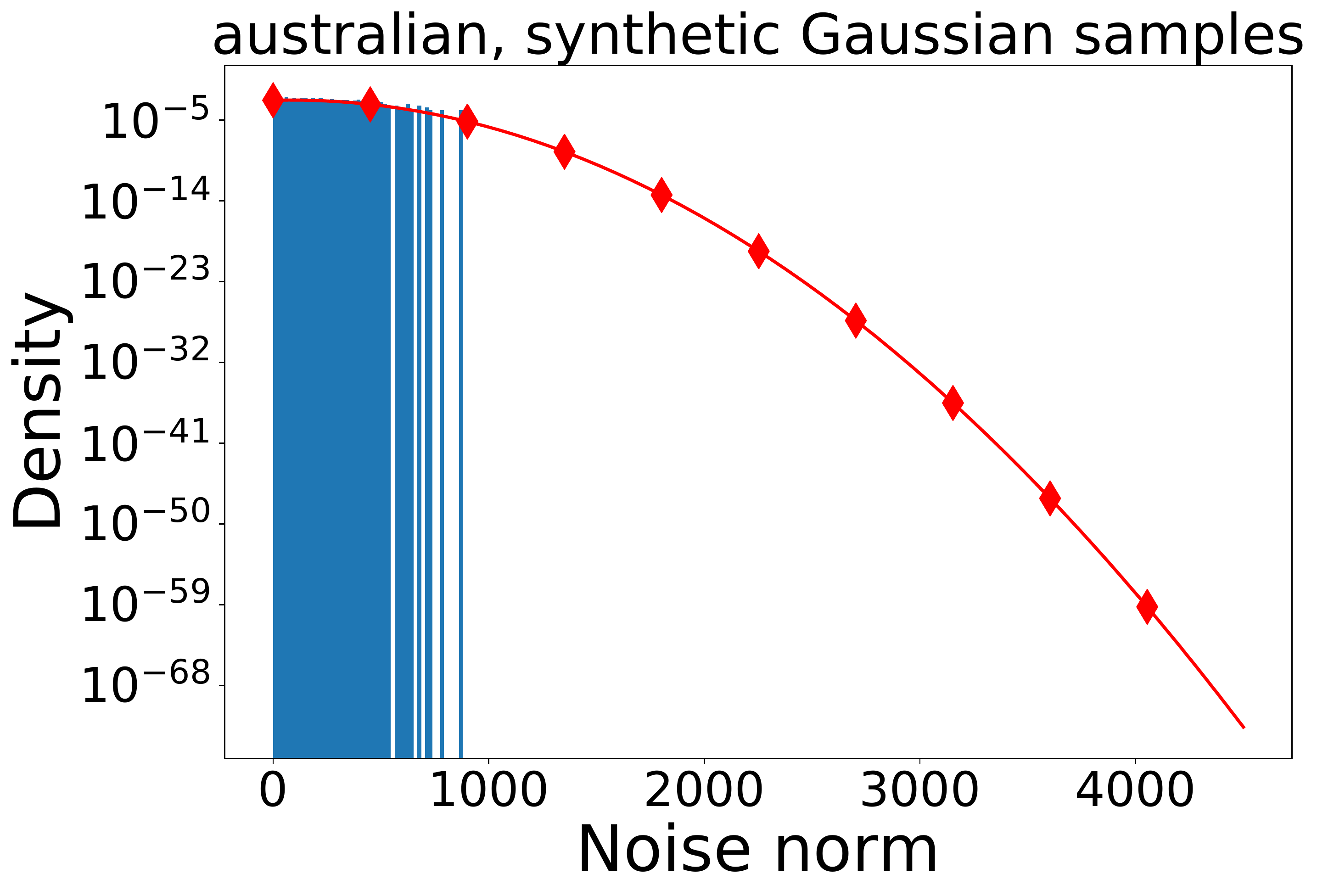}
    \caption{Histograms of $\|\nabla f_i(x^k) - \nabla f(x^k)\|_2$ for {\tt australian} dataset and synthetic Gaussian samples with mean and variance estimated via empirical mean and variance of real samples $\|\nabla f_1(x^k) - \nabla f(x^k)\|_2,\ldots,\|\nabla f_r(x^k) - \nabla f(x^k)\|_2$ where $x^k$ is the last point produced by {\tt SGD} with $\gamma = \frac{20}{L}$. Red lines correspond to probability density functions of normal distributions with empirically estimated means and variances.}
    \label{fig:noise_distrib_last_sgd_big_step}
\end{figure}
These histograms imply that the noise in stochastic gradients is heavy-tailed and explain an unstable behavior of {\tt SGD} in this case.

Finally, we conducted experiments on larger datasets: {\tt a9a} and {\tt w8a}. The results of our numerical test are reported on Figures~\ref{fig:a9a_w8a_logreg}~and~\ref{fig:noise_distrib_a9a_w8a}. We notice that {\tt SSTM} with given stepsize and batchsize suffers from noise accumulation, while {\tt clipped-SSTM} does not have this drawback and shows comparable performance with {\tt SGD} on {\tt a9a} and much better performance on {\tt w8a}.  

Figure~\ref{fig:noise_distrib_a9a_w8a} shows the gradient's noise distributions for both datasets. While the distribution of stochastic gradients at the optimum for {\tt a9a} have sub-Gaussian-like distribution, for {\tt w8a} they have heavy-tailed distribution.

\begin{figure}[h]
    \centering
    \includegraphics[width=0.45\textwidth]{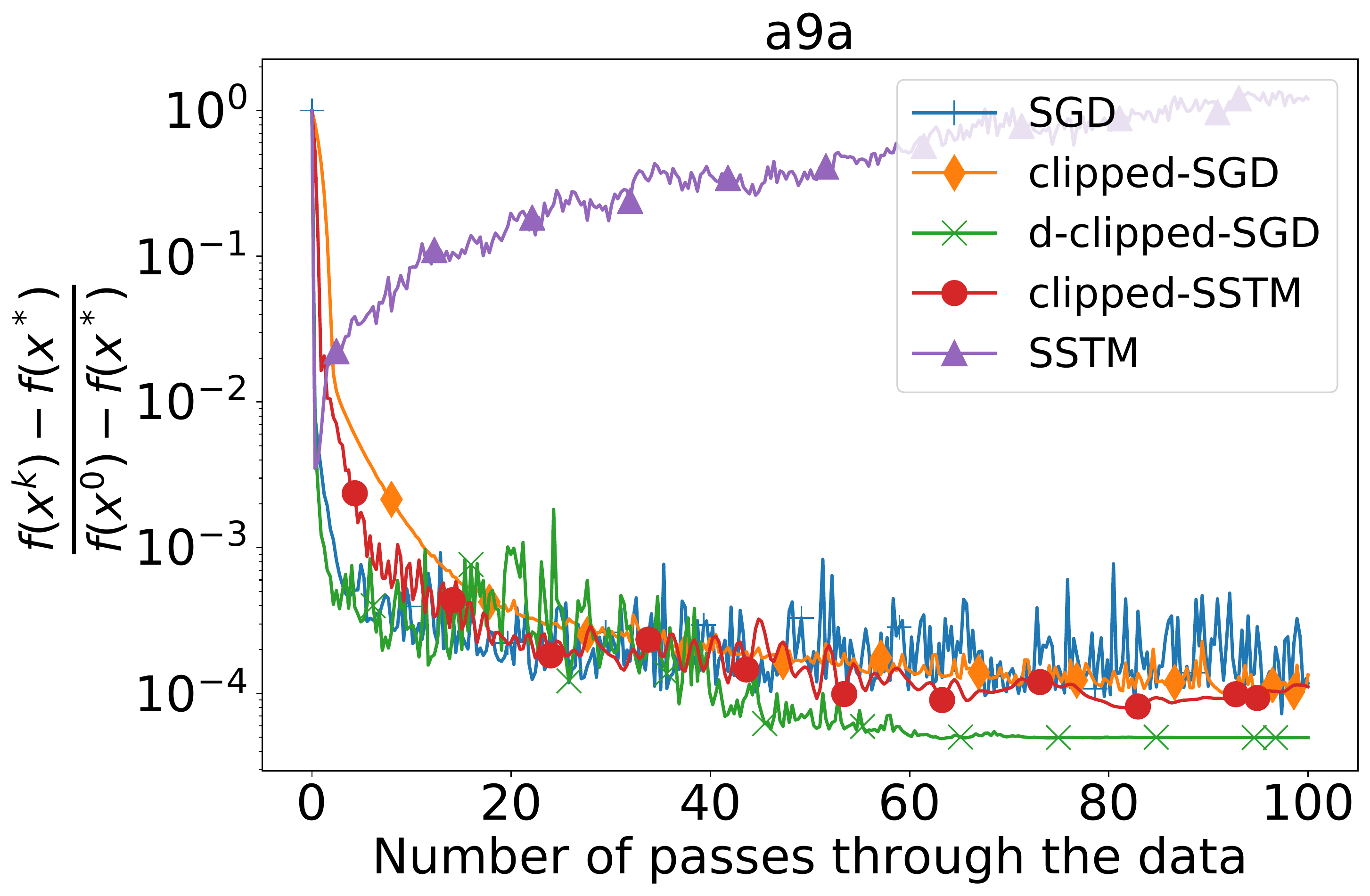}
    \includegraphics[width=0.45\textwidth]{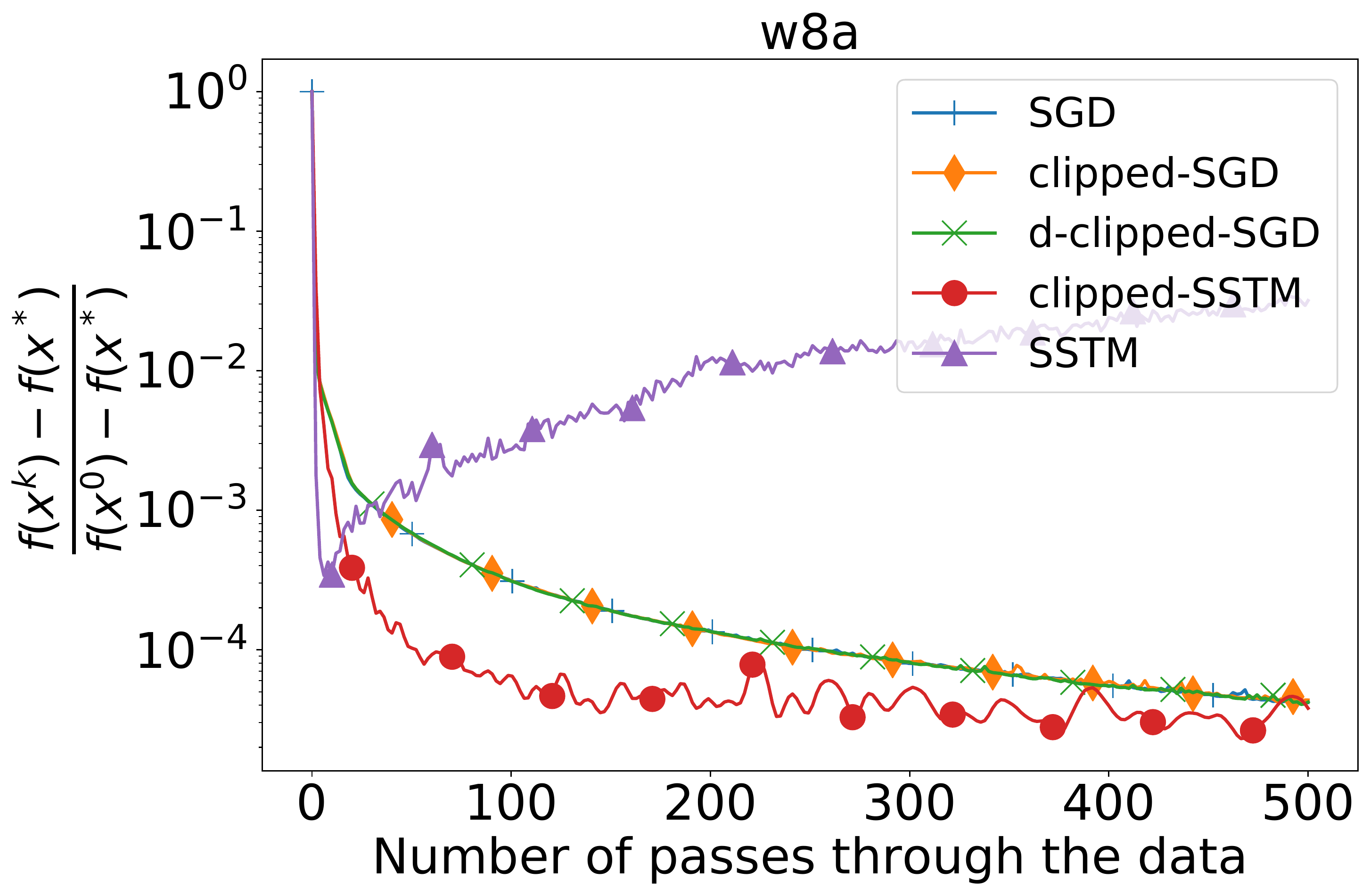}
    \caption{Trajectories of {\tt SGD}, {\tt clipped-SGD}, {\tt d-clipped-SGD} and {\tt clipped-SSTM} applied to solve logistic regression problem on {\tt a9a} and {\tt w8a} datasets. Parameters of the methods used in experiments are presneted in Table~\ref{tab:logreg_params}.}
    \label{fig:a9a_w8a_logreg}
\end{figure}

\begin{figure}[h]
    \centering
    \includegraphics[width=0.49\textwidth]{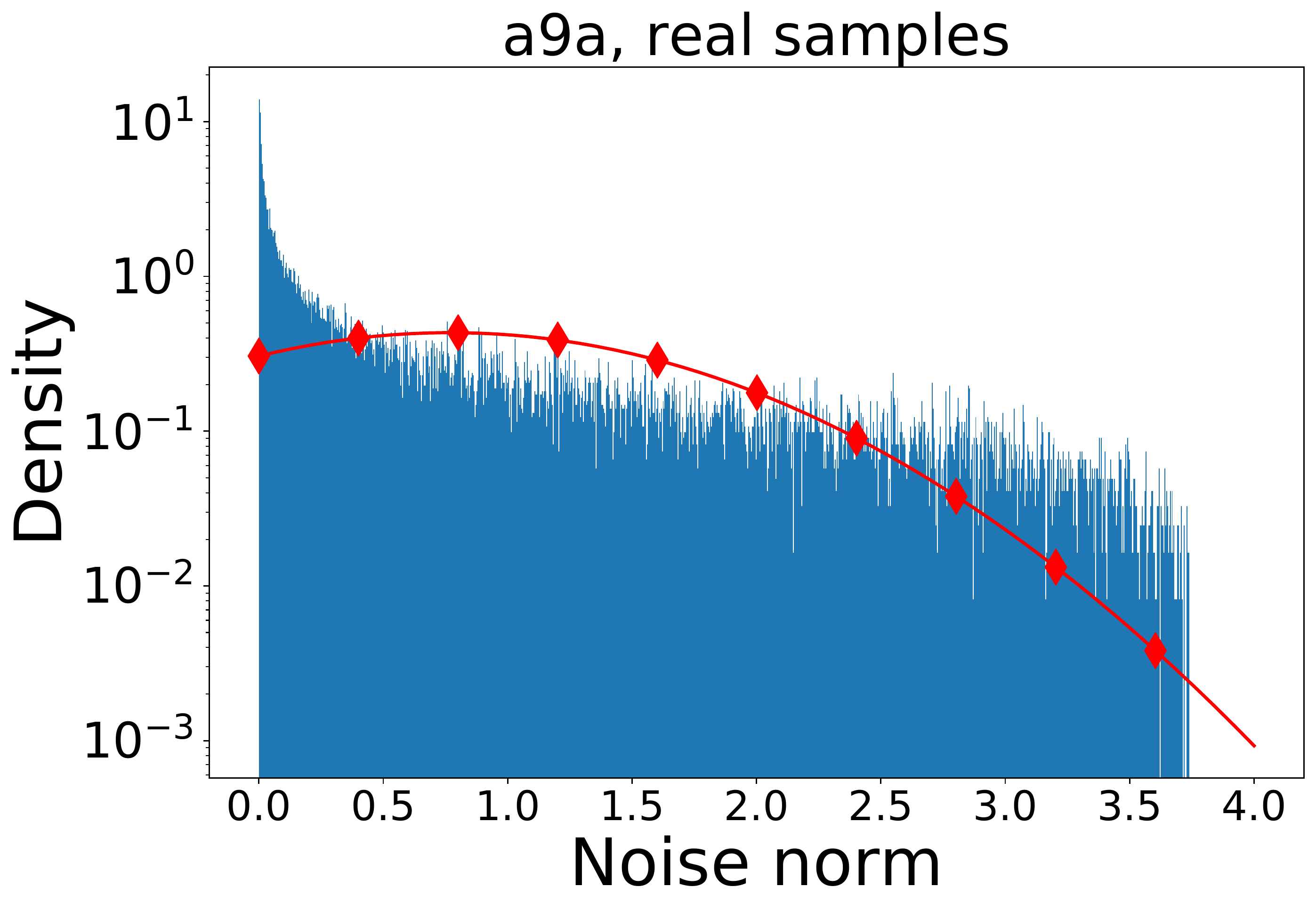}
    \includegraphics[width=0.49\textwidth]{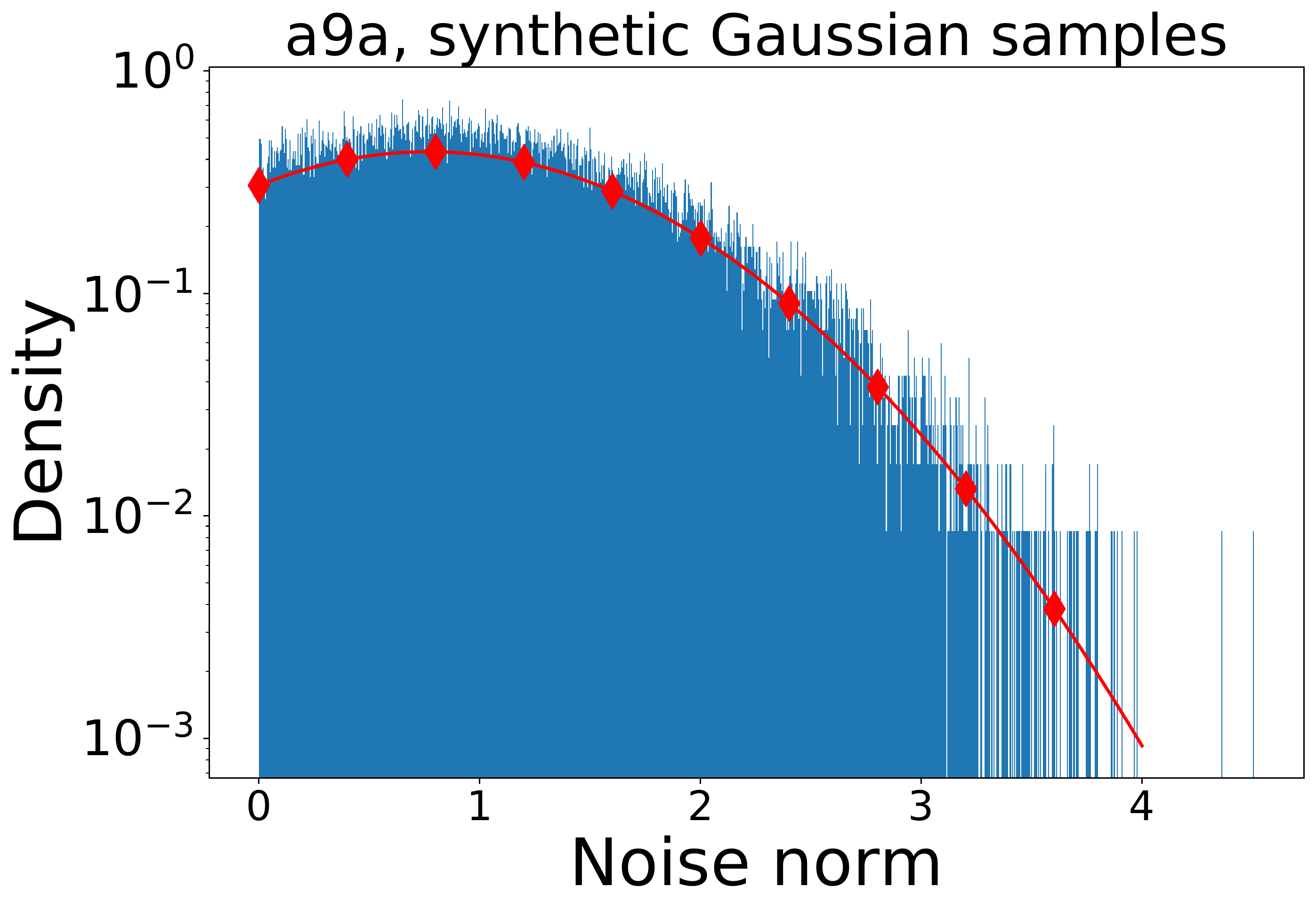}
    \includegraphics[width=0.49\textwidth]{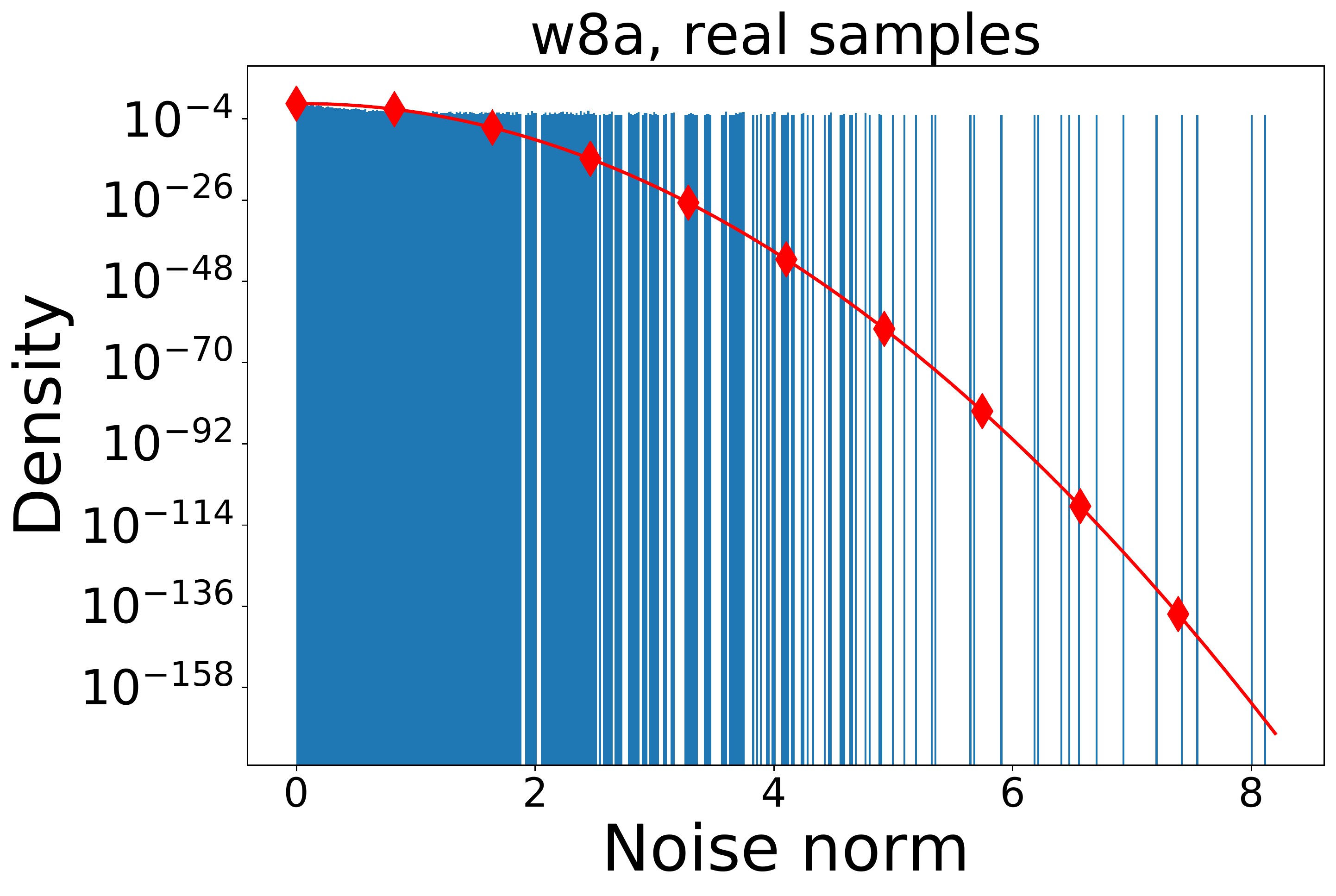}
    \includegraphics[width=0.49\textwidth]{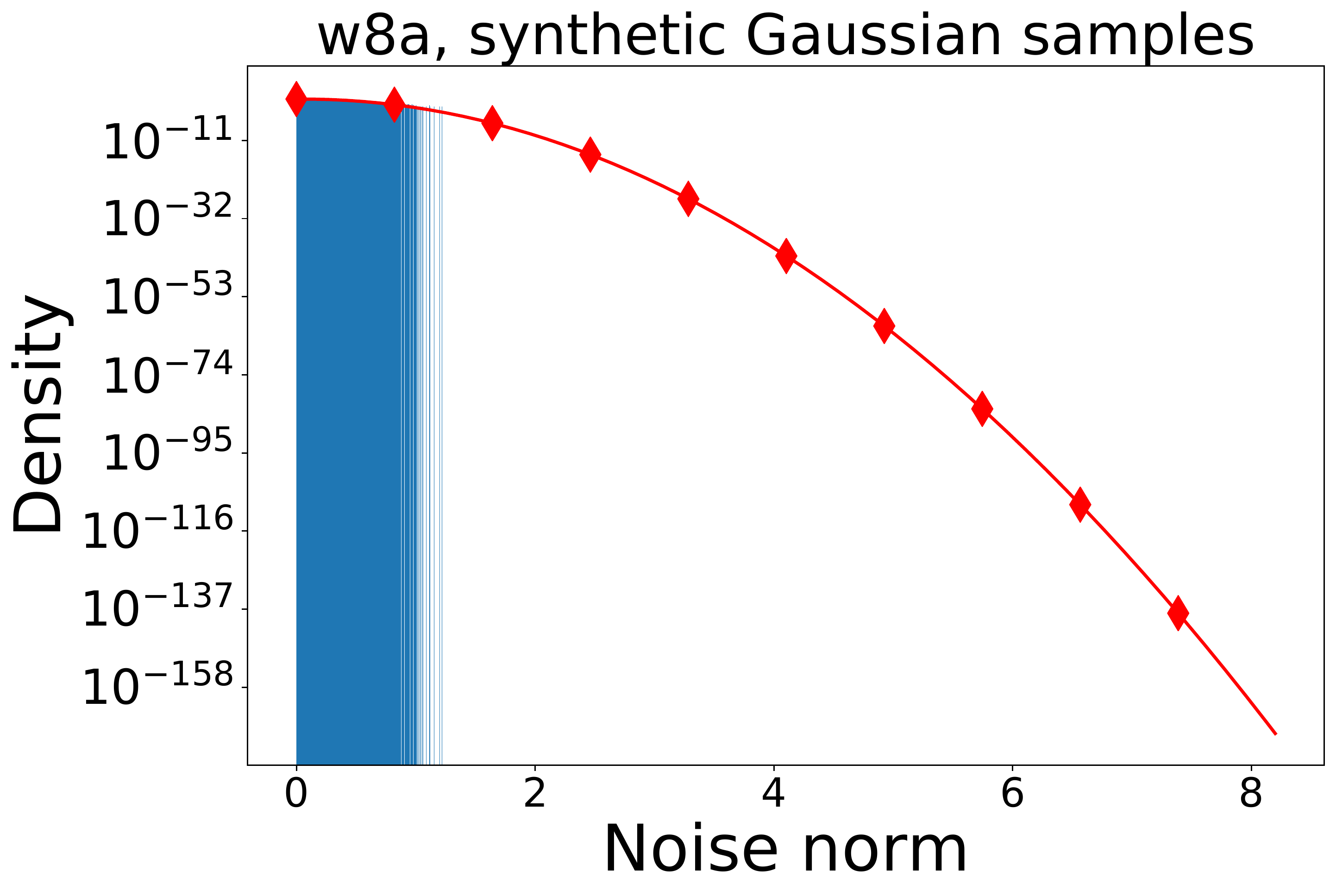}
    \caption{Histograms of $\|\nabla f_i(x^*)\|_2$ for {\tt a9a} and {\tt w8a} dataset and synthetic Gaussian samples with mean and variance estimated via empirical mean and variance of real samples $\|\nabla f_1(x^*)\|_2,\ldots,\|\nabla f_r(x^*)\|_2$. Red lines correspond to probability density functions of normal distributions with empirically estimated means and variances.}
    \label{fig:noise_distrib_a9a_w8a}
\end{figure}

\end{document}